\documentclass[11pt]{article}
\usepackage{latexsym}
\usepackage{index} 
\makeindex
\usepackage{pdfsync}
\usepackage{amsmath,amsthm,amsfonts,amssymb}
\usepackage[T1]{fontenc}
\usepackage[utf8]{inputenc}
\usepackage{url}
\usepackage{enumerate}
\usepackage{graphicx}
\usepackage[a4paper,textwidth=15cm,textheight=25cm]{geometry}
\usepackage{export}
\usepackage{hyperref}
\usepackage[depth=0]{bookmark}

\newtheorem{thm}{Theorem}[section]
\newtheorem{thmbis}{Theorem}
\newtheorem*{thm*}{Theorem}
\newtheorem{dfn}[thm]{Definition} 
\newtheorem*{dfn*}{Definition}

\newtheorem{cor}[thm]{Corollary}
\newtheorem*{cor*}{Corollary}

\newtheorem{prop}[thm]{Proposition} 
\newtheorem*{prop*}{Proposition} 
\newtheorem*{properties*}{Properties} 
\newtheorem{propbis}[thmbis]{Proposition} 
\newtheorem{lem}[thm]{Lemma} 
\newtheorem*{lem*}{Lemma} 
 
\newtheorem{claim}[thm]{Claim} 
\newtheorem*{claim*}{Claim} 
 
\newtheorem*{fact*}{Fact}

\newtheorem*{qst*}{Question}

\newtheorem*{pb*}{Problem}

\theoremstyle{remark}
 \newtheorem*{const*}{Construction}

\newtheorem*{rem*}{Remark}
\newtheorem{rem}[thm]{Remark}
\newtheorem*{example*}{Example}

\newtheorem{example}[thm]{Example}

\newlength{\espaceavantspecialthm}
\newlength{\espaceapresspecialthm}
{\normalfont \vskip \espaceapresspecialthm}

\newenvironment{specialthm*}[1]{
\vskip\espaceavantspecialthm \noindent \textbf{#1} \itshape}%
{\normalfont \vskip \espaceapresspecialthm}

\newlength{\espaceavantenonce}
\newlength{\espaceapresenonce}
\newcommand{\fontetitreun}[1]{\textbf{#1}} 
\newcommand{\fontetitredeux}[1]{\textit{#1}} 

{\normalfont \vskip \espaceapresenonce}

\newenvironment{enonce1*}[1]{
\vskip\espaceavantenonce \noindent \fontetitreun{#1} \itshape}%
{\normalfont \vskip \espaceapresenonce}

{\vskip \espaceapresenonce}

\newenvironment{enonce2*}[1]{
\vskip\espaceavantenonce \noindent \fontetitredeux{#1} }%
{\vskip \espaceapresenonce}

\makeatletter
\edef\@tempa#1#2{\def#1{\mathaccent\string"\noexpand\accentclass@#2 }}
\@tempa\rond{017}
\makeatother

\newcommand{\es}{\emptyset}
\renewcommand{\phi}{\varphi} 
\newcommand{\m} {^{-1}} 
\newcommand{\eps} {\varepsilon}

\newcommand {\ra} {\rightarrow}

\newcommand {\onto} {\twoheadrightarrow}
\newcommand {\xra} {\xrightarrow}

\newcommand{\actson}{\curvearrowright}

\newcommand{\ol}[1]{\overline{#1}}

\newcommand{\normal} {\vartriangleleft}
\renewcommand{\subsetneq}{\varsubsetneq}

\newcommand{\dunion}{\sqcup}

\newcommand{\Dunion}{\bigsqcup} 

\newcommand{\ie} {i.e.\ }

\newcommand {\cala} {{\mathcal {A}}}   
\newcommand {\calb} {{\mathcal {B}}}   
\newcommand {\calc} {{\mathcal {C}}}   
\newcommand {\cald} {{\mathcal {D}}}   
\newcommand {\cale} {{\mathcal {E}}}   
\newcommand {\calf} {{\mathcal {F}}}   
   
\newcommand {\calh} {{\mathcal {H}}}

\newcommand {\calk} {{\mathcal {K}}}   
\newcommand {\call} {{\mathcal {L}}}   
\newcommand {\calm} {{\mathcal {M}}}

\newcommand {\calp} {{\mathcal {P}}}   
   
\newcommand {\calr} {{\mathcal {R}}}   
\newcommand {\cals} {{\mathcal {S}}}   
\newcommand {\calt} {{\mathcal {T}}}   
\newcommand {\calu} {{\mathcal {U}}}   
\newcommand {\calv} {{\mathcal {V}}}

\newcommand {\calz} {{\mathcal {Z}}}

\newcommand {\bbH} {{\mathbb {H}}}

\newcommand {\bbR} {{\mathbb {R}}}

\newcommand {\bbZ} {{\mathbb {Z}}}   

\newcommand{\grp}[1]{\langle #1 \rangle}

\newcommand{\Isom} {\mathop{\mathrm{Isom}}}

\newcommand{\Fix}{\mathop{\mathrm{Fix}}}

\newcommand{\Out} {\mathop{\mathrm{Out}}}
\newcommand{\Aut} {\mathop{\mathrm{Aut}}}

\newcommand{\Zc}{{\mathcal{Z}}}
\newcommand{\Zmax}{{\mathcal{Z}_{\mathrm{max}}}}

\newcommand{\Rt}{$\R$-tree}
\newcommand {\F} {{\mathbb {F}}}
\newcommand {\N} {{\mathbb {N}}} 
\newcommand {\Z} {{\mathbb {Z}}}
\newcommand {\R} {{\mathbb {R}}}
\newcommand {\Q} {{\mathbb {Q}}}
\newcommand{\tco}{T_{co}}
\newcommand{\inc}{\subset}
\newcommand{\bo}{\partial}
\newcommand{\VPC} {\mathrm{VPC}}

\newcommand{\AH}{\ensuremath{(\cala,\calh)}}

\newcommand{\Dco}{\cald_{\mathrm{co}}}
\newcommand{\Tco}{\tco}
\newcommand{\Om}{\Omega}

\newcommand{\wrt}{with respect to\ }

\newcommand\elli{{\mathrm{ell}}}

\newcommand\smally{smally\ }

\newcommand{\Comm} {\mathop{\mathrm{Comm}}}

\newcommand{\Inc}{\mathrm{Inc}}

\newcommand{\Inch}{\Inc^\calh}

\renewcommand {\cale} {{\mathcal {A}}_\infty}  
\newcommand {\cla} {{[A]}}  
\newcommand{\peripheral}{incidence}
\newcommand{\Peripheral}{Incidence}

\renewcommand{\labelenumi}{(\theenumi)}

\begin{document}

\title{JSJ decompositions of groups}
\author{Vincent Guirardel, Gilbert Levitt}
\date{}

\maketitle

\begin{abstract}

%
%
%
%
%
%
%

This is an account of the  theory of JSJ decompositions of finitely generated groups, as developed in the last twenty years or so.

We give a simple general definition of JSJ decompositions (or rather of their Bass-Serre trees),
as maximal universally elliptic trees.
  In general, there is no preferred JSJ decomposition, and
the right object to consider is the whole set of JSJ decompositions,
 which forms a contractible space: the JSJ deformation space (analogous to Outer Space).

We prove that   JSJ decompositions exist
for any finitely presented group, without any   assumption  on edge groups.  When edge groups are slender, we describe flexible vertices of JSJ decompositions as quadratically hanging extensions of 2-orbifold groups.

Similar results hold in the presence of acylindricity, in particular for splittings of torsion-free CSA groups over abelian groups, and splittings of relatively
hyperbolic groups over virtually cyclic or parabolic subgroups. Using trees of cylinders, we obtain   canonical
JSJ trees (which are invariant under automorphisms).

  We introduce a variant in which the property of being universally elliptic is replaced by the more restrictive and rigid property
of being universally compatible.
This yields a canonical compatibility JSJ tree, not just
  a deformation space.
We show that it exists for any finitely presented group.

We give many examples, and we work throughout with relative decompositions (restricting to trees where certain subgroups are elliptic).
\end{abstract}

\section*{Introduction }

JSJ decompositions first appeared in 3-dimensional topology with the theory of the characteristic submanifold
by {\bf J}aco-{\bf S}halen and {\bf J}ohannson \cite{JaSh_JSJ,Johannson_JSJ} (the terminology JSJ was popularized by Sela). We start with a quick review (restricting to manifolds without boundary). 

\subsection*{From 3-manifolds to groups}
Let $M$ be a closed orientable 3-manifold. Given a  finite collection of disjoint embedded 2-spheres, one may cut $M$ open along them, and glue balls to the boundary of the pieces to make them boundaryless. This expresses $M$ as a connected sum of closed manifolds $M_i$. 
The prime decomposition theorem (Kneser-Milnor) asserts that one may choose the spheres so that each $M_i$ is 
either irreducible ($M_i\ne S^3$, and every embedded 2-sphere bounds a ball)
or homeomorphic to $S^2\times S^1$; moreover, up to a permutation, the summands $M_i$ are uniquely determined up to homeomorphism. 

On the group level, one obtains a  decomposition of $G=\pi_1(M)$ as a free product $G=G_1*\dots*G_p*\F_q$, with $G_i$ the fundamental group of an irreducible $M_i$ and  $\F_q$ a free group of rank $q$ coming from the $S^2\times S^1$ summands.  This 
  decomposition is
  a \emph{Grushko decomposition}\index{Grushko decomposition, deformation space}  of $G$, in the following sense: each $G_i$ is  freely indecomposable (it cannot be written as a non-trivial free product), non-trivial, and not isomorphic to $\Z$ (non-triviality of $G_i$ is guaranteed by the Poincar\'e conjecture, proved by Perelman).   Any finitely generated group has a Grushko decomposition, with $q$ well-defined and the $G_i$'s well-defined up to conjugacy (and a permutation).

The prime decomposition implies that one should focus on irreducible manifolds. Since all spheres bound balls, one  now considers embedded tori. In order to avoid trivialities (such as a torus bounding a tubular neighborhood   of a curve), tori should be incompressible: the embedding $T^2\to M$ induces an injection on fundamental groups. 

The theory of the characteristic submanifold now says that, given an irreducible $M$,  there exists 
a finite family $\calt $ of disjoint non-parallel incompressible tori such that each component $N_j$ of the manifold (with boundary) obtained by cutting $M$ along $\calt$  is  either   atoroidal (every incompressible torus is boundary parallel) or a Seifert fibered space,   
\ie a $3$-manifold having a singular fibration  by circles over a 2-dimensional surface $\Sigma$ (better viewed as a 2-dimensional orbifold).\footnote{Thurston's geometrization conjecture, whose proof was completed by Perelman, asserts that each $N_j$ has a geometric structure. In particular, every atoroidal $N_j$ with infinite fundamental group is hyperbolic (its interior  admits a complete metric with finite volume and constant curvature $-1$).}  Moreover, any incompressible torus may be isotoped to be disjoint from $\calt$. 

With groups in mind, let us point out an important feature of this decomposition. If two incompressible tori cannot be made disjoint by an isotopy, they may be isotoped to be contained in a Seifert piece. 
Conversely,  in a Seifert fibered space, 
preimages  of intersecting simple curves on $\Sigma$ are intersecting tori.

We thus see that the presence of  intersecting tori in $M$ forces some surface $\Sigma$ to appear. One of the remarkable facts about JSJ theory for groups is that a similar phenomenon occurs. For instance, if a  finitely generated one-ended 
    group   admits two   
splittings over $\Z$ that ``intersect'' each other in an essential way,\footnote{  More precisely: the edge  groups of each splitting should be hyperbolic in the Bass-Serre tree of the other splitting.} 
 then it must contain the fundamental group of a compact surface,  
 attached to the rest of the group along the boundary 
(see Theorem \ref{slenderintro} below and Section \ref{Fujpap}, in particular Proposition \ref{prop_RN}).

We also point out that the family $\calt$ mentioned above is unique up to isotopy. On the other hand, a family of spheres defining the prime decomposition is usually not unique. 
Similarly, the Grushko decompositions of a group usually form a large \emph{outer space}\index{outer space} \cite{CuVo_moduli,GL1}, whereas one may often construct canonical splittings of one-ended groups, which are in particular invariant under automorphisms   (see  Theorems \ref{dur1} and \ref{dur2}  
 below).

These topological ideas were carried over  to group theory by Kropholler \cite{Kro_JSJ}
for some Poincar\'e duality groups of dimension at least $3$, and by Sela for torsion-free hyperbolic groups \cite{Sela_structure}.
Constructions of JSJ decompositions were  given in more general settings by many authors
\index{Rips-Sela}\index{Dunwoody-Sageev}\index{Fujiwara-Papasoglu}\index{Scott-Swarup}\index{Bowditch}
(Rips-Sela \cite{RiSe_JSJ}, Bowditch \cite{Bo_cut}, Dunwoody-Sageev \cite{DuSa_JSJ}, Fujiwara-Papasoglu \cite{FuPa_JSJ},
Dunwoody-Swenson \cite{DuSw_algebraic}, Scott-Swarup \cite{ScSw_regular+errata}, Papasoglu-Swenson \cite{PaSw_boundaries}\dots). 
This 
has had a vast influence and range of applications, 
from the isomorphism problem and the structure of the group of automorphisms of hyperbolic groups, 
to diophantine geometry over groups,
  and many others.

In this group-theoretical context, one has a finitely generated group $G$ and a class of subgroups $\cala$ (such as cyclic groups, abelian groups,  ...), 
and one tries to understand splittings   (i.e.\   graph of groups decompositions)  of $G$ over groups in $\cala$.
The family of tori $\calt$ of the $3$-manifold is replaced by a splitting  of $G$ over groups in $\cala$.
The   authors construct   a
 splitting 
 enjoying a long list of properties, rather  specific to each case.

Our first goal is to give   a simple general \emph{definition}  of JSJ decompositions
 stated by means of a  universal maximality property,  together with general 
  \emph{existence} and \emph{uniqueness} statements   in terms of \emph{deformation spaces} (see below).
  The JSJ decompositions constructed in \cite{RiSe_JSJ,Bo_cut,DuSa_JSJ,FuPa_JSJ} are JSJ decompositions in our sense (see Subsection \ref{autres}).

The regular neighbourhood\index{regular neighbourhood} of \cite{ScSw_regular+errata} 
is of a different nature.
In  \cite{DuSw_algebraic,ScSw_regular+errata}, one looks at almost invariant sets rather than splittings, 
in closer analogy to
a $3$-manifold situation  where one
 wants to understand all \emph{immersed} tori, not just the embedded ones.
One obtains a canonical 
splitting of $G$ rather than just a canonical deformation space.
See Parts \ref{partacyl} and \ref{chap_compat} for canonical splittings, and \cite{GL5} for the   relation between \cite{ScSw_regular+errata} and   usual JSJ decompositions.

\subsection*{Definition of JSJ decompositions}

To motivate the definition, let us  first consider free decompositions\index{free splitting} of a group $G$,
\ie decompositions of $G$ as the fundamental group of  a graph of groups with trivial edge groups, 
or equivalently   actions of $G$ on a simplicial  tree $T$ with trivial edge stabilizers.

Let $G=G_1*\dots * G_p *\F_q$ be a Grushko decomposition,\index{Grushko decomposition, deformation space}   as defined   above ($G_i$ is non-trivial, not $\Z$, freely indecomposable). One may view $G$ as the fundamental group of  one of the 
graphs of groups
 pictured on Figure \ref{fig_grushko}.   
The corresponding Bass-Serre trees
have  trivial edge stabilizers, and the vertex stabilizers
 are precisely the conjugates of the $G_i$'s; we call a tree with these properties a \emph{Grushko tree} (if $G$ is freely indecomposable, Grushko trees are points).

 \begin{figure}[htbp]
   \centering
   \includegraphics{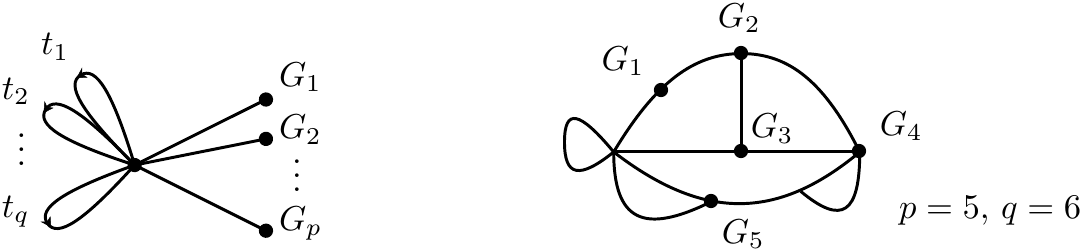}
   \caption{Graph of groups decompositions corresponding to  two Grushko trees.}
   \label{fig_grushko}
 \end{figure}

Since the $G_i$'s are freely indecomposable,  Grushko trees  $T_0$ have the following  maximality property: if $T$ is any tree on which $G$ acts with trivial edge stabilizers, $G_i$ fixes a point in $T$, and  therefore $T_0$  \emph{dominates}\index{domination} $T$ in the sense that there is a $G$-equivariant map $T_0\to T$.
 In other words, 
among free decompositions of $G$,  a Grushko  tree 
$T_0$ is as  far  as possible from the trivial  tree (a point):
its vertex stabilizers are as small as possible (they are conjugates of the $G_i$'s). This maximality property does not determine $T_0$ uniquely, as it is shared by all Grushko trees; we will come back to this key fact later, when discussing uniqueness.

When more general decompositions are allowed,   for instance when one considers  splittings over $\bbZ$, there may not exist a   tree with the same maximality property.
 The  fundamental example is the following. 
Consider an   orientable  closed surface $\Sigma$, and  two simple closed curves $c_1,c_2$ in $\Sigma$ with non-zero intersection number. 
Let $T_i$ be the Bass-Serre tree of the 
 associated 
splitting of $G=\pi_1(\Sigma)$ over $\bbZ\simeq\pi_1(c_i)$.
Since $c_1$ and $c_2$ have positive intersection number, $\pi_1(c_1)$ is hyperbolic in $T_2$ (it does  not fix a point) and vice-versa.
Using  the fact that $\pi_1(\Sigma)$ is freely indecomposable, it is an easy exercise to check that there is no splitting of $\pi_1(\Sigma)$
which dominates both $T_1$ and $T_2$. In this case 
  there is no hope of having a  
  splitting over cyclic groups  similar to $T_0$ above.

To overcome this difficulty, one restricts to  \emph{universally elliptic} splittings, defined  as follows.

We   consider   trees with an action of a finitely generated group $G$, and we require that   edge stabilizers be in $\cala$ (a given family of subgroups of $G$, closed under conjugating and taking subgroups); we call such a tree an   \emph{$\cala$-tree}.\index{0AT@$\cala $-tree} By Bass-Serre theory, this corresponds to splittings of $G$ over groups in $\cala$. Unless otherwise indicated, all trees are assumed to be $\cala$-trees.

\begin{dfn*}
 An   $\cala$-tree is \emph{universally elliptic}\index{universally elliptic subgroup, tree} if its edge stabilizers are elliptic in every $\cala$-tree.
\end{dfn*}

Recall that $H$ is \emph{elliptic}\index{elliptic element or subgroup} in $T$ if it fixes a point in $T$ (in terms of graphs of groups, $H$ is contained in a conjugate of a vertex group). Free decompositions are universally elliptic, but the trees $T_1,T_2$ introduced above are not.

\begin{dfn*}
  A \emph{JSJ decomposition} (or JSJ tree)\index{JSJ decomposition, tree, deformation space} of $G$ over $\cala$ is an $\cala$-tree $T$ such that: 
  \begin{itemize}
  \item $T$ is universally elliptic;
  \item $T$ dominates  any other universally elliptic tree $T'$. 
  \end{itemize}

We   call the quotient graph of groups $\Gamma=T/G$ a \emph {JSJ splitting}, or a JSJ decomposition.
\end{dfn*}

  Recall  that $T$ dominates $T'$ if there is an equivariant map $T\to T'$; equivalently, any group which is elliptic in $T$ is also elliptic in $T'$. 
The second condition   in the definition 
is a maximality condition
expressing that vertex stabilizers of $T$ are as small as possible (they are elliptic in every universally elliptic tree).

If $\cala$ consists of all subgroups with a given property (being cyclic, abelian, slender, ...), we refer to, say, cyclic trees, cyclic JSJ decompositions when working over $\cala$. 

When $\cala$ only contains the trivial group, JSJ trees are the same as  Grushko trees. If $G=\pi_1(\Sigma)$ as above, and $\cala$ is the family of cyclic subgroups, the JSJ decomposition is trivial (the point is the only JSJ tree).

\subsection*{Existence.}

\emph{JSJ trees do not always exist:} the   finitely generated inaccessible group $D$ 
constructed by Dunwoody \cite{Dun_inaccessible} has no JSJ tree over finite groups, and $D\times\Z$ is a one-ended group with no JSJ decomposition over virtually cyclic subgroups. On the other hand, it follows rather easily from Dunwoody's accessibility   \cite{Dun_accessibility}
that a finitely presented group has JSJ decompositions 
over any class $\cala$ of subgroups (we emphasize that   no assumption on $\cala$, such as smallness,  is needed).

\begin{thmbis}[Theorem \ref{thm_exist_mou}] \label{existmouintro}
Let $\cala$ be an arbitrary family  of subgroups of $G$, stable under taking subgroups and under conjugation.
If $G$ is finitely presented,  
 it has  a JSJ decomposition    over $\cala$. In fact, there exists  a  JSJ tree whose edge and vertex stabilizers are finitely generated.
\end{thmbis}

In Part \ref{partacyl} we shall present a different way of constructing JSJ decompositions, based on Sela's acylindrical accessibility,\index{acylindrical accessibility} which applies in some more general situations   (the existence of such JSJ decompositions   for limit groups is mentioned  
and used in \cite{Sela_diophantine1}, and we give a complete proof).
 We will give more details later in this introduction, but we mention  two typical results here. A group is   \emph{CSA} if maximal abelian subgroups are malnormal, \emph{small} if it has no  free non-abelian subgroup (see Subsection \ref{pti} for variations).

\begin{thmbis}[Theorem \ref{JSJ_CSA}] \label{mou2}
Let $G$ be a torsion-free finitely generated CSA group. There is a JSJ decomposition of $G$ over abelian subgroups.
\end{thmbis}

\begin{thmbis}[Theorem \ref{thm_JSJr}] \label{mou3}
Let $G$ be hyperbolic relative to a finite family of finitely generated small subgroups. If $\cala$ is either  the family of all virtually cyclic subgroups of $G$, or the family of all small subgroups, there is 
a JSJ decomposition of $G$ over $\cala$.
\end{thmbis}

\subsection*{Uniqueness.}

 \paragraph{JSJ trees are not unique\dots} Returning to the example of free decompositions, one obtains trees 
with the same maximality property as $T_0$ by precomposing the action of $G$ on $T_0$ with any automorphism of $G$. One may also change the topology of the quotient graph $T_0/G$ (see Figure \ref{fig_grushko}). The canonical object is   not a single tree, but the set of all Grushko trees (trees with trivial edge stabilizers, and non-trivial vertex stabilizers conjugate to the $G_i$'s), a  {deformation space}.

\begin{dfn*}[Deformation space \cite{For_deformation}]
The \emph{deformation space}\index{deformation space} $\cald$ of a 
tree   $T$ is the set of 
trees 
$T'$ such that $T$ dominates $T'$ and $T'$ dominates $T$. Equivalently, two 
trees are in the same deformation space
 if and only if they have the same elliptic subgroups. 
\end{dfn*}

More generally,  given a family of subgroups   $\widetilde\cala\inc\cala$, 
one considers deformation spaces over $\widetilde\cala$ by restricting to trees in $\cald$ with edge stabilizers in $\widetilde\cala$.

For instance, Culler-Vogtmann's outer space\index{outer space} (the set of free actions of $\F_n$ on trees) is a deformation space. Just like outer space, any deformation space   may be viewed as a complex in a natural way, and it is contractible (see \cite{Clay_contractibility,GL2}).

If $T$ is a JSJ tree, another  tree $T'$ is a JSJ tree if and only  if  $T'$ is universally elliptic,  $T$ dominates $T'$, and $T'$ dominates $T$. 
In other words, $T'$ should belong to the   deformation space of $T$ over $\cala_\elli$, where $\cala_\elli$ is the family of universally elliptic groups in $\cala$.

\begin{dfn*}
If non-empty, the set of all JSJ trees   is a deformation space over $\cala_{\elli}$ called the \emph{JSJ deformation space} (of $G$ over $\cala$). We denote it by $\cald_{JSJ}$.
\end{dfn*}

The canonical object is therefore not a particular JSJ decomposition, but the   JSJ deformation space.
 For instance, the JSJ deformation space of $\F_n$ over any $\cala$ is outer space    (see Subsection \ref{free}). 

It is a general fact that two trees belong to the same deformation space if and only if one can pass from one to the other by applying a finite sequence of moves of certain types, see \cite{For_deformation,GL2,For_splittings,ClFo_Whitehead}  and Remark \ref{moves} (this may be viewed as a connectedness statement, but as mentioned above deformation spaces are actually contractible). 
The statements about uniqueness of the JSJ decomposition up to certain moves which appear in \cite{RiSe_JSJ,DuSa_JSJ,FuPa_JSJ}, as well as the non-uniqueness results of \cite{For_uniqueness},  are special cases of this general fact. 

Another general fact is the following: two trees belonging to the same deformation space over $\cala$ have the same vertex stabilizers, provided one restricts to groups not in $\cala$. It thus makes sense to study vertex stabilizers of JSJ trees. We do so in   Part \ref{part_QH} (see below).

  If $\cala$ is invariant under the group of automorphisms of $G$ (in particular if $\cala$ is defined by restricting the isomorphism type), then so is
the deformation space $\cald_{JSJ}$.  As in the case of outer space, precomposing actions on trees with automorphisms of $G$  yields an action  of  
 $\Aut(G)$ on $\cald_{JSJ}$, a  contractible complex. This action factors through
an action of $\Out(G)$,
thus providing  information  about  $\Out(G)$
\cite{CuVo_moduli,McCulloughMiller_symmetric,GL1,Clay_deformation}. 

We stress once again that, in general, the canonical object associated to $G$ and $\cala$ is the deformation space consisting of all 
JSJ trees, and it may be quite large. 
\paragraph{\dots but sometimes there is a canonical JSJ tree.}
In some 
nice situations 
one can construct a  \emph{canonical} 
JSJ tree $T$ in $\cald_{JSJ}$.

By canonical,\index{canonical tree} we essentially mean ``defined in a natural, uniform way''. 
In particular, given any isomorphism $\alpha:G\ra G'$ sending $\cala$ to $\cala'$,
canonicity implies that there is  a unique $\alpha$-equivariant isomorphism $H_\alpha:T\ra T'$ between the canonical JSJ trees. 
Applying this with $G'=G$ (assuming that $\cala$ is invariant under automorphisms), one gets an action of $\Aut(G)$ on $T$.  
The canonical tree $T$ is a fixed point for the action of $\Out(G)$ on the JSJ deformation space.

The existence of such a canonical splitting
gives 
precise information about $\Out(G)$, see \cite{Sela_structure,Lev_automorphisms,GL4} for applications. 
A particularly nice example,  due to   Bowditch\index{Bowditch} \cite{Bo_cut}, is the  construction of a canonical JSJ decomposition of a one-ended hyperbolic group
over virtually cyclic 
subgroups,  from the structure of local cut points in  the Gromov boundary. 

This is not a consequence of the sole fact that $G$ is one-ended, and that one considers splittings over virtually cyclic groups:
the cyclic splittings of $G=\F_n\times \bbZ$ are in one-to-one correspondence with the free splittings of $\F_n$,
and the JSJ deformation space of $G$ is  the
outer space of $\F_n$.
Generalized Baumslag-Solitar groups\index{generalized Baumslag-Solitar group} are other striking examples where a strong non-uniqueness occurs, with surprising algebraic consequences
like the fact  (due to Collins-Levin \cite{CollinsLevin_automorphisms})   
  that the group $\Out(BS(2,4))$ of outer automorphisms of the Baumslag-Solitar group\index{Baumslag-Solitar group} $BS(2,4)=\grp{a,t\mid ta^2t\m=a^4}$
is not finitely generated
\cite{For_splittings,Clay_deformation}.

A method to  produce 
a canonical tree from a deformation space will be  given 
in Part \ref{partacyl},
using a construction called the \emph{tree of cylinders}\index{tree of cylinders} \cite{GL4}. 
 In particular, it yields canonical JSJ decompositions of one-ended  CSA groups\index{CSA group} and relatively hyperbolic groups\index{relatively hyperbolic group} 
(see Theorems \ref{dur1} and  \ref{dur2} below).
The compatibility JSJ decomposition introduced below also   yields a canonical tree.

\subsection*{Description: Quadratically hanging vertex groups.}

  As mentioned above, Grushko trees have a strong maximality property: their
vertex stabilizers 
are elliptic in any free splitting of $G$.
This does not hold any longer when one   considers  JSJ decompositions over infinite groups, in particular cyclic groups:
 a vertex stabilizer $G_v$ of a JSJ tree may fail to be elliptic in some splitting (over the chosen family $\cala$).

  If this happens, we say that the
vertex stabilizer $G_v$ (or the corresponding vertex $v$,  or the vertex group of the quotient graph of groups) is
\emph{flexible}.\index{flexible vertex, group, stabilizer}
The other stabilizers  (which are elliptic in every splitting  over $\cala$)  are called \emph{rigid}.\index{rigid vertex, group, stabilizer}
 In particular, all vertices of    Grushko trees are rigid  (because their stabilizers are freely indecomposable).
On the other hand, in the example of $G=\pi_1(\Sigma)$, the unique vertex stabilizer $G_v=G$  is  flexible.

Because all JSJ decompositions lie in the same deformation space over $\cala_{\elli}$, they have the same flexible vertex stabilizers (and the same  rigid vertex stabilizers not in $\cala$).

The essential feature of JSJ theory is the description of  flexible vertices,
in particular  the fact that   flexible vertex stabilizers are often ``surface-like'' 
  (\cite{RiSe_JSJ,DuSa_JSJ,FuPa_JSJ};
  \index{Rips-Sela}\index{Dunwoody-Sageev}\index{Fujiwara-Papasoglu}  see below, and
  Theorem \ref{thm_description_slender}, for more precise statements).
   In other words, the example of   trees $T_1,T_2$ given above using intersecting curves on a surface is often  the only source of flexible vertices.
  
This is formalized through the notion of \emph{quadratically hanging} (QH) groups,\index{QH, quadratically hanging} a terminology due to Rips and Sela \cite{RiSe_JSJ}.\index{Rips-Sela}

For cyclic splittings of a torsion-free group, a  vertex group $G_v$
is QH if it may be viewed as the fundamental group of a (possibly non-orientable) compact surface with  boundary $\Sigma$, in such a way 
  that any incident edge group is  trivial or contained (up to conjugacy) in a boundary subgroup,\index{boundary subgroup} i.e.\ the fundamental group $B=\pi_1(C)$ of a boundary component $C$ of $\Sigma$. 
The terminology  ``{quadratically hanging}'' describes the way in which $\pi_1(\Sigma)$ is attached to the rest of the group,
since boundary subgroups are generated by elements which  are quadratic words in a suitable basis of the free group $\pi_1(\Sigma)$.

In a more general setting,
one extends this notion as follows: $G_v$ is an extension  $1\to F\to G_v\to \pi_1(\Sigma)\to1$, where
$\Sigma$ is  a compact hyperbolic $2$-orbifold\index{orbifold} (usually with boundary), and $F$ is an arbitrary   group    called the \emph{fiber}.\index{fiber (of a QH subgroup)} The    
  condition on the attachment is  that the image of any incident edge group
  in $\pi_1(\Sigma)$ is finite or contained in a 
  boundary subgroup 
  (see Section \ref{sec_QH} for details).

Recall that a group is \emph{slender}\index{slender group}  if all its subgroups are finitely generated.

\begin{thmbis}[\cite{FuPa_JSJ}, see Corollary \ref{cor_slender}] \label{slenderintro}
Let $\cala$ be the class of all slender subgroups of a finitely presented group $G$.
  Let $G_v$ be a flexible vertex group   of a    JSJ decomposition of $G$ over $\cala$. Then $G_v$   is either slender  
or QH with slender fiber.
\end{thmbis}

One may  replace $\cala$ by  a subfamily, provided that  it  satisfies a suitable stability condition.\index{stability condition} In particular, $\cala$ may be the family of cyclic subgroups, virtually cyclic subgroups, polycyclic subgroups. 
Failure of this stability condition explains why the result does not apply to JSJ decompositions over abelian groups  in general (see Subsection \ref{ab}). On the other hand, non-abelian flexible vertex groups in Theorem \ref{mou2} are QH with trivial fiber, and in Theorem \ref{mou3} non-small flexible vertex groups are QH with finite fiber.

Theorem \ref{slenderintro} says that flexible subgroups of the JSJ decomposition are QH, but one can say more:
they are maximal in the following sense.

\begin{propbis}[see Corollary \ref{cor_qhe2}]\label{prop_qhe_intro}
Let $G$ be one-ended, 
and let $\cala$ be the class of all virtually cyclic groups.
Let $Q$ be a QH vertex  stabilizer with finite fiber in an arbitrary $\cala$-tree.

Then $Q$ is contained in a QH   vertex stabilizer of any cyclic JSJ decomposition of $G$.
\end{propbis}

 This does not hold without the one-endedness assumption: non-abelian free groups contain  many non-trivial QH subgroups.

Our proof of Theorem \ref{slenderintro} is based on the approach by Fujiwara and Papasoglu using products of trees \cite{FuPa_JSJ},  but with several simplifications. In particular, we do not  have to construct a group enclosing more than two splittings. 

The characteristic property of slender groups is that, whenever they act  on  a tree with no fixed point, there is an invariant  line. Using these lines, one may construct subsurfaces in the product of two trees and this explains (at least philosophically) the appearance of surfaces, hence of QH vertex groups, in Theorem \ref{slenderintro}. This is the content of Proposition \ref{prop_RN}.

This approach does not work if edge groups are not slender (unless there is acylindricity, as in Theorem \ref{dur2}  below), and the following problem is open.

\begin{pb*}
  Describe flexible vertices of   JSJ decompositions of a finitely presented group over small subgroups.\\
\end{pb*}

\subsection*{Relative decompositions}
For many applications, it is important not to consider all $\cala$-trees, but only
those in which subgroups of $G$ belonging to  a given family $\calh$   are elliptic (\ie every $H\in\calh$ fixes a point in the tree).
We say that such a tree is \emph{relative to $\calh$},\index{relative tree or splitting} and we call it an  $(\cala,\calh)$-tree.\index{0AH@$(\cala ,\calh )$-tree}
Working in a relative setting is important for applications (see e.g.\ \cite[Section 9]{Sela_diophantine1}, \cite{Pau_theorie,Perin_elementary},
 and Theorems \ref{dur1} and \ref{dur2}  below), 
and also needed in our proofs.  In the proof of Theorem \ref{slenderintro}  describing flexible vertex groups $G_v$ of slender JSJ decompostions,  for instance, we do not work with splittings of $G$, but rather with splittings of $G_v$ relative to incident edge groups.

Definitions extend naturally to relative trees:
a tree is universally elliptic if its edge groups are elliptic in every $(\cala,\calh)$-tree, and
 a JSJ decomposition of $G$ over $\cala$ relative to $\calh$ is
an $(\cala,\calh)$-tree which is universally elliptic and maximal for domination.

The   theorems stated above remain true, but  $\calh$ must be a finite family of finitely generated subroups in Theorems \ref{existmouintro}  and \ref{slenderintro}, and one must take $\calh$ into account when defining QH vertices (see Definition \ref{dfn_qh}).  

In this text we consistently work in a relative setting, so   the reader does not have to take it as an act of faith that   arguments also work for relative trees. For simplicity, though, we limit ourselves to the non-relative case  in this introduction (except in Theorems \ref{dur1} and \ref{dur2} where obtaining 
 canonical trees definitely requires working in a relative setting).

\subsection*{Acylindricity}

We have explained how Dunwoody's accessibility may be used to construct the JSJ deformation space, which  contains no preferred tree in general. In  Part \ref{partacyl} we use a different approach, based on the trees of cylinders of \cite{GL4} and on Sela's acylindrical accessibility\index{acylindrical accessibility}  \cite{Sela_acylindrical}, which  yields more precise results  when applicable. Unlike Dunwoody's accessibility, acylindrical accessibility  only  requires finite generation of $G$ (see Subsection \ref{accessi}).

Let $G$ be a one-ended group, and $T$ any tree with (necessarily infinite) virtually cyclic edge stabilizers.
Say that two edges of $T$ are equivalent if their stabilizers are commensurable (i.e.\ have infinite intersection).
One easily checks that equivalence classes are connected subsets of $T$, which we call \emph{cylinders}.\index{cylinder in a tree} Two distinct cylinders intersect in at most one point. 
Dual to this partition of $T$ into subtrees is another tree $T_c$ called \emph{the tree of cylinders}\index{tree of cylinders} (it is sometimes necessary to use a collapsed tree of cylinders, see Definition \ref{treecyl}, but we neglect this here).

This construction
works with other
equivalence relations among infinite edge stabilizers.
Here are  two   examples:   $G$ is a relatively hyperbolic group, and the equivalence relation is co-elementarity\index{co-elementary subgroups} among  infinite elementary subgroups  ($A\sim B$ if and only if $\grp{A,B}$ is elementary, \ie virtually cyclic or parabolic), or $G$ is a torsion-free CSA group and the relation is commutation among non-trivial abelian subgroups.

\begin{figure}[htbp]
  \centering
\includegraphics{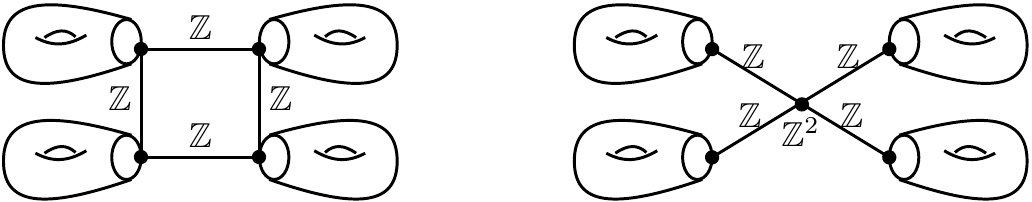}
  \caption{A JSJ splitting of a toral relatively hyperbolic group and its tree of cylinders.}
  \label{fig_ex}
\end{figure}

Here is  an example (which already appears in \cite{GL4}). Let $T$ be the Bass-Serre tree of the graph of groups $\Gamma$ pictured on the left of Figure \ref{fig_ex}  (all punctured tori have the same boundary subgroup, equal to the edge groups of $\Gamma$). The fundamental group of $\Gamma$ is a torsion-free one-ended group $G$ (which is toral relatively hyperbolic and CSA),  and $T$ is a
cyclic JSJ decomposition of $G$. 
In this case all previous equivalence relations on the set of edge stabilizers  of $T$  (commensurability, co-elementarity, commutation) reduce to equality, and the quotient graph of groups of the tree of cylinders $T_c$ is pictured on the right of Figure \ref{fig_ex}; there is a new vertex group, isomorphic to $\Z^2$.

There are three main benefits in passing from  the JSJ tree $T$ to $T_c$. First, two trees in the same deformation space   always have the same tree of cylinders. In particular,     the tree $T_c$ is invariant under all automorphisms of $G$  (whereas  $T$   is only invariant up to deformation). 
Second, $T_c$ is acylindrical:\index{acylindrical tree, splitting} all segments of length 3 have trivial stabilizer (whereas $T$ contains   lines with infinite cyclic pointwise stabilizer).  Third,   trees of cylinders enjoy nice compatibility properties, which will be important later. 
There is a small price to pay in order to replace $T$ by the  better tree $T_c$, namely changing the deformation space by  creating a new vertex group $\Z^2$.

This example is typical of the method we use in Part \ref{partacyl}  to prove Theorems \ref{mou2} and \ref{mou3}. Using trees of cylinders, we show that one may associate to any tree $T$ an acylindrical  tree $T^*$ in such a way that groups elliptic in $T$ are also elliptic in $T^*$, and groups elliptic in $T^*$ but not in $T$ are \emph{small}\index{small, small in $(\cala ,\calh )$-trees}  ($T^*$ is smally dominated by $T$ in the sense of Definition \ref{sm}). 
 Applying acylindrical accessibility to trees $T^*$ is   one of the key ingredients in our construction 
of JSJ decompositions, but the proof involves much more than a bound on the complexity of acylindrical
splittings.

In the example, the tree $T$ is a cyclic JSJ tree, but the preferred tree $T_c$ is in  a slightly different deformation space; it is  
 a   JSJ tree relative to
the $\Z^2$  subgroup. In general, we show:
 
 \begin{thmbis}[Theorem \ref{JSJ_CSA}] 
 \label{dur1}
Let $G$ be a torsion-free finitely generated one-ended CSA group.\index{CSA group} There is a
 canonical JSJ tree  over abelian subgroups\index{abelian tree} relative to all non-cyclic abelian subgroups. Its non-abelian flexible vertex stabilizers are QH with trivial fiber.
\end{thmbis}

The CSA property (maximal abelian subgroups are malnormal)\index{CSA group}  holds for any torsion-free hyperbolic group $\Gamma$, and any $\Gamma$-limit group.\index{GA@$\Gamma$-limit group}
If $\Gamma$ is a hyperbolic group with torsion, $\Gamma$ and all $\Gamma$-limit groups have the weaker property of being 
 $K$-CSA   (for some integer $K$), as defined in
 Subsection \ref{aKcsa}.
Theorem \ref{dur1}   generalizes to this setting (see Theorem \ref{JSJ_KCSA}) and implies the following result:
\begin{thmbis} \label{thmKCSA}
 Let $\Gamma$ be a hyperbolic group, and let $G$ be a one-ended $\Gamma$-limit group.
Then $G$ has a
 canonical JSJ tree over  virtually abelian subgroups relative to all  virtually abelian subgroups which are not virtually cyclic. 
Its  flexible vertex stabilizers   are   virtually abelian or QH with finite fiber.
\end{thmbis}

We have a similar statement for relatively hyperbolic groups.\index{relatively hyperbolic group}

\begin{thmbis}[Theorem \ref{thm_JSJr}]
\label{dur2}
Let $G$ be one-ended and hyperbolic relative to a finite family of finitely generated small subgroups. If $\cala$ is either  the family of all virtually cyclic subgroups of $G$, or the family of all small subgroups, there is 
a canonical
JSJ tree  over $\cala$ relative to all parabolic subgroups. Its   flexible vertex stabilizers are small or QH with finite fiber.
\end{thmbis}

The trees   produced by these theorems are defined in  a uniform, natural way, and are  canonical 
 (as discussed above). In particular, they are
 invariant under automorphims.  
When  $G$ is  a one-ended hyperbolic group, the  canonical  JSJ tree (non-relative in this case)
coincides with the tree constructed by Bowditch\index{Bowditch}  \cite{Bo_cut} using the topology of $\bo G$.

\subsection*{Compatibility JSJ}

A \emph{refinement}\index{refinement} $\hat T$ of a tree $T$ is a tree obtained by blowing up vertices of $T$ (beware  that in \cite{FuPa_JSJ} 
a refinement of $T$ is what we call a tree dominating $T$; on the other hand, their elementary unfoldings are refinements in our sense). 
There is a map   $p$ from $ \hat T$  to $ T $, so $\hat T$ dominates  $T$, but this map is very special: it maps any segment   $[x,y]$ of $\hat T$ onto the segment $[p(x),p(y)]$ of $T$ (in particular, it does not fold). We call such a map a \emph{collapse map,}\index{collapse map} as it is obtained by collapsing certain edges to points.

If a tree $T_1$ is universally elliptic, then given any tree $T_2$ there is a refinement $\hat T_1$ of $T_1$  which dominates $T_2$: there is an equivariant map $f: \hat T_1\to T_2$ (see Proposition \ref{prop_refinement}).
  If $T_2$ has finitely generated edge stabilizers, one may obtain $T_2$ from $\hat T_1$ by a finite sequence of folds and collapses \cite{BF_bounding}.
In this sense, one can \emph{read} $T_2$ from $T_1$.
In particular, one may read any tree $T_2$ from any JSJ tree $T_1$. 

In general, the map $f: \hat T_1\to T_2$ is not a collapse map (there are folds). We say that $T_1$ and $T_2$ are \emph{compatible}\index{compatible trees} if there exists a refinement $ \hat T_1$ with a collapse map $f: \hat T_1\to T_2$. In other words, 
 $T_1,T_2$ are  {compatible} when they have a common refinement.
This implies that edge stabilizers of each tree are elliptic in the other, but is much more restrictive.

For instance,  if $\Sigma$ is a compact surface with boundary, free splittings of $\pi_1(\Sigma)$ dual to properly embedded arcs are always elliptic with respect to each other (edge groups are trivial), but they are compatible only if the arcs are disjoint (up to isotopy).
On the other hand,
splittings of a  hyperbolic  surface group associated to  two   simple closed geodesics  $\gamma,\gamma'$ are compatible if and only if $\gamma$ and $\gamma'$  are disjoint or equal; in this specific case, compatibility is equivalent to the trees being elliptic with respect to each other. 

In Part \ref{chap_compat} we introduce another type of JSJ decomposition, 
which encodes compatibility of splittings rather than  ellipticity. The new feature is that, except in degenerate cases,  it will lead to a  canonical 
tree $\Tco$ (not just a deformation space). As above, we fix a family $\cala$   and all trees are assumed to be  $\cala$-trees.

We say that a  tree  $T$ (or the corresponding  graph of groups $\Gamma$) is  \emph{universally compatible}\index{universally compatible tree} if it is compatible with  every tree. One may then  obtain any tree from $T$ by refining and collapsing, and view any one-edge splitting of $G$ as coming from an edge of $\Gamma$ or from a splitting of a vertex group of $\Gamma$. 

  \begin{dfn*}  
  The \emph{compatibility JSJ deformation space}\index{compatibility JSJ deformation space}\index{0DC@$\Dco$: compatibility JSJ space} $\Dco$ is the maximal deformation space (for domination) containing a  
  universally compatible tree (such a maximal deformation space, if it exists, is unique).
\end{dfn*}

In other words, $\Dco$  contains a universally compatible tree $T$, and $T$ dominates all universally compatible trees.

\begin{thmbis}[Theorem \ref{thm_exists_compat}]\label{intro1}
  Let $G$ be finitely presented, and let $\cala$ be any conjugacy-invariant class of subgroups of $G$, stable under taking subgroups.
Then the compatibility JSJ deformation space $\Dco$ of $G$ over $\cala$ exists.
\end{thmbis}

Although existence of the usual JSJ deformation space  is a fairly direct   consequence of accessibility, proving 
existence of  the compatibility JSJ deformation space is more delicate.
Among other things, we use a limiting argument, and we need to know that a limit of universally compatible trees is universally compatible. This is best expressed in terms of \Rt s (see the appendix).

 As mentioned above,   the deformation space $\Dco$  contains a
 canonical element $\Tco$, 
 except  in degenerate cases. Recall that a tree $T$ is \emph{irreducible}\index{irreducible tree, deformation space} if   $G$ acts on $T$ with no fixed point, no fixed end, and no invariant line; we say that a deformation space $\cald$ is irreducible if some (equivalently, every) $T\in\cald$ is irreducible. 
 
 \begin{thmbis}[Corollary \ref{corpref}]\label{intro1.5}  
 If $\Dco$ exists and is irreducible, it contains a canonical
tree\index{canonical tree} $\Tco$, the \emph{compatibility JSJ tree}.\index{compatibility JSJ tree}\index{0TC@$\Tco$: compatibility JSJ tree} 
In particular, if 
 $\cala$ is invariant under automorphisms of $G$, 
 so is $\Tco$.
 \end{thmbis}
 
 This is because a deformation space $\cald$ containing an irreducible universally compatible tree has a preferred element.
 
 We develop an analogy with arithmetic, viewing a refinement of $T$ as a multiple of $T$, and one-edge splittings as primes. We define the \emph{greatest common divisor (gcd)}\index{gcd of two trees} of two trees, the \emph{least common multiple (lcm)}\index{lcm of pairwise compatible trees} of a family of pairwise compatible trees, and $\Tco$ is the lcm of the reduced universally compatible trees contained in $\Dco$ (see Subsection \ref{defspa} for the definition of  ``reduced'').

This tree $\Tco$ is similar to the canonical tree $T_{SS}$ constructed by Scott and Swarup \cite{ScSw_regular+errata},\index{Scott-Swarup}
which has the property of being
{compatible} with
almost-invariant sets 
(Scott and Swarup
use the word \emph{enclosing}, which generalizes compatibility).
In general, $\Tco$ dominates $T_{SS}$, but    it may be non-trivial when $T_{SS}$ is trivial (this happens for instance for the Baumslag-Solitar group\index{Baumslag-Solitar group} $BS(m,n)$ when none of $m,n$ divides the other). See \cite{GL5} for the relation between $T_{SS}$ and JSJ decompositions.

Being invariant under automorphisms sometimes forces $\Tco$ to be trivial (a point). 
This happens for instance when $G$ is free. On the other hand, we give simple examples with $\Tco$ non-trivial:  certain  virtually free groups,   generalized Baumslag-Solitar groups,\index{generalized Baumslag-Solitar group} Poincar\'e duality groups... Trees of cylinders\index{tree of cylinders} also provide many examples. In particular, the
 canonical trees of Theorems \ref{dur1} and \ref{dur2} are very closely related to $\Tco$.

\subsection*{Contents of the paper} 

For the reader's convenience we now describe the detailed contents of each section. This includes a few results not directly related to JSJ decompositions, which are of independent interest: 
 relative finite presentation (\ref{fingv}), small orbifolds (\ref{liste}), orbifolds with finite mapping class group (\ref{fmg}), $K$-CSA groups (\ref{aKcsa}), compatibility and length functions  (\ref{sec_compatible_via_longueurs}), arithmetic of trees (\ref{sec_arith}).

This   is meant as a description, so statements may be imprecise or incomplete.

$\bullet$  In the preliminary section (Section \ref{prel}), we collect basic facts about groups acting on trees, and we define $(\cala,\calh)$-trees (trees with edge stabilizers in $\cala$ relative to $\calh$), collapse maps and refinements, compatibility, domination,  deformation spaces. We discuss slenderness and smallness of subgroups, one-endedness, and we recall the main accessibility results. We also define and discuss relative finite generation and presentation (relative finite presentation of vertex groups is studied in Subsection 4.2.2).

$\bullet$ Section \ref{de}   starts with a very useful fact:  
if all edge stabilizers of $T_1$ are elliptic in $T_2$, there is a refinement of $T_1$ which dominates $T_2$.  After defining universal ellipticity, we define JSJ trees and the JSJ deformation space by a maximality property, as explained above. We prove the existence of JSJ decompositions under a finite presentability assumption, first in the non-relative case and then in general. This relies on a version of Dunwoody's accessibility due to Fujiwara-Papasoglu  \cite{FuPa_JSJ}
which we state and prove.  We also explain why the JSJ decompositions constructed by 
 Rips-Sela \cite{RiSe_JSJ},  Dunwoody-Sageev \cite{DuSa_JSJ}, Fujiwara-Papasoglu \cite{FuPa_JSJ} are JSJ decompositions in our sense.

$\bullet$ Section \ref{exa} is devoted to simple examples. We first consider Grushko decompositions (over the trivial group) and Stallings-Dunwoody decompositions (over finite groups), explaining how to interpret them as JSJ decompositions. We also consider small groups, and locally finite trees (such as those associated to cyclic splittings of generalized Baumslag-Solitar groups). All these examples of JSJ decompositions only have rigid vertices. At the end of the section we work out an example where the JSJ decomposition has QH flexible vertices.

$\bullet$ Section \ref{usef} contains various useful technical results. Given a vertex of    a graph of groups or of a tree, we define the incident edge groups  and we point out that any splitting of the vertex group which is relative to the incident edge groups extends to a splitting of $G$. 
Given a universally elliptic splitting of $G$, one may obtain a JSJ decomposition of $G$ from relative JSJ decompositions of vertex groups. In particular, one may usually restrict to one-ended groups when studying JSJ decompositions.

$\bullet$ Section \ref{sec_QH} is devoted to QH groups. 
We first study 2-dimensional hyperbolic  orbifolds $\Sigma$, in particular the relation between splittings of $\pi_1(\Sigma)$ and simple closed geodesics on $\Sigma$. 
We classify orbifolds $\Sigma$ with no essential simple closed geodesic (their groups  do  not split over a cyclic subgroup relative to the fundamental groups of boundary components, so they do  not appear in QH vertices of JSJ decompositions); when $\Sigma$ is a surface, only the pair of pants (thrice-punctured sphere) occurs, 
but the classification is more complicated for singular, possibly non-orientable, orbifolds. We also classify orbifolds with finite mapping class group.

We then define QH subgroups, and study their basic properties (universal ellipticity of the fiber, used boundary components, existence of simple geodesics, universally elliptic subgroups). In particular, we show that, if $Q$ is a QH vertex group in a JSJ decomposition over slender groups, and $G$ acts on a tree with slender edge stabilizers, then either $Q$ fixes a point or its action on its minimal subtree is dual to a family of simple closed geodesics of $\Sigma$. We also prove a more general version  of Proposition \ref{prop_qhe_intro} showing that, under suitable assumptions, any QH vertex stabilizer $Q$ of any tree is elliptic in JSJ trees.

Using a filling construction, we give examples of possible peripheral structures of QH vertex groups, and we show that flexible vertex groups of JSJ decompositions over abelian subgroups do not have to be QH (the filling construction was introduced in Section \ref{usef}, in order to provide an alternative construction of relative splittings). 

$\bullet$ In Section \ref{Fujpap} we study flexible vertices of JSJ decompositions over slender groups;
 in particular, we prove Theorem \ref{slenderintro}. For completeness we also describe slender flexible subgroups. 
We   allow edge stabilizers which are only ``slender in trees'': whenever $G$ acts on a tree, they fix a point or leave a line invariant. This is useful when   working in a relative setting, as groups in $\calh$ are automatically slender in trees.

To prove Theorem \ref{slenderintro}, we follow the approach by Fujiwara-Papasoglu \cite{FuPa_JSJ}, with simplifications. In particular,  the a priori knowledge that JSJ decompositions exist allows us to reduce to totally flexible groups. Following  \cite{FuPa_JSJ}, we define the core, and the enclosing group of two splittings (which we call their regular neighborhood). We then construct a filling pair of splittings, and we show that their regular neighborhood is the required QH group. For technical reasons we replace Fujiwara-Papasoglu's  notion of minimal splittings     by the slightly stronger notion of  \emph{minuscule splittings}. 

In Theorem  \ref{slenderintro} the family $\cala$ does not have to be the family of all slender subgroups of $G$. The theorem holds if $\cala$ is a family of slender groups satisfying one of two stability conditions $(SC $) and $(SC_\calz$); the familly of  cyclic (resp.\  virtually cyclic) subgroups satisfies $(SC_\calz$) (resp.\ $(SC $)). 
These  conditions ensure that the regular neighborhood of two splittings with edge groups in $\cala$ also has edge groups in $\cala$.

$\bullet$ Section \ref{sec_cyl} is devoted to the tree of cylinders. Given an admissible equivalence relation on the set $\cale$ of infinite groups in $\cala$, one may associate a tree of cylinders $T_c$ to any tree $T$ with stabilizers in $\cale$. This tree only depends on the deformation space of $T$. We give conditions ensuring that $T_c$ is acylindrical and smally dominated by $T$ (this means in particular that groups elliptic in $T_c$ but not in $T$ are small). We also study the compatibility properties of $T_c$.

$\bullet$ In Section \ref{jsjac} we show that JSJ decompositions exist, and non-small flexible vertex groups are QH with finite fiber, under the assumption that one may associate to any tree $T$ an acylindrical tree $T^*$ smally dominated by $T$. We first construct a relative tree as in Theorems \ref{dur1} and \ref{dur2}, and we refine it in order to get the required JSJ tree.

$\bullet$  This is applied in Section \ref{exam}, with $T^*$ the tree of cylinders, and used to prove Theorems \ref{dur1} and \ref{dur2}. We study torsion-free CSA groups, relatively groups, as well as cyclic splittings of commutative transitive groups. We introduce $K$-CSA groups, for $K$ an integer, which are better suited than CSA groups to study groups with torsion,
 and we prove Theorem \ref{thmKCSA}. 
We also discuss a slightly different type of JSJ decompositions of one-ended hyperbolic groups, where edge groups are required to be maximal virtually cyclic subgroups with infinite center.

$\bullet$ In Section \ref{compt} we define universal compatibility and we show Theorem \ref{intro1} (existence of the compatibility JSJ deformation space $\Dco$). We then construct the compatibility JSJ tree $\Tco$ and we give examples. In particular, we use the compatibility properties of the tree of cylinders to identify $\Dco$ for abelian splittings of    CSA groups, elementary splittings of relatively hyperbolic groups, and cyclic splittings of  commutative transitive groups. 

$\bullet$ In the appendix we view trees as metric rather than combinatorial objects, and we actually consider \Rt s.
We first give simple proofs of two standard results: a tree is determined by its length function (\cite{AlperinBass_length}, \cite{CuMo}), the axes topology agrees with the equivariant Gromov-Hausdorff topology \cite{Pau_Gromov}. We then  study compatibility (which is defined for \Rt s using collapse maps): we prove that two \Rt s are compatible if and only if the sum of their length functions is a length function (i.e.\ it comes from some \Rt). In particular, compatibility is a closed property in the space of trees (this is used in the proof of Theorem \ref{intro1}).

Using the core introduced by the first author \cite{Gui_coeur}, we then show that a finite family of pairwise compatible \Rt s   has a common refinement. Going back to simplicial trees, we  develop an   analogy with basic arithmetics: we define the prime factors  of a tree $T$ (the one-edge splittings corresponding to edges of the quotient graph $T/G$), we show that trees are square-free, and we define gcd's and lcm's. 

We conclude by some remarks about combining   JSJ theory and Rips theory to describe small actions on \Rt s. This gives another, more general, approach to the main result of \cite{Gui_reading}.

\subsection*{What is new} 

JSJ theory was developed by several people, but to the best of our knowledge the following 
is original material:

$\bullet$
The definition of JSJ decompositions as a deformation space satisfying a maximality property. 

$\bullet$
The systematic study of relative JSJ decompositions.

$\bullet$
Most of   Section \ref{usef}. 

$\bullet$ The classification of small orbifolds and orbifolds with finite mapping class group in Section \ref{sec_QH}.

$\bullet$ The stability conditions $(SC $) and $(SC_\calz)$ (but compare \cite{DuSa_JSJ}), and the description of slender flexible groups. 

$\bullet$ Everything from Subsection \ref {sec_Tc_smally} on, with the exception of \ref{marb} and \ref{fred}.   In particular 
the detailed proof of  existence of the JSJ decomposition under acylindricity  assumptions (compare \cite{Sela_hopf,Sela_diophantine1}), the compatibility JSJ tree,   the arithmetic of trees.

$\bullet$ The $K$-CSA property (but see also \cite{ReiWei_MR}).

\subsection*{Acknowledgements}

The first author acknowledges support from the Institut universitaire de France, ANR project ANR-11-BS01-013, and membership to the Henri Lebesgue center 11-LABX-0020.
The second author acknowledges support from
ANR-07-BLAN-0141-01 and ANR-2010-BLAN-116-03.

\tableofcontents

\part{Preliminaries}   

\setcounter{section}{0}
\section{Preliminaries}  \label{prel}

 In   this paper, $G$ will always be a finitely generated group. Sometimes finite presentation will be needed, for instance
to prove existence of JSJ decompositions in full generality.

\subsection{Basic notions and notations} \label{basic}

Two subgroups $H,H'$ of a group $G$ are \emph{commensurable}\index{commensurable}
  if $H\cap H'$ has finite index in $H$ and in $H'$.
The \emph{commensurator}\index{commensurator} of $H$ (in $G$) is the set of elements $g\in G$ such that $gHg\m$ is commensurable to $H$.

We denote by $\F_n$ the free group on $n$ generators.  \index{0Fn@$\F_n$, the free group on $n$ generators}

A group $H$ is \emph{virtually cyclic}\index{virtually cyclic}  if it has  a cyclic subgroup of finite index:  $H$ is finite or has a finite index subgroup isomorphic to $\bbZ$  (infinite virtually cyclic groups are characterized as being two-ended).
 There are two types of infinite virtually cyclic groups: those with infinite center map onto 
$\bbZ$ with finite kernel $K$, those with finite center map onto   the \emph{infinite dihedral group} $D_\infty=(\bbZ/2)*(\bbZ/2)$ with finite kernel $K$  (see \cite{ScottWall}).
\index{dihedral group}
The kernel $K$ is the unique maximal finite normal subgroup of $H$.
In Subsection \ref{sec_Tc_acyl} we shall say that $H$ is $C$-virtually cyclic for some constant $C\geq 1$ if $K$ has cardinality at most $C$.
\index{virtually cyclic@ $C$-virtually cyclic}\index{0CV@ $C$-virtually cyclic}

A group $G$ is \emph{relatively hyperbolic}\index{relatively hyperbolic group} with respect to  a family of finitely generated subgroups $\calp=\{P_1,\dots,P_n\}$
if it acts properly by isometries on a proper Gromov hyperbolic space with an invariant collection
 of disjoint horoballs,   the action is cocompact on the complement of these horoballs,
and the stabilizers of the horoballs are exactly the conjugates of the $P_i$'s
\cite{Bow_relhyp,Farb_relatively,Osin_relatively,Hruska_quasiconvexity}.
A subgroup is called \emph{parabolic}\index{parabolic} if it is conjugate  to a subgroup of some  $P_i$,
\emph{elementary}\index{elementary subgroup} if it is parabolic  or virtually cyclic.
Non-elementary subgroups contain free subgroups acting by hyperbolic isometries.  
 
 A group is \emph{CSA}\index{CSA group} (for \emph{conjugately separated abelian}) if its maximal abelian subgroups are malnormal.
For example, torsion-free hyperbolic groups are CSA.
See Subsection \ref{aKcsa} for a generalisation in presence of torsion.

\subsection{Trees} \label{tre}

We consider actions of $G$ on simplicial trees $T$. We identify two  trees if there is an equivariant isomorphism between them.
 In the appendix we will view trees as metric spaces and work with  \Rt s, but in most of the paper trees are considered as combinatorial objects.   
Still, it  is   useful to think of a tree as a geometric object, for instance to define the arc between any two points, or the \emph{midpoint} of an edge.  See  \cite{Serre_arbres,ScottWall,DicksDunwoody_groups,Sh_dendrology,Sh_introduction,Chi_book} 
for basic facts about trees. 

Given two points 
$a,b\in T$, there is a unique \emph{segment}\index{segment} $[a,b]$ joining them; it is  \emph{degenerate}\index{degenerate segment} if $a=b$. 
  Non-degenerate segments are also called \emph{arcs}\index{arc}  (they are homeomorphic to $[0,1]$).
If $A,B$ are disjoint  simplicial    subtrees, the \emph{bridge}\index{bridge} 
between them is the unique segment $I=[a,b]$ such that $A\cap I=\{a\}$ and $B\cap I=\{b\}$.

We
always assume that $G$ acts without inversion (if $e=vw$ is an edge, no element of $G$ interchanges $v$ and $w$), and often that $T$ has no \emph{redundant vertex}\index{redundant vertex}:
if a vertex has valence $2$, it is the unique fixed point of some element of $G$. We denote 
by $V(T)$\index{0V@$V(T)$: vertex set of $T$} and $E(T)$\index{0E@$E(T)$: edge set of $T$} the set of vertices and (non-oriented) edges of $T$ respectively, by  
 $G_v$ or $G_e$\index{0G@$G_v,G_e$: vertex or edge stabilizer/group} the stabilizer of a vertex $v$ or an edge $e$,  by $G_x$ the stabilizer of an arbitrary point $x\in T$. 

By Bass-Serre theory, the action of $G$ on $T$ can be viewed as a \emph{splitting}\index{splitting} of $G$ as a marked graph of groups,\index{graph of groups}
\ie an isomorphism between $G$ and the fundamental group of a graph of groups $\Gamma$.
A  \emph{one-edge splitting}\index{one-edge splitting} (when $\Gamma$ has one edge) is an amalgam $G=A*_CB$ or an HNN-extension $G=A*_C$. We also denote by  $G_v$ or $G_e$   the groups carried by    vertices and edges of $\Gamma$.

The action $G\actson T$ is \emph{trivial}\index{trivial action} if $G$ fixes a point (hence a vertex since we assume that there is no inversion), \emph{minimal}\index{minimal action} if there is no proper $G$-invariant subtree.

Unless otherwise indicated, all trees 
are endowed with a \emph{minimal action of $G$ without inversion} (we allow the trivial case when  $T$ is a point). Finite generation of $G$ implies that the quotient graph $\Gamma=T/G$ is finite.  Note, however, that the restriction of the action to a subgroup of $G$ does not have to be minimal.

An element  $g$ or a subgroup  $H$ 
of $G$ is \emph{elliptic}\index{elliptic element or subgroup} in $T$ if it fixes a point. This is equivalent to $g$ or $H$ being contained in a conjugate of a vertex group of $\Gamma$. We    denote by $\Fix  g$ or $\Fix  H$ its fixed point set in $T$. If $H$ is elliptic, any $H$-invariant subtree meets $\Fix  H$.
 If $H_1\inc H_2\inc G$ with $H_1$ of finite index in $H_2$, then $H_1$ is elliptic if and only if $H_2$ is. 
 
Finite groups, and groups with Kazhdan's property (T), have   \emph{Serre's property (FA)}:\index{0FA@(FA): Serre's property}
 there is a fixed point in every tree on which they act. 
 
 If $H$ is finitely generated, it is elliptic if and only if all its elements are. This follows from \emph{Serre's lemma}\index{Serre's property (FA), Serre's lemma} 
\cite[6.5, corollaire 2]{Serre_arbres}: if $g_1,\dots, g_n$, as well as all products $g_ig_j$, are elliptic, then $\langle g_1,\dots,g_n\rangle$ is an elliptic subgroup. 
 
An element $g$ which is not elliptic is \emph{hyperbolic}\index{hyperbolic element, subgroup}, it  has a unique \emph{axis} 
$A(g)$\index{0A@$A(g)$: axis of $g$}\index{axis} on which it acts by translation.  We also denote by $A(g)$ the \emph{characteristic set}\index{characteristic set} of $g$: its fixed point set if it is elliptic, its axis if it is hyperbolic.

Translation lengths and length functions will only be used in Subsections \ref{unac} and \ref{exdur}, we discuss them in the appendix (for \Rt s). 

We will need to consider the restriction of the action of $G$ to subgroups. Let therefore $H$ be an arbitrary group  (possibly infinitely generated) acting on a tree $T$. We only assume that the action is not trivial (there is no global fixed point).

If $H$
contains no hyperbolic element,   it is not  finitely generated by Serre's lemma,  and it fixes a unique end of $T$: there is a ray $\rho$ such that each finitely generated subgroup of $H$ fixes a subray of $\rho$ (an \emph{end}\index{end of a tree} of $T$ is an equivalence class of rays, with $\rho,\rho'$ equivalent if their intersection is a ray).   In this case there are $H$-invariant subtrees, but no minimal one.

Assume now that  $H$ contains a hyperbolic element (we sometimes say that $H$ \emph{acts hyperbolically},\index{hyperbolic element, subgroup} or is hyperbolic); 
  by Serre's lemma, this always holds if $H$ is finitely generated and acts non-trivially.
Then there is a \emph{unique minimal $H$-invariant subtree} $\mu_T(H)$, namely the union of axes of hyperbolic elements of $H$. \index{0M@$\mu_T(H)$: minimal $H$-invariant subtree}\index{minimal subtree} 

The  tree $T$ (with the action of $H$) is \emph{irreducible}\index{irreducible tree, deformation space} if there exist  two hyperbolic elements $g,h\in H$ whose axes are disjoint or intersect in a finite segment.
Suitable powers of $g$ and $h$ then generate a non-abelian free group $\F_2\inc H$   acting freely on $T$.
 It follows that  $T$ is irreducible if and only if there exist two hyperbolic elements $g,h$ whose commutator $[g,h]$ is hyperbolic.

If $T$ is not irreducible, then  $H$ preserves a line of $T$  or fixes a unique   end (recall that  by assumption there is no global fixed point, and note that there is an invariant line if there are two fixed ends). 

There are two types of   non-trivial actions on a line. If orientation is preserved, the action is by translations and both ends of the line are invariant. All points and edges have the same stabilizer. If not, there are reflections. The action is said to be \emph{dihedral},\index{dihedral action} it factors through an action of the infinite dihedral group $D_\infty =\Z/2*\Z/2$. There is no invariant end. All edges have the same stabilizer, but vertex stabilizers may contain the edge stabilizer with index 2.

If $H$ fixes an end of $T$, there is an associated homomorphism $\chi:H\to\Z$ measuring how much an element $h$ pushes towards the end. More precisely, for $h\in H$, one defines $\chi(h)$ as the difference between the number of edges in   $\rho\setminus(\rho\cap h\rho)$ and the number of edges in $h\rho\setminus(\rho\cap h\rho)$, with $\rho$ any ray going to the end. 

The map $\chi$ is non-trivial if and only if $H$ contains a hyperbolic element.
 In this case the quotient graph of groups $\Gamma$ is homeomorphic to a circle, and one may orient the circle so that  the inclusion $G_e\to G_v$ is onto whenever $e=vw$ is a positively oriented edge. When $\Gamma$ is a single edge, it defines 
 an ascending HNN extension
 and $\chi$ is (up to sign) the exponent sum of the stable letter.

To sum up:

\begin{prop} \label{trois}
If $H$ acts on a tree, one of the following holds:
\begin{enumerate}
\item there is a global fixed point;
\item    there are hyperbolic elements $h\in H$, and $T$ contains   a unique minimal   $H$-invariant subtree $\mu_T(H)$; 
\item   $H$ is infinitely generated  and fixes a unique end. \qed
   \end{enumerate}
\end{prop}

\begin{prop} \label{cinq}
If $H$ acts minimally on a tree $T$ (i.e.\ there is no proper $H$-invariant subtree), 
there are five possibilities:
\begin{enumerate}
\item   $T$ is a point (trivial action);
\item $T$ 
 is a line, and $H$ acts by translations; the action factors through $\Z$;
\item  $T$ 
 is a line, and some $h\in H$ reverses orientation (dihedral action);   the action factors through    $D_\infty =\Z/2*\Z/2$;
\item  there is a unique invariant end, but no global fixed point;    the quotient graph of groups is homeomorphic to a circle;
\item $T$ is irreducible (this implies $\F_2\inc H$).
\end{enumerate}
In all cases except (1), some $h\in H$ is hyperbolic. \qed
\end{prop}

In particular:
\begin{cor} \label{pasirr}
 If $H$ acts on $T$, and $T$ is not irreducible, then there is a fixed point, or a  unique fixed end, or a  unique  invariant line.   
\end{cor}

\begin{proof}
 Existence follows from the propositions. We prove  the uniqueness statements. If there are two invariant ends, the line joining them is invariant. We show that there is a fixed point if there exist two invariant lines  $\ell_1,\ell_2$.

 If $\ell_1$ and $\ell_2$ are disjoint, the midpoint of the bridge between them is fixed. If their intersection is a segment of finite length, its midpoint is fixed. If the intersection is a ray, its origin is a fixed point.
\end{proof}

\subsection{$(\cala,\calh)$-trees,  one-endedness} \label{AH}

Besides $G$, we usually  also  fix a  (nonempty)  family $\cala$\index{0A@$\cala$: allowed edge stabilizers} of subgroups of $G$ which is  stable under
conjugation and under taking subgroups. 
An
\emph{$\cala$-tree}\index{0AT@$\cala$-tree} is a tree $T$ whose edge stabilizers belong to $\cala$. We  often  say
that
$T $, or the corresponding splitting\index{splitting over $\cala$} of $G$,  is
\emph{over $\cala$}, or over groups in $\cala$. We  \index{cyclic tree}\index{abelian tree}
\index{slender group}
say cyclic tree (abelian 
 tree, slender tree, \dots) when $\cala$ is the family of cyclic (abelian,  
 \dots) subgroups.

We    also   fix an arbitrary   set $\calh$\index{0H@$\calh$: relative structure}
 of subgroups of $G$, and we restrict to  
$\cala$-trees   
$T$
 such that each $H\in\calh$ is elliptic in $T$  (in terms of graphs of groups, $H$ is contained in a conjugate of a vertex group;  if $H$ is not finitely generated, this is stronger than requiring that every $h\in H$ be elliptic). We call such a tree an \emph{$(\cala, \calh)$-tree},\index{0AH@$(\cala,\calh)$-tree}
 or a tree over $\cala$ \emph{relative to $\calh$.}\index{relative tree or splitting}
 The set of $(\cala, \calh)$-trees does not change if we replace a group of $\calh$ by a conjugate, or if we enlarge $\calh $ by making it invariant under conjugation. 
 
 If $G$ acts non-trivially on an $(\cala,\calh)$-tree, we say that $G$
  \emph{splits over $\cala$ (or over a group of $\cala$) relative to $\calh$}.
  
  The group $G$ is \emph{freely indecomposable relative to $\calh$}\index{freely indecomposable relative to $\calh$} if it does not split over the trivial group relative to $\calh$. Equivalently (unless $G=\Z$ and $\calh$ is trivial), one cannot write $G=G_1*G_2$ with $G_1,G_2$ non-trivial, and every group in $\calh$ contained in a conjugate of $G_1$ or $G_2$. 
  
  One says that $G$ is \emph{one-ended relative to $\calh$}\index{one-ended relative to $\calh$} if 
$G$ does not split over a finite group relative to $\calh$ (when $\calh$ is empty, this is equivalent to $G$ being finite or one-ended 
 by a theorem of Stallings, see e.g.\ \cite{ScottWall}).

\subsection{Maps between trees, compatibility, deformation spaces} \label{comp}

\subsubsection{Morphisms, collapse maps, refinements,   compatibility}\label{sec_morphisms}

Maps between trees will always be $G$-equivariant,  send vertices to vertices 
  and edges to edge paths (maybe a point).\index{map between trees}
 By minimality of the actions, they are always surjective; 
each
 edge of $T'$ is contained in the image of
an edge of $T$.
Any edge stabilizer $G_{e'}$ of $T'$ contains an edge stabilizer $G_e$ of $T$. 
Also note that any edge or vertex stabilizer of $T$ is contained in a vertex stabilizer of $T'$.

 We  mention  two particular classes of maps.
 
A map $f:T\ra T'$ between two trees is a \emph{morphism}\index{morphism between trees} 
if and only if one may subdivide
$T$  
so that $f$ maps each edge onto an edge;   equivalently,
no edge of $T$ is collapsed to a point. 
Folds are examples of morphisms (see   \cite{Stallings_topology,BF_bounding}).\index{fold} If  
$f$ is a morphism, any edge stabilizer of $T$ is contained in an edge stabilizer of $T'$.

A \emph{collapse}\index{collapse map} map $f:  T\to T'$ is a map obtained by collapsing  certain edges to points, 
  followed by an isomorphism
 (by equivariance, the set of collapsed edges is $G$-invariant).
Equivalently, $f$ \emph{preserves alignment}:\index{alignment preserving map}
 the image of any arc $[a,b]$ is a point or  the arc $[f(a),f(b)]$. Another characterization is that the preimage of every subtree is a subtree.  
 In terms of graphs of groups, one obtains $T'/G$ by collapsing edges in $T/G$.   If $T'$ is irreducible, so is $T$. If $T$ is irreducible, one easily checks that  $T'$ is  trivial (a point) or irreducible  (compare Lemma \ref{po}).

A tree $T'$ is a \emph{collapse}\index{collapse of a tree} of $T$ if there is a collapse map  $T\ra T'$; conversely, we say that $T$ \emph{refines}\index{refinement} $T'$. 
In terms of graphs of groups, one passes from $\Gamma=T/G$ to $\Gamma'=T'/G$ by collapsing edges;
  for each vertex $v$  of $\Gamma'$, 
the vertex group $G_{v}$ is the fundamental group of the graph of groups $\Gamma_{v}$ occurring as the preimage of $v$ in $\Gamma$.

Conversely, suppose that   $v$ is a vertex  of a splitting $\Gamma'$, and $\Gamma_{v}$ is  a splitting of $G_{v}$ in which   incident edge groups are elliptic. One may then refine $\Gamma'$ at $v$ using $\Gamma_{v}$, so as to obtain a splitting $\Gamma$ whose edges are those of $\Gamma'$ together with those of $\Gamma_{v}$ (see Lemma \ref{extens}). Note that $\Gamma$ is not uniquely defined because there is flexibility in the way edges of $\Gamma'$ are attached to vertices of $\Gamma_{v}$;  this is   discussed in    \cite[Section 4.2]{GL_vertex}.

  Two trees $T_1,T_2$ are \emph{compatible}\index{compatible trees} if they have a \emph{common refinement}\index{common refinement}: there exists a tree $\hat T$
with   collapse maps $g_i:\hat T\to T_i$. There is such a $\hat T$ with the additional property that no edge of $\hat T$ gets collapsed in both $T_1$ and $T_2$ (this is discussed in  Subsection \ref{sec_arith}).

\subsubsection{Domination, deformation spaces  \cite{For_deformation, GL2} }
\label{defspa}

A  tree $T_1$ \emph{dominates}\index{domination}
  a tree $T_2$ if there is an equivariant map   $f:T_1\ra T_2$ from $T_1$ to $T_2$.
  We call $f$
a \emph{domination map}.\index{domination}
Equivalently, $T_1$ dominates $T_2$ if every vertex stabilizer of $T_1$ fixes a point in $T_2$: every subgroup which is elliptic in $T_1$ is also elliptic in $T_2$.
  In particular, every refinement of $T_1$ dominates $T_1$. 
Beware that domination is defined by considering ellipticity of subgroups, not just of elements (this may make a difference if vertex stabilizers are not finitely generated).

 Deformation spaces are defined by saying that two  trees   belong to the same \emph{deformation space}\index{deformation space} $\cald$   if   they have the same elliptic subgroups (\ie each one dominates the other).  When we restrict to $\cala$-trees, we say that $\cald$ is a deformation space over $\cala$. If a tree in $\cald$ is irreducible, so are all others, and we say that $\cald$ is \emph{irreducible}.\index{irreducible tree, deformation space} 

 For instance, all   trees with a free action of $\F_n$ belong to the same deformation space $\cald=CV_n$, Culler-Vogtmann's outer space\index{outer space} \cite{CuVo_moduli}. Note, however, that only finitely many trees  are compatible with a given $T\in CV_n$. 

We have defined deformation spaces as combinatorial objects, but (just like outer space) they may be viewed as geometric objects (see e.g.\ \cite{GL2}).
We will not use this point of view.

A deformation space $\cald$  {dominates}\index{domination} a space $\cald'$ if trees in $\cald$ dominate those of $\cald'$.  Every deformation space dominates 
the deformation space of the trivial tree, which  is called the \emph{trivial deformation space.}\index{trivial deformation space} It is the only deformation space in which $G$ is elliptic.

A tree $T$ is \emph{reduced}\index{reduced tree} \cite{For_deformation} if no proper collapse of $T$ lies in the same deformation space as $T$ (this is different from being reduced in the sense of \cite{BF_bounding}).
 Observing that the inclusion from $G_v$ into $G_u*_{G_e}G_v$ is onto if and only if the inclusion $G_e\to G_u$ is, one sees that $T$ is reduced if and only if, whenever
 $e=uv$ is an edge with $G_e=G_u$, then $u$ and $v$ belong to the same orbit (i.e.\ $e$ projects to a loop in $\Gamma=T/G$). Another characterization is that,
  for any edge $uv$ 
such that $\grp{G_u,G_v}$ is elliptic, there exists a hyperbolic element $g\in G$ sending $u$ to $v$
(in particular the edge maps to a loop in $\Gamma$).

If $T$ is not reduced, one obtains a reduced tree $T'$ in the same deformation space by  collapsing  certain  orbits of edges ($T'$ is not uniquely defined in general).

\subsection{Slenderness, smallness} \label{sec_slender}

\subsubsection{Slenderness}

A group $H$ is \emph{slender}\index{slender group} if $H$ and all its subgroups are finitely generated.
Examples of slender groups include finitely generated virtually abelian groups, finitely generated virtually nilpotent groups,
and virtually polycyclic groups. A slender group cannot contain a non-abelian free group.

Slender groups have the   characteristic property that, whenever they act on a tree,   they fix a point 
or there is  an invariant line.

\begin{lem}[\cite{DuSa_JSJ},\index{Dunwoody-Sageev} Lemma 1.1]\label{lem_slender}
 Let $H$ be a slender group acting on a tree $T$.
If $H$ does not fix a point,   there is a  unique $H$-invariant line $\ell\subset T$.
\end{lem}

Since $H$ is finitely generated, there is a minimal subtree  $\mu_T(H)$, and with the terminology
 of Proposition \ref{cinq},  only cases (1), (2), (3) are possible  for  $\mu_T(H)$. 

\begin{proof}
The action of $H$ cannot be irreducible since $H$ does not contain $\F_2$. If there is no fixed point and no invariant line, there is a fixed end, and the associated $\chi:H\to\Z$ is non-trivial because $H$ is finitely generated. Each element of $\ker\chi$ fixes a ray going to the fixed end. Being finitely generated,   $\ker\chi$ is elliptic  by Serre's lemma. Its fixed point set is a subtree $T_0$, which is $H$-invariant because $\ker\chi$ is normal in $H$. The action of $H$ on $T_0$ factors through an action of the cyclic group $H/\ker\chi$, so $T_0$ contains an $H$-invariant line.
Uniqueness follows from Corollary \ref{pasirr}.
\end{proof}

It is convenient to use this lemma to define a weaker notion for subgroups of $G$. We say that a subgroup $H\inc G$ (possibly infinitely generated)  is \emph{slender in $(\cala,\calh)$-trees}\index{slender in $(\cala,\calh)$-trees} if, whenever $G$ acts on  an $(\cala,\calh)$-tree $T$, there is a point fixed by $H$ or an $H$-invariant line. In particular, any slender group,  any group contained in a group of $\calh$,     any group with property (FA), 
is slender in $(\cala,\calh)$-trees.

The following lemma will be used in Subsection \ref{slt}.
\begin{lem} \label{extslen}
  Let $A\normal A'$ be subgroups of $G$, with $A'/A$   slender 
and $A$ slender in  $(\cala,\calh)$-trees.
Then $A'$ is slender in  $(\cala,\calh)$-trees.
\end{lem}

\begin{proof}
Let $T$ be an  $(\cala,\calh)$-tree. If $A$ is elliptic, its fixed point set is an $A'$-invariant subtree  because $A$ is normal in $A'$. The action of 
  $A'/A$ on this subtree  fixes a point or leaves a line invariant,   so the same is true for  $A'$. If $A$ is not elliptic, it   preserves
a unique line, which is $A'$-invariant since  $A\normal A'$. 
\end{proof}

\subsubsection{Smallness}\label{pti}

One  defines an abstract group as being \emph{small}\index{small, small in $(\cala,\calh)$-trees} if it does not contain $\F_2$. Such a group cannot act irreducibly on a tree. As above, we   use trees to give a weaker  definition   for subgroups of $G$.  In particular, we want groups of $\calh$ to be small.

Given a tree $T$ on which $G$ acts, we say (following \cite{BF_bounding} and \cite{GL2}) that a subgroup $H<G$ is \emph{small in $T$} if its action on $T$ is not irreducible. As mentioned above, $H$ then fixes a point, or an end, or leaves  a line invariant (see Corollary \ref{pasirr}). We say that $H$ is
\emph{small in $(\cala,\calh)$-trees}\index{small, small in $(\cala,\calh)$-trees} if  it is small in    every $(\cala,\calh)$-tree  on which $G$ acts. 
Every subgroup not containing $\F_2$, and every group contained in a group of $\calh$,  
is small in $(\cala,\calh)$-trees. Moreover, $H$ is small in $(\cala,\calh)$-trees if and only if all its finitely generated subgroups are.

 \subsection{Accessibility} \label{accessi}

Constructions of JSJ decompositions are based on accessibility\index{accessibility} theorems stating that, given  suitable $G$ and $\cala$, there is 
an a priori 
bound for the number of orbits of edges of $\cala$-trees,  
under the assumption that there is no redundant vertex (if $v$ has valence 2, it is the unique fixed point of some $g\in G$).
This holds in particular:
\begin{enumerate} 
\item if $G$ is finitely generated and all groups in $\cala $ are finite with bounded order \cite{Linnell};

\item if $G$ is finitely presented and all groups in $\cala $ are finite \cite{Dun_accessibility};
\item if $G$ is finitely presented, all groups  in $\cala $ are small, and the trees are  \emph{reduced} in the sense of \cite{BF_bounding};
\item if $G$ is finitely generated and the trees are $k$-acylindrical for some $k$ \cite{Sela_acylindrical};\index{acylindrical accessibility}
\item if $G$ is finitely generated and the trees are $(k,C)$-acylindrical  \cite{Weidmann_accessibility} (\cite{Delzant_accessibilite} for finitely presented groups). 
\end{enumerate}

 A tree is \emph{$k$-acylindrical}\index{acylindrical tree, splitting} (resp.\  \emph{$(k,C)$-acylindrical}) if 
 the pointwise  stabilizer of any segment of length $>k$ is trivial 
(resp.\  has
  order $\le C$). 
 
In this paper, we use a version of Dunwoody's accessibility given in \cite{FuPa_JSJ} (see Proposition \ref{prop_accessibility}).
 In Section \ref{jsjac} we use acylindrical accessibility.

\subsection{Relative finite generation and presentation}\label{sec_prelim_relatif}

As mentioned above, we always assume that $G$ is finitely generated, or finitely presented. 
  However, these properties
  are not always inherited by vertex groups. We therefore consider 
\emph{relative}\index{relative finite generation, presentation} finite generation (or presentation), which behave better in that respect (see Subsection \ref{fingv}).

  Let $G$ be a group with  a finite family  of subgroups $\calh=\{H_1,\dots,H_p\}$.

\begin{dfn}
One says that $G$ is \emph{finitely generated relative to $\calh$}
if there exists a finite set $\Om\subset G$ 
such that $G$ is generated by $\Om \cup H_1\dots\cup H_p$. Such a subset $\Om $ is a \emph{relative generating set.}\index{relative generating set}
\end{dfn}

Clearly, if $G$ is finitely generated, then it is finitely generated relative to any $\calh$.  If the $H_i$'s are finitely generated, relative finite generation is equivalent to finite generation.

By adding the conjugators to $\Om $, one sees that relative finite generation does not change if one replaces the subgroups $H_i$ by conjugate subgroups.

As in Subsection \ref{tre}, we have:
\begin{prop}  \label{arbtf}
Suppose that $G$ is finitely generated relative to $\calh$, and   acts   on a tree $T$    relative to $\calh$. If there is no global fixed point,   then $G$ contains hyperbolic elements,   there is a   unique  minimal invariant subtree $\mu_T(G)$, and the  quotient  $\mu_T(G)/G$ is a finite graph.  \qed 
\end{prop}

Recall that $T$ is relative to $\calh$ if every $H_i$ is elliptic in $T$.

 We now consider relative finite presentation\index{relative finite generation, presentation} (see \cite{Osin_relatively}). Note that, 
if $\Om $ is a relative finite generating set, then the natural morphism   ${\F(\Om) }*H_1*\dots *H_p\ra G$ is an epimorphism (with $\F(\Om)$ the free group on $\Om$).

\begin{dfn}\label{dfn_relfp}
One says that $G$ is \emph{finitely presented relative to $\calh$}
if there exists a finite relative generating set $\Om \subset G$, 
such that the kernel of the epimorphism  $\F(\Om)*H_1*\dots *H_p\onto G$
is normally generated by a finite subset $\calr\subset \F(\Om)*H_1*\dots *H_p$.
\end{dfn}

 In particular, any group which is   hyperbolic   with respect to a finite family $\calh$ is finitely presented relative to $\calh$ 
\cite{Osin_relatively}.

One easily checks that relative finite presentation does not depend on the choice of $\Om$,
and is not affected if one replaces the subgroups $H_i$ by conjugate subgroups.
If $G$ is finitely presented, then it is finitely presented relative to any
finite collection of \emph{finitely generated} subgroups.
Note, however, that the free group $\F_2$ is \emph{not} finitely presented relative to an infinitely generated free subgroup $H_1$ (only finitely many generators of $H_1$ may appear in $\calr$).
 Conversely, if $G$ is finitely presented relative to any
finite collection of \emph{finitely presented} subgroups, then $G$ is finitely presented.

 The following lemma will be used in Subsection \ref{fingv}.

\begin{lem} \label{menage}
Suppose that $G$ is finitely presented relative to $\{H_1,\dots,H_p\}$, and $H_p\inc H_i$  for some  $i<p$. 
Then $G$  is finitely presented relative to $\{H_1,\dots,H_{p-1}\}$.
\end{lem}

\begin{proof}
We show that $H_p$ is finitely generated. The lemma then follows by applying Tietze transformations.

Let $x_1,\dots,x_q$ be the set of elements of $H_p$ which appear as a letter in one of the relators of $\calr$ 
(expressed as elements of $\F(\Om)*H_1*\dots *H_p$). Each $x_j$ is equal to an element $y_j\in H_i$. We define a new finite set of relators   by replacing each $x_j$ by $y_j$  and  adding the relations $x_j=y_j$. 
This new presentation expresses $G$ as the amalgam of $\langle \Om ,H_1,\dots,H_{p-1}\rangle $ with $H_p$ over a finitely generated group $H$. The inclusion $H_p\inc H_i$ now implies $H_p=H$.
\end{proof}

 Suppose that  a finitely generated group $G$ splits as a finite graph of groups. It is well-known 
  (see for instance \cite{Cohen_combinatorial}, Lemma 8.32 p.\ 218)
 that vertex groups are finitely generated if   one assumes that  edge groups are   finitely generated, but this is false  in general without this assumption.
However, vertex groups  are always
  \emph{finitely generated  relative to the incident edge groups}, and  there is a similar statement for relative finite presentation (see  Subsection \ref{fingv}).

\part{The   JSJ deformation space}

 We start this  part
 by introducing standard refinements: if edge stabilizers of $T_1$ are elliptic in $T_2$, there is a tree $\hat T_1$ which refines $T_1$ and dominates $T_2$. We then define the JSJ deformation space, and we show that it exists under some finite presentability assumption. We give examples in cases when there is no flexible vertex (flexible vertices are the subject of Part \ref{part_QH}). We conclude this part by collecting a few useful facts. In particular, given a tree $T$, we discuss finite presentation of vertex stabilizers, and we relate their splittings relative to incident edge groups to splittings of $G$. We also explain why one may usually restrict to one-ended groups when studying JSJ decompositions. 

We fix a   finitely generated group $G$,   a family $\cala$ of subgroups of $G$  (closed under conjugating and taking subgroups), and another family $\calh$. All trees will be  minimal $(\cala,\calh)$-trees  (see Subsection \ref{AH}). 
Whenever we construct a new tree (for instance in   Propositions \ref{prop_refinement} and \ref{prop_accessibility}) we check that it is  a  minimal  $(\cala,\calh)$-tree.

\section{Definition and existence} \label{de}

\subsection{Standard refinements}
\label{sec_blowup}

Let $T_1,T_2$ be trees. 

\begin{dfn}[Ellipticity of trees] \label{et}
  $T_1$ is \emph{elliptic}\index{elliptic with respect to a tree} with respect to $T_2$ if every edge stabilizer of $T_1$ 
fixes a point in $T_2$.
\end{dfn}

 Note that $T_1$ is  {elliptic} with respect to $T_2$ whenever  there is a refinement $\Hat T_1$  of $T_1$ that dominates $T_2$  (see Subsection \ref{comp} for the difference between refinement and domination):
edge stabilizers of $T_1$ are elliptic in $\Hat T_1$, hence in $T_2$.
 We show a converse statement.

\begin{prop}\label{prop_refinement} \index{refinement}
If $T_1$ is elliptic with respect to $T_2$, there is  a tree $\Hat T_1$ with maps $p:\Hat T_1\ra T_1$ and $f:\Hat T_1\ra T_2$ such that:
\begin{enumerate}
\item $p$ is a collapse map; 
\item for each $v\in T_1$, the  restriction of $f$ to the subtree $Y_v=p\m(v)$ is injective.
\end{enumerate}
In particular:
\renewcommand{\theenumi}{(\roman{enumi})}
\renewcommand{\labelenumi}{\theenumi}
\begin{enumerate}
\item \label{it_ref} $\Hat T_1$ is a refinement of $T_1$ that dominates $T_2$;
\item \label{it_Ge} the stabilizer of any edge of $\Hat T_1$ fixes an edge in $T_1$ or in $T_2$;
\item \label{it_Ge2} every edge stabilizer of $T_2$ contains an edge stabilizer of $\Hat T_1$;
\item \label{it_ell} a subgroup of $G$ is elliptic in $\Hat T_1$ if and only if it is elliptic in both $T_1$ and
$T_2$.
\end{enumerate}

\end{prop}

Assertions \ref{it_Ge} and \ref{it_ell} guarantee that $\Hat T_1$ is an $(\cala,\calh)$-tree   since $T_1$ and $T_2$ are.

\begin{rem}\label{rem_fold}
   If edge stabilizers of $T_2$ are finitely generated, then $T_2$ can be obtained from
$\Hat T_1$ by a finite number of collapses and folds \cite{BF_bounding}.
\end{rem}

\begin{proof}
We construct $\Hat T_1$ as follows.

For each vertex $v\in V(T_1)$, with stabilizer $G_v$, choose any
$G_v$-invariant subtree $\Tilde Y_v$ of $T_2$ (for instance, $\Tilde Y_v$ can be a minimal $G_v$-invariant subtree,
or the whole of  $ T_2$).
For each   edge $e=vw\in E(T_1)$, 
choose     vertices $p_v\in \Tilde Y_{v}$ and $p_w\in \Tilde Y_w$ fixed by $G_e$; this is possible
because
$G_e$ is elliptic in
$T_2$ by assumption, so has a fixed point in any  $G_v$-invariant subtree.   We make these choices $G$-equivariantly.

We can now define a  tree $\Tilde T_1$
by blowing up each vertex $v$ of $T_1$ into $\Tilde Y_v$, and attaching     edges   of $T_1$ using the
points $p_v$.   Formally, we consider the disjoint union $( \Dunion_{v\in V(T_1)}\Tilde Y_v) \cup (\Dunion _{e\in E(T_1)}e)$, and for
each edge $ e=vw$ of
$T_1$ we identify $v$ with $p_v\in \Tilde Y_v$ and $w$ with $p_w\in\Tilde  Y_w$.
We define $\Tilde p:\Tilde T_1\ra T_1$ by sending $Y_v$ to $v$, and sending $e\in E(T_1)$ to itself.  We also define a map 
 $\Tilde f:\Tilde T_1\ra T_2$  as equal to the inclusion $\Tilde Y_v\hookrightarrow T_2$   on $\Tilde Y_v$, and   sending the edge $e=vw$ to the segment $[p_v,p_w]\subset T_2$.

In general, $\Tilde T_1$ may fail to be minimal, so we define $\Hat T_1\subset \Tilde T_1$
as the unique minimal $G$-invariant subtree $\mu_{ \Tilde T_1}(G)$ (the action of $G$ on $\Hat T_1$ is non-trivial unless $T_1$ and $T_2$ are both points).
We then  define $p$ and $f$ as the restrictions of $\Tilde p$ and $\Tilde f$ to $\Hat T_1$.
These maps clearly satisfy the  first two requirements.

Let us check that the other properties follow.
Assertion \ref{it_ref} is clear.

If $e$ is an edge of $\Hat T_1$ that is not collapsed by $p$, then $G_e$ fixes an edge of $T_1$.
Otherwise, $f$ maps $e$ injectively to a non-degenerate segment of $T_2$, so $G_e$ fixes an edge in $T_2$, and Assertion \ref{it_Ge} holds.

 Assertion \ref{it_Ge2} is true for any surjective map $f$ between   trees.

To prove  the non-trivial direction of Assertion \ref{it_ell}, assume that $H$ is elliptic in $T_1$ and $T_2$. Then $H\subset G_v$ for some $v\in T_1$,
so $H$ preserves the subtree $Y_v\subset \Hat T_1$. 
Since $f$ is
 injective in restriction to $Y_v$, it is enough to prove
that $H$ fixes a point in $f(Y_v)$. This holds because $H$ is elliptic in $T_2$.
\end{proof}

\begin{rem}\label{rem_bas}
One may think of this construction in terms   of graphs of groups, as follows. Starting from the graph of group $\Gamma_1=T_1/G$,
one replaces each vertex $v\in \Gamma_1$ by the graph of groups $\Lambda_v$ dual to the action of $G_v$ on its minimal subtree
in $T_2$, and one attaches each edge $e$ of $\Gamma_1$ incident to $v$ onto a vertex of $\Lambda_v$ whose group contains a conjugate of $G_e$.

 Since any vertex group $G_v$ of $\Gamma_1$ is  finitely generated relative to its incident edge groups (see Lemma \ref{relfg}), and   these edge groups are elliptic in $T_2$,
there is a minimal $G_v$-invariant subtree in $T_2$  by Lemma \ref{arbtf}. Thus, one may require that $G_v$  act minimally on the preimage $Y_v\inc \Hat T_1$.

\end{rem}

\begin{dfn}[Standard refinement]\label{dfn_std_blowup}\index{standard refinement}
Any tree $\Hat T_1$
  as in   Proposition \ref{prop_refinement} 
 will be called a
 \emph{standard refinement of $T_1$ dominating $T_2$.}

\end{dfn}

In general, there is no uniqueness of standard refinements.
However, by Assertion \ref{it_ell}   of Proposition \ref{prop_refinement}, all standard refinements belong to the same deformation space, which is 
the
lowest deformation space dominating the deformation spaces containing $T_1$ and $T_2$ respectively. If $T_1$ dominates $T_2$ (resp.\ $T_2$ dominates $T_1$),
then
$\hat T_1$ is in the same deformation space as $T_1$ (resp.\ $T_2$).
Moreover,  there is some
symmetry: if $T_2$ also happens to be elliptic with respect to $T_1$,  then   any standard refinement $\Hat T_2$ of $T_2$ dominating $T_1$ is in the
same deformation space as $\Hat T_1$.

\begin{lem} \label{cor_Zor} \ 
\begin{enumerate}
\item  
If $T_1$ refines $T_2$ and does not belong to the same deformation space, some $g\in G$ is hyperbolic in $T_1$ and elliptic in $T_2$.
\item

If  $T_1$ is elliptic with respect to  $T_2$, and every $g\in G$ which is elliptic in $T_1$ is also elliptic in $T_2$, then $T_1$ dominates $T_2$.  

\item
If $T_1$ is elliptic with respect to  $T_2$, but 
 $T_2$ is not elliptic with respect to $T_1$, 
 then $G$ splits over a group    which has infinite index in an edge stabilizer of $T_2$.
\end{enumerate}
\end{lem}

 Recall that 
all trees   are assumed to be $(\cala,\calh)$-trees; 
  the splitting obtained in (3) is   relative to $\calh$.

\begin{proof} 
One needs only   prove the first assertion when $T_2$ is obtained from $T_1$ by collapsing the orbit of an edge $e=uv$.
If $u$ and $v$ are in the same orbit, or if $G_e\neq G_u$ and $G_e\neq G_v$, then some hyperbolic element of $T_1$
becomes elliptic in $T_2$. Otherwise, $T_1$ and $T_2$ are in the same deformation space  (see Subsection \ref{defspa}).

For the second assertion, assume that $T_1$ does not dominate $T_2$,  and let     $\Hat T_1$  be a standard refinement 
of $T_1$ dominating $T_2$. It
does not belong to the same deformation space as $T_1$. 
Since it is a refinement of $T_1$, we have just seen that some $g\in G$ is elliptic in $T_1$ and hyperbolic in $\Hat T_1$. By assumption $g$ is elliptic in $T_2$, contradicting  Assertion \ref{it_ell}  of Proposition
\ref{prop_refinement}.

For (3) (which is Remark 2.3 of \cite{FuPa_JSJ}),
let $\hat T_1$ be as in Proposition \ref{prop_refinement}.  Let $G_e$ be an edge stabilizer of $T_2$ which is
not elliptic in $T_1$.
It contains an edge stabilizer $J$ of  
$\hat T_1$.  Since $J$ is elliptic in $T_1$ and $G_e$ is not, the index of
$J$ in $G_e$ is infinite.
\end{proof}

\subsection{Universal ellipticity }

\begin{dfn}[Universally elliptic]
A subgroup $H\subset G$ is \emph{universally elliptic}\index{universally elliptic subgroup, tree}   if 
it is elliptic in every  
\AH-tree.
A tree $T$ is \emph{universally elliptic}\index{universally elliptic subgroup, tree} if its edge stabilizers are universally elliptic, 
\ie if $T$ is elliptic with respect to every 
\AH-tree. 
\end{dfn}

 When we need to be specific, we say universally elliptic over $\cala$ relative to $\calh$, or $(\cala,\calh)$-universally elliptic. Otherwise we just say universally elliptic, recalling that all trees are assumed to be $(\cala,\calh)$-trees.

Groups with Serre's property (FA), in particular finite groups,  
are universally elliptic. If $H$ is universally elliptic and $H'$ contains $H$ with finite index, then $H'$ is universally elliptic.

\begin{lem}\label{lem_sup} Consider  two trees $T_1,T_2$.
\begin{enumerate}
\item If $T_1$ is universally elliptic,
then some refinement  of $T_1$ dominates $T_2$.
\item   If $T_1$ and $T_2 $ are universally elliptic, any standard refinement  $\Hat T_1$ of $T_1$ dominating $T_2$
 is  universally elliptic.  In particular, there is a universally elliptic tree $\hat T_1$ dominating both $T_1$ and $T_2$. 

\item If $T_1$ and $T_2 $ are universally elliptic and have the same elliptic elements, they belong to the same deformation space.
\end{enumerate}
\end{lem}

\begin{proof}   The first two assertions follow directly from  
Assertions  (i) and (ii)  of Proposition \ref{prop_refinement}.
The last one follows from the second assertion of Lemma \ref{cor_Zor}.
\end{proof}

The following lemma will be used in Subsection \ref{exdur}. 

\begin{lem} \label{lem_Zorn} 
Let $(T_i )_{i \in I}$ be any family of   trees. There exists a countable subset
$J\subset I$ such that, if 
$T$ is elliptic  with respect to  every $T_i$ ($i\in I$),
and $T$ dominates every $T_{j}$ for $j \in J$, 
then $T$ dominates $T_i $ for all $i\in I$.
\end{lem}

\begin{proof} Since $G$ is countable, we can   find a countable  $J$ such that, if an element $g\in G$ is
hyperbolic in some $T_i$, then it is hyperbolic in some $T_{j}$ with $j \in J$. 
If $T$ dominates every $T_{j }$ for $j \in J$, any $g$ which is elliptic in $T$ is elliptic in every $T_i$. 
By  (2) of Lemma  \ref{cor_Zor}, the tree $T$ dominates every $T_i$. 
\end{proof}

For many purposes, it is enough to consider one-edge splittings, \ie
trees with only one orbit of edges.\index{one-edge splitting}

\begin{lem}\label{lem_oneed}
 Let $S$ be a
 tree. \begin{enumerate}
\item
$S$ is universally elliptic if and only if
it is elliptic with respect to every one-edge splitting.
\item
$S$ dominates every universally elliptic 
tree if and only if it dominates every universally elliptic one-edge
splitting.
\end{enumerate}
\end{lem}

\begin{proof} For the non-trivial direction, one proves that $S$ is elliptic with respect to $T$ (resp.\ dominates $T$)
by induction on the  number of orbits of
edges of $T$, using the following lemma.
\end{proof}

\begin{lem}
Let $T$ be a  tree, and $H$ a subgroup of $G$.
  Let $E_1\dunion E_2$ be a partition of $E(T)$ into two $G$-invariant sets.
Let $T_1,T_2$ be the trees obtained from $T$ by collapsing $E_1$ and $E_2$ respectively.
\begin{enumerate}
\item If a subgroup $H$ is elliptic in $T_1$ and $T_2$, then $H$ is elliptic in $T$.
\item If a tree $T'$ dominates $T_1$ and $T_2$, then it dominates $T$.
\end{enumerate}
\end{lem}

\begin{proof}
Let $x_1\in T_1$ be a vertex fixed by $H$. Let $Y\subset T$ be its preimage under the collapse map $T\ra T_1$. It is a subtree.
Now $Y$ is $H$-invariant and embeds into $T_2$.
Since $H$ is elliptic in $T_2$, it fixes a point in $Y$, so is elliptic in $T$.  One shows (2) 
by applying (1) to the vertex stabilizers of $T'$.
\end{proof}

\subsection{The JSJ deformation space} \label{jsjdf}

Having  fixed $\cala$, 
  we define $\cala_\elli\subset \cala$\index{0AE@$\cala_\elli$: universally elliptic groups in $\cala$} 
as the set of   groups in $\cala$ which are
universally elliptic (over $\cala$ relative to $\calh$); $\cala_\elli$ is stable under conjugating and taking subgroups. 
A tree is universally elliptic if and only if it is an $\cala_\elli$-tree.

\begin{dfn}[JSJ deformation space]\label{prop_def}
If there exists a  deformation space $\cald_{JSJ}$\index{0DJSJ@$\cald_{JSJ}$: the JSJ deformation space} 
of $(\cala_\elli,\calh)$-trees which is maximal for  
domination, it is unique  by   
the second assertion of Lemma \ref{lem_sup}.
It  is called \emph{the JSJ deformation space}\index{JSJ decomposition, tree, deformation space} of $G$ over $\cala$ relative to $\calh$.

Trees in $\cald_{JSJ}$ are called \emph{JSJ trees}\index{JSJ decomposition, tree, deformation space}  (of $G$ over $\cala$ relative to $\calh$). 
They are precisely those
trees $T $ which are universally
elliptic, and   which dominate every universally elliptic tree.  We also say that trees $T\in\cald_{JSJ}$, and  the associated graphs of groups $\Gamma=T/G$, are \emph{JSJ decompositions}.\index{JSJ decomposition, tree, deformation space}
\end{dfn}

We will show that the JSJ deformation space exists if $G$ is finitely presented (Theorems \ref{thm_exist_mou}  and \ref{thm_exist_mou_rel}), or in the presence of acylindricity (see Section \ref{jsjac}). See Subsection \ref{Dunw} for an example where there is no JSJ deformation space.

   In general there are
many JSJ trees, but they  all belong to    the same deformation space
and therefore have a lot in common (see Section 4 of \cite{GL2}). In
particular \cite[Corollary 4.4]{GL2},  they have the same vertex stabilizers, except possibly  for 
vertex stabilizers in $\cala_{\elli}$.

\begin{rem} \label{moves}
There are results saying that two trees belong to the same deformation space $\cald$ if and only if one can pass from one to the other by a finite sequence of moves of a certain type (\cite{For_deformation}, see also \cite{ClFo_Whitehead}); in particular, if $\cald$ is non-ascending as defined in   \cite[Section 7]{GL2}, 
for instance when all groups in $\cala$ are finite, any two reduced trees in $\cald$ may be joined by a finite sequence of slide moves. These results may be interpreted as saying that a JSJ tree is unique up to certain moves. This is the content of the uniqueness statements of
\cite{Sela_structure,DuSa_JSJ,For_uniqueness,FuPa_JSJ}.
\end{rem}

\begin{dfn} [Rigid and flexible vertices] Let $H=G_v$ be a vertex stabilizer   of  a JSJ  tree $T$ (or a vertex group  of the graph of groups $\Gamma=T/G$).
We say that $H$ is \emph{rigid}\index{rigid vertex, group, stabilizer} 
if it is universally elliptic,
\emph{flexible}\index{flexible vertex, group, stabilizer} if it is not. 
We also say that the vertex $v$ is rigid  (flexible). If $H$ is flexible, we say that it  is a  flexible  subgroup of $G$
(over $\cala$ relative to $\calh$).
 \end{dfn} 

The definition of  flexible subgroups of $G$ does not
depend on the choice of the  JSJ tree $T$.
The heart of JSJ  theory is to understand   flexible groups.   They will be discussed in
Part  \ref{part_QH}.

We record the following simple facts  for future reference.

\begin{lem}\label{lem_rafin}
 Let  $T$  be a JSJ tree, and $S$ any tree.
  
  \begin{itemize}
  \item   There is a tree $\hat T$ which refines $T$ and dominates $S$.
  
  \item If $S$  is 
universally elliptic,
it  may be refined  to a JSJ tree.
\end{itemize}
\end{lem}

\begin{proof} 
  Since $T$ is elliptic with respect to $S$, one can construct a standard refinement $\hat T$ of $T$ dominating $S$ 
(Proposition \ref{prop_refinement}). It satisfies the first assertion.

For the second assertion, since $S$ is elliptic with respect to $T$,  we can consider  a standard refinement $\hat S$ of $S$ dominating $T$.
It is universally elliptic  by the second assertion of  Lemma \ref{lem_sup},
and dominates $T$, so it is a JSJ tree.
\end{proof}

There sometimes exists a universally compatible JSJ tree (see Sections \ref{exam}  and  \ref{exemp}). In this case, one may require that $\hat T$ also be a refinement of $S$.

\subsection{Existence of the JSJ deformation space: the non-relative case}

We prove the existence of JSJ decompositions,  first assuming $\calh=\es$.  

  \begin{thm}\label{thm_exist_mou}
   If $G$ is finitely presented, then the JSJ deformation
    space $\cald_{JSJ}$ of $G$ over $\cala$ exists. It contains a tree
    whose edge and vertex stabilizers are finitely generated.
  \end{thm}

There is no   hypothesis, such as smallness or finite generation, on
the elements of $\cala$.
Recall that, $G$ being finitely generated, finite generation of edge stabilizers implies  finite generation of vertex stabilizers.

The existence of $\cald_{JSJ}$ will be deduced from the
following version of   Dunwoody's accessibility,  whose proof will be  given  in the next subsection.

\begin{prop}[Dunwoody's accessibility]\label{prop_accessibility}\index{accessibility (Dunwoody)}
  Let $G$ be finitely presented.
  Assume that $T_1\leftarrow\dots\leftarrow T_k \leftarrow T_{k+1}\leftarrow \dots$
is a sequence of \emph{refinements} of 
trees.
There exists a
tree 
$S $ such that:
\begin{enumerate}
\item for   $k$ large enough, there is a morphism $S\ra T_k$ (in particular, $S$ dominates $T_k$);
\item  each edge and vertex  stabilizer of $S$ is finitely generated.
\end{enumerate}
\end{prop}

Note that the maps $T_{k+1}\ra T_k$ are required to be \emph{collapse maps}.

Recall (Subsection \ref{comp}) that $f:S\ra T_k$ is  a morphism if
$S$ may be  subdivided 
so that $f$ maps edges to edges
(a collapse map is not a morphism). In particular, edge stabilizers of $S $ fix an edge in $T_k$, so $S$ is an $\cala$-tree since  every $T_k$ is. It is universally elliptic if every $T_k$ is.

\begin{rem}
 Unfortunately, it is not true that the deformation space of $T_k$ must stabilize as $k$ increases, even if all edge stabilizers are cyclic. 
For example (as on pp.\ 449-450 of \cite{BF_bounding}), let $A$ be  a group with a sequence of nested infinite cyclic groups $A\supsetneqq  C_1\supsetneqq C_2\supsetneqq\dots$, let $G=A*B$ with $B$ non-trivial, and let $T_k$ be the Bass-Serre tree of the iterated amalgam 
$$G=A*_{C_1} C_1 *_{C_{2}} C_{2} *_{C_{3}} \cdots *_{C_{k-1}} C_{k-1} *_{C_k} \grp{C_k,B}.$$
These trees refine each other, but are not in the same deformation space since $\grp{C_k,B}$ is not elliptic in $T_{k+1}$.
They are dominated by the tree $S$ dual to the free decomposition $G=A*B$, in accordance with the proposition.
\end{rem}

Applying Proposition \ref{prop_accessibility} to a constant sequence
yields the following standard result:

\begin{cor}\label{cor_fg}
  If $G$ is finitely presented, and   $T$ is a
  tree,
  there exists  a morphism $f:S\ra T$ where $S$ is a
  tree   with finitely generated
edge and  vertex stabilizers.
\qed
\end{cor}

If $T$ is a universally elliptic $\cala$-tree, so is $S$.

Proposition \ref{prop_accessibility} is basically Proposition 5.12 of \cite{FuPa_JSJ}.
We omit its proof, and refer to the  more general Proposition \ref{prop_accessibility_rel}.

 \begin{proof}[Proof of Theorem \ref{thm_exist_mou}]
Let $\calu $ be the set of   universally elliptic
trees 
with finitely generated edge and vertex stabilizers, up to equivariant isomorphism.
It is non-empty since it contains the trivial tree. An element of $\calu$ is described by a
finite graph of groups with finitely generated edge and vertex groups. Since  $G$  only has
countably many finitely generated subgroups, and there are countably many 
 homomorphisms from a given finitely generated  group to another, the set $\calu$ is
countable.  

  By Corollary \ref{cor_fg},   every universally elliptic tree is dominated by one in $\calu$, so
it suffices  to produce a universally elliptic
  tree dominating every $U\in\calu$.
Choose an enumeration $\calu=\{U_1,U_2,\dots,U_k,\dots\}$.
 We define inductively a universally elliptic
 tree $T_k$ which refines $T_{k-1}$ and
dominates $U_1,\dots, U_k$ 
 (it may have infinitely generated edge or vertex stabilizers).
We start with $T_1=U_1$. Given $T_{k-1}$ which dominates $U_1,\dots,U_{k-1}$,
we  let $T_k$ be a  standard 
  refinement of $T_{k-1}$ which dominates $U_k$ (it exists by Proposition \ref{prop_refinement} because $T_{k-1}$ is universally elliptic).
Then $T_k$ is universally elliptic   by the second assertion of Lemma  \ref{lem_sup}, 
and it dominates $U_1,\dots, U_{k-1}$
because
$T_{k-1}$ does.

  Apply  Proposition \ref{prop_accessibility} to the sequence $T_k$. The tree $S$ is universally
elliptic 
and it  dominates every $T_k$, hence every $U_k$. 
It follows that $S$ is a JSJ tree over $\cala$.
\end{proof}

\subsection{Existence: the relative case} \label{rela}

In this section, we prove the existence of a relative JSJ deformation space under a relative finite presentation assumption.

\begin{thm}\label{thm_exist_mou_rel} 
Assume that  $G$ is finitely presented relative to $\calh=\{H_1,\dots,H_p\}$.
Then the JSJ deformation space $\cald_{JSJ}$ of $G$ over $\cala$ relative to 
$\calh$ exists.  It contains a tree with finitely generated edge stabilizers.
\end{thm}

Recalling that a  finitely presented group  is finitely presented relative to any
finite collection of  {finitely generated} subgroups, we get:

\begin{cor}\label{cor_exist_mou_rel} 
Let  $G$ be finitely presented. Let    $\calh=\{H_1,\dots,H_p\}$ be a finite family of finitely generated subgroups.
Then the JSJ deformation space $\cald_{JSJ}$ of $G$ over $\cala$ relative to 
$\calh$ exists.  It contains a tree with finitely generated edge  (hence  vertex) stabilizers. \qed
\end{cor}

\begin{rem}  We will give a different approach to existence in Subsection \ref{filling}, by comparing the relative JSJ decomposition of $G$ to a non-relative JSJ decomposition of a larger group $\hat G$.
\end{rem}

\begin{rem}    The corollary does not apply when groups in $\calh$  are   infinitely generated. See Section \ref{jsjac} for existence results with $\calh$   arbitrary. 
\end{rem}

The theorem is proved as in the non-relative case (Theorem \ref{thm_exist_mou}),  with $\calu$ the set  of $(\cala,\calh)$-trees with finitely generated edge stabilizers. It is countable because vertex stabilizers are relatively finitely generated  by {\cite[Lemmas 1.11, 1.12]{Gui_actions}} (this is explained in  Remark \ref{den}).
Proposition  \ref{prop_accessibility} is replaced by the following result.

\begin{prop}[Relative Dunwoody's accessibility]\label{prop_accessibility_rel}\index{accessibility (Dunwoody)}
  Let $G$ be finitely presented relative to $\calh=\{H_1,\dots, H_p\}$.
  Assume that $T_1\leftarrow\dots\leftarrow T_k \leftarrow T_{k+1}\leftarrow \dots$
is a sequence of refinements of 
\AH-trees.
There exists an \AH-tree $S $ such that:
\begin{enumerate}
\item for   $k$ large enough, there is a morphism $S\ra T_k$; 
\item  each edge stabilizer of $S$ is finitely generated.
\end{enumerate}
\end{prop}

As above, applying the proposition to a constant sequence, we get:

\begin{cor}\label{cor_fg_rel}
  If $G$ is finitely presented relative to $\calh$, and   $T$ is an
  \AH-tree,
  there exists a morphism $f:S\ra T$ where $S$ is an
  \AH-tree   with finitely generated
edge stabilizers.
\qed
\end{cor}

Before proving the proposition, we
recall that $G$ is finitely presented relative to its subgroups
  $H_1,\dots,H_p$ if there exists a finite subset $\Om\subset G$ such that
  the natural morphism  $\F{(\Om)}*H_1*\dots *H_p\to G$ is onto, and its
  kernel is normally generated by a finite subset $\calr$.

  Here is an equivalent definition: $G$ is finitely presented relative
  to $\calh$ if and only if  it is the fundamental group of a
 connected $2$-complex $X$  (which may be assumed to be simplicial) containing 
  disjoint connected subcomplexes $Y_1,\dots,Y_p$
  (possibly infinite) with the following properties:  $X\setminus (Y_1\cup\dots\cup Y_p)$ contains only finitely many open cells,
  $\pi_1(Y_i)$ embeds into $\pi_1(X)$, and its image is
  conjugate to $H_i$.

  The fact that $G$ is relatively finitely presented if such a space
  $X$ exists follows from Van Kampen's theorem.  Conversely, if $G$ is
  finitely presented relative to $\calh$, one can construct 
 $X$ as follows.   Let  $(Y_i,u_i)$ be a pointed
 $2$-complex with 
  $\pi_1(Y_i,u_i)\simeq H_i$.
Starting from the disjoint
  union of the $Y_i$'s, add $p$  edges joining the $u_i$'s to an  additional vertex $u$,
  and $\#\Om$
  additional edges joining $u$ to itself.   We get a complex whose  fundamental group is
  isomorphic to the free product  $\F(\Om)*H_1*\dots *H_p$.  Represent
  each element of $\calr$ by a loop in this space, and glue a
  disc
  along this loop to obtain the desired space $X$.

\begin{proof}[Proof of Proposition \ref{prop_accessibility_rel}]
  Let  $\pi:\Tilde X\ra X$ be the universal cover of a  simplicial  $2$-complex $X$ as above, with $G$ acting on $\Tilde X$ by deck transformations.
For $i\in\{1,\dots,p\}$, consider  a connected component $\Tilde Y_i$  of $\pi\m(Y_i)$ whose stabilizer is $H_i$.  Also fix lifts  $v_1,\dots,v_q\in \Tilde X$   of    all the vertices in $X\setminus Y_1\cup\dots\cup Y_p$.

We denote by $p_k:T_{k+1}\ra T_k$ the collapse map. Note that the preimage of the midpoint  of an edge $e_k$ of $T_k$ is a single point, namely the midpoint of the edge of $T_{k+1}$ mapping onto $e_k$.

 We shall now construct  equivariant maps $f_k:\Tilde X\ra T_k$
 such that 
$f_k$ maps each $\Tilde Y_i$ to a vertex fixed by $H_i$, sends each $v_j$ to a vertex, and 
sends each edge of $\Tilde X$  either to a point or injectively onto
a   segment in $T_k$. We further require that $p_k( f_{k+1}(x))=f_k(x)$ if $x$ is a vertex of $\Tilde X$ or if $f_k(x)$ is the
midpoint of an edge of $T_k$.

 We construct $f_k$   inductively.  We start with $T_0$   a point,  and   $f_0$, $p_0$ the 
   constant maps.
 We then assume that $f_k:\Tilde X\ra T_k$ has been constructed, and we construct $f_{k+1}$.

To define $f_{k+1}$ on $\Tilde Y_i$, note that $f_k(\Tilde Y_i) $ is a vertex of $T_k$ fixed by $H_i$. 
Since $p_k$ preserves alignment, $p_{k}\m( f_k(\Tilde Y_i) )$ is an $H_i$-invariant subtree.
Since $H_i$ is elliptic in $T_{k+1}$, it fixes some vertex   in this subtree,
and  we map   $\Tilde Y_i$ to  such a  vertex.
We then  define $f_{k+1}(v_j)$ as any vertex in $p_k\m(f_k(v_j))$, and we extend by equivariance.

Now consider an edge $e$ of $  X$  not contained in any $Y_i$, and a lift $\Tilde e\inc \Tilde X$. The map $f_{k+1}$ is already defined on the endpoints of $\Tilde e$, we explain how to define it on $\Tilde e$.

 The restriction of $p_k$ to the segment of $T_{k+1}$ joining the images of the endpoints of $\Tilde e$ is a collapse map. In particular, the preimage of the midpoint of an edge $e_k$ of $T_k$ is the midpoint of the edge of $T_{k+1}$ mapping onto $e_k$. Recalling that $f_k$ is constant or injective on $\Tilde e$,  this allows us to define $f_{k+1}$ on $\Tilde e$, as a map which is either constant or injective, and satisfies $p_k( f_{k+1}(x))=f_k(x)$ if   $f_k(x)$ is the
midpoint of an edge of $T_k$. Doing this equivariantly, we have now defined $f_{k+1}$ on the 1-skeleton of  $\Tilde X$.

We then extend $f_{k+1}$ in a standard way to every triangle $abc$ not contained in  a $\pi\m(Y_i)$; in particular, if $f_{k+1}$ is not constant on $abc$, preimages of midpoints of edges of $T_{k+1}$ are straight arcs joining two distinct sides. This completes the construction of the maps $f_k$.

We  now define $\Tilde \tau_k\subset \Tilde X$
as the preimage (under $f_k$) of the midpoints of all edges of   $T_k$. This is a \emph{pattern} in the sense of Dunwoody    \cite{Dun_accessibility}. It does not intersect any  $\Tilde Y_i$, and the maps $f_k$ were constructed so that  
$\Tilde \tau_k\subset \Tilde\tau_{k+1}$. We denote by $\tau_k=\pi(\Tilde \tau_k)$ the projection   in $X$. It is a finite graph because it is contained in the complement of $Y_1\cup \dots\cup Y_p$.

Let $S_k$ be the tree dual to the pattern $\tilde\tau_k$ in $\Tilde X$. We claim that it is an \AH-tree with finitely generated edge stabilizers. By construction, $f_k$ induces a map $\phi_k:S_k\ra T_k$ sending edge to edge, so  edge stabilizers of $S_k$ are  in $\cala$. They are finitely generated because they are generated by fundamental groups of components of $\tau_k$. Every $H_i$ is elliptic in $S_k$ because   $\Tilde Y_i$ does not intersect $\Tilde \tau_k$. This proves the claim.

Let $X'\subset X$ be the closure of the complement of $Y_1\cup\dots\cup Y_p$.
By construction, this is a finite complex. By \cite[Theorem 2.2]{Dun_accessibility},
there is a bound on the number of non-parallel tracks in $X'$.
This implies that there exists $k_0$ such that, for all $k\geq k_0$, 
for every connected component $\sigma$ of $\tau_k\setminus \tau_{k_0}$, there exists
a connected component $\sigma'$ of $ \tau_{k_0}$ such that $\sigma\cup \sigma'$ bounds 
a product region containing no vertex of $X'$. 
It follows that, for $k\geq k_0$,  one can  obtain   $S_{k}$ from $S_{k_0}$ by subdividing edges.
We then take  $S=S_{k_0}$.
\end{proof}

\subsection{Relation with other constructions} \label{autres}
\label{sec_gens}
Several authors have constructed JSJ splittings of finitely presented groups in various
settings. We explain here (in the non-relative case) why those splittings are JSJ splittings in the sense of Definition \ref{prop_def} (results in the literature are often
stated only for one-edge splittings, but this is not a restriction by  Lemma \ref{lem_oneed}).

\index{Rips-Sela}\index{Dunwoody-Sageev}\index{Fujiwara-Papasoglu}\index{Scott-Swarup} 
In \cite{RiSe_JSJ}, Rips and Sela consider cyclic splittings of a one-ended group $G$ (so
$\cala$ consists of all  cyclic subgroups of $G$,  including the trivial group). Theorem 7.1 in
\cite{RiSe_JSJ} says that their JSJ splitting  is universally elliptic (this is statement
(iv)) and maximal (statement (iii)).  The uniqueness up to deformation is statement (v).

In  the work of Dunwoody-Sageev  \cite{DuSa_JSJ}, the authors consider splittings of a  
group $G$ over slender subgroups in a class $\calz\calk$ 
such that $G$ does not split over finite
extensions of infinite index subgroups of $\calz\calk$ (there are restrictions on the class
$\calz\calk$,  but one can typically take $\calz\calk=VPC_n$,
see \cite{DuSa_JSJ} for details). 
In our notation, $\cala$ is the set of subgroups of elements of $\calz\calk$.
Universal ellipticity of the splitting they construct   follows from statement (3) in the Main Theorem of  \cite{DuSa_JSJ}, and from the fact that 
any edge group is contained in a white vertex group.
Maximality follows from the fact that white vertex groups are universally elliptic (statement (3)) 
and that black vertex groups   either are in $\calz\calk$
(in which case they are universally elliptic by the non-splitting assumption made on $G$), or are
$\calk$-by-orbifold groups  and hence are necessarily elliptic in any JSJ tree (see Proposition \ref{qhe} below). 

In \cite{FuPa_JSJ}, Fujiwara and Papasoglu consider all   splittings of a   group
over the class $\cala$ of its slender subgroups.
Statement (2) in \cite[Theorem 5.13]{FuPa_JSJ}  says that the   JSJ splitting they obtain  is elliptic
with respect to any   splitting  which is minimal (in their sense). By Proposition 3.7 in \cite{FuPa_JSJ}, any
splitting is dominated  by a minimal splitting, so universal ellipticity holds.
Statement (1) of Theorem 5.15   in \cite{FuPa_JSJ}  implies maximality.

As mentioned in the introduction, the regular neighbourhood of Scott-Swarup  \cite{ScSw_regular+errata}  is closer to the decompositions constructed in   
 Parts \ref{partacyl} and \ref{chap_compat}.

\section{Examples of JSJ decompositions} \label{exa}

Recall that we have fixed $\cala$ and $\calh$, and   we only consider  \AH-trees.  Unless otherwise indicated, $G$ is only assumed to be finitely generated. 

At the end of this section we shall give two examples of JSJ decompositions having flexible vertices, but most examples here will   have all vertices rigid.
  The fact that they are indeed  JSJ decompositions
will be a consequence of the following simple fact.
 
\begin{lem} \label{lem_rigid} 
Any 
tree $T$ with
universally elliptic \emph{vertex} stabilizers is a JSJ tree.
\end{lem}

\begin{proof} By assumption, $T$ dominates every
tree.
In particular, $T$ is universally elliptic and dominates every universally elliptic tree, so it is a JSJ tree.
\end{proof}

We also note:

\begin{lem}\label{geuel} 
  Assume that all groups in $\cala$ are
  universally elliptic.
If $T$ is a JSJ tree,
then its vertex stabilizers are universally elliptic. 
\end{lem}

This applies in particular to splittings over finite groups. 
 
\begin{proof}
If a vertex stabilizer $G_v$ of $ T$ is flexible,   consider $T'$ such that
$G_{v}$  is not elliptic in $T$. Since $T$ is a JSJ decomposition, it is universally elliptic,
so one can consider  a standard refinement $\Hat T$ of $T$ dominating $T'$.
By our assumption on $\cala$, the
tree $\Hat T$ is universally elliptic, so by definition of the JSJ deformation space $T$ dominates $\Hat T$.
This implies that $G_v$ is elliptic in $\Hat T$, hence in $T'$, a contradiction.
\end{proof}

\subsection{Free groups} \label{free}

Let $G=\F_n$ be a finitely generated free group,
  let $\cala$ be arbitrary, 
  and $\calh=\es$.
Then the JSJ deformation space of $\F_n$ over $\cala$ is the space of free actions (unprojectivized Culler-Vogtmann's 
outer space \cite{CuVo_moduli}).\index{outer space}

More generally, if $G$ is virtually free and 
$\cala$ contains all finite subgroups, then $\cald_{JSJ}$ is the space of trees with finite vertex
stabilizers.

\subsection{Free splittings: the Grushko deformation space} \label{freep}
Let $\cala$ consist only of the trivial subgroup
of $G$, and $\calh=\es$.   Thus $\cala$-trees are trees with trivial edge stabilizers, also called \emph{free splittings}.\index{free splitting}
Then the JSJ deformation space exists, it
is the outer space introduced in \cite{GL1} (see  \cite{CuVo_moduli} when $G=\F_n$, and \cite{McCulloughMiller_symmetric} when no free factor of $G$ is $\bbZ$). We call it the \emph{Grushko deformation space}.\index{Grushko decomposition, deformation space} 
It consists of trees $T$ such that
edge stabilizers are trivial, and vertex stabilizers are freely indecomposable and different from $\Z$ (one often considers 
$\bbZ$ as freely decomposable since it splits   as an HNN extension over the trivial group). 

 Denoting by 
$G=G_1*\dots*G_p*\F_q $   a  decomposition of $G$ given by Grushko's theorem (with $G_i$ non-trivial and
freely indecomposable, $G_i\ne\Z$, and $\F_q$ free), the quotient graph of groups $T/G$ 
 is homotopy equivalent to a wedge of $q$ circles; 
 it has one vertex 
 with group $G_i$ for each $i$,  and all other vertex groups are trivial  (see Figure \ref{fig_grushko} in the introduction). 

If $\calh\ne\es$, the JSJ deformation space is the Grushko deformation space relative to $\calh$. Edge stabilizers of JSJ trees are trivial, groups in $\calh$ fix a point, and vertex stabilizers are freely indecomposable relative to their subgroups which are conjugate to a group in $\calh$.

\subsection{Splittings over finite groups:  the Stallings-Dunwoody  deformation space}
\label{Dunw}
If $\cala$ is the set
of finite subgroups of $G$, and $\calh=\es$, we call   
the JSJ deformation space the \emph{Stallings-Dunwoody  deformation space}.\index{Stallings-Dunwoody deformation space} 
 It is the set of  trees whose edge groups are finite and whose vertex
groups have 0 or 1 end (one deduces from  Stallings's theorem that a
tree is maximal for domination if and only if its vertex stabilizers have at most one end).

The JSJ deformation space exists if $G$ is finitely
presented by Dunwoody's original accessibility result \cite{Dun_accessibility}. If $G$ is only finitely generated, it exists if and only if $G$ is accessible. In particular, the inaccessible group constructed in \cite{Dun_inaccessible} has no JSJ decomposition over finite groups. 

\begin{rem} \label{rem_linn}
Even if $G$ is inaccessible, there is a JSJ deformation space over $\cala$ if $\cala$ is a family of finite subgroups of bounded order.  The reason is that 
Proposition \ref{prop_accessibility} remains true, because  
 $T_{k+1}$ is just a subdivision of  $T_k$ for $k$ large by Linnell's accessibility 
\cite{Linnell}. 
\end{rem}

If $\calh\ne\es$, the JSJ deformation space is the Stallings-Dunwoody deformation space relative to $\calh$. Edge stabilizers are finite, groups in $\calh$ fix a point, and vertex stabilizers are one-ended  relative to their subgroups which are conjugate to a group in $\calh$
  (one-endedness is in the sense of Subsection \ref{AH}: they do not split over finite subgroups
relative to their subgroups which are conjugate to a group in $\calh$).

As above, 
the relative JSJ  space exists  if $G$ is finitely generated  and  $\cala$ consists of finite groups with bounded order (and $\calh$ is arbitrary).
If $\cala$ contains finite groups of arbitrary large  order, the JSJ space exists if   relative accessibility holds.

\subsection{Splittings of small  groups}\label{sec_G_small}

Recall that $G$ is small  in $(\cala,\calh)$-trees if $G$ has no irreducible action on a tree: 
there   always is a fixed point, or a fixed end, or an invariant line (see Corollary \ref{pasirr} and   Subsection \ref{pti}). \index{small, small in $(\cala ,\calh )$-trees}
  This is in particular the case if $G$ is small, \ie contains no non-abelian free group. 
If $G$ acts with a fixed end or an invariant line,
then every vertex stabilizer has a subgroup of index at most $2$ fixing an edge  (the index is 2 if $G$ acts dihedrally on  a line).

 \begin{lem} \label{sma}
 If $G$ is  small in $(\cala,\calh)$-trees,
 and $T$ is
a  non-trivial universally elliptic tree, then $T$ dominates every
tree. 
\end{lem}
 
 \begin{proof}
Since the action of $G$ on $T$ is non-trivial but is not irreducible, every vertex stabilizer   
contains an edge stabilizer with index at most 2.
It follows that every vertex stabilizer is universally elliptic, so $T$ dominates every
tree. 
\end{proof}

\begin{cor}  \label{sma2}
If $G$ is  small in $(\cala,\calh)$-trees, 
  there is at most one non-trivial deformation space containing a universally elliptic 
 tree. \qed
\end{cor}

In this situation, the JSJ deformation space always exists:
if there is a deformation space as in the corollary, it is the JSJ space; otherwise, the JSJ space is trivial. 

Consider for instance  (non-relative) cyclic splittings of  solvable Baumslag-Solitar groups\index{Baumslag-Solitar group}   $BS(1,n)=\grp{a,t\mid tat\m=a^n}$. 
If $n=1$ (so   $G\simeq\Z^2$), 
there are infinitely many deformation spaces (corresponding to epimorphisms $G\ra\bbZ$) and there is no non-trivial universally elliptic tree. 
If $n=-1$ (Klein bottle group), there are   exactly two  non-trivial deformation spaces: one contains the  Bass-Serre tree of 
the HNN extension $\grp{a,t\mid tat\m=a\m}$, the other contains the tree associated to  the amalgam $\grp{t}*_{t^2=v^2}\grp{v}$, with $v=ta$.
None of these trees is universally elliptic ($t$ and $v$ are hyperbolic in the HNN extension, and $a$ is hyperbolic in the amalgam).

  Thus for $n=\pm 1$ (when $G$ is $\Z^2$ or the Klein bottle group) the cyclic JSJ deformation space of $BS(1,n)$ is the trivial one, and $G$ is flexible (see Subsection \ref{pslflex}  for generalizations of these examples).
If $n\ne\pm1$, the JSJ space is non-trivial, as we shall now see.

 \subsection{Generalized Baumslag-Solitar groups}\label{gbs}

Let $G$ be a generalized Baumslag-Solitar group,\index{generalized Baumslag-Solitar group}
\ie a finitely generated group which acts on  a tree $T$ 
with all  vertex and edge stabilizers infinite cyclic.
Let $\cala$ be the set of cyclic subgroups of $G$ (including the trivial subgroup), and $\calh=\es$.
Unless $G$ is   isomorphic to  $\Z$, $\Z^2$, or   the   Klein bottle group,
  the deformation space of $T$  is  the JSJ deformation space \cite{For_uniqueness}.

Here is a short proof (the arguments are contained in \cite{For_uniqueness}). We show
  that every vertex stabilizer $H$ of $T$ is universally elliptic.
  The commensurator\index{commensurator} of $H$ is $G$ because
  the intersection of any pair of vertex stabilizers has finite index in both of them.
If $H$ acts hyperbolically in a
tree $T'$, 
its commensurator $G$ preserves its axis,
so $T'$ is a line. Edge stabilizers being cyclic,
vertex stabilizers are virtually cyclic, hence cyclic since $G$ is torsion-free. 
This implies that
$G$ is  $\Z$, $\Z^2$, or  a Klein bottle group.

\subsection{Locally finite trees} \label{lf}

We generalize the previous example to locally finite trees with small edge stabilizers. 

We suppose that 
$G$ acts irreducibly on a  locally finite tree $T$ with small edge stabilizers 
(local finiteness
is equivalent to edge stabilizers having finite index in neighboring vertex stabilizers; in particular,
vertex stabilizers are small).
In \cite[Lemma 8.5]{GL2}, we proved that all such trees $T$ belong
to the same deformation space. This happens to be the JSJ deformation space.

\begin{prop}  Suppose that all groups of $\cala$ are  small in $(\cala,\calh)$-trees.\index{small, small in $(\cala ,\calh )$-trees} 
Then any locally finite irreducible
tree $T$ belongs to the JSJ deformation space.
\end{prop}

\begin{proof} We show that every vertex stabilizer $H$ of  $T$ is universally elliptic.  Since $T$ is locally finite, $H$ contains an edge stabilizer with finite index, so is small (for simplicity we do not write ``small in $(\cala,\calh)$-trees'' in this proof).   By way of contradiction, assume that $H$
is not elliptic in some
tree $T' $.  Being small, $H$ fixes a unique end of $T'$ or preserves a unique line  (see Corollary \ref{pasirr}).   Any finite index subgroup of $H$ preserves the same unique end or line.

As in the previous subsection, local finiteness implies that 
  $G$ commensurates $H$, so $G$ preserves the $H$-invariant end or  line  of $T'$ (in particular, $T'$ is not
irreducible).
We now define a small normal subgroup $G'\inc G$.
If $G$ does not act dihedrally on  a line, there is a fixed end and we let $G'=[G,G]$ be the commutator subgroup.
It  is small because
any finitely generated subgroup pointwise fixes a ray of
$T'$, so is contained in an edge stabilizer $G_e\in\cala$.  If $G$ acts dihedrally, we let $G'$ be the kernel of the action, so that $G/G'$ is an infinite dihedral group.

 Consider the action of the normal subgroup $G'$ on $T$. 
If it is elliptic,   its fixed point set  is $G$-invariant, so by minimality
the action of $G$ factors through the action of an abelian  or dihedral group; this contradicts the irreducibility of
$T$.  Otherwise  $G'$ preserves a unique end or  line; this end or  line is
$G$-invariant because
$G'$ is normal, again    contradicting the irreducibility of $T$.
\end{proof}

 \subsection{RAAGs}

Let $\Delta$ be a finite graph. The associated right-angled Artin group $A_\Delta$ (RAAG, also called graph group, or partially commutative group)\index{RAAG} is the group presented as follows: there is one generator $a_v$ per vertex, and a relation $a_va_w=a_wa_v$ if there is an edge between $v$ and $w$. See  \cite{Charney_introduction} for an introduction.

The decomposition of $\Delta$ into connected components induces a decomposition of $A_\Delta$ as a free product of freely indecomposable RAAGs (which may be infinite cyclic), so to study JSJ decompositions of $A_\Delta$ one may assume that $\Delta$ is connected (see Corollary \ref{cor_unbout} below). 
M.\ Clay  \cite{Clay_when} determines the (non-relative) cyclic JSJ decomposition of $A_\Delta$; he   gives a characterization of RAAGs with non-trivial cyclic JSJ decomposition, and he shows that there is no flexible vertex.   See \cite{GrHu_abelian} for   abelian splittings of RAAGs relative to the generators.

\subsection{Parabolic splittings}
\label{parab}

Assume that $G$ is hyperbolic relative to a family of finitely generated subgroups $\calh=\{H_1,\dots,H_p\}$.   
Recall  that a subgroup of $G$ is
\emph{parabolic}\index{parabolic} if it is contained in a conjugate of an $H_i$. 
We let $\cala$ be the family of parabolic subgroups.

JSJ trees over parabolic subgroups, relative to $\calh$ (equivalently, to $\cala$), exist  by Theorem \ref{thm_exist_mou_rel} because $G$ is finitely presented relative to $\calh$ \cite{Osin_relatively}. 
Parabolic subgroups are universally elliptic (all splittings  are relative to $\calh$!), so
 JSJ trees
do not have flexible vertices by Lemma \ref{geuel}.
See \cite{Bo_peripheral}, where such a JSJ decomposition is related to the cut points of the boundary of $G$.

See Subsection \ref{sec_relh} for the case when virtually cyclic groups are added to $\cala$.

\subsection{Non-rigid examples} \label{exflex}

\index{flexible vertex, group, stabilizer}
Unlike the previous subsections, we now consider examples with flexible vertices. This may be viewed as an introduction to Part \ref{part_QH}.  We consider    cyclic splittings with $\calh=\es$. 

Suppose that $G$ is the fundamental group of a closed orientable hyperbolic surface $\Sigma$. Any simple closed geodesic $\gamma$ on $\Sigma$ defines a  dual cyclic splitting (an amalgam or an HNN-extension, depending on whether $\gamma$ separates or not). A non-trivial  element $g\in G$, represented by an immersed closed geodesic $\delta$, is elliptic in the splitting dual to $\gamma$ if and only if $\delta\cap\gamma=\es$ or $\delta=\gamma$.
Since any $\delta$ meets transversely some simple $\gamma$,  this 
shows that $1$ is the only universally elliptic element of $G$, so the JSJ decomposition of $G$ is trivial and its vertex is flexible. 
Similar considerations apply to 
splittings of fundamental groups of compact hyperbolic surfaces with boundary which are  relative to
the fundamental groups  of boundary components (the pair of pants is special because it    contains no essential simple geodesic,  see Subsection \ref{orb}).

The content of Section \ref{Fujpap} is that this example is somehow a universal example. 

\begin{figure}[htbp]
  \centering
  \includegraphics{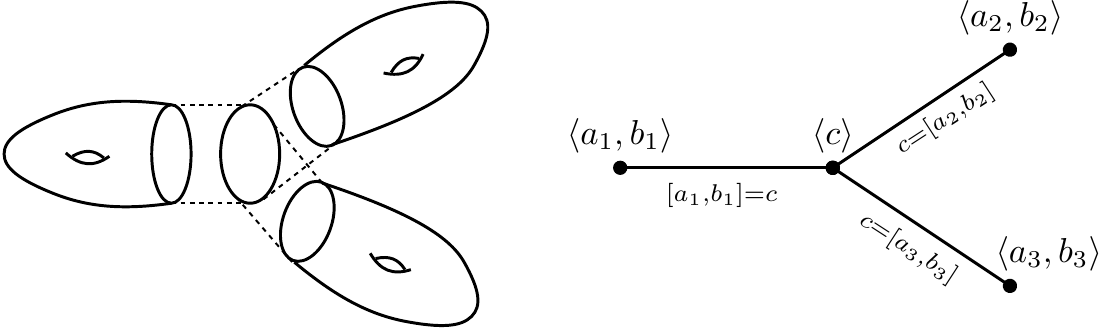}
  \caption{Three punctured tori attached along their boundaries, and the corresponding JSJ decomposition.}
  \label{fig_3tores}
\end{figure}
Now suppose that $G$ is the fundamental group of the space pictured on   Figure \ref{fig_3tores} (consisting of three punctured tori attached along their boundaries). A presentation of $G$ is $$\langle a_1,b_1,a_2,b_2,a_3,b_3\mid [a_1,b_1] = [a_2,b_2]  =[a_3,b_3]\rangle.$$  It is the fundamental group of a graph of groups $\Gamma$ with one central vertex $v$ and three terminal vertices $v_i$. All edges, as well as $v$, carry the same cyclic group  $C=\langle c\rangle$. We claim that 
 \emph{$\Gamma$ is a JSJ decomposition of $G$, with 3 flexible vertices $v_i$.}

 Let us first show that $\Gamma$
is universally elliptic, \ie that 
 $C$ is universally elliptic. Using Lemma \ref{lem_oneed}, we
consider a one-edge cyclic splitting $\Lambda$ of $G$  in which $C$ is not elliptic, and we argue towards a contradiction. The group $G_{12}$ generated by $G_{v_1}$ and $G_{v_2}$ is the fundamental group of a closed surface (of genus 2), in particular it is one-ended. It follows that the edge group   $\grp{a}$ of $\Lambda$ is  non-trivial  
and has non-trivial intersection with some conjugate of $G_{12}$, so  some $a^k$ lies in a conjugate of $G_{12}$. 

We now consider the action of  $a^k$ on the Bass-Serre tree $T$ of $\Gamma$.
If $a^k$ fixes an edge in $T$, then $\grp{a^k}$ has finite index in some conjugate of $C$, so $C$ is elliptic in $\Lambda$, a contradiction.
Denote by  $A({a^k})$ the characteristic set of $a^k$ in $T$, \ie its unique fixed point or its axis.
Then $A({a^k})$ is contained in the minimal subtree of a conjugate of $G_{12}$,  so    contains no lift of the vertex $v_3$.
Permuting indices   shows that  all vertices 
in
$A({a^k})$ are lifts of $v$, so 
 $A({a^k})$   is a single point (a lift of $v$).
This implies that  $\grp{a^k}$ has finite index in some conjugate of $C$, which is again a contradiction.

To prove maximality, consider    a universally elliptic tree $S$  dominating $T$.  If the domination is strict (\ie $S$ and $T$ are in different deformation spaces), some $G_{v_i}$ is non-elliptic in $S$.
Being universally elliptic, $C$ is elliptic in $S$, so by a standard fact (see Proposition \ref{dual}) the action of $G_{v_i}$ on its minimal subtree in $S$ is dual to an essential 
simple closed curve  $\gamma_i
\inc \Sigma_i$ (there is only one curve because $\Sigma_i$ is a punctured torus). Considering the splitting of $G$ dual to
a curve $\gamma\inc\Sigma_i$  intersecting   $\gamma_i$ non-trivially shows that $S$ is not universally elliptic, a contradiction.

\section{A few useful facts} \label{usef}

 In this section, we first describe the behavior of the JSJ deformation space when we change the class $\cala$ of allowed edge groups.
We then introduce the \peripheral{} 
structure inherited by a vertex group of a graph of groups,
and we relate JSJ decompositions of $G$ to JSJ decompositions of   vertex groups relative to their \peripheral{} structure. We also discuss relative finite presentation of vertex groups. 
Finally, we give an alternative construction of relative JSJ decompositions, obtained by embedding   $G$ into a larger group.

\subsection{Changing edge  groups}\label{ceg}

We fix two families of subgroups $\cala$ and $\calb$, with $\cala\inc\calb$, 
and we   compare   
JSJ splittings over $\cala$ and over $\calb$   (universal ellipticity, and all   JSJ decompositions,   are relative to some fixed family $\calh$). 
   For  example:

$\bullet$ $\cala$ consists of the finitely generated abelian subgroups of $G$, and $\calb$ consists of
the slender subgroups. This will be useful to describe the abelian JSJ decomposition in
Subsection \ref{ab}. 

$\bullet$ groups in $\calb$ are locally slender (their finitely generated subgroups are slender), and $\cala$ is the family of slender subgroups.

$\bullet$  $G$ is relatively hyperbolic, $\cala $ is the family of parabolic
subgroups, and $\calb$ is the family of elementary subgroups (subgroups which are parabolic or virtually cyclic).

$\bullet$ $\cala$ consists of the trivial group, or   the finite subgroups of $G$ (see
Corollary \ref{cor_unbout}).

There are now two notions of universal ellipticity, so we  shall distinguish between $\cala$-universal ellipticity (being elliptic in all   $(\cala,\calh)$-trees) and $\calb$-universal ellipticity   (being elliptic in all $(\calb,\calh)$-trees).
 
Of course, $\calb$-universal ellipticity implies $\cala$-universal ellipticity.

Recall that two trees are \emph{compatible} if they have a common refinement. 

\begin{prop}\label{prop_embo}
 Assume   $\cala\subset\calb$. Let $T_\calb$ be a JSJ tree over $\calb$. 
 \begin{enumerate}
 \item \label{it_embo1}
  If there is a JSJ tree over $\cala$, there is one  which is compatible with $T_\calb$. It may be
  obtained    by refining $T_\calb$, and then collapsing all edges whose
stabilizer is not in $\cala$.
\item \label{it_embo2} If   
every $\cala$-universally elliptic $\cala$-tree is $\calb$-universally
elliptic, the tree 
  $\bar T _\calb$ obtained  from $T_\calb$ by 
collapsing all edges whose stabilizer is not in $\cala$  is a JSJ tree over $\cala$.
\end{enumerate}
\end{prop}

Note that (\ref{it_embo2}) applies  if $\cala$ consists of finite groups (more generally, of groups with Serre's property (FA)).

\begin{proof} Let $T_\cala$ be a JSJ tree over $\cala$. 
  Since $\calb\supset \cala$, the tree $T_\calb$ is elliptic with respect to $T_\cala$.
Let $\hat T_\calb$ be a standard refinement of $T_\calb$ dominating $T_\cala$.
 Consider an edge $e $ of
$\hat T_\calb$ whose stabilizer is not in $\cala$. Then $G_e$ fixes a unique
point of $T_\cala$,   so any equivariant map   from $\hat
T_\calb$ to
$T_\cala$ 
is constant on $e$. It  follows that the tree obtained from $\hat T_\calb$ by
collapsing all edges whose stabilizer is not in $\cala$  dominates
$T_2$. Being $\cala$-universally
elliptic by Assertion \ref{it_Ge} of Lemma \ref{prop_refinement}, it  is a JSJ tree over $\cala$. 

For (\ref{it_embo2}), first note that $\bar T_\calb$ is an $\cala$-universally
elliptic $\cala$-tree. If $T'$ is another   one, it is $\calb$-universally
elliptic, hence dominated by $T_\calb$. As above, any map from $T_\calb$ to $T'$ factors through $\bar T_\calb$, so $\bar T_\calb$ is a JSJ tree over $\cala$.
\end{proof} 

\begin{prop} \label{loc}
Let $\cala\inc\calb$. Assume that every  finitely generated group in $\calb$ belongs to $\cala$. 
  If $G$ is finitely presented relative to $\calh$,
then any JSJ tree over $\cala$ (relative to $\calh$) is a JSJ tree over $\calb$ (relative to $\calh$). 
\end{prop}

\begin{proof}
  By  Corollary \ref{cor_fg_rel}, for every $\calb$-tree $T$ (as always,  
  relative to $\calh$), 
 there is  a morphism $S\ra T$ with $S$ a
tree 
with finitely generated edge stabilizers. 
Edge stabilizers of $S$ fix an edge in $T$, so $S$ is an $\cala$-tree.
The proposition easily follows.
\end{proof}

This applies in particular if $\calb$ is the family of all groups  which are locally in $\cala$ (their finitely generated subgroups are in $\cala$). For instance, 
 $ \calb$ may be the family of all locally cyclic  (resp.\  locally abelian,  resp.\ locally slender) subgroups, and $\cala$   the family of cyclic (resp.\  finitely generated abelian, resp.\ slender) subgroups.

\subsection{\Peripheral{} structures of vertex groups}

\index{incidence structure} 
Given a vertex stabilizer $G_v$ of a tree, it is useful to consider splittings of $G_v$ relative to incident edge stabilizers, as they extend to splittings of $G$ (Lemma \ref{extens}). In this subsection we give   definitions  and we  show that $G_v$ is finitely presented relative to incident edge groups if $G$ is finitely presented and edge stabilizers are finitely generated (Proposition \ref{Gvrelfp}).

\subsubsection{Definitions}\label{sec_def_periph}

\label{blowup}
Let $T$ be a tree  (minimal, relative to $\calh$, with edge stabilizers in $\cala$).  Let $v$ be a vertex, with stabilizer $G_v$.

\begin{dfn} [Incident edge groups $\Inc_v$]  \label{dfn_incv}
Given a vertex $v$ of a tree $T$, there are finitely many $G_v$-orbits of edges with origin $v$. 
We choose representatives $e_i$ and we define $\Inc_v$ (or $\Inc_{G_v}$) as the family of stabilizers $G_{e_i}$. 
We call $\Inc_v$ the set of \emph{incident edge groups.}\index{incident edge groups}  It is a finite family of subgroups of $G_v$, each well-defined up to conjugacy. 
\end{dfn}

Alternatively, one can define $\Inc_v$ from the quotient graph of groups $\Gamma=T/G$ as 
the image in $G_v$ of the groups  carried by all oriented edges with origin $v$. 
The peripheral structure, defined in \cite{GL2} and studied  in Subsection \ref{periph}, is a more sophisticated invariant
derived from $\Inc_v$;  unlike $\Inc_v$, it does not change when we replace $T$ by another tree in the same deformation space.

\begin{dfn}[Restriction  $\calh_{|G_v}$] \label{indu}\index{0HG@$\calh_{\vert G_v}$}\index{restriction of $\calh$} 
 
Given $v$, consider the family of conjugates of groups in $\calh$ that fix $v$ and no other vertex of $T$.
We define the \emph{restriction} $\calh_{|G_v}$ by choosing a representative for each $G_v$-conjugacy class in this family.
\end{dfn}

\begin{dfn}[$\Inch_{v}$] \label{indup} \index{0INC@$\Inc_v$, $\Inch_{v}$, $\Inc_{\vert Q}$, $\Inch_{\vert Q} $: incident structures} 
We define $\Inch_{v}=\Inc_v\cup \calh_{|G_v}$.  We will sometimes write  $\Inc_{ | Q}$ and $\Inch_{ | Q}$ rather than $\Inc_{v}$ and $\Inch_{v}$, with $Q=G_v$.
\end{dfn}

We also view   $\calh_{|G_v}$ and $\Inch_{v}$ as  families of subgroups of $G_v$, each well-defined up to conjugacy.

\begin{rem} \label{ashv}
 We emphasize that $\calh_{|G_v}$ only contains groups having $v$ as their unique fixed point. Two such groups are conjugate in $G_v$ if they are conjugate in $G$. In particular, 
the number   of $G_v$-conjugacy classes of groups in $\calh_{|G_v}$ is bounded by the number of     $G$-conjugacy classes of groups in $\calh$.

Also note that any subgroup of $G_v$ which is conjugate to  a group of $\calh$ is contained (up to conjugacy in $G_v$) in a group belonging to $\Inch_{v}$, so is elliptic in any splitting of $G_v$ relative to $\Inch_{v}$. 
\end{rem}

\subsubsection{Finiteness properties}
\label{fingv}

 Assume that $G$ is the fundamental group of a finite graph of groups.
It is well-known that, if   $G$ and all edge groups are finitely generated (resp.\ finitely presented), then so are all vertex groups. 

The goal of this subsection is to extend these results to relative finite generation (resp.\ finite presentation),\index{relative finite generation, presentation} as defined in Subsection \ref{sec_prelim_relatif}. This is not needed if all groups in $\cala$ are assumed to be finitely presented.

\begin{lem}[{\cite[Lemmas 1.11, 1.12]{Gui_actions}}]\label{relfg}
   If a finitely generated group $G$ acts on  a tree $T$,  
then every vertex stabilizer $G_v$   is finitely generated relative to the incident edge groups.
 
More generally, if $G$ is finitely generated relative to $\calh$, and $T$ is a
tree relative to $\calh$,
then $G_v$ is finitely generated  relative to $\Inch_v$.  \qed
\end{lem}

\begin{rem}\label{den}
 If $\calh=\{H_1,\dots,H_p\}$, then the family $\calh_{|G_v}$ consists of at most $p$ groups, each conjugate to some $H_i$. It follows that the set of vertex stabilizers of $(\cala,\calh)$-trees with finitely generated edge stabilizers is countable. This was used in the proof of Theorem \ref{thm_exist_mou_rel}. 
\end{rem}
 
There is a similar statement for relative finite presentation.

\begin{prop}\label{Gvrelfp}
    If $G$ is finitely presented, and $T$ is a  
    tree with finitely generated edge stabilizers,
then every vertex stabilizer $G_v$   is finitely presented relative to the incident edge groups.
 
More generally, if $G$ is finitely presented relative to $\calh$, and $T$ is a
tree relative to $\calh$  with finitely generated edge stabilizers,
then $G_v$  is finitely presented relative to $\Inch_v$. 
\end{prop}

We will  use the following fact.

\begin{lem}\label{relfp_topo}
  Let $X$ be a cellular complex, and let $U\subset X$ be a compact connected subcomplex.
Let $G=\pi_1(X)$, and let $H<G$ be the image of $\pi_1(U)$ in $G$.

Then $H$ is finitely presented relative to the image of the fundamental groups of the connected components
of the topological boundary $\partial U$.

More generally, if $U$ is not compact but there exist finitely many connected disjoint subcomplexes
 $Z_1,\dots,Z_r\subset U$, disjoint from $\partial U$, such that $U\setminus (Z_1\cup\dots\cup Z_r)$ has compact closure,
then $H$ is finitely presented relative to the  images of the fundamental groups of $Z_1,\dots,Z_r$ and of the connected components
of $\partial U$.
\end{lem}

\begin{proof}
 
Denote by $B_1,\dots,B_r$ the connected components of $\partial U$.  
In the special case when each map $\pi_1(B_i)\to\pi_1(X)$ is injective, a standard argument shows that   the map $\pi_1(U)\to\pi_1(X)$ is also injective, so   $H\simeq \pi_1(U)$ is finitely presented, and finitely presented  relative to its   finitely generated subgroups $\pi_1(B_i)$. One reduces to the special case by gluing (possibly infinitely many) discs to the $B_i$'s.
The proof of the second assertion is similar.
\end{proof}

\begin{proof}[Proof of Proposition \ref{Gvrelfp}]

Assume that  $G$ is   finitely presented relative to $\{H_1,\dots,H_p\}$.
As in  the proof of Proposition \ref{prop_accessibility_rel}, we consider a $2$-complex $X$ and disjoint subcomplexes $Y_1,\dots,Y_p$ with $\ol{X\setminus (Y_1,\dots,Y_p)}$ compact and 
  $\pi_1(Y_i)\simeq H_i$. We let $\pi:\Tilde X\ra X$ be the universal cover, and   $\Tilde Y_i$  a connected component of $\pi\m(Y_i)$ whose stabilizer is $H_i$.
Let  $f:\Tilde X\ra T$ be an equivariant map which sends vertex to vertex, sends
each $\Tilde Y_i$ to a vertex fixed by $H_i$, is constant or  injective
on each edge,
and is standard on each triangle. 

Consider the pattern $\tilde \tau\inc \tilde X$ obtained as the preimage of all midpoints of edges of $T$, and its projection $\tau\inc X$. Let $W\inc X$ be the closure   of the complement of a regular neighborhood of $\tau$. 

Let $S$ be the   tree dual to $\tilde \tau$. 
First suppose that $S=T$. Applying Lemma \ref{relfp_topo} to the component $U$ of $W$ corresponding to $v$ shows that $G_v$ is finitely presented relative to the family $\calh'_v$ consisting of  $\Inc_v$ and those $H_i$'s for which $Y_i$ is contained in $U$. If $\calh\ne\es$ this may not be quite the required family $\Inch_v$: we have to remove $H_i$ if it fixes an edge of $T$ (see Subsection \ref{sec_def_periph}). But such an $H_i$ is contained (up to conjugacy in $G_v$) in a group belonging to $\Inch_v$, and we can use Lemma \ref{menage}.

Now suppose $S\ne T$, and recall that edge stabilizers of $T$ are finitely generated. When $\calh=\es$ we may  use \cite[Theorem 0.6]{LP}, saying that $T$ is geometric,  to construct $\tilde X$ and  $f$ such that $S=T$. In general, we may write the induced map $g:S\to T$ as a finite composition of folds (see the proposition on page 455 of   \cite{BF_bounding}). It therefore suffices to show that, 
given a factorization $g= h\circ\rho$ with $\rho$ a fold,   we may change $\tilde X$ to a new complex $\tilde X_1$, with $f_1: \tilde X_1\to T$, so that  the associated map   $g_1$ equals $h$.

Let $e$, $e'$ be adjacent edges of $S$ which are folded by $\rho$. They are dual to components 
 of $\tilde\tau$ adjacent to a component $Z$ of $\tilde X\setminus \tilde\tau$. 
Let $ab$ and $a'b'$ be edges of $\tilde X$ mapping to $e$ and $e'$ in $S$ respectively, with $a$ and $a'$ in $Z$ (subdivide $X$ if needed). Let $\gamma$ be an edge path joining $a$ to $a'$ in $Z$. Note that $\gamma$ is mapped to a single vertex in $S$, and that $b$, $b'$ have the same image in $T$. 

Now glue a square $[0,1]\times[0,1]$  to $\tilde X$, with vertical edges glued to $e$, $e'$ and the edge $[0,1]\times\{0\}$ glued to $\gamma$ (the edge $[0,1]\times\{1\}$ is free). The map $f:\tilde X\to T$ extends to the square, with all vertical arcs $\{t\}\times[0,1]$ mapped to the edge  $f(e)=f(e')$.  
Doing this gluing $G$-equivariantly yields the desired $\tilde X_1$. 
\end{proof}

\subsection{JSJ decompositions of vertex groups}\label{sec_JSJ_vx}

 Given an $(\cala, \calh)$-tree $T$,
we compare   splittings  of $G$ and relative splittings of vertex stabilizers $G_v$.

Recall that $\Inch_v$ is the family of incident edge stabilizers together with $\calh_{|G_v}$ (see  Subsection \ref{sec_def_periph}), and 
  that
  $G_v$ is 
finitely generated relative to $\Inch_v$. In particular, whenever $G_v$ acts on an $(\cala, \calh)$-tree $S$ relative to $\Inc_v$ with no global fixed point, there is
a unique minimal $G_v$-invariant subtree $\mu_S(G_v)\inc S$ by Proposition \ref{arbtf}. 
We view  $\mu_S(G_v)$ as a tree with an action of $G_v$; if $G_v$ is elliptic in $S$, we let $\mu_S(G_v)$ be any   fixed point.
 
 \begin{dfn}[$\cala_v$]
  We denote by  $\cala_v$\index{0AV@$\cala_v$: groups of $\cala$ in $G_v$}    the family consisting of subgroups of $G_v$ belonging to $\cala$. All splittings of $G_v$ will be over groups in $\cala_v$.
\end{dfn}

\begin{lem} \label{extens} 
Let $G_v$ be a vertex stabilizer of an $(\cala, \calh)$-tree $T$.
 Any splitting of $G_v$
relative to $\Inch_v$ extends (non-uniquely)  to a splitting of $G$
relative to $\calh$.

 More precisely, given an $(\cala_v,\Inch_v)$-tree $S_v$,
there exist an $(\cala,\calh)$-tree $\hat T$ and 
a collapse map  $p:\hat T\to T$ such that $p\m(v)$ 
is $G_v$-equivariantly isomorphic to $S_v$. 
\end{lem}

We say that $\hat T$ is obtained by \emph{refining $T$ at $v$ using $S_v$}.\index{refining $T$ at a vertex $v$}
 More generally, one may choose a splitting for each orbit of vertices of $T$, and refine $T$ using them. 
Any refinement of $T$ may be obtained by this construction (possibly with non-minimal trees $S_v$).

\begin{proof}
We construct $\hat T$ as in the proof of Proposition \ref{prop_refinement} , with $Y_v=S_v$. It is relative to $\calh$ because any group in $\calh$ which is conjugate to 
a subgroup of $G_v$ is conjugate to a subgroup of a group 
belonging to $\Inch_{v}$. Non-uniqueness comes from the fact that there may be several ways of attaching
edges of $T$
to $S_v$ (see \cite[Section 4.2]{GL_vertex}).
\end{proof}

\begin{lem}\label{lem_passage}
 Let $ G_v$ be a vertex stabilizer   of a universally elliptic tree $T$.  
\begin{itemize}
\item  The groups in $\Inch_v$ are elliptic in every $(\cala,\calh)$-tree $S$.
\item
 A subgroup  $H< G_v$ is $(\cala,\calh)$-universally elliptic\index{universally elliptic subgroup, tree} (as a subgroup of $G$)
if and only if it is $(\cala_v,\Inch_v)$-universally  elliptic (as a subgroup of $G_v$).
\end{itemize}
\end{lem}

The second assertion says that $H$ is 
elliptic in every $(\cala,\calh)$-tree on which $G$ acts if and only if it is 
elliptic in every $(\cala_v,\Inc_v\cup \calh_{|G_v})$-tree on which $G_v$ acts. If this holds, we simply say that $H$ is universally elliptic.

\begin{proof}
The first assertion is clear: $S$ is relative to $\calh_{|G_v}$,
and also to $\Inc_v$ because $T$ is universally elliptic.

Suppose that $H$ is $(\cala_v,\Inch_v)$-universally  elliptic, as a subgroup of $G_v$.
Let $S$ be any $(\cala,\calh)$-tree.
It is relative to $\Inch_v$ by the first assertion. 
Since $G_v$ is finitely generated relative to $\Inch_v$, it has a minimal subtree $S_v\subset S$ by Proposition \ref{arbtf}.
The action of $G_v$ on $S_v$ is an $(\cala_v,\Inch_v)$-tree, so by  assumption $H$ fixes a point in $S_v$, hence in $S$.

 We have proved the ``if'' direction  in the second assertion, and 
the converse   follows from Lemma \ref{extens}.
\end{proof}

\begin{cor}\label{cor_flex}
  Let $G_v$ be a  vertex stabilizer of a JSJ tree $T_J$.  
  
\begin{enumerate}
\item
$G_v$ does not split over a universally elliptic subgroup relative to $  \Inch_v$.
\item
$G_v$ is flexible\index{flexible vertex, group, stabilizer} if it splits relative to $ \Inch_v$, rigid otherwise.
\end{enumerate}
\end{cor}

\begin{proof}
If there is a splitting as in (1), we may use it to refine $T_J$ to a universally elliptic tree (see Lemma \ref{extens}). This tree must be in the same deformation space as $T_J$, so the splitting of $G_v$ must be trivial. (2) follows from Lemma \ref{lem_passage} applied with $H=G_v$.
\end{proof}

\begin{prop}\label{prop_JSJ_sommets}
  Let $T$ be a universally elliptic $(\cala,\calh)$-tree.
\begin{enumerate}

\item
Assume that every vertex stabilizer $G_v$ of $T$ has a JSJ tree $T_v$ relative to $ \Inch_v$. One can then refine\index{refinement} $T$ using these decompositions so as to obtain a JSJ tree of $G$ relative to $\calh$. 

\item Conversely,
if $T_J$ is a JSJ tree of $G$ relative to $\calh$, and $G_v$ is a vertex stabilizer of $T$,  one obtains a JSJ tree for $G_v$ relative to $ \Inch_v$ by considering 
the action of  $G_v$ on its minimal subtree\index{minimal subtree} $T_v=\mu_{T_J}(G_v)$ in $T_J$ (with $T_v$ a point if $G_v$ is elliptic).
\end{enumerate}
\end{prop}

 \begin{proof} 
To prove (1), let $\Hat T$ be the tree obtained by refining $T$ using the $T_v$'s as in Lemma  \ref{extens}.
It is relative to $\calh$, 
and universally elliptic by Lemma \ref{lem_passage}  since 
its edge stabilizers are edge stabilizers of $T$ or of some $T_v$.
To show maximality, we consider  another  universally elliptic $(\cala,\calh)$-tree $T'$,
 and we show that any vertex stabilizer $H$ of $\Hat T$ is elliptic in $T'$. It is a vertex stabilizer of  some $T_v$, with $v$ a vertex of $T$.
If $G_v$ is not elliptic in $T'$, 
its minimal subtree $Y_v$ is a universally elliptic  $(\cala_v, \Inch_v)$-tree.   Since $T_v$ is a JSJ tree, it  dominates $Y_v$  so $H$ is elliptic in $T'$.
This proves (1).

Now let $T_J$
and $T_v\subset T_J$ be as in (2).
By Lemma \ref{lem_passage} $T_v$  is relative to $  \Inch_v$, 
and
it is $(\calh_v,\Inch_v)$-universally elliptic because its edge stabilizers are contained in edge stabilizers of $T_J$. 
To prove maximality of $T_v$,  consider another tree $S_v$ with an action of $G_v$ which is relative to $  \Inch_v$  and universally elliptic. Use it to refine $T$ to a tree $\Hat T$ as in Lemma \ref{extens}. As above,
$\Hat T$ is relative to $\calh$,
and 
  universally elliptic by Lemma \ref{lem_passage}.
Being a JSJ tree,  $T_J$ dominates $\Hat T$. Vertex stabilizers of $T_v$ are elliptic in $T_J$, hence in $\Hat T$, hence in $S_v$, so $T_v$ dominates $S_v$. This proves (2).
 \end{proof}
 
 The following corollary says that one may usually restrict to one-ended\index{one-ended relative to $\calh $} groups when studying JSJ decompositions.
 
\begin{cor}\label{cor_unbout} Suppose that $\cala$ contains all finite subgroups of $G$.

 If $\calh=\es$, 
refining a Grushko decomposition\index{Grushko decomposition, deformation space} of $G$ using  JSJ decompositions of free factors
yields a JSJ decomposition of $G$.

Similarly,
refining a Stallings-Dunwoody\index{Stallings-Dunwoody deformation space} decomposition of $G$  using   JSJ decompositions of the vertex groups $G_v$ 
yields a JSJ decomposition of $G$.

  If $\calh\ne\es$, one must use relative Grushko or Stallings-Dunwoody decompositions, 
and JSJ decompositions of vertex groups $G_v$ relative to $\calh_{|G_v}$.
  
  Every flexible subgroup of $G$ is a flexible subgroup of  some  $G_v$. 
\end{cor}

As mentioned  in Subsection \ref{Dunw},  Stallings-Dunwoody decompositions only  exist under some accessibility assumption.

 \begin{proof} 
This follows from the proposition,   applied with $T$ in the Grushko or Stallings-Dunwoody deformation space:   finite groups are universally elliptic, and every splitting of $G_v$ is relative to $\Inc_v$. The assertion about flexible subgroups follows from Corollary \ref{cor_flex}.
\end{proof}

\subsection{Relative JSJ decompositions through fillings}\label{filling}

  Fix   a finitely
presented group $G$ and a family $\cala$. In Subsection \ref{rela} we have shown the existence of  the JSJ deformation space of  $G$  relative to a finite set  $\calh=\{H_1,\dots,H_p\}$  of
finitely generated subgroups.  
We now give an alternative construction, using     (absolute) JSJ
decompositions  of another group $\hat G$ obtained by a \emph{filling} construction.  This construction will be used in Subsections \ref{periph} and \ref{ab} to provide examples of flexible groups.

\begin{figure}[htbp]
  \centering
  \includegraphics{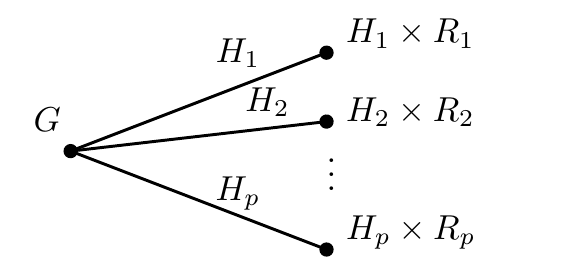}
  \caption{The group $\hat G$ obtained by the \emph{filling} construction.}
  \label{fig_fill}
\end{figure}

\paragraph{The filling construction.}\index{filling construction}
Let $G$  and $\calh$ be as above. 
For $i\in\{1,\dots,p\}$, we let $R_i$ be  a  non-trivial finitely presented group with property (FA),
and we define a group
$\hat G$ by amalgamating  $G$ with  $K_i=H_i\times R_i$ over $H_i$ for $i=1,\dots,p$ (see Figure \ref{fig_fill});  in other words, $\hat G=((G*_{H_1}K_1)*\dots)*_{H_p} K_p$. 
It  is finitely presented. We denote by $T$ the Bass-Serre tree of this amalgam, and by $v$ the vertex of $T$ with stabilizer $G$.   The stabilizer of an edge  with origin $v$ is conjugate  to 
    one of the $H_i$'s in $G$.

Fix a family $\calb$ of
subgroups of $\hat G$ such that
$\calb_{v}=\cala$ and $R_i\notin \calb$,  for instance 
the family of subgroups of $\Hat G$ having a  conjugate in $\cala$ (note that  two subgroups of $G$ which are conjugate in $\hat G$ are also conjugate in $G$  because $H_i$ is central in $K_i$,  so this family $\calb$
  induces $\cala$).

A subgroup of $G$ is $(\cala,\calh)$-universally elliptic if it is elliptic in all
splittings of $G$ over $\cala$ relative to $\calh$. 
A subgroup of $\hat G$ is $\calb$-universally elliptic
if it is elliptic in all
splittings of $\hat G$
with edge groups in $\calb$ 
(splittings of $\hat G$ are non-relative).

\begin{lem} \label{lem_relu} $H_i\times R_i$ is
$\calb$-universally elliptic. 
A subgroup $J\subset G$ is $(\cala,\calh) $-universally elliptic
if and only if $J$ (viewed as a subgroup of $\hat G$) is
$\calb $-universally elliptic.
\end{lem}

\begin{proof}
Consider any $\calb$-tree.
The group $R_i$ fixes a point by property (FA), which is
unique since
$R_i\notin \calb$. This point is also fixed by $H_i$ since $H_i$
commutes with $R_i$. This proves the first assertion.  Since $T$ is $\calb $-universally elliptic, and $\calh$ is the family of incident edge groups at $v$, the second assertion follows from Lemma \ref{lem_passage}.
\end{proof}

 Being finitely presented, $\hat G$ has a  JSJ decomposition $T_J$ over $\calb$ 
 by Theorem \ref{thm_exist_mou}. We let $T_G=\mu_{T_J}(G)$  be the minimal $G$-invariant subtree (a point if $G$ is elliptic in $T_J$).
 
  \begin{prop} The tree $T_G$ is a JSJ tree of $G$ over $\cala$ relative to $\calh$.
 \end{prop}

\begin{proof}
The tree $T_G$ has edge stabilizers in $\calb_v=\cala$, is relative to $\calh$ and   $(\cala,\calh)$-universally elliptic by Lemma \ref{lem_relu}. We have to show that it dominates any universally elliptic $(\cala,\calh)$-tree $S_G$. Use $S_G$ to refine $T$ at $v$ into a tree $\hat T$, as in Lemma \ref{extens}. The tree $\hat T$ has two types of edges, those coming from $T$ and those in the $\hat G$-orbit of $S_G$. 

Define a new tree $T'$ by collapsing all edges of $\hat T$ coming from $T$, and note that a subgroup of $G$ is elliptic in $T'$ if and only if it is elliptic in $S_G$.  
  Indeed, a subgroup $H$  of $\hat G$ is elliptic in $\hat T$ if and only if it is elliptic in both   $T'$ and   $T$; for $H<G$ (hence elliptic in $T$), being elliptic in $T'$ is equivalent to being elliptic in $\hat T$, i.e.\ in $S_G$.

 The tree $T'$ is a $\calb$-tree, and it is universally elliptic by Lemma \ref{lem_relu}, so it is dominated by $T_J$. Vertex stabilizers of $T_G$ are then elliptic in $T'$, hence in $S_G$, so $T_G$ dominates $S_G$.
\end{proof}

\part{Flexible vertices} \label{part_QH}

\index{flexible vertex, group, stabilizer}
 Flexible vertex groups   of   JSJ decompositions are most important, as understanding 
 their splittings conditions the understanding of the splittings of $G$.
 The key result  is that, in many cases, flexible vertex groups are ``surface-like''.

 For instance, first consider cyclic splittings of a torsion-free group $G$.
We will see that, if $G_v$ is a flexible vertex group of a
JSJ decomposition  $\Gamma=T/G$,
then $G_v$ may be viewed as $\pi_1(\Sigma)$, with $\Sigma$ a compact (possibly non-orientable) surface. Moreover, incident edge groups are trivial or contained (up to conjugacy) in a boundary subgroup, i.e.\ the fundamental group $B=\pi_1(C)$ of a boundary component $C$ of $\Sigma$. Boundary subgroups being generated by elements which  are quadratic words in a suitable basis of the free group $\pi_1(\Sigma)$, Rips and Sela called the  subgroup $G_v$
\emph{quadratically hanging} (QH).\index{QH, quadratically hanging}
 
If $\Sigma$ is not too simple, there are infinitely many isotopy classes of essential two-sided simple closed curves. Each such curve defines a cyclic splitting of $\pi_1(\Sigma)$ relative to incident edge groups, which extends to a cyclic splitting of $G$ by Lemma \ref{extens}.  Two curves which cannot be made disjoint by an isotopy define two splittings with are not elliptic with respect to each other, and this makes $G_v$  flexible
 (see Corollary \ref{cor_flex}). It turns out  that this construction is basically  the only   source of 
 flexible vertices.

 If $G$ is allowed to have torsion, or if non-cyclic edge groups are allowed,   the definition of ``surface-like'' must be adapted. 
 First,  $\Sigma$ may be a 2-dimensional orbifold rather than a surface. Second, $G_v$ is not always equal to $\pi_1(\Sigma)$; it only maps onto $\pi_1(\Sigma)$, with a possibly  non-trivial kernel $F$ (called the fiber).  
 
 Various authors have called such groups $G_v$   hanging surface groups,
hanging Fuchsian groups,
hanging $K$-by-orbifold groups, $VPC$-by-Fuchsian type vertex groups... We choose to extend Rips and Sela's initial terminology of QH groups,  to emphasize the way in which $G_v$ is attached to the rest of the group $G$ (in Rips-Sela   \cite{RiSe_JSJ}, $F$ is trivial and $\Sigma$ has no mirror,  see Theorem \ref{thm_RiSe}; on the other hand, we insist that a QH group be based on  a hyperbolic orbifold, not on a Euclidean one).

In Section \ref{sec_QH} we formalize the definition of QH vertices, and prove general properties of such vertices. 
In Section \ref{Fujpap} we show that, indeed, flexible vertices of the JSJ deformation space over nice classes of slender subgroups are QH.

As before,  we   fix a family $\cala$ closed  under conjugating and passing to subgroups, and another family $\calh$.
All trees are   assumed to be 
$(\cala,\calh)$-trees.

\section{Quadratically hanging  vertices}\label{sec_QH}

 In this section, after preliminaries about $2$-orbifold groups, we give a  definition of QH subgroups.
 We study their basic properties, in particular we relate their splittings to families of simple geodesics on the underlying orbifold $\Sigma$. 

We then show that, under natural hypotheses, any QH subgroup has to be elliptic in the JSJ deformation space. 

In Subsection \ref{periph}
we give
examples of possible incident edge groups for QH vertex groups of a JSJ decomposition.
This will be relevant in     Section \ref{Fujpap}, where we   show that flexible subgroups of the   slender JSJ decomposition   are 
QH.
We also show that flexible subgroups of abelian JSJ decompositions do not have to be QH.

 \subsection{2-orbifolds and their splittings}  \label{orb}

\subsubsection{ Hyperbolic 2-orbifolds}\index{orbifold}
Most compact 2-dimensional orbifolds $\Sigma$, including all those that will concern us, are Euclidean or hyperbolic 
(we refer to \cite{Scott_geometries} and \cite[Ch.\ 13]{Thurston_notes}
for basic facts about orbifolds).\index{orbifold}
Euclidean orbifolds whose fundamental group is not virtually cyclic have empty boundary, so they can only arise from flexible vertices in a trivial way. For instance, in the case of cyclic splittings of torsion-free groups, $\Z^2=\pi_1(T^2)$ and the Klein bottle group may appear as flexible vertex groups (see Subsections \ref{sec_G_small} and  \ref{pslflex}), but only as free factors (all incident edge groups must be trivial). 

 We therefore restrict to hyperbolic orbifolds: 
  $\Sigma$ is a compact 2-dimensional orbifold, equipped with  a hyperbolic metric with totally geodesic boundary. It
is the quotient of a convex subset  $\Tilde \Sigma\subset\bbH^2$
by a proper discontinuous group 
of isometries 
$G_\Sigma\inc \text{Isom}(\bbH^2)$ (isometries may reverse orientation); we denote by $p:\Tilde\Sigma\to\Sigma$ the quotient map. By definition, the (orbifold) fundamental group of $\Sigma$ is $\pi_1(\Sigma)=G_\Sigma$.
We may also view $\Sigma$ as the quotient of a compact orientable hyperbolic
 surface $ \Sigma_0$ with geodesic boundary 
by a finite group of isometries $\Lambda$. A point of $\Sigma$ is \emph{singular} if its preimages in 
$\Tilde \Sigma$ (or in $\Sigma_0$) have non-trivial stabilizer.

\begin{figure}[htbp]
  \centering
  \includegraphics{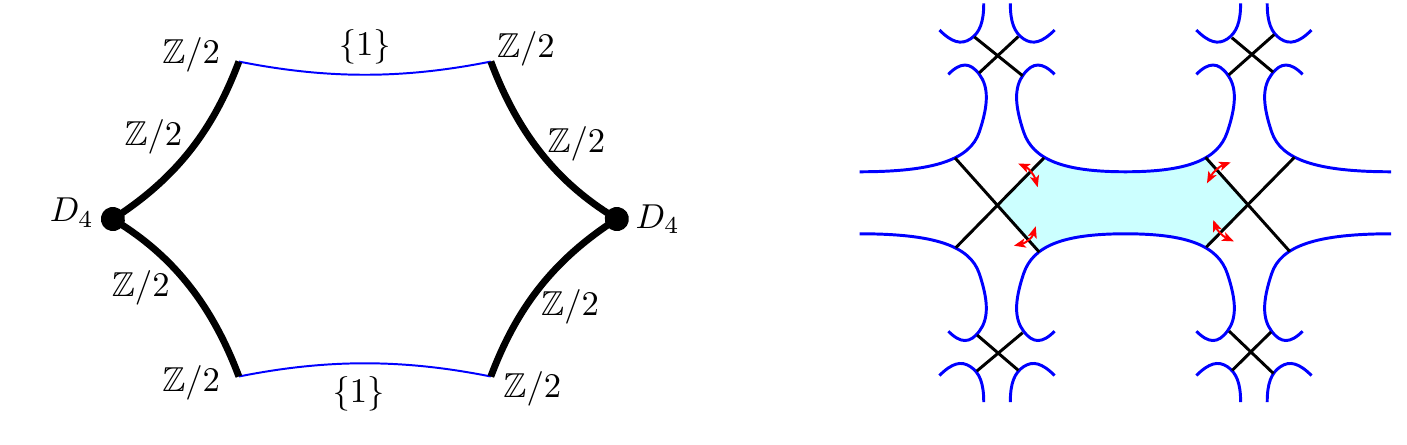}
 \caption{An orbifold with 4 mirrors (in bold), 2 boundary components, and 2 corner reflectors carrying $D_4$. 
Its fundamental group is the Coxeter group generated by reflections over 4  sides of a right-angled hexagon in $\bbH^2$.}
  \label{fig_mirrors}
\end{figure}

If we forget the orbifold structure, $\Sigma$ is homeomorphic to a surface   $\Sigma_{top}$ (a disc in Figure \ref{fig_mirrors}). 
  The boundary of $\Sigma_{top}$ comes from the boundary $\partial \Tilde \Sigma$, and from  {mirrors} corresponding to reflections in $G_\Sigma$ (see below).
We define the \emph{boundary}   \index{boundary (in an orbifold)}
$\partial\Sigma$ of $\Sigma$ as the image of $\partial \Tilde \Sigma$ in $\Sigma$ (thus excluding mirrors).
Equivalently, it is the image of $\bo\Sigma_0$.
Each component $C$ of $\bo\Sigma$ is either a component of $\bo\Sigma _{top}$ (a circle) or an arc contained in $\bo\Sigma _{top}$.
The  (orbifold) fundamental group of $C$ is $\Z$ or an infinite dihedral group $D_\infty$   accordingly.  
A   \emph{boundary subgroup}\index{boundary subgroup} is a subgroup   $B\inc \pi_1(\Sigma)$ 
which is conjugate to  the fundamental group of a component $C$ of $\partial \Sigma $.
Equivalently, it is 
the setwise stabilizer of a connected component of $\partial \Tilde \Sigma$.

The  closure of the complement of $\bo\Sigma$ in $\bo\Sigma_{top}$ is a union of \emph{mirrors}:\index{mirror}
a mirror is the image of a component of the fixed point set of an orientation-reversing element of  $\pi_1(\Sigma)$  in $\Tilde \Sigma$.
Equivalently, a mirror is the image of a component of the fixed point set  of an orientation-reversing element of  $\Lambda$ in $\Sigma_0$.
Each mirror is itself a circle or  an arc  contained in $\bo\Sigma _{top}$.
Mirrors may be adjacent in $\bo\Sigma_{top}$, whereas boundary components of $\Sigma$ are disjoint.

Singular points not contained in mirrors are
  \emph{conical points};\index{conical point}
   the stabilizer of their    preimages  in $\bbH^2$   is a finite cyclic group consisting of orientation-preserving maps (rotations). Points belonging to two mirrors are \emph{corner reflectors};\index{corner reflector} the associated stabilizer is a finite dihedral group  $D_{2r}$ of order $2r$.

As in the case of surfaces, hyperbolic orbifolds may be characterized in terms of their Euler characteristic (see \cite{Thurston_notes,Scott_geometries}). 

\begin{dfn}[Euler characteristic] \label{defeuler}\index{Euler characteristic}
The  Euler characteristic $\chi(\Sigma)$  is defined as the Euler characteristic of the underlying topological surface $\Sigma_{top}$, minus contributions coming from the singularities:
a conical point of order $q$ (with isotropy group $\Z/q\Z$) contributes $1-\frac1q$,   a corner reflector with isotropy group the dihedral group $D_{2r}$ of order $2r$ contributes $\frac12(1-\frac1r)$, and a point adjacent to a mirror and a component of $\bo\Sigma$ contributes $\frac14$.
\end{dfn}

 \begin{prop}\label{euler}
 A compact 2-dimensional orbifold $\Sigma$ is hyperbolic if and only if  $\chi(\Sigma)<0$. \qed
 \end{prop}

\subsubsection{Curves and splittings} \label{cid}

We now generalize the fact that any essential 2-sided simple closed curve on a surface $\Sigma$ defines a cyclic splitting of $\pi_1(\Sigma)$.

Let $\Sigma$ be a hyperbolic 2-orbifold as above. A \emph{closed geodesic}\index{geodesic (in an orbifold)!closed} 
$\gamma\subset \Sigma$ is the image in $\Sigma$ of 
a bi-infinite geodesic $\Tilde \gamma\subset \Tilde \Sigma$ whose image in $\Sigma$ is compact.
It is \emph{simple} \index{geodesic (in an orbifold)!simple} 
if  $h\Tilde \gamma$ and $\Tilde\gamma$ are equal or disjoint for all $h\in \pi_1(\Sigma)$. 
If $\Tilde \gamma\not\subset\bo\tilde\Sigma$, we say that  $\gamma$ is an \emph{essential simple closed geodesic} \index{geodesic (in an orbifold)!essential simple closed} 
in $\Sigma$ (possibly one-sided). For brevity we often just call $\gamma$ a geodesic.

If  $\gamma$ is an essential simple closed geodesic, then  $p\m(\gamma)$, the orbit of $\Tilde\gamma$ under $\pi_1(\Sigma)$, is a family of disjoint geodesics. There is a simplicial tree $T_\gamma$ dual to this family: vertices of $T_\gamma$ are components of $\tilde\Sigma\setminus p\m(\gamma)$, edges   are components of $p\m(\gamma)$. The group $\pi_1(\Sigma)$ acts non-trivially on $T_\gamma$, but there are inversions if $\gamma$ is one-sided; in this case  we subdivide edges so as to get an action without inversions  
(this may viewed 
as replacing $\gamma$  by the boundary of a regular neighborhood, 
a connected $2$-sided simple $1$-suborbifold  bounding a M\"obius band). 

We have thus associated to $\gamma$ a one-edge splitting of $\pi_1(\Sigma)$, which is clearly relative to the boundary subgroups. We call it the  \emph{splitting dual to the essential simple closed geodesic $\gamma$}\index{splitting!dual to a family of geodesics}\index{dual splitting} (or the splitting determined by $\gamma$). Its edge group (well-defined up to conjugacy) is the subgroup 
 $H_{\gamma}\subset 
 G_\Sigma$ consisting of elements which preserve $\Tilde\gamma$ and each of the   half-spaces bounded by    $\Tilde \gamma$.
It  is   isomorphic to $\bbZ$ or $D_\infty$  (the infinite dihedral group).

 More generally, there is a splitting dual to any family $\call$ of disjoint essential simple closed geodesics $\gamma_i$. For simplicity, we sometimes just say that the splitting is dual to a family of geodesics.

 Recall that a group is small if it does not contain $\F_2$. The next result says that this construction yields all small splittings of $\pi_1(\Sigma)$ relative to the boundary subgroups  (note that all small subgroups of $\pi_1(\Sigma)$ are virtually cyclic). 
\begin{rem}\label{rem_small}
  Also note that   any subgroup of $\pi_1(\Sigma)$ that 
  preserves a line or an end in such a   splitting is virtually
  cyclic. 
\end{rem}

\begin{prop} \label{dual}  Let $\Sigma$ be a compact hyperbolic 2-orbifold. Assume that $\pi_1(\Sigma)$ acts on a tree $T$ non-trivially, without inversions, minimally, with small edge stabilizers, and
with all boundary subgroups elliptic. 

Then $T$ is  equivariantly isomorphic to the  Bass-Serre tree of the splitting dual
  to a family $\call$   of disjoint essential
simple closed geodesics  of $\Sigma$.

If  edge stabilizers are not assumed to be small,
$T$ is still dominated by a 
tree dual to a family of geodesics.
\end{prop}

\begin{rem} 
  In this statement, we assume that $T$ has no redundant vertex, and we do not allow multiple parallel simple closed curves (because
we consider geodesics).
If $T$ is allowed to have redundant vertices, then 
it is only isomorphic
to a \emph{subdivision} of the tree dual to a family of geodesics.
\end{rem}

\begin{proof} 
When $\Sigma$ is an orientable surface, this follows from 
Theorem III.2.6 of \cite{MS_valuationsI}.
If $\Sigma$ is a  $2$-orbifold, we consider a covering  surface $\Sigma_0$ as above. 
The action of $\pi_1(\Sigma_0)$ on $T$ is dual to a family of  
closed geodesics on $\Sigma_0$. 
This family is $\Lambda$-invariant and projects to the required family on $\Sigma$. 
The action of $\pi_1(\Sigma)$ on $T$ is dual to this family.

The second statement follows from standard arguments (see the proof of   \cite[Theorem III.2.6]{MS_valuationsI}).
\end{proof}

\begin{cor} \label{courbe}
 $\pi_1(\Sigma)$ has a non-trivial splitting relative to the boundary subgroups if and only  if $\Sigma$ contains an essential
simple closed geodesic. \qed
\end{cor}

 Orbifolds with no   essential simple closed geodesic are classified in the next subsection.

Proposition \ref{dual} implies in particular  that $\pi_1(\Sigma)$ is one-ended relative to its boundary subgroups. This does not remain true if we set one boundary component aside.

\begin{lem} \label{arc}
Let $C$ be a boundary component of a compact hyperbolic 2-orbifold $\Sigma$. 
There exists a non-trivial splitting of $\pi_1(\Sigma)$ over $\{1\}$ or $\Z/2\Z$ 
relative to the fundamental groups $B_k$  of all boundary components distinct from $C$.
\end{lem}

\begin{proof}
Any arc $\gamma$ properly embedded in $\Sigma_{top}$ and with endpoints on $C$ defines a free splitting of $\pi_1(\Sigma)$  relative to the groups  $B_k$. In most cases one can choose $\gamma$ so that this splitting is non-trivial. We study the exceptional cases:
$\Sigma_{top}$ is   a disc or an annulus, and $\Sigma$ has no conical point. 
 
If $\Sigma_{top}$ is   a disc, its boundary circle consists of components of $\bo\Sigma$ and mirrors. 
 Since $\Sigma$ is hyperbolic, there must be a mirror $M$ not adjacent to $C$ 
(otherwise $\bo\Sigma_{top}$ would consist of $C$ and one or two mirrors, or two boundary components and two mirrors, and  $\chi(\Sigma)$ would not be negative). 
An arc $\gamma$ with one endpoint on $C$ and the other on $M$ defines a  splitting over $\Z/2\Z$, which is non-trivial because $M$ is not adjacent to $C$.

If $\Sigma_{top}$ is   an annulus, there are two cases. If  $C$ is an arc, one can find an arc $\gamma$ from $C$ to $C$ as in the general case. If $C$ is a  circle in $\bo\Sigma_{top}$, the other circle contains a mirror  $M$ (otherwise $\Sigma$ would be  a regular annulus) and an arc $\gamma$ from $C$ to $M$ yields a splitting over $\Z/2\Z$.
\end{proof}

\begin{rem}\label{arcb}
 If the splitting constructed is over $K=\Z/2\Z$,  then $\Sigma$ contains a mirror and   $K$ is contained in an infinite dihedral
  subgroup (generated by $K$ and a conjugate).
\end{rem}

\begin{dfn}[Filling geodesics]\label{dfn_fill} 
Let   $\Sigma$ be a compact hyperbolic  $2$-orbifold,  and let $\calc$ be a non-empty collection of  (non-disjoint)  essential simple closed  geodesics in $\Sigma$.
We say that $\calc$ \emph{fills}\index{geodesic (in an orbifold)!filling the orbifold} $\Sigma$ if the following equivalent conditions hold:
\begin{enumerate}
\item For every essential simple closed geodesic $\alpha$ in $\Sigma$, there exists $\gamma\in \calc$ that intersects $\alpha$ non-trivially  
 (with $\gamma\neq \alpha$).
\item For every element $g\in\pi_1(\Sigma)$ of infinite order that is not conjugate into a  boundary subgroup,
 there exists $\gamma\in \calc$ such that $g$ acts hyperbolically in the splitting of $\pi_1(\Sigma)$ dual to $\gamma$.
\item  The full preimage $\Tilde \calc$ of $\calc$ in the universal covering   $\Tilde\Sigma$
is connected.
\end{enumerate}
\end{dfn}

The equivalence between these conditions is well known. We include a proof for completeness.
\begin{proof}
$(2)\Rightarrow(1)$ is clear,  using a $g$ representing $\alpha$.

To prove $(1)\Rightarrow (3)$, consider  a connected component  $\calc_0$ of $\Tilde \calc$, and  its convex hull $A$ in $\Tilde \Sigma\subset \bbH^2$. 
If $\Tilde \calc$ is not connected, 
then $A\neq\Tilde \Sigma$; indeed, $A$ is contained in a half space bounded by a geodesic in  $\Tilde \calc \setminus \calc_0$.
Let  $\alpha$ be a connected component of the boundary of $A$  in $\Tilde \Sigma$, a bi-infinite geodesic.
We note that $\alpha$ cuts no geodesic of $\Tilde \calc$. Indeed, if $\gamma\in\Tilde\calc$ cuts $\alpha$, then $\gamma\notin \calc_0$,
so all elements of $\calc_0$ are disjoint from $\gamma$, and  $A$ is contained in a half-space bounded by $\gamma$, contradicting    $\alpha\subset \bar A$.
Now $\alpha$ projects to a simple closed geodesic in $\Sigma$ (if $g\alpha$ did intersect $\alpha$  transversely, then $A$ would be contained in
the intersection of the half spaces bounded by $\alpha$ and $g\alpha$, a contradiction).
Assumption $(1)$ ensures that $\alpha$ is a boundary component of $\Tilde \Sigma$. This implies that  $A=\Tilde \Sigma$, and that $\Tilde \calc$ is connected.

To prove $(3)\Rightarrow(2)$, consider $g\in\pi_1(\Sigma)$ of infinite order, and let $A(g)$ be its axis in $\Tilde \Sigma$.  If (2) does not hold, 
it intersects no geodesic in $\Tilde \calc$. By connectedness, $\Tilde \calc$ is contained in one of 
the half-spaces bounded by $A(g)$. It follows that the convex hull $A$ of $\Tilde \calc$ is properly contained in $\Tilde \Sigma$.
Since $A$ is  $\pi_1(\Sigma)$-invariant, this is a contradiction.
\end{proof}

\begin{cor}\label{lem_fill}
If $\Sigma$ contains at least one essential simple closed geodesic,
then the set of simple closed geodesics fills $\Sigma$.
\end{cor}

Using the first definition of filling, this follows immediately from the following lemma.

\begin{lem}[Lemma 5.3 of \cite{Gui_reading}] \label{courbes}
If $\gamma_0$ is any essential  simple closed geodesic, 
 there exists another essential simple closed geodesic $\gamma_1$  
intersecting $\gamma_0$ non-trivially. \qed
\end{lem}

\subsubsection{Small orbifolds} \label{liste}

\index{orbifold!small orbifolds}\index{small orbifold}
 It is well-known that the pair of pants is the only compact hyperbolic surface containing no essential  simple closed geodesic. In this subsection we classify hyperbolic 2-orbifolds which do not contain an essential geodesic. Their fundamental groups do not split relative to the boundary subgroups (Corollary \ref{courbe}),  and they do not appear in flexible vertex groups of JSJ decompositions (see Subsection \ref{quah}). 

As above, we work with compact orbifolds with geodesic boundary, but we could equally well consider orbifolds with cusps. 

\begin{figure}[htbp]
  \centering
\includegraphics{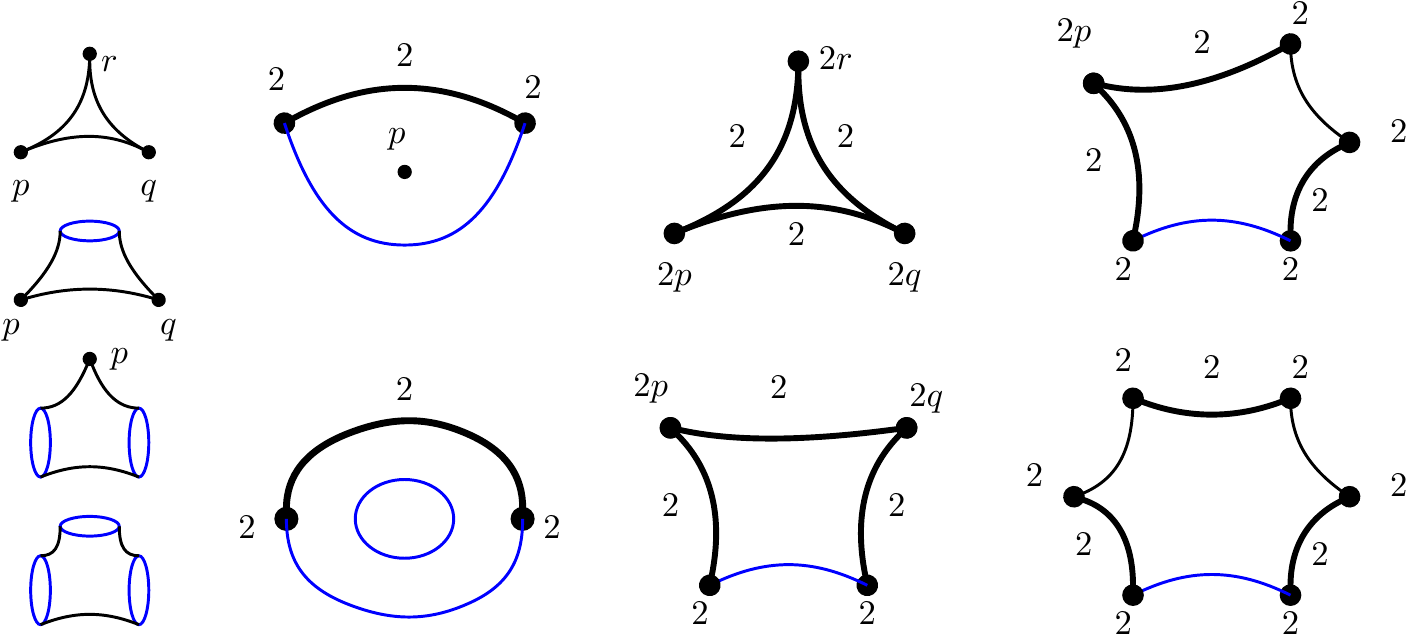}
  \caption{Orbifolds with no splittings (mirrors in bold, labels $=$ isotropy size).}
  \label{fig_rigid_orbifolds}
\end{figure}

\begin{prop} \label{smallorb}
 A compact hyperbolic 2-orbifold  with geodesic boundary contains no  simple closed essential geodesic if and only if it belongs to  the following list (see Figure \ref{fig_rigid_orbifolds}):
  \begin{enumerate}
  \item a sphere with
  3 conical points, a disc with 2 conical points, an annulus with 1 conical point, a pair of pants ($\Sigma$ has no mirror);
  
  \item a disk whose boundary circle is the union of a single mirror with a single boundary segment,   with
 exactly one conical point;

 \item an annulus with one mirror and no conical point; 
  \item a disk whose boundary circle is the union of three mirrors, together with at most $3$ boundary segments (no conical point).  \end{enumerate}
\end{prop}

 Orbifolds $\Sigma$ in this list are hyperbolic
if and only if $\chi(\Sigma)<0$ (see Proposition \ref{euler}).

\begin{proof}
   Let $\Sigma$ be an orbifold with no closed geodesic, and let  $\Sigma_{top}^*$ be the underlying topological surface with the conical points removed.  It  has to be  orientable, since otherwise 
 it contains an embedded M\"obius band, whose core yields an essential simple closed geodesic.
Similarly,   $\Sigma_{top}^*$  has to be planar,   with at most 3 boundary components or punctures.

If $\Sigma$ has no mirror, we must be in case (1): 
the  total number of  conical points and boundary components must be three because $\chi(\Sigma)$ is  negative.

 We therefore assume that $\Sigma$ has a mirror $m$. Recall that each component $c$ of $\bo\Sigma$ is contained in a component $b$ of $\bo  \Sigma_{top}^*$ (a circle). If $c\ne b$, then $c$ is a segment whose endpoints belong to (possibly equal) mirrors contained in $b$.

The classification  when there are mirrors relies on the inequality $\chi(\Sigma)<0$ and the following basic observation.
Consider a properly embedded arc joining $m$ to a mirror $n$ (possibly equal to $m$). Since it cannot be isotopic to an essential simple closed geodesic, it may be isotoped (in the complement of the conical points, with its endpoints remaining on $m$ and $n$ respectively)  to a boundary segment of $\Sigma$ or to an arc contained in $m$.

This  implies that 
  the connected component $b$ of  $\bo\Sigma_{top}$ containing $m$ is
the only connected component of $\bo\Sigma_{top}$ containing a mirror,  and $b$ cannot contain more than 3 mirrors.
Considering arcs joining $m$ to $m$, it also implies
that  $\Sigma_{top}^*$ is a disc or an annulus (not a pair of pants).

If $\Sigma_{top}^*$ is  an annulus, the only possibility is that $b=m\cup s$ with $s$ a boundary segment of $\Sigma$.  This corresponds to cases (2) and (3), depending on whether $\Sigma$ has a conical point (necessarily of order $>2$) or a second boundary component.

The only remaining possibility is  that $\Sigma_{top}$ is a disc with boundary $b$, and there is no conical point. 
Any pair of mirrors in $b$ have to be adjacent, or joined by a boundary segment of $\Sigma$.
It follows that $b$ contains at most $3$ mirrors, and at most as many boundary segments. There must be 3 mirrors in order for $\chi(\Sigma)$ to be negative.
\end{proof}

\subsubsection{Orbifold with finite mapping class group}
 \label{fmg}

\index{orbifold!with finite mapping class group}
Suppose that $\Sigma$ is a compact hyperbolic surface. Unless it is a pair of pants, it contains a simple closed geodesic $\gamma$. If $\gamma$ is 2-sided, it defines a splitting of $\pi_1(\Sigma)$ over a maximal cyclic subgroup $H_\gamma$, and  the Dehn twist around $\gamma$ defines an infinite   order element of the mapping class group. If $\gamma$ is one-sided, the dual splitting is over an index 2 subgroup of $\pi_1(\gamma)$, corresponding to the boundary curve $\hat \gamma$ of a regular neighborhood; the Dehn twist around $\hat \gamma$ is homotopically trivial.  

The  pair of pants and the  twice-punctured projective plane contain  no 2-sided geodesic and have finite mapping class group.  All other hyperbolic surfaces contain  a   2-sided geodesic and have infinite mapping class group (note that the once-punctured Klein bottle contains a unique 2-sided geodesic; like the closed non-orientable surface of genus 3, it has no pseudo-Anosov mapping class).

In this subsection we generalize this discussion to orbifolds.

There are more examples of simple closed geodesics  which do not yield non-trivial twists.
If $\gamma$ is a mirror which is a full circle in $\partial \Sigma_{top}$, its fundamental group is isomorphic to $\bbZ\oplus \bbZ/2\bbZ$; 
  it defines a splitting of $\pi_1(\Sigma)$ over the index 2 subgroup isomorphic to $\bbZ$, but the associated Dehn twist is once again trivial.  If $\gamma$ is a geodesic arc whose endpoints belong to mirrors, its fundamental group is the dihedral group $D_\infty$ and there is no associated Dehn twist (because $D_\infty$ has trivial center). If $\gamma$ is a mirror joining two corner reflectors carrying $D_4$, it yields a splitting over $D_\infty$ but no twist.

 We refer to \cite{GL6} (especially Theorem 7.14) 
for a  general discussion of the relations between splittings and automorphisms   (see also \cite{DG2} or \cite{Lev_automorphisms},  and the discussion in Subsection \ref{zmax}). 
Here we limit ourselves to classifying orbifold groups which do not split over a maximal cyclic subgroup relative to the boundary subgroups; equivalently, we classify orbifolds  $\Sigma$ 
such that the interior of $\Sigma_{top}^*$ contains no 2-sided geodesic. 
These are the orbifolds whose mapping class group is finite.

Let $\Sigma$ be such an orbifold. We first note that
$\Sigma$ must be planar with at most 3 conical points or boundary components, or be a M\"obius band with one conical point, or a M\"obius band with one open disc removed,    or a projective plane with 2  conical points.

Now say that  a   component $b$ of $\bo\Sigma_{top}$ is \emph{simple} 
if $b$ contains no mirror (it is a component of $\bo \Sigma$) or $b$ is a single  mirror;
  we say that $b$ is a \emph{circular boundary component} or a \emph{circular mirror} accordingly. 
Note that a simple boundary component does not  contribute to $\chi(\Sigma)$ (see Definition \ref{defeuler}). In particular, if $\Sigma$ is planar and all boundary components of $\Sigma_{top}$ are simple, the total number of boundary components and conical points is 3.

If $b$ is a non-simple boundary component, consider a simple closed curve parallel to it inside $\Sigma_{top}$. Since it is not isotopic to a geodesic, it must be parallel to a simple boundary component $b'\ne b$, or bound a M\"obius band containing no conical point, or bound a disc containing at most one conical point.

One can now check that $\Sigma$ must be in the following list.  We write, e.g., $(S^2,3)$ to mean that the surface obtained from $\Sigma_{top}$ by removing the boundary and the conical points is a sphere minus 3 points. 
\begin{itemize}
\item $(S^2,3)$: a sphere with a choice of 3 conical points, circular boundary components, or circular mirrors;
\item $(S^2,2)$: an annulus one of whose boundary component is non-simple, and the other is either  a conical point or a circular boundary component or a circular mirror;
\item $(S^2,1)$:  a disc whose boundary is non-simple, with no conical point;
\item $(P^2,2)$: a projective plane with a choice of 2 conical points,  circular boundary components, or circular mirrors;

\item $(P^2,1)$: a M\" obius band whose boundary is non-simple  (no conical point).
\end{itemize}

\subsection{Definition and properties of quadratically hanging   subgroups}  \label{quah}

 As usual, we fix $G$, $\cala$, and $\calh$. 
  Let $Q$ be a subgroup of $G$.

\begin{dfn}[QH subgroup, fiber, extended boundary subgroup]\label{dfn_qh} 
\index{QH, quadratically hanging}\index{fiber (of a QH subgroup)}\index{extended boundary subgroup}

We say  that $Q$ is a \emph{QH subgroup} (over $\cala$, relative to  $\calh$) if: 
  \begin{enumerate}
  \item $Q=G_v$ is the stabilizer of a vertex $v$ of an $(\cala,\calh)$-tree $T$; 
  \item $Q$ is an extension $1\to F\to Q \to
  \pi_1(\Sigma )\to 1$, with $\Sigma$ a compact hyperbolic 2-orbifold;  we call $F$ the \emph{fiber},  and $\Sigma$ the underlying orbifold;
  \item each incident edge stabilizer, and each intersection $Q\cap gHg\m$ for $H\in \calh$, is an \emph{extended boundary subgroup}:
  by definition, this means that its image in $\pi_1(\Sigma)$ is   
  finite or contained in a boundary subgroup $B$ of $\pi_1(\Sigma)$. 
\end{enumerate}

 \end{dfn}

Condition (3) may be rephrased as saying that all groups in
 $\Inch_v$ (see Definition \ref{indup})
are extended boundary subgroups.

  In   full generality, the isomorphism type of $Q$ does not necessarily determine $F$ and $\Sigma$; when we refer to a QH subgroup, we always consider $F$ and $\Sigma$ as part of the   structure.   If $F$ is small, however,   it may be characterized as the largest normal subgroup of $Q$ which is small; in particular, any automorphism of $Q$ leaves $F$ invariant.  The group $Q$ does not have to be finitely generated if $F$ is not, but it is finitely generated relative to  $\Inch_v$  by Lemma \ref{relfg}, so Proposition \ref{arbtf} (guaranteeing  the existence of a minimal subtree) applies to actions of $Q$ relative to $\Inch_v$.

A vertex $v$ as above, as well as its image in $\Gamma=T/G$,  is called a \emph{QH vertex}.\index{QH, quadratically hanging}
It is the only point of $T$ fixed by $Q$ because extended boundary subgroups are proper subgroups of $Q$. Also note that the preimage in $Q$ of the finite   group carried by  a conical point or a corner reflector of
   $\Sigma$ is an extended boundary subgroup.

Any incident edge stabilizer $G_e$  is contained in an extension of $F$ by a virtually cyclic subgroup  of $\pi_1(\Sigma)$. But, even if $G$ is one-ended,  $G_e$ may meet $F$ trivially, or have trivial image in $\pi_1(\Sigma)$ (see Subsection  \ref{periph}). In particular,   in full generality, 
$F$ does not have  to belong to $\cala$,   or be universally elliptic if $T$ is a JSJ tree.

\begin{dfn}[Dual splitting of $G$, group $Q_\gamma$] \label{dfg}
A splitting  of $\pi_1(\Sigma)$ dual to a family of geodesics $\call$   as in Subsection \ref{cid} induces  a splitting of $Q$ relative to $F$.  By the third condition of Definition \ref{dfn_qh}, this splitting is also relative to    $\Inch_{|Q}$ (see Definition \ref{indup}), so extends to a splitting of $G$ relative to $\calh$ by Lemma \ref{extens}. We   say that this splitting of $G$ is \emph{dual to $\call$}\index{splitting!dual to a family of geodesics}\index{dual splitting}   
(or determined by $\call$). The edge group associated to $\gamma\in\call$ is denoted $Q_\gamma$.
\end{dfn}

 The 
   edge groups    $Q_\gamma$  
are extensions of the fiber $F$ by $\Z$ or $D_\infty$, and in general they do not have to be in $\cala$.
For many natural classes of groups $\cala$, though, such extensions are still in $\cala$ (this is the content of the stability  conditions\index{stability condition} of Subsection \ref{star}).

Conversely, we have seen (Proposition \ref{dual}) that small splittings of an orbifold group relative to the boundary subgroups are dual  to families of geodesics.  
We will prove  a similar statement (Lemma \ref{dual2}) for splittings of QH groups relative to incident edge stabilizers and groups in $\calh_{ | Q}$ (i.e.\ to $\Inch_{|Q}$),
but one has to make   additional assumptions. 

First,  the fiber has to be elliptic in the splitting; this is not automatic in full generality, but this holds
as long as $F$ and groups in $\cala$ are slender, see Lemma \ref{uesle} below.
Next, we need to ensure that boundary subgroups are elliptic;  this motivates
 the following definition.

\begin{dfn}[Used boundary component]\label{dfn_used}\index{used boundary component} 
 A boundary component $C$ of $\Sigma$ is \emph{used}  if the  group $B=\pi_1(C)$   (isomorphic to $\bbZ$ or $D_\infty$) 
 contains with   finite index the image of   an incident edge stabilizer 
or of a subgroup of $Q$ conjugate (in $G$) to a group in $\calh$. 
\end{dfn}

Equivalently,
$C$ is used if there exists some subgroup $H\in \Inch_v$
whose image in $\pi_1(\Sigma)$ is infinite and contained in $\pi_1(C)$ up to conjugacy.

Using Lemma \ref{arc}, we get:
\begin{lem}\label{lem_arcbis} 
  Let $Q=G_v$ be a QH vertex group of a tree $T$, with fiber $F$. If some boundary component is not used, then
$G$ splits relative to $\calh$ over  a group $F'$ containing $F$ with index at most 2.
Moreover, $T$ can be refined at $v$ using this splitting.
\end{lem}

\begin{rem}
 $F'=F$ whenever  the underlying orbifold has no mirror.
In general,  $F'$ does not have to  be in $\cala$.
\end{rem}

\begin{proof}
   Let $C$ be a boundary component of $\Sigma$. Lemma \ref{arc} yields 
a non-trivial splitting of $Q$ over a group  $F'$
 containing F with index	$\le2$   (with $F'=F$ if there is no mirror).  
If  $C$ is not used,
this splitting is relative to $\Inch_{ | Q}$. 
 By Lemma \ref{extens}, one may use it to refine $T$ at $v$. One obtains a splitting of $G$ relative to $\calh$, which may be collapsed to a one-edge splitting over $F'$.
\end{proof}

Proposition \ref{dual} implies:
   
   \begin{lem} \label{dual2} 
      Let $Q $ be a QH vertex group, with fiber $F$. 
   \begin{enumerate}
 \item Any non-trivial (minimal) 
 splitting of $Q$ relative to $F$ 
factors through   a splitting of $\pi_1(\Sigma)=Q/F$.
 \item If the splitting 
 is also relative to
 $\Inch_{|Q}$, 
 and every boundary component of $\Sigma$ is used, 
then the splitting of $\pi_1(\Sigma)$ is dominated by a splitting
 dual to a family of geodesics. In particular, $\Sigma$ contains an essential simple closed geodesic.

 \item  If, moreover,  the splitting of $Q$ has small edge groups, the splitting of $\pi_1(\Sigma)$  is dual to a family of geodesics.
\end{enumerate}
 \end{lem}
   
 \begin{rem}
 The induced splitting of $\pi_1(\Sigma)$ is   relative to the boundary subgroups if (2) holds, but not necessarily if only (1) holds. 
 \end{rem}

 \begin{proof}  Let $S$ be the Bass-Serre tree of the splitting of $Q$.  The group $F$ acts as the identity on $S$:   its fixed point set is nonempty, and  $Q$-invariant because $F$ is normal in $Q$.
 We deduce that the action of 
 $Q$ on $S$ factors through an action of $\pi_1(\Sigma)$. 
 
 Under the   assumptions of (2), this action is relative to all boundary subgroups because all boundary components of $\Sigma$ are used (by an incident edge stabilizer or a conjugate of a group in $\calh$).
 We apply  Proposition \ref{dual}.
 \end{proof}

  Flexible QH vertex groups $G_v$ of JSJ decompositions have non-trivial splittings relative to   $\Inch_v$
   (see Corollary \ref{cor_flex}).  As in Corollary \ref{courbe} we wish to  deduce that $\Sigma$ 
   contains an essential
simple closed geodesic, so as to rule  out  the small orbifolds of Proposition \ref{smallorb}.

    \begin{prop} \label{bu}
Let $Q=G_v$ be a QH vertex stabilizer of a JSJ tree over $\cala$ relative to $\calh$. 
Assume that $F$, and any subgroup  of $G_v$ containing $F$ with index 2, belongs to  $\cala$. Also assume that  $F$  is universally elliptic. Then:
\begin{enumerate}
 \item every boundary component of $\Sigma$ is used;
 \item if $Q$ is flexible, then $\Sigma$ contains an essential simple closed geodesic;
  \item 
    let $T$ be an $(\cala,\calh)$-tree such that $Q$ acts on $T$ with small edge stabilizers; if $Q$ is not elliptic in $T$, 
then the action of $Q$ 
on its minimal subtree 
 is dual, up to subdivision,
to a family of essential simple closed geodesics of $\Sigma$.
\end{enumerate}
  \end{prop}
 
 \begin{proof}
 If some boundary component of $\Sigma$ is not used, Lemma \ref{lem_arcbis} yields
a refinement of the JSJ tree. The new edge stabilizers contain $F$ with index at most 2, so belong to $\cala$ and are universally elliptic.  
This  contradicts the maximality of the JSJ tree.

If $Q=G_v$ is flexible, it acts non-trivially on an $(\cala,\calh)$-tree. This tree is also relative to $F$ because $F$ is universally elliptic. This yields a splitting of $G_v$ relative to $F$, which is relative to $\calh_{ | G_v}$ and to incident edge stabilizers because they are universally elliptic. Since every boundary component is used, we obtain a geodesic by  Lemma \ref{dual2}. 

 Applying the previous argument to $T$ as in (3) shows that
the splitting of $Q$ is dual to a family of geodesics by the third assertion of Lemma \ref{dual2}.
\end{proof}

The  first part of the following proposition   shows that, conversely, the existence of an essential  geodesic implies flexibility.

\begin{prop}\label{prop_QHUE} 
  Let   $Q$ be a QH vertex group.
Assume that $\Sigma$ contains an essential simple closed geodesic, 
and that, for all essential simple closed geodesics $\gamma$, 
the group $Q_\gamma$  (see Definition \ref{dfg}) belongs to $\cala$.

\begin{enumerate}
\item
Any universally elliptic element (resp.\ subgroup) of $Q$ is contained in an extended boundary subgroup. 
 In particular, $Q$ is not universally elliptic. 
 \item  If $J<Q$ is small in $(\cala,\calh)$-trees, its image in $\pi_1(\Sigma)$  is virtually cyclic.
 \end{enumerate}
 
\end{prop}

\begin{rem}
 This holds
under the weaker assumption that
the set of essential simple closed geodesics $\gamma$ such that $Q_\gamma\in \cala$
\emph{fills} $\Sigma$ (in the sense of Definition \ref{dfn_fill}).
\end{rem}

  \begin{proof}
 Any subgroup of $Q$ contained in the union of all extended boundary subgroups is contained in a single extended boundary subgroup, so  to prove (1) it suffices to show that, if  an element $g$   does not lie in an extended boundary subgroup, then $g$ is not   universally elliptic. 
 
 The image of $g$ in $\pi_1(\Sigma)$ has infinite order, so acts non-trivially in a splitting of  $\pi_1(\Sigma)$ dual to a geodesic $\gamma$ (see Definition \ref{dfn_fill}). 
 The splitting of $G$ dual to $\gamma$    is relative to $\calh$, and the edge group  $Q_\gamma$ is assumed to be in $\cala$, 
  so $g$ is not universally elliptic.
  
    If  $J$ as in (2) has infinite image   in $\pi_1(\Sigma)$, it acts non-trivially in a splitting of $G$  dual to $\gamma$ as above. The action preserves a line or an end by smallness of $J$, so the image of $J$ in $\pi_1(\Sigma)$ is virtually cyclic by Remark \ref{rem_small}.
\end{proof}

 Proposition \ref{bu} requires  $F$ to be universally elliptic. We show that this  
 is   automatic for splittings over slender groups.

\begin{lem} \label{uesle} 
Let $Q$ be a QH vertex group.
  If $F$ and all groups in $\cala$ are slender, then $F$ is universally elliptic. 
\end{lem}

  \begin{proof}
  Suppose that $F$ is slender   but not   universally
elliptic. It acts non-trivially  on some
 tree, and there is a unique $F$-invariant  line $\ell$. This line is  $Q$-invariant because   $F$ is normal in $Q$. It follows that $Q$ is an extension of a group in $\cala$ by its image in $\Isom(\ell)$, a   virtually cyclic group.
If groups in  $\cala$ are slender, we deduce that $Q$ is slender, a contradiction because $Q$ maps onto $\pi_1(\Sigma)$ with $\Sigma
$ hyperbolic.
\end{proof}

\begin{cor} \label{ufs}
Let $Q$ be a QH vertex group $G_v$ of  a JSJ tree $T$.  
 Assume that all groups in $\cala$ are slender,\index{slender group} and      that every extension of the fiber $F$ by a virtually cyclic group belongs to $\cala$. 
 Then:
 \begin{enumerate}
\item  $F$   is universally elliptic; it is the largest slender normal subgroup of $Q$;
\item all boundary components of the underlying orbifold $\Sigma$ are used;
\item if $G$ acts on a tree $S$, and $Q$ does not fix a point, then the action of $Q$ on its minimal subtree is dual to a family of essential simple closed geodesics of $\Sigma$;
\item $Q$ is flexible if and only if $\Sigma$ contains an essential simple closed geodesic; 
\item if $Q$ is flexible, than any universally elliptic subgroup of $Q$ is
an extended boundary subgroup.
 \end{enumerate}
\end{cor}

 This follows from   results proved above.

\subsection{Quadratically hanging subgroups are elliptic in the JSJ}
\label{sec_QH_ell}

The goal of this subsection is to prove that,  
under suitable hypotheses, any QH vertex group  
is elliptic in the JSJ deformation space 
(if we do not assume existence of the JSJ deformation space, we obtain ellipticity in every universally elliptic tree).

We start with the following fact: 
 
\begin{lem} \label{pass}  If $G$ splits over a group $K\in\cala$, but does not split over any infinite index subgroup of $K$, 
then $K$ is elliptic in the JSJ deformation space.  
\end{lem} 

\begin{proof} Apply Assertion 3 of Lemma \ref{cor_Zor} with   $T_1$   a JSJ tree  and $T_2$ a one-edge splitting over $K$. 
\end{proof}

\begin{rem} \label{uni} If $K$ is not universally elliptic, it fixes a unique point in any JSJ tree. Also note that being elliptic or universally elliptic is a commensurability invariant, so the same conclusions hold for groups commensurable with $K$. 
\end{rem}

In \cite{Sela_structure,RiSe_JSJ,DuSa_JSJ}
\index{Rips-Sela}\index{Dunwoody-Sageev}\index{Fujiwara-Papasoglu}  it is proved that,  if  $Q$
 is a  QH vertex group in some splitting (in the class considered), then $Q$ is 
elliptic in the JSJ deformation space.

This is not true in general, even if $\cala$ is the class of cyclic groups: $\F_n$ contains many
 QH subgroups, none of them elliptic in the JSJ deformation space (which consists of free actions, see Subsection \ref{free}).
This happens because $G=\F_n$ splits over groups in $\cala$ having infinite index in each other, 
something which is prohibited by the hypotheses of the papers mentioned above (in \cite{FuPa_JSJ},  $G$ is allowed to split over a subgroup of infinite index in a group in $\cala$, but $Q$ has to be the enclosing group of \emph{minimal splittings}, see  
\cite[Definition  4.5 and Theorem  5.13(3)]{FuPa_JSJ}).

A different counterexample will be given in Example \ref{exq} of Subsection \ref{ab}, with $\cala$ the family of abelian groups. 
 In that example the QH subgroup   $Q$
 only has one abelian splitting, which is universally elliptic, so $Q$ is not elliptic in the JSJ  space.  

These examples explain the hypotheses in the following  
result. 

\begin{thm}\label{qhe} Let   $Q$ be a   QH vertex group.\index{QH, quadratically hanging}
Assume that, if  $\hat J\inc Q$ is the preimage of a virtually cyclic subgroup   $J\inc \pi_1(\Sigma)$,  
  then $\hat J$  belongs to $\cala$   
and  $G$ does not split    over a subgroup of
infinite index of $\hat J$.
If one of the following conditions holds, then 
$Q$ is elliptic in the JSJ deformation space: 
\begin{enumerate}
\item \label{ass_main}   $\Sigma$ contains an essential  simple closed geodesic $\gamma$;
\item \label{ass_QG}  $\calh=\es$, and the boundary of $\Sigma$ is nonempty;
\item \label{ass_FUE}  the fiber $F$ is elliptic in the JSJ deformation space;
\item \label{ass_slender} all groups in $\cala$ are slender.
\end{enumerate}

\end{thm}

 Note that $  \hat J$ contains $F$, so $F\in\cala$.

\begin{proof}  
Fix a JSJ tree $T_J$.
We start by proving ellipticity of $Q$ in $T_J$  assuming the existence of an essential simple closed geodesic $\gamma$. 

 Given
a lift $\Tilde \gamma\inc \Tilde \Sigma$, we have denoted by $H_{\gamma}$ 
the subgroup of $\pi_1(\Sigma) $ consisting of elements which preserve $\Tilde \gamma$ and each of the   half-spaces bounded by    $\Tilde \gamma$,
and by $Q_{\gamma}$   the preimage of $H_{\gamma}$ in $Q$.  We now write $Q_{\tilde \gamma}$ rather than $Q_{\gamma}$ because we want it to be well-defined, not just up to conjugacy.

The group $Q_{\tilde \gamma}$  belongs to $\cala$,
and $G$ splits
over $Q_{\tilde
\gamma}$. 
By Proposition \ref{prop_QHUE}, $Q_{\tilde\gamma}$ is not universally elliptic.  
By Lemma \ref{pass} and Remark \ref{uni}, $Q_{\tilde\gamma}$  fixes a  unique point $c_{\tilde \gamma}$ in  $T_J$.

\begin{lem} \label{above}
 Let $\tilde\gamma,\tilde\gamma'\inc\tilde\Sigma\setminus\bo\tilde\Sigma$ be lifts of simple geodesics. If $\tilde\gamma$ and $\tilde\gamma'$ intersect, 
then $c_{\tilde \gamma}=c_{{\tilde \gamma}'}$.
 \end{lem}

\begin{proof} We can assume $\tilde\gamma\neq\tilde\gamma'$. 
Let $T'$ be the Bass-Serre tree of the splitting of $G$ determined by $\tilde \gamma'$. 
It contains a unique  edge $e'$ with stabilizer $Q_{\tilde \gamma'}$. Since ${\tilde \gamma}$ and ${\tilde \gamma}'$ intersect, 
the group $Q_{\tilde \gamma}$ acts hyperbolically on $T'$, and its minimal subtree $M$ (a line) contains $e'$. 
Let $T_1$ be a refinement of $T_J$ which dominates $T'$,    as   in Lemma \ref{prop_refinement},  and let $  M_1\inc  T_1$ be the minimal
subtree of $Q_{\tilde \gamma}$. 

The image of $  M_1$ in $T_J$ consists of the single point $c_{\tilde \gamma}$ (because $  T_1$ is a refinement of $T_J$), 
and its image by any equivariant map $f:  T_1\to T'$ contains $M$. 
Let $  e_1$ be an edge of $  M_1$ such that $f(  e_1)$ contains $e'$. 
The stabilizer $G_{  e_1}$ of $  e_1$ is contained in $Q_{{\tilde \gamma}'}$,
so $G_{e_1}$ fixes $c_{\tilde \gamma'}$.
Since $G$ does not split over infinite index subgroups of $Q_{\tilde\gamma'}$,
the index of $G_{e_1}$   in  $Q_{\tilde\gamma'}$ is finite. The group $Q_{\tilde\gamma'}$ is not universally elliptic (it acts non-trivially in the splitting dual to $\gamma$), 
so $c_{{\tilde \gamma}'}$ is the unique fixed point of $G_{e_1}$ in $T_J$ by Remark \ref{uni}. 
But $G_{  e_1}$ fixes $c_{\tilde \gamma}$ because $e_1$ is mapped to $c_{\tilde\gamma}$ in $T_J$,
so $c_{\tilde \gamma}=c_{{\tilde \gamma}'}$.
\end{proof}

We can now conclude. Since the set of all essential simple geodesics of $\Sigma$ fills $\Sigma$ (Corollary \ref{lem_fill}),
the union of their lifts is a connected subset of $\Tilde\Sigma$ (Definition \ref{dfn_fill}).
In particular, given   any pair $\tilde\gamma,\tilde\gamma'$ of such lifts,
there exists a finite sequence $\tilde \gamma=\tilde \gamma_0,\tilde \gamma_1,\dots,\tilde \gamma_p=\Tilde\gamma'$
where $\Tilde \gamma_i$ intersects $\Tilde \gamma_{i+1}$. Lemma \ref{above} implies that $c_{\tilde\gamma}=c_{\tilde\gamma'}$, so   $c_{\tilde \gamma}$ does not depend on $\tilde\gamma$.
It is fixed by $Q$, so $Q$ is elliptic.

We have   proved the theorem in   case (\ref{ass_main}) (when $\Sigma$ contains a geodesic).   
 We reduce   cases
(\ref{ass_QG})   and (\ref{ass_FUE})
 to this one, using Lemma \ref{dual2} to find a geodesic 
   (case (\ref{ass_slender}) reduces to (\ref{ass_FUE}) by  Lemma \ref{uesle}).

We   first show that every boundary component of $\Sigma$ is used (see Definition \ref{dfn_used}). 
If not, Lemma   \ref{lem_arcbis} yields a splitting of $G$
over a group   $F'$ containing $F$ with index 1 or 2.  
By Remark \ref{arcb}, the   group $F'$ is contained in  the preimage $\hat J$ of a 2-ended subgroup $J\inc \pi_1(\Sigma)$, 
and this   contradicts the assumptions of the theorem since $F'$ has infinite index in $\hat J$.

If $\calh=\es$ and $\partial \Sigma\ne \es$, every boundary component $C$ is used by   an incident edge stabilizer $G_e$ whose image in $\pi_1(\Sigma)$ is contained  in $B=\pi_1(C)$ with finite index.
 Let $\Hat B\subset Q$ be the preimage of $B$.
 Folding yields  a non-trivial splitting of
$G$ over $\Hat B$ (if for instance $G=R*_{G_e} Q $, then
$G=(R*_{G_e} \Hat B)*_{\Hat B}Q$).
Since $\hat B\in\cala$,
  Lemma \ref{pass} implies that 
$\hat B$ is elliptic in the JSJ deformation space; in particular, $F$ is elliptic. If $Q$ is not elliptic,
 Lemma \ref{dual2}   implies that $\Sigma$ contains a geodesic. 

Now suppose that $F$ is   elliptic in the JSJ space, but $Q$ is not. The action of $Q$ on any JSJ tree then factors through $\pi_1(\Sigma)$ by Assertion (1) of Lemma \ref{dual2}. Every boundary subgroup $B=\pi_1(C)$ is elliptic:  
this follows from the previous argument if $C$ is used by an incident edge stabilizer, and holds if $C$ is used by a group in $\calh$ because splittings are relative to $\calh$. Proposition \ref{dual} yields a geodesic.
  \end{proof}
  
\begin{rem} \label{rem_qhe} Under the assumptions of Theorem \ref{qhe},
assume moreover that all groups in $\cala$ are slender, 
 and that $G$ does not split over a subgroup of $Q$ whose
 image in $\pi_1(\Sigma)$ is finite.
Then $Q$ fixes a unique point $v_J\in T_J$,  so $Q\subset G_{v_J}$.
We claim that, \emph{if
$G_{v_J}$ (hence also $Q$) is universally elliptic, then $v_J$ is a QH vertex of $T_J$.}
This is  used in \cite{GL5} (proof of Theorem 4.2).

Let $T$ be a tree in which $Q$ is a   QH vertex group $G_w$. 
Note that $Q=G_{v_J}$ because $G_{v_J}$ is elliptic in $T$.
We have  to show that, if $e$ is an edge of $T_J$ containing $v$, then $G_e$ is 
 an extended boundary subgroup of $Q$.
Let $\hat T$ be a refinement of $T_J$ which dominates $T$.
Let $\hat w$ be the unique  point of $\hat T$ fixed by $Q$, and let $f:\hat T\to T$ be an equivariant map.  Let $\hat e$ be the lift of $e$ to $\hat T$.

If $f(\hat e)\ne\{w\}$,  then $G_e$ fixes an edge of $T$ adjacent to $w$, so is
an extended boundary subgroup of $Q$. Otherwise, consider
 a segment $x\hat w$, with $f(x)\ne w$, which contains $\hat e$. Choose such a segment of minimal length,
and let $\varepsilon=xy\ne \hat e$ be its initial edge (so $f(y)=w$).
We have $G_\varepsilon\inc G_y\inc G_w =Q$, and $G_\varepsilon$ fixes an edge of $T$ adjacent to $w$.
Since $G$ does not split over groups mapping to finite groups in $\pi_1(\Sigma)$,
the image of $G_\varepsilon$ in $\pi_1(\Sigma)$ is a finite index subgroup of a boundary
subgroup $B\inc\pi_1(\Sigma)$.
But we also have $G_\varepsilon\inc G_w=G_{\hat w}$, so that $G_\varepsilon\inc G_{\hat e}=G_e$.
Being slender and containing a finite index subgroup of $B$,  the image of $G_e$ in $\pi_1(\Sigma)$ has to be contained in $B$.
\end{rem}

  \begin{cor}\label{cor_qhe2}
Let $\cala=VPC_{\leq n}$\index{VPC} be the class of all virtually polycyclic groups of Hirsch length $\leq n$.
 Assume that $G$ does not split over a group in $VPC_{\leq n-1}$.
Let $Q$ be a QH vertex group with fiber in $VPC_{\leq n-1}$, in some splitting of $G$ over $VPC_{\leq n}$.

If $T$ is any JSJ tree of $G$ over $\cala$, 
then $Q$ is contained in a QH vertex stabilizer of $T$.
  \end{cor}
 
  Proposition 
\ref{prop_qhe_intro} of the introduction is the case $n=1$.

\begin{proof}
 The group $Q$ fixes a point $v$ in $T$ by Theorem \ref{qhe}. 
 If $G_v$ is universally elliptic, we apply  Remark \ref{rem_qhe}. If not, $v$ is flexible and we will prove in Theorem \ref{thm_description_slender} that $G_v$ is QH.
\end{proof}

\subsection{Peripheral structure of quadratically hanging  vertices}\label{periph}

   Suppose that 
$Q$ is a QH vertex group of a JSJ decomposition.
The incident edge groups are extended boundary subgroups  (see Definition \ref{dfn_qh}).
However,
the collection of incident edge groups (the family $\Inc_v$ of Definition \ref{dfn_incv}) may change 
when the JSJ tree $T$ varies in the JSJ deformation space (though the collection of extended boundary subgroups usually   does not change, see Assertion (6) of Theorem \ref{thm_description_slender}).

\index{peripheral structure}
In Section 4 of \cite{GL2} we introduced a collection $\calm_0$ of subgroups of $Q$, which is related to the incident edge groups,
but which does not depend on the tree in the JSJ deformation space. We called this collection    the \emph{peripheral structure} of $Q$.
 
 The goal of this subsection is to show that the peripheral structure contains more information than the collection of all   extended subgroups, and may be fairly arbitrary.

In the examples, $\cala$ is the class of slender groups, and no slender group is conjugate to a proper subgroup of itself.
In this context, the peripheral structure $\calm_0$ of a non-slender vertex stabilizer $G_v$  of a tree $T$ 
may be determined as follows. We first collapse   edges so as to make $T$ reduced (see Subsection \ref{defspa}). The image of $v$ has the same   stabilizer,
and we still denote it by $v$. Then $\calm_0$ is the set of conjugacy classes of incident edge stabilizers $G_e\inc G_v$  
that are not properly contained in another incident edge stabilizer of $G_v$ 
(see \cite{GL2} for details).

We have seen (Proposition \ref{bu})
 that $\calm_0$ often uses   every boundary component of $\Sigma $.
Apart from that, the peripheral structure of $Q$ may be fairly arbitrary.
We shall now give examples. In particular, Example \ref{perd2} will show the following proposition:

\begin{prop}   If $Q$ is a   QH vertex group  of a slender JSJ decomposition  of a one-ended group, it is possible for an incident edge
group to meet
$F$ trivially, or   to have trivial image in $\pi _1(\Sigma )$.
\end{prop}

We work with $\calh=\es$ and $\cala$ the family of slender subgroups.

\begin{figure}[htbp]
  \centering
  \includegraphics{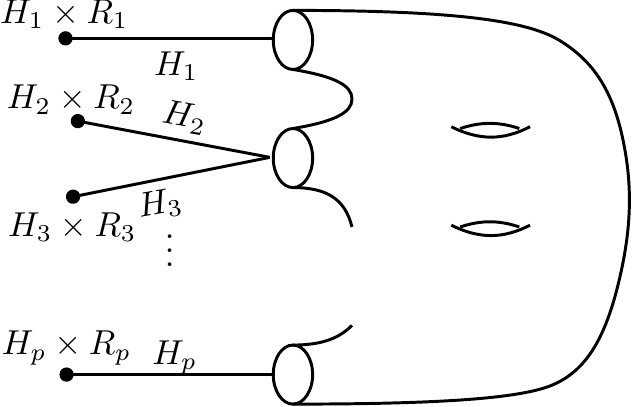}
  \caption
  {An example of a QH group in a JSJ decomposition.}
  \label{fig_exple_QH}
\end{figure}

To give examples, we use the filling construction described in Subsection \ref{filling}, see Figure \ref{fig_exple_QH}.
We start with   an extension $1\to F\to Q \to \pi_1(\Sigma )\to 1$,  with $F$   slender and $\Sigma $ a compact orientable surface  
(with genus $\ge 1$, or with at least 4  boundary components). 
Let $H_1,\dots,H_p$ be a finite
family of infinite   extended boundary subgroups of $Q$  as defined in Definition \ref{dfn_qh} (note that they are slender).   
  We impose that, for each boundary subgroup $B$ of
$\pi _1(\Sigma )$, there is an $i$ such that 
$H_i$ maps onto
a finite index subgroup of $B$ (i.e.\  every boundary component is used by some $H_i$  in the sense of  Definition \ref{dfn_used}). 
Let $R_i$ be a non-slender finitely presented group with Serre's property (FA) (it has no non-trivial action on a tree), for
instance $SL(3,\Z)$. 
As in Subsection  \ref{filling}, we define a finitely presented group  $\hat G$ by amalgamating $Q$ with   $K_i=H_i\times R_i$ over $H_i$ for
each $i$.

\begin{lem}\label{lem_amal}
The Bass-Serre tree $T$ of the amalgam defining $\hat G$ is a slender JSJ tree,  
$Q$ is a flexible QH subgroup, and $\hat G$ is one-ended.  If no $H_i$ is conjugate  in $Q$ to a subgroup of
  $H_j$ for $i\ne j$, the peripheral structure $\calm_0$ consists of
the conjugacy classes of the $H_i$'s.
\end{lem}

\begin{proof} 
Let
$T'$ be any  tree. Each
$R_i$ fixes a unique point because it is not slender, and this point is also fixed by $H_i$. In particular,
$H_i\times R_i$, $H_i$  and $T$  are universally elliptic. 
To prove that $T$ is a JSJ tree, it suffices to see that $Q$ is elliptic in any universally elliptic  tree $T'$.

By Lemma \ref{uesle}, $F$  is  
universally elliptic.  If $Q$ is not elliptic in $T'$, then by Lemma \ref{dual2} 
the action of $Q$ on its minimal subtree    $T_Q\subset T'$
   factors through a nontrivial action of $\pi _1(\Sigma )$ with slender (hence cyclic) edge stabilizers. Since every $H_i$, hence every
boundary subgroup of $\pi _1(\Sigma )$, is elliptic, this action  is dual to a
system of disjoint 
geodesics on $\Sigma$  by Proposition \ref{dual}.
  By Proposition \ref{prop_QHUE}, 
no edge stabilizer of $T_Q$ is universally elliptic,  contradicting  universal ellipticity of $T'$.

This shows that $T$ is a JSJ tree, and 
 $Q$ is flexible because $\Sigma $ was chosen to contain
intersecting simple closed curves (see Corollary \ref{ufs}).

By Proposition  \ref{prop_embo}, one obtains a JSJ tree of $\hat G$ over finite groups
by collapsing all edges of $T$ with infinite stabilizer. Since each $H_i$ is infinite,
this JSJ tree  is trivial, so $\hat G $ is one-ended. 

The assertion about $\calm_0$ follows from the definition of $\calm_0$ given in
\cite{GL2}. 
\end{proof}

\begin{example}\label{perd}
Let $\Sigma $ be a punctured torus, with fundamental group  $\langle a,b\rangle$.
Write $u= [a,b]$. Let $Q=F\times\langle a,b\rangle$, with $F$ finite and
non-trivial. Let $H_1=\langle F, u^2\rangle$ and $H_2=\langle u\rangle$. 
The peripheral structure of $Q$ in the JSJ tree $T$ consists of
two elements, though $\Sigma $ only has one boundary component. There is a JSJ
tree $T'$ such that incident edge groups are conjugate to $\langle F, u \rangle$
(the quotient $T'/G$ is a tripod), but it does not display the peripheral structure of $Q$.
\end{example}

\begin{example} \label{perd2}
Let $\Sigma , a,b, u$ be as above. Again write $Q=F\times\langle
a,b\rangle$, but now $F=\langle t\rangle$ is infinite cyclic. Let $H_1=\langle  
u\rangle$ and $H_2=\langle t\rangle$. Then
$H_1$ meets
$F$ trivially, while $H_2$ maps trivially to $\pi _1(\Sigma )$.
\end{example}

\begin{example}\label{perd3}
Assume that the orbifold $\Sigma$ has a conical point $x$ carrying a finite  cyclic  group $F_x$ and that the fiber $F$ is infinite.
Then one can attach an edge to $x$:
one   chooses  any infinite subgroup $H_x$ 
  of the preimage   of $F_x$ in $Q$, and one constructs  an amalgam $Q*_{H_x}(H_x\times R_x)$. Similar constructions are possible
with  $x$ a  corner reflector, a point on a mirror, or even  a point with trivial isotropy.
\end{example}

\subsection{Flexible vertices of  abelian JSJ decompositions} 
\label{ab}

\index{flexible vertex, group, stabilizer}
We shall see in Section \ref{Fujpap}  that, 
if $\cala$ is the family of  cyclic subgroups, or virtually cyclic
subgroups, or slender subgroups, then non-slender flexible vertex groups $Q$ of JSJ decompositions   over $\cala$ are QH with slender fiber (we   say that $Q$ is slender-by-orbifold). 
Things are more complicated 
when $\cala$ is the family of   abelian subgroups (or equivalently of  finitely generated abelian subgroups, see Proposition \ref{loc}). 

The
basic reason  is the following: if a group $Q$ is an extension $1\to F\to Q \to \pi
_1(\Sigma )\to 1$ with $F$  a finitely generated abelian group and $\Sigma $ a
surface, a    splitting of
$\pi _1(\Sigma )$ dual to  a simple closed curve induces a splitting of $Q$ over a   subgroup which is
slender (indeed polycyclic) but not necessarily abelian.

Using the terminology of Definition \ref{dfn_RN},
 the regular neighbourhood\index{regular neighbourhood} of two abelian splittings is not necessarily an abelian splitting.

In fact, we shall now   construct examples showing:   
\begin{prop}
\begin{enumerate}
\item  {Flexible subgroups of abelian JSJ
trees are not always  slender-by-orbifold groups.} 
\item {One cannot
always   obtain
an abelian JSJ tree\index{abelian tree} by collapsing edges of a slender JSJ tree.}
\end{enumerate}
\end{prop}

By Proposition \ref{prop_embo}, 
 one can obtain an abelian JSJ tree by refining and collapsing  a slender JSJ tree. The point here is that collapsing alone is not always sufficient.  It may be shown that collapsing suffices when $G$ is  torsion-free, finitely presented, and CSA (see Proposition 8.12 of \cite{GL3a}).

We use the same construction as in the previous subsection, but now   $\pi _1(\Sigma )$ will act non-trivially on the fiber $F$.

\begin{example}  In this example $F\simeq\Z$. Let $\Sigma $ be obtained by
gluing a once-punctured torus to one of the boundary components of a pair of
pants. Let
$M$ be a circle bundle over $\Sigma $ which is trivial over the punctured torus
but non-trivial over the two boundary components of $\Sigma $. Let $Q=\pi
_1(M)$, and let
$H_1,H_2$ be the fundamental groups of the components of $\partial M$
(homeomorphic to Klein bottles). Note that $H_1,H_2$ are non-abelian. Construct
 $\hat G$ by amalgamation  with $H_i\times R_i$ as above. 
We claim that \emph{the abelian JSJ
decomposition of
$\hat G$ is trivial, and $\hat G$ is flexible (but not slender-by-orbifold). }

We argue as in the proof of Lemma \ref{lem_amal}.
We know that $H_1\times R_1$ and $H_2\times R_2$ (hence also $F$) are universally elliptic.  If  $T$ is any tree with abelian edge stabilizers, the  action of $Q$  on its minimal subtree factors through $\pi_1(\Sigma)$, and the action of $\pi_1(\Sigma)$ is dual to a system of simple closed curves.
But not all simple closed curves give rise to an abelian splitting of $Q$: they have to be ``positive'', in the sense that  the bundle is trivial over them.

To
prove that none of these splittings is universally elliptic, hence that the abelian JSJ space of $\hat G$  is trivial  and $\hat G$ is flexible,   it suffices to see that any positive
curve intersects (in an essential way) some other positive curve. This is true for
the curve $\delta $ separating the pair of pants from the punctured torus (one easily
constructs a positive curve meeting $\delta $ in 4 points). It is  also 
true for
curves meeting $\delta $. Curves disjoint from $\delta $ are contained in the
punctured torus, and the result is true for them.

 Every map from $SL(3,\Z)$ to a 2-orbifold group has finite image, so $\hat G$ is not isomorphic to a slender-by-orbifold group if we choose $R_i=SL(3,\Z)$.

\begin{rem} If one performs the construction adding a third group $H_3=F$, then $\hat G$
becomes a flexible vertex group in a group 
whose JSJ decomposition is non-trivial. 
\end{rem}
\end{example}

\begin{example}  \label{exq}
Now $F=\Z^2$. Let $\Sigma $ be a surface of genus $\ge2$ with two boundary
components $C_1,C_2$. Let  $\gamma $ be  a simple closed curve 
separating $C_1$ from $C_2$. Let $\Sigma '$ be the space obtained from $\Sigma
$ by collapsing $\gamma $ to  a point. 
Map $\pi _1(\Sigma )$ to  $SL(2,\Z)\inc\Aut(\Z^2)$ by projecting to  $\pi
_1(\Sigma ')$ and embedding the free group  $\pi _1(\Sigma ')$ into $SL(2,\Z)$.
Let $Q$ be the associated semi-direct product $\Z^2\rtimes \pi _1(\Sigma )$, and
$H_i=\Z^2\rtimes \pi _1(C_i)$.  Construct $\hat G$ as before. 

Abelian splittings of $\hat G$ now
come from simple closed curves on $\Sigma $ belonging to the kernel of $\rho: \pi _1(\Sigma )\to SL(2,\Z)$. 
But it is easy to see that $\gamma $ is the only
such curve.
 It follows that  the one-edge splitting dual to
$\gamma $ is an abelian JSJ  decomposition of $\hat G$. It has two rigid vertex groups. It cannot be obtained by collapsing a
slender JSJ splitting. 
\end{example}

\section{JSJ decompositions over slender groups} \label{Fujpap}

\index{slender group}
The main result of this section is the description of   JSJ decompositions over slender groups.
Recall  (Subsection \ref{sec_slender}) that a subgroup $A\inc G$ is slender if $A$ and all its subgroups are finitely generated. Whenever $G$ acts on a tree, $A$ fixes a point or leaves a line invariant.

Our approach essentially follows Fujiwara and Paposoglu \cite{FuPa_JSJ},\index{Fujiwara-Papasoglu} but with   simplifications. In particular, we do not have to ``deal with a third splitting''  
(see below for further discussion).

\subsection{Statement of results} \label{star}

Let $G$ be a finitely generated group, $\cala$ a family of subgroups stable under conjugation and taking subgroups,
and $\calh$ a finite set of finitely generated subgroups of $G$ such that $G$ is finitely presented relative to $\calh$.

The goal of this section is to show that non-slender \emph{flexible vertex groups $Q$ of JSJ decompositions over $\cala$ relative to $\calh$ are QH} (see Subsection \ref{quah}).  We need two assumptions on $\cala$. First, groups in $\cala$ should be slender  (or at least slender in $(\cala,\calh)$-trees, see Subsection \ref{slt}). The second is a stability condition  involving a subfamily $\calf\inc\cala$ (we will show that    fibers of QH flexible vertex groups  belong to $\calf$).

\begin{dfn}[Stability Condition ($SC$)]\label{dfn_SC}\index{stability condition ($SC$)}\index{0SC@($SC$), ($SC_\calz$): stability conditions}
We say that $\cala$ satisfies the \emph{stability condition ($SC$)}, with fibers\index{fiber (of a QH subgroup)} in  a family of subgroups $\calf$,  if the following hold for every short exact sequence $$1\to F\to A\to K\to1$$ with $A< G$:
\begin{enumerate}
\item
if $A\in \cala$, and 
  $K$ is isomorphic to $\Z$ or $D_\infty$, then $F\in\calf$; 

\item
if $F\in \calf$ and $K$ is isomorphic to  a quotient of $\Z$ or $D_\infty$, then $A\in\cala$.

\end{enumerate}
\end{dfn}

The group $\Z$ acts on the line in an orientation-preserving way. 
  The infinite dihedral group $D_\infty$ also  acts on the line, but orientation is not preserved.
In (2), the group $K$ may be finite (cyclic or dihedral) or infinite (isomorphic to $\Z$ or $D_\infty$). 

If $Q$ is a QH subgroup with fiber $F$, recall that any simple geodesic $\gamma$ on $\Sigma$ defines a splitting of $G$ over a group $Q_\gamma$ which is an extension of $F$ by $\Z$ or $D_\infty$. If there is a geodesic $\gamma_0$ such that $Q_{\gamma_0}\in\cala$, the stability condition ensures first that $F\in\calf$, and then that $Q_\gamma\in\cala$ for every $\gamma$ 
(compare the assumptions of Proposition \ref{prop_QHUE}, Corollary \ref{ufs} and Theorem \ref{qhe}). Failure of the stability condition explains why Theorem \ref{thm_description_slender} below does not apply to abelian JSJ splittings\index{abelian tree} (see Subsection \ref{ab}).

On the other hand, one easily checks that the stability condition holds   in the following cases:
  \begin{itemize}
\item $\cala$=\{slender\} (i.e.\ $\cala$ consists of all slender subgroups of $G$), with $\calf=\cala$;
\item $\cala$=\{virtually cyclic\},\index{virtually cyclic} with $\calf$=\{finite\};
\item $\cala$=\{virtually polycyclic\}, with $\calf=\cala$;
\item $\cala=\{VPC_{\leq n}\}$,\index{VPC} the virtually polycyclic subgroups of Hirsch length at most $n$, with 
 $\calf$=\{VPC$_{\leq n-1}$\};
 \item   $G$ is a torsion-free CSA group,\index{CSA group} $\cala$=\{finitely generated abelian\}, with $\calf=\cala$ (recall that a group is CSA if all maximal abelian subgroups are malnormal).\index{CSA group} 

\end{itemize}
 See Definition \ref{dfn_SCZ} for a different condition, which applies to $\cala$=\{cyclic\}.
Also recall that, since $G$ is relatively finitely presented, a JSJ decomposition  over   $\cala$ is also one over groups locally in $\cala$
 (Proposition \ref{loc}).

\begin{thm}\label{thm_description_slender}
   Suppose that all groups in $\cala$ are slender,\index{slender group} and $\cala$ satisfies the stability condition ($SC$), with fibers in a family $\calf$. Let $\calh$ be a finite family  of finitely generated subgroups. Let $G$ be finitely presented (or only finitely presented relative to  $\calh$).

If $Q$ is a non-slender flexible\index{flexible vertex, group, stabilizer} vertex group of a JSJ decomposition of $G$ over $\cala$ relative to $\calh$, then:
\begin{enumerate}
\item $Q$ is QH\index{QH, quadratically hanging} with fiber in $\calf$ (it maps onto $\pi_1(\Sigma)$, where $\Sigma$ is a compact hyperbolic 2-orbifold, with kernel $F\in \calf$; the image of an incident edge group in $\pi_1(\Sigma)$ is finite or contained in a boundary subgroup);
\item  if $G$ acts on a tree and $Q$ does not fix a point, the action of $Q$ on its minimal subtree is dual to a family of geodesics of $\Sigma$;
\item $F$ and extended boundary subgroups are universally elliptic;
\item every boundary component of $\Sigma$ is used; 
\item  $\Sigma$ contains an essential closed geodesic;
\item every universally elliptic subgroup of $Q$ is  an extended boundary subgroup. 
\end{enumerate} 

\end{thm}

 Once (1) is known, Assertions (2)-(6) are direct consequences of
 the results of Subsection \ref{quah}.
See 
  Subsection \ref{liste} for the list of orbifolds containing no essential closed geodesic, 
and Subsection \ref{pslflex} for a description of slender flexible vertex groups.

\begin{cor}\label{cor_slender} 
Let $\cala$ be the class of all slender subgroups of $G$. 
Let $\calh$ be a finite family  of finitely generated subgroups. Let $G$ be finitely presented (or only finitely presented relative to  $\calh$).

Then every flexible vertex of any JSJ decomposition over $\cala$ relative to $\calh$ is either slender  
or QH with slender fiber.
\qed
\end{cor}

  In view of the examples given above, one has a similar description of JSJ decompositions over virtually cyclic groups (non-slender flexible groups are QH with finite fiber), 
over $VPC_{\leq n}$-groups (non-slender flexible groups are QH with $VPC_{\leq n-1}$ fiber), etc.

The result by Fujiwara-Papasoglu \cite{FuPa_JSJ} is the case when $\calh=\es$ and $\cala$ is the class of all slender groups. Dunwoody-Sageev \cite{DuSa_JSJ} consider (a generalisation
of) the case when $\calh=\es$ and $\cala=\{VPC_{\leq n}\}$;
if $G$ does not split over a $VPC_{\leq n-1}$-subgroup, flexible vertex groups are QH with   $VPC_{ n-1}$ fiber. 

If groups in $\cala$ are only assumed to be slender in $(\cala,\calh)$-trees, the conclusions of the theorem apply to groups $Q$ which are not slender in $(\cala,\calh)$-trees (see Subsection \ref{slt}).

The class of all (finite or infinite) cyclic groups does not satisfy the stability condition ($SC$) 
if $G$ contains
dihedral subgroups. 
To recover Rips and Sela's\index{Rips-Sela} description of   JSJ decompositions
over cyclic groups, we   introduce  
a modified stability condition $(SC_\calz$) preventing groups in $\cala$ from acting dihedrally on a line.

\begin{dfn}[Stability Condition ($SC_\calz$)]\label{dfn_SCZ}
\index{stability condition ($SC_\calz$)}\index{0SC@($SC$), ($SC_\calz$): stability conditions}
We say that $\cala$ satisfies the \emph{stability condition ($SC_\calz$)}, with fibers in $\calf$,  if the following hold:
\begin{enumerate}
\item no group of $\cala$ maps onto $D_\infty$; given a short exact sequence $1\to F\to A\to \Z\to1$ with $A\in \cala$, we have $F\in\calf$;
\item given a short exact sequence $1\to F\to A\to K\to1$ with $F\in \calf$ and $K$ isomorphic to  a quotient of $\Z$, we have $A\in\cala$.

\end{enumerate}
\end{dfn}

This condition is satisfied by the   classes consisting of  
\begin{itemize}
 \item all cyclic subgroups,\index{cyclic tree} with $\calf=\{1\}$,
 \item all subgroups which are finite or cyclic, with $\calf=\{1\}$,
 \item  all virtually cyclic subgroups which do not map onto $D_\infty$ (i.e.\ are finite or have infinite center),
 with $\calf=$\{finite\},
\end{itemize}
 and  any class satisfying ($SC$)  and consisting of groups which do not map onto $D_\infty$.

\begin{thm}\label{thm_description_slenderZ}
Theorem \ref{thm_description_slender} holds if $\cala$ satisfies 
the stability condition $(SC_\calz)$ rather than ($SC$). In this case the underlying orbifold $\Sigma$ has no mirror\index{mirror} (all singular points are conical points).
\end{thm}

In particular:

\begin{thm}[\cite{RiSe_JSJ}]\label{thm_RiSe}
 Let $G$ be finitely presented relative to a finite family $\calh$ of finitely generated subgroups. Let $\cala$ be the class of all finite or cyclic subgroups of $G$.
 
 If $Q$ is a   flexible vertex group of a JSJ decomposition of $G$ over $\cala$ relative to $\calh$, and $Q$ is not virtually $\Z^2$, then $Q$ is QH with trivial fiber. Moreover, the underlying orbifold $\Sigma$ has no mirror, every boundary component of   $\Sigma$ is used, and $\Sigma$ contains an essential simple closed geodesic.
 \qed
\end{thm}

  Indeed, it follows from (\ref{qh}) of Proposition \ref{prop_RN}, or from Proposition \ref{slflex}, that slender flexible groups are virtually $\Z^2$.

In the next subsection we shall reduce Theorems \ref{thm_description_slender} and \ref{thm_description_slenderZ} to Theorem \ref{thm_totally_flex}, which will be proved in
Subsections   \ref{fpc} through \ref{sec_all_minus}.

In Subsections \ref{fpc} through \ref{tfg}, and \ref{sec_all_minus},
 we only assume that $G$ is finitely generated  (not finitely presented), and $\calh$ may be arbitrary (we use finite presentability only in   Lemma \ref{lem_Zorn_minuscule} and Subsection \ref{preuve}).
 This will be useful   in Subsection \ref{unac}, where finite presentability will not be assumed. In particular, we will see (Theorem \ref{thm_RiSe2}) that Theorem \ref{thm_RiSe} is true if $G$ is only  finitely generated and $\calh$ is an arbitrary family of subgroups.

\subsection{Reduction to totally flexible groups}\label{sec_reduc}

Unlike Fujiwara-Papasoglu, we know  in advance that JSJ decompositions exist, and we only need to show that flexible vertex groups $G_v$ are QH.  
This allows us to  forget $G$ and concentrate on $G_v$,
but we have to remember the incident edge groups 
and we therefore consider splittings of $G_v$ 
that are relative to $\Inch_v$ (see Definition \ref{dfn_incv}): since incident edge groups are universally elliptic, these are exactly the splittings that extend to splittings of $G$ relative to $\calh$ 
(see Lemma \ref{extens}).   Even if we are interested only in non-relative JSJ decompositions  of $G$ ($\calh=\es$), it is important here that we work in a relative context. In Subsection \ref{sec_JSJ_acyl_desc} we will even have to allow groups in $\calh$ to be infinitely generated. 

The fact that $G_v$ is a flexible vertex of a JSJ decomposition says that 
 $G_v$ splits relative to $  \Inch_v$, but not over a universally elliptic subgroup
(Corollary \ref{cor_flex}).  This motivates the following  general definition.

\begin{dfn}[Totally flexible]\index{totally flexible group}
 $G$ is \emph{totally flexible} (over $\cala$ relative to $\calh$) if it admits a non-trivial splitting, but none over a universally elliptic subgroup.  Equivalently, the JSJ decomposition of $G$ is trivial, and $G$ is flexible.
\end{dfn}

 The example to have in mind is   the fundamental group of a compact hyperbolic surface other than a pair of pants, with $\cala$  the class of cyclic groups and $\calh$ consisting of the fundamental groups of the boundary components.

 We shall deduce Theorems \ref{thm_description_slender} and \ref{thm_description_slenderZ}
from the following result, which  says that totally flexible groups are QH.

\begin{thm}\label{thm_totally_flex}
  Let $G$ be finitely presented relative to a finite family $\calh$ of finitely generated subgroups.
Let $\cala$ be a class of slender groups
satisfying the stability condition $(SC)$ or  ($SC_\calz$)  with fibers in 
  $\calf$ (see Definitions \ref{dfn_SC} and \ref{dfn_SCZ}).
  
Assume that $G$ is totally flexible over $\cala$ relative to $\calh$, and not slender.
Then  $G$ is QH with fiber in $\calf$, 
and $\Sigma$ has no mirror in the $SC_\calz$ case.
\end{thm}

Since there are no incident edge groups, being QH means that 
  $G$ is an extension 
 of a group $F\in \calf$ by
    the fundamental group
of a hyperbolic  orbifold $\Sigma$, and   the image of each group  $H\in\calh$ in $\pi_1(\Sigma)$   is either finite or contained in a boundary subgroup.  By Corollary \ref{ufs}, every component of $\bo \Sigma$ is used by a group of $\calh$ (if $\calh=\es$, then $\Sigma$ is a closed orbifold).

This theorem implies Theorems \ref{thm_description_slender} and \ref{thm_description_slenderZ}: we apply it to $G_v$, with  $\cala=\cala_v$ 
(the family of subgroups of $G_v$ belonging to $\cala$) and $\calh=   \Inch_{v}$ (Definition \ref{dfn_incv}).  Note that $ \Inch_v$ is a finite family of finitely generated subgroups by Theorem \ref{thm_exist_mou_rel} and Remark \ref{ashv}, that $G_v$ is finitely presented relative to $\Inch_v$ by Proposition \ref{Gvrelfp},
and  that $\cala_{ v}$ satisfies   $(SC)$ or $(SC_\calz)$ with fibers in $\calf_{ v}$.
\\

There are three main steps in the proof of Theorem \ref{thm_totally_flex}.

$\bullet$  Given two trees $T_1,T_2$ such than no edge stabilizer of one tree
  is elliptic in the other, we   follow Fujiwara-Papasoglu's construction of a \emph{core}\index{core in the product of two trees}  $\calc\inc
  T_1\times T_2$. 
  This core happens to be a surface away from its vertices.
  More precisely, if one removes from $\calc$ its cut vertices and the vertices  whose link is homeomorphic to a line,
one gets a surface whose connected components are simply connected. 
   The decomposition of   $\calc$ dual to  its cut points yields a   tree $R$
having QH vertices coming from the surface components.
  Borrowing Scott and Swarup's terminology \cite{ScSw_regular+errata}, we call this tree $R$  the \emph{regular neighborhood}\index{regular neighbourhood} 
     of $T_1$ and $T_2$ 
  (see Example  \ref{exsurf} for an explanation of this name).   
The stability condition is used to ensure that edge stabilizers of $R$ are in $\cala$.

$\bullet$
  Given a totally flexible $G$, we construct two splittings $U$ and $V$ of $G$
  which ``fill'' $G$, and we show that their regular neighborhood is
  the trivial splitting of $G$.
  We deduce that $G$ itself is QH,  as required (in the case of cyclic splittings of a torsion-free group, $G$ is  the fundamental group of a compact surface   and one should think of $U$ and $V$ as dual to transverse pair of pants decompositions, see Example \ref{pant}). 

$\bullet$
  The previous   steps require  that splittings of $Q$ be
  \emph{minuscule}. This is a condition which controls the way in which $Q$ may
  split over subgroups $A,A'$ with $A'\inc A$, and we check that it is
  satisfied.

 Our first step is the same as \cite{FuPa_JSJ}.
But
total flexibility allows us to 
  avoid the more complicated part of their paper where, given a 
QH vertex group of a tree (obtained for instance as the regular neighborhood of two splittings), one has to make it larger to  
enclose   a third splitting.

In the second step, the first splitting $U$ is obtained by a maximality argument which requires finite presentability. 

Fujiwara and Papasoglu work with a minimality condition for splittings that allows the  construction  of regular neighborhoods.
This condition is not sufficient for our purpose, and we replace it by a stronger
 condition (minuscule).

\subsection{Fujiwara-Papasoglu's core} \label{fpc}

As mentioned earlier, we let $\calh$ be arbitrary and we only assume that $G$ is finitely generated. Groups in $\cala$ are slender, and one of the stability conditions $(SC)$ or  ($SC_\calz$)  is satisfied. 
\subsubsection{Minuscule splittings}

\begin{dfn}[Minuscule] \index{minuscule} 
Given $A,A'\in\cala$, 
we say that $A'\ll A$ if $A'\subset A$ and there exists
 a tree $S$ such that
$A'$ is elliptic in $S$ but $A$ is not. We also write $A\gg A'$.

We say that $A\in\cala$ is \emph{minuscule} if, 
whenever a subgroup $A'\ll A$  fixes an edge $e$ of a tree, $A'$ has infinite index in the stabilizer $G_e$. Equivalently, $A$ is minuscule if and only if no $A'\ll A$ is commensurable with an edge stabilizer.

A tree is minuscule if   its edge stabilizers are minuscule.  We say that all trees  (or all splittings) are minuscule if all $(\cala,\calh)$-trees are minuscule. 
\end{dfn}

If $A'\ll A$, then $A'$ has infinite index in $A$. If $A'\inc A\ll B\inc B'$, then $A'\ll B'$. 
In particular, the relation $\ll $ is  transitive.

Being minuscule is a commensurability invariant. It is often used in the following way: if an edge stabilizer $G_e$ of a tree $T$ is elliptic in another tree $S$, then any minuscule $A$ containing $G_e$ is also elliptic in $S$.

\begin{rem} \label{boncas}
 If $G$ does not split (relative to $\calh$) over a subgroup commensurable with a subgroup of infinite index of a group in $\cala$, then
every  $A\in\cala$ is minuscule (because, if $A$ is not minuscule, some $G_e$ has a finite index subgroup $A'$ contained in $A$ with infinite index). 
This holds for instance under the assumptions of \cite{DuSa_JSJ},
in particular for splittings of one-ended groups over virtually cyclic subgroups. 
The reader interested only in this case may therefore
ignore Subsection \ref{sec_all_minus}, where we  prove that all splittings of a totally flexible group are minuscule.
\end{rem}

The following lemma says that ellipticity is a symmetric relation among minuscule trees 
(in the terminology of \cite{FuPa_JSJ},  minuscule one-edge splittings are \emph{minimal}).

\begin{lem}\label{lem_minuscule_minimal}
 If $T_1$ is elliptic \wrt $T_2$, and $T_2$ is minuscule,  then $T_2$ is elliptic \wrt $T_1$.
\end{lem}

\begin{proof}
Let $\Hat T_1$ be a standard refinement  of $T_1$ dominating $T_2$ (Definition \ref{dfn_std_blowup}).
Let $e_2$ be an edge of $T_2$, and
$e$ an edge of $\Hat T_1$ with $G_e\subset G_{e_2}$. 
The group $G_e$ is elliptic in $T_1$, and since $G_{e_2}$ is minuscule 
 $G_{e_2}$ itself is elliptic in $T_1$.
This argument applies to any edge   of $T_2$, so  $T_2$ is elliptic \wrt $T_1$.
\end{proof}

\begin{lem}\label{lem_Zorn_minuscule}
Assume that $\calh$ is a finite family of finitely generated subgroups, 
 $G$ is finitely presented relative to $\calh$, and all trees are minuscule.
  There exists a    tree $T$ which is maximal for domination:
if  $T'$ dominates $T$, then $T$ dominates $T'$ (so $T$ and $T'$ are in the same deformation space).
\end{lem}

Finite presentation is necessary as evidenced by Dunwoody's inaccessible group \cite{Dun_inaccessible}, with $\calh=\es$ and $\cala$ the family of finite subgroups.

\begin{example} \label{pantalons}
If we consider cyclic splittings of the fundamental group of a closed orientable surface  $\Sigma$, then $T$ is maximal if and only if it is dual to a pair of pants decomposition of $\Sigma$.
\end{example}

\begin{proof} 
Let $\calt$ 
be the set   of all  trees with finitely generated edge and vertex stabilizers, up to equivariant isomorphism.  By Dunwoody's accessibility  (Corollary \ref{cor_fg_rel})
 every tree is dominated by a tree in $\calt$, so it suffices to find a maximal element in $\calt$.

As pointed out in the proof of Theorem \ref{thm_exist_mou}, the set $\calt$ is countable, so it suffices to show that,  given any  sequence $T_k  $   such that $T_{k+1}$ dominates $T_k$, there exists a tree $T$ 
  dominating every $T_k$.

We produce inductively  a tree $S_k$ in the same deformation space as $T_k$ 
which refines $S_{k-1}$.
Start with $S_0=T_0$, and assume that $S_{k-1}$ is already defined.
Since $T_{k}$ dominates $T_{k-1}$, hence  $S_{k-1}$,
it  is elliptic \wrt $S_{k-1}$.
All  trees being assumed to be  minuscule, $S_{k-1}$ is elliptic \wrt $T_{k}$ by Lemma \ref{lem_minuscule_minimal}. We may therefore define $S_k$ as a standard refinement of $S_{k-1}$ dominating $T_k$. It belongs to the same deformation space as $T_k$ by Assertion \ref{it_ell} of Proposition \ref{prop_refinement}.

By Dunwoody's accessibility   
(Proposition \ref{prop_accessibility_rel}),
there exists a tree $T\in\calt$ dominating every $S_k$, hence every $T_k$. 
\end{proof}

\subsubsection{Definition of the core} \label{defco}

Let $T_1$ and $T_2$ be trees. Recall that $T_1$ is elliptic \wrt $T_2$ if every edge stabilizer of $T_1$ is elliptic in $T_2$ (Definition \ref{et}).  If $T_1$ has several orbits of edges, it may happen that certain edge stabilizers are elliptic in $T_2$ and others are not (being slender, they act on $T_2$ hyperbolically, leaving a  line invariant). 
This motivates the following definition.

\begin{dfn} [Fully hyperbolic]\index{fully hyperbolic tree}
 $T_1$ is \emph{fully hyperbolic} \wrt $T_2$ if
every edge stabilizer of $T_1$ acts hyperbolically on $T_2$. 

\end{dfn}

This implies that  edge stabilizers of $T_1$ are infinite, and no vertex stabilizer of $T_1$ is elliptic in $T_2$ (except if $T_1$ and $T_2$ are both 
 trivial). 

\begin{example} \label{pant2}
Suppose that $G$ is a surface group and $T_1$, $T_2$ are dual to families of disjoint geodesics $\call_1$, $\call_2$ as in Subsection \ref{orb}. Then $T_1$ is elliptic \wrt $T_2$ if each curve in $\call_1$ is either contained in $\call_2$ or disjoint from $\call_2$.  
It is fully hyperbolic  \wrt $T_2$ if each curve in $\call_1$ meets $\call_2$ and every intersection is transverse.
\end{example}

Let $T_1$ be fully hyperbolic \wrt $T_2$. We consider the product $T_1\times T_2$. We view it as a   complex made of squares $e_1\times e_2$, with the diagonal action of $G$. An edge of the form $e_1\times \{v_2\}$ is \emph{horizontal}, and an edge $\{v_1\}\times e_2$ is \emph{vertical}.
\index{horizontal edge of the core}
\index{vertical edge of the core}

Following \cite{FuPa_JSJ}, we define the \emph{asymmetric core}\index{core in the product of two trees}\index{0C1@$\calc(T_1,T_2)$: core of $T_1\times T_2$}
\index{0C2@$\check\calc(T_1,T_2)$: flipped core of $T_1\times T_2$}
 $\calc(T_1,T_2)$ as follows.
For each edge $e$ and each vertex $v$ of $T_1$, let $\mu_{T_2}(G_e)$ and $\mu_{T_2}(G_v)$ 
be the minimal subtrees of $G_e$ and $G_v$ respectively in $T_2$ (with $\mu_{T_2}(G_e)$ a line since $G_e$ is slender).

\begin{dfn} [Asymmetric cores $\calc(T_1,T_2)$, $\check\calc(T_1,T_2)$]
Let $T_1$ be fully hyperbolic \wrt $T_2$. The \emph{asymmetric core} $\calc(T_1,T_2)\subset T_1\times T_2$ is
$$\calc(T_1,T_2)=\left( \bigcup_{v\in V(T_1)} \{v\}\times \mu_{T_2}(G_v) \right)
\cup \left(
\bigcup_{e\in E(T_1)}e\times \mu_{T_2}(G_e)\right).$$

If we also assume that $T_2$ is fully hyperbolic \wrt $T_1$, we denote by 
$\check\calc(T_1,T_2)\subset T_1\times T_2$ the  opposite
construction:
$$\check\calc(T_1,T_2)=\left(\bigcup_{v\in V(T_2)}\mu_{T_1}(G_v)\times \{v\}\right) \cup 
\left(\bigcup_{e\in E(T_2)}  \mu_{T_1}(G_e)\times e\right).$$

\end{dfn}

When no confusion is possible, we use the notations $\calc$, $\check\calc$ instead of $\calc(T_1,T_2)$, 
$\check\calc(T_1,T_2)$.
Note that $\calc$ consists of all $(x,y)\in T_1\times T_2$ such that $y$ belongs to the minimal subtree of $G_x$.

Every  $\mu_{T_2}(G_v) $, $ \mu_{T_2}(G_e)$ being a   non-empty subtree, with $ \mu_{T_2}(G_e)\inc \mu_{T_2}(G_v) $ if $v$ is an endpoint of $e$, a standard argument shows that $\calc(T_1,T_2)$ is simply connected.

The group $G$ acts diagonally on $T_1\times T_2$, and $\calc$ is $G$-invariant. Since $G_e$ acts   cocompactly on the line  $\mu_{T_2}(G_e)$ 
for $e\in E(T_1)$, there are finitely many $G$-orbits of squares in $\calc$.
\begin{rem}\label{rem_fix}
If $H\subset G$ is elliptic in $T_1$ and $T_2$,   it fixes a vertex in $\calc$ (because $H$ fixes some $v_1\in T_1$, and being elliptic in $T_2$ it fixes a point in the $G_{v_1}$-invariant subtree $\mu_{T_2}(G_{v_1})\inc T_2$).
\end{rem}

\subsubsection{Symmetry of the   core} \label{symmco}

The goal of this subsection is to prove that the asymmetric core of minuscule splittings is actually symmetric: $\check\calc= \calc$. Before doing so, we note the following basic consequence of symmetry.

\begin{lem}\label{prop_surface}
Let  $T_1,T_2$ be  fully hyperbolic \wrt each other. Assume that 
  $\check\calc= \calc$.
Then 
$\calc$ is a pseudo-surface: if $V(\calc)$ is its  set of vertices,
then $\calc\setminus V(\calc)$ is a  surface (which does not have to be  connected or   simply connected). 
\end{lem}

\begin{proof}
  By construction, $\calc\cap( \rond{e_1}\times T_2)=\rond{e_1}\times \mu_{T_2}(G_{e_1})\simeq \rond{e_1}\times \bbR$ if $e_1$ is an edge of $T_1$ and $\rond{e_1}$ denotes the open edge.
It follows that all horizontal edges of $\calc$ are contained in exactly two squares.
The symmetric argument shows that all vertical edges of $\check\calc$ are contained in exactly two squares of $\check\calc$.
Since $\check\calc= \calc$,  all open edges of $\calc$ are contained in exactly two squares,
so $\calc\setminus V(\calc)$ is a surface.
\end{proof}

\begin{prop}\label{prop_minuscule_hh}
  Let $T_1,T_2$ be two minuscule  trees that are fully hyperbolic with respect to each other.
Then $\check\calc 
=\calc 
$.
\end{prop}

In other words, the relation ``$y$ belongs to the minimal subtree of $G_x$'' is symmetric.

 The remainder of this subsection is devoted to the proof of 
 the proposition. We 
always assume  that $T_1$ and $T_2$ are  fully hyperbolic with respect to each other, but they are only  assumed to be minuscule when indicated. 

We denote by $\calc^{(2)}\subset\calc$\index{0C3@$\calc^{(2)}$: union of squares in $\calc(T_1,T_2)$}
the union of all closed squares $e_1\times e_2$ of $T_1\times T_2$ which are contained
in $\calc$, or equivalently  such that $e_2\subset \mu_{T_2}(G_{e_1})$.   It contains 
all horizontal edges $e\times \{v_2\}$  of $\calc$, but only contains those  vertical edges $\{v\}\times e_2$ which bound a square in $\calc$. 
Any  open edge $\{v\}\times \rond e_2$ which does not bound a square in $\calc$ disconnects   $\calc$.

We define $\check\calc^{(2)}$ analogously.

\begin{rem*} The minimality condition of \cite{FuPa_JSJ} is weaker than requiring trees to be minuscule, and they only conclude  $\check\calc^{(2)}=\calc^{(2)}$. \end{rem*}

If $y$ is a point of $T_2$, 
define $Y_{y}=\calc\cap(T_1\times\{y\})$
and $Y_{y}^{(2)}=\calc^{(2)}\cap(T_1\times\{y\})$. They are invariant under $G_{y}$ 
(we do not claim that they are connected). If $y$ is not a vertex, $Y_y\setminus Y_{y}^{(2)}$ is a union of isolated vertices, so
any connected component of $Y_y^{(2)}$ is a connected component of $Y_y$.

\begin{lem}\label{lem_fibres_cnx}
 Let $m$ be the midpoint of an edge $e_2$ of $T_2$.
 If
  $Y_{m}$
  is not  connected, then $T_2$ is not minuscule.
Moreover, if $Y_{m}^{(2)}$ is not connected, then  some edge stabilizer of $T_1$ is hyperbolic in 
a one-edge splitting over a group $ 
A\ll G_{e_2}$.
\end{lem}

The ``moreover''  will only be needed in Subsection \ref{symm}.

\begin{proof}
Assuming that $Y_m$ is not connected, 
let   $Z $ 
 be a connected component  
that does not contain $\mu_{T_1}(G_{e_2})\times\{m\}$   (there exists one since $\mu_{T_1}(G_{e_2})\times\{m\}$ is connected; 
  we do not claim   $\mu_{T_1}(G_{e_2})\times\{m\}\subset \calc$). 
If $Y^{(2)}_m$ is disconnected, we choose $Z\subset Y^{(2)}_m$.
Let $G_Z\subset G_{e_2}$ be the global stabilizer of $Z$.

We first  show  that $G_Z$ is elliptic in $T_1$, so in particular, $G_Z\ll G_{e_2}$. 
If $G_Z$ contains an element $h$ which is hyperbolic in $T_1$,
the projection of $Z$ in $T_1$    contains the axis $\mu_{T_1}(h)$. Since $G_{e_2}$ is slender, $\mu_{T_1}(h)=\mu_{T_1}(G_{e_2})$,
contradicting our choice of $Z$.

To prove that $T_2$ is not minuscule, it is therefore enough to construct a (minimal) tree $S$ in which $G_Z$ is an edge stabilizer.
Consider $\Tilde Z=G.Z\subset \calc$, and note that $Z$ is a connected component of $\Tilde Z$.
Since $\calc$ is simply connected, each connected component of $\Tilde Z$ separates $\calc$
($Z$ is a track in $\calc$).
Let $S $ be the tree dual to $\Tilde Z$: its vertices are the connected components
of $\calc\setminus\Tilde Z$, and its edges are the connected components of $\Tilde Z$.

By construction, $S $ is a one-edge splitting, $G_Z$ is the stabilizer of the edge corresponding to $Z$,
so in particular  $S $ is an $\cala$-tree. Since any $H\in\calh$ fixes a vertex in $\calc$ (see Remark \ref{rem_fix}),
$S $ is an $(\cala,\calh)$-tree. 

There remains to check that $S $ is non-trivial (hence minimal).
Let $\{v_1\}\times\{m\}$ be a vertex of $Z$. By definition of $\calc$, the point $m$ belongs to $\mu_{T_2}(G_{v_1})$, so there is 
  $g\in G_{v_1}$
whose axis in $T_2$ contains $m$.
The line $\{v_1\}\times\mu_{T_2}(g)\subset \calc$ naturally defines an embedded $g$-invariant line in $S $,
on which $g$ acts as a non-trivial translation. This proves that $g$ is hyperbolic in $S $, so $S $ is non-trivial.
It follows that $T_2$ is not minuscule.

Under the stronger assumption that $Y_m^{(2)}$ is disconnected, there is an edge $e_1\times\{m\}$ in $Z$,
and we can choose the element $g$ in $G_{e_1}$. The tree $S $ provides the required splitting, with edge group $G_Z\ll G_{e_2}$.
\end{proof}

\begin{lem}\label{lem_fibre_edge2vertex}
If
$Y_m$ is connected whenever $m$ is the midpoint of an edge of $T_2$, then 
$Y_y$ is connected for every   $y\in T_2$.
\end{lem}

\begin{proof}
Clearly $Y_y$ is homeomorphic to some $Y_m$ if $y$ belongs to the interior of an edge, so assume that $y$ is a vertex of $T_2$. We sketch the
argument, which is standard. 
  Consider $a,b\in Y_y$, and
join them by a piecewise linear path $\gamma$  in $\calc$. The projection of $\gamma$ to $T_2$ is a loop   based at $y$. 
If this projection is not $\{y\}$, then  $\gamma$ has a subpath $\gamma_0$  whose endpoints project to $y$ and whose initial and terminal segments project
to the same edge $e_2=yv$.
Using connectedness of $Y_m$, for $m$ the midpoint of $e_2$, we may replace $\gamma_0$ by a path contained in $Y_y$. Iterating yields a path joining $a$ and $b$ in $Y_y$.
\end{proof}

\begin{lem}
 \label{lem_inclus}
  Assume that $T_1,T_2$ are fully hyperbolic with respect to each other, and  that $T_2$ is minuscule.
Then $\check\calc\inc\calc$.
\end{lem}

\begin{proof}
 Lemmas \ref{lem_fibres_cnx} and  \ref{lem_fibre_edge2vertex} imply that $Y_y$ is connected for every $y\in T_2$. 
Being connected and $G_y$-invariant, $Y_y$  
contains $\mu_{T_1}(G_{y})\times \{y\}$, so  $\calc$ contains $\check\calc$. 
  \end{proof}

 Proposition \ref{prop_minuscule_hh} follows immediately from Lemma \ref{lem_inclus} by symmetry.

\subsection{The regular neighborhood} \label{rn}

In this subsection we assume that $T_1,T_2$ are   non-trivial,  fully hyperbolic \wrt each other, 
and that $\check\calc= \calc$. We use the core to construct a tree $R$, which we call the regular neighborhood of $T_1$ and $T_2$. Its main properties are summarized in Proposition \ref{prop_RN}.

We have seen (Lemma \ref{prop_surface}) that $\calc$ is a  surface away from its vertices. 
It follows that the link of any vertex $v\in\calc$ is a one-dimensional manifold, i.e.\ a disjoint union of lines and circles. Since $\calc$ is simply connected, the vertices whose link is  disconnected are precisely the cut points of $\calc$. We
  define the regular neighborhood $R$ as the tree dual to the decomposition of $\calc$ by its cut points.  

\begin{dfn}[Regular neighborhood]\label{dfn_RN} The \emph{regular neighborhood}\index{regular neighbourhood}\index{0RN@$RN(T_1,T_2)$: the regular neighbourhood of two trees}  
  $R=RN(T_1,T_2)=RN(T_2,T_1)$ is the bipartite tree  with vertex set $\calv\dunion\cals$, where  $\calv$ is the set of cut vertices $x$ of $\calc$, and $\cals$ is the set of connected components $Z$ of $\calc\setminus\calv$. There is an edge between $x$ and $Z$ in $R$ if and only if $v$ is in the closure $\bar Z$ of $Z$. 
\end{dfn}

$\calv$ may be empty, but $\cals$ is always non-empty  (unless $T_1,T_2$ are trivial).

\begin{example} \label{exsurf}
 The basic example is the following. As in Example \ref{pant2}, let $T_1, T_2$ be  cyclic splittings of   the fundamental group of a closed orientable surface  $\Sigma$ dual to families of geodesics $\call_1,\call_2$. Then $R$ is dual to the family $\call$ defined as follows (more precisely, $R$ is a subdivision of the tree dual to $\call$): consider the boundary of  a regular neighborhood of $\call_1\cup\call_2$, disregard homotopically trivial curves, and isotope each non-trivial one  to a geodesic. Orbits of vertices in $\cals$ correspond to   components of $\Sigma\setminus\call$ meeting $\call_1\cup\call_2$. The other components correspond to  orbits of vertices in $\calv$. 
 \end{example}

\begin{prop}\label{prop_RN}
Suppose that all groups in $\cala$ are slender, and $\cala$ satisfies ($SC $) or ($SC_\calz$).
 Let $T_1$, $T_2$ be non-trivial trees that are  fully hyperbolic \wrt each other, 
with $\check\calc= \calc$.

The bipartite tree  $R $ is a minimal  $(\cala,\calh)$-tree satisfying the following properties: 

  \begin{enumerate}
\renewcommand{\theenumi}{RN\arabic{enumi}}
    \item \label{qh}
    
        If  $v\in\cals$, then $G_v$  acts hyperbolically in $T_1$ and $T_2$. It is an extension $1\ra F \ra G_v\ra O\ra 1$, where    the fiber  $F $ is in $\calf$, is contained in every incident edge stabilizer,  and fixes an edge in $T_1$ and in $T_2$. There are two possibilities.
       
       If $G_v$ is not slender, it is QH with fiber $F$.
               The underlying  orbifold $\Sigma $ is hyperbolic, with fundamental group $O$,  and contains an essential simple closed geodesic. There is no mirror if ($SC_\calz$) holds. Extended boundary subgroups of $G_v$ are elliptic in $T_1$ and $T_2$. 
    
  If   $G_v$ is   slender,   then
  $O$ is virtually $\Z^2$ (it is the fundamental group of a Euclidean orbifold $\Sigma$ without boundary). Incident edge stabilizers are finite extensions of $F$.

  \item\label{el1} 
  The stabilizer of a vertex $v\in\calv$  is elliptic in $T_1$ and $T_2$; in particular,  $R $ is elliptic \wrt $T_1$ and $T_2$.  Conversely, any group $H\subset G$ elliptic in $T_1$ and $T_2$ is elliptic in $R $; more precisely,   $H$ fixes a vertex $v\in \calv$,
    or $H$ fixes a  vertex $v\in\cals$ and   the image of $H$ in $O$ is   finite or contained  in a boundary subgroup.

  \item\label{compat} 
  For $i=1,2$, one passes from $R$ to $T_i$ by refining $R$ at   vertices $v\in\cals$ using  families of essential simple closed geodesics in the underlying orbifolds, and then collapsing all original edges of $R$. In particular, 
   $R $ is compatible with $T_i$. Any edge stabilizer  of   $T_i$  fixes some  $v\in\cals$, and no other vertex of $R$; it contains the associated fiber.

\end{enumerate}
\end{prop}

   The tree $R$ may be trivial (in   Example \ref{exsurf}, this happens precisely when $\call_1$ and $\call_2$ fill $\Sigma$). In this case, $G$ itself is   slender or   QH with fiber in $\calf$.

\begin{proof}
The action of $G$ on $\calc$ induces an action of $G$ on $R$. It is not obvious   that $R$ is minimal or has edge stabilizers in $\cala$. 
Note, however, that $R$ is relative to  $\calh$. In fact, any subgroup $H$ which is elliptic in $T_1$ and $T_2$ is elliptic in $R$, because $H$ fixes a vertex of $\calc$ by Remark \ref{rem_fix}. 

The stabilizer of a vertex $v\in \calv$ fixes a point in $\calc$, so is elliptic in $T_1$ and $T_2$. In particular,   edge stabilizers of $R$   are elliptic in $T_1$ and $T_2$; in other words, $R$ is elliptic \wrt $T_1$ and $T_2$. 

The heart of the proof is to show that $G_v$ is QH or slender for   $v\in\cals$ (\ref{qh}). Along the way we will show that $R$ is an $\cala$-tree and complete the proof of (\ref{el1}). We will then prove (\ref{compat}) and minimality of the action of $G$ on $R$.

If  $v\in\cals$,
the group $G_v$ acts on $\bar Z$,   the closure of  the  connected component $Z$ of $\calc\setminus\calv$ associated to $v$.
Thus $\bar Z$ may be viewed as one of the pieces one obtains when cutting  $\calc$ open  at its cut points.  
Stabilizers of edges of $R$ incident to $v$ fix a vertex in $\bar Z$.  The group $G_v$ contains an edge stabilizer of $T_1$ (resp.\ $T_2$), so acts hyperbolically in $T_2$ (resp.\ $T_1$).

Recall that $\bar Z$ is made of squares.
Since $G$ acts on $T_1$ and $T_2$ without inversions, any $g\in G_v$ leaving a square $S$ invariant is the identity on $S$, hence on adjacent squares
(because $\bar Z$ is a pseudo-surface), and therefore on the whole of $\bar Z$. We let $F$ be the pointwise stabilizer of $\bar Z$, so that $O=G_v/F$ acts on $\bar Z$.

We have pointed out in Subsection \ref{defco} that $G$ acts on $\calc$ with finitely many orbits of squares, so the same is true for the action of $G_v$ (and $O$) on $\bar Z$. The action of $O$ on   $\bar Z $ is proper in the complement of vertices, and free in the complement   of the $1$-skeleton (an element may swap two adjacent squares).
We consider this action near a vertex $x$. Since we now view $x$ as a vertex of $\bar Z$ rather than one of $\calc$, its link $L_x$ is connected. The stabilizer $O_x$ of $x$ for the action of $O$ acts on  $L_x$ with trivial edge stabilizers. 

If $L_x$ is a circle (so $\bar Z$ is a surface near $x$),  $O_x$ is a finite group (cyclic or dihedral) and the action is proper near $x$. 
If $L_x$ is a line, the stabilizer $O_x$   acts on it by translations or dihedrally, and the image of $x$ in $\bar Z/O$ must be viewed as a puncture. Since we want a compact orbifold with boundary, we remove an $O_x$-invariant open neighborhood of $x$ from $\bar Z$. 

After doing this $G_v$-equivariantly near all vertices of $\bar Z$ whose link is a line, we get an effective  proper action of $O$ on  a simply connected surface, with quotient a compact orbifold $\Sigma$. The fundamental group of $\Sigma$ is isomorphic to $O$. 

If $H$ is a subgroup of $G_v$ which is elliptic in $T_1$ and $T_2$ (in particular if $H\in\calh$), it fixes a vertex $x\in\bar Z$, so its image in $\pi_1(\Sigma)$ is finite (if the link of $x$ is a circle) or contained in a boundary subgroup of $\pi_1(\Sigma)$ (if the link is a line).  Conversely, any subgroup of $G_v$ whose image in $\pi_1(\Sigma)$ is finite or contained in a boundary subgroup of $\pi_1(\Sigma)$ is elliptic in $T_1$ and $T_2$. This completes the proof of  (\ref{el1}).

Before concluding that $G_v$ is slender or QH, we use  ($SC $) or ($SC_\calz$) to prove   that edge stabilizers of $R$ are in $\cala$.

We first show $F\in\calf$. 
Given any square $e_1\times e_2\inc \bar Z$, the group $F$ is the stabilizer of $e_2$ for the action of $G_{e_1}$ on the line $ \mu_{T_2}(G_{e_1})$; in particular, $F$ fixes an edge in $T_2$ (and in $T_1$). Moreover,    $F$ is the kernel of an epimorphism from $G_{e_1}$ to $\Z$ or $D_\infty$. The 
stability condition implies $F\in\calf$. If ($SC_\calz$) holds, $G_{e_1}$ acts on   $ \mu_{T_2}(G_{e_1})$ by translations, and  $G_{e_2}$ acts on   $ \mu_{T_1}(G_{e_2})$ by translations. This is true for all squares $e_1\times e_2\inc \bar Z$, so the orbifold $\bar Z/O$ contains no mirror. 

Any  incident edge stabilizer   of $v$ in $R$ is the stabilizer $G_x$ of a vertex $x\in \bar Z$ for the action of $G_v$ on $\bar Z$ (as pointed out earlier, it is    elliptic in $T_1$ and $T_2$).  It contains $F$, and $O_x=G_x/F$ acts on the link of $x$ with trivial edge stabilizers, so is cyclic or dihedral (finite or infinite). The stability condition ($SC $) implies $G_x\in\cala$. If ($SC_\calz$) holds, $O_x$ is cyclic because there is no mirror and we also get $G_x\in\cala$. The tree $R$ being bipartite, it is an $\cala$-tree.

Given $v\in\cals$, we would have proved that $G_v$ is a QH subgroup (in the sense of Definition \ref{dfn_qh}) if we  knew that the orbifold $\Sigma$ is hyperbolic. Assume otherwise. Since $\Sigma$ 
  is the quotient of a simply connected surface by the infinite group $O=\pi_1(\Sigma)$, it  is   Euclidean. 
Its fundamental group  is not virtually cyclic because it contains   the image of $G_{e_1}$ and $G_{e_2}$ for any square $e_1\times e_2$  in $\bar Z$. This implies that  $\Sigma$ is a quotient of the torus.  
In particular,  $\Sigma$ has empty boundary, and   incident edge stabilizers of $v$ in $R$ are finite extensions of $F$. We also deduce that 
$ \pi_1(\Sigma)$ contains $\Z^2$ with finite index, so $G_v$ is slender because it is an extension of $F$ by $\pi_1(\Sigma)$.

 To sum up: for $v\in\cals$, the group $G_v$ is slender (if $\Sigma$ is Euclidean) or QH
(if $\Sigma$ is hyperbolic).
 
We   now prove  (\ref{compat}) for $i=1$ (the proof for $i=2$ is the same). In the process we will show that all orbifolds $\Sigma$ contain an essential geodesic, thus completing the proof of (\ref{qh}), and that the action of $G$ on $R$ is minimal.

 Let $\calm_1\subset \calc$ be the preimage 
of all midpoints of edges of $T_1$ under the first projection $p_1:\calc\ra T_1$. It is a $G$-invariant collection of disjoint properly embedded lines  containing no vertex. One may view  $T_1$ as the tree dual to  $\calm_1$, with vertices the components of $\calc\setminus\calm_1$ and edges the components of $\calm_1$: the projection $p_1$ induces a map from the dual tree to $T_1$, which is an isomorphism because point preimages of $p_1 $ are connected. 

It immediately follows that any edge stabilizer $G_{e_1}$ of $T_1$ fixes a vertex $v\in\cals$ (associated to the component of $\calc\setminus\calv$ containing the preimage by $p_1$ of the midpoint of $e_1$) and contains the associated fiber; in particular, $T_1$ is elliptic  \wrt $R$. Since $G_{e_1}$ is hyperbolic in $T_2$ and edge stabilizers of $R$ are elliptic, $v$ is the only vertex   of $R$ fixed by $G_{e_1}$. 
This shows that, for any $v\in\cals$, there exists a subgroup  of $G$ (an edge stabilizer of $T_1$) having $v$ as its unique fixed point. Minimality of $R$ easily follows from this observation, since no vertex in $\calv$ is terminal. 

We define   $\hat T_1$ refining both $T_1$ and $R$ as the tree dual to the decomposition of $\calc$ given by its cut points and $\calm_1$. Its vertices are 
  the elements of $\calv$ together with the
components of $\calc\setminus(\calv\cup\calm_1)$, its edges are those of $R$ and components of $\calm_1$.   One obtains $T_1$ from $\hat T_1$ by collapsing edges coming from  edges of $R$, and $\hat T_1$ from 
   $T_1$ by refining at vertices $v\in\cals$. There remains to  show that the refinement  at $v$ is dual to a family of geodesics on the associated orbifold $\Sigma$. 

Recall that $v$ is associated to a  component $Z$ of $\calc\setminus\calv$.   Consider a component $\ell_1$ of $\calm_1$ contained in $Z$.
  Its stabilizer (an edge stabilizer $G_{e_1}$ of $T_1$)  acts cocompactly on $\ell_1$, so the image of $\ell_1$ in the orbifold $\Sigma$ associated to $v$ is a simple 1-suborbifold (simple closed curve if there is no mirror). This suborbifold is not homotopically trivial, so is isotopic to a geodesic; it is not boundary parallel because $G_{e_1}$ is hyperbolic in $T_2$ but extended boundary subgroups of $G_v$ are elliptic. 
    Components of $\calm_1$ contained in $Z$ thus yield the required   non-empty    family of disjoint essential simple closed geodesics in $\Sigma$. 
  \end{proof}

\begin{rem} 
Note that a given orbifold $\Sigma$ is filled by the images of $\calm_1$ (defined above) and $\calm_2$ (defined similarly using $T_2$).
The proof also shows that the number of orbits of vertices in $\cals$ is bounded by the number of orbits of edges of $T_1$. 
  If $T_1$ and $T_2$ are one-edge decompositions, then $R $   only has one orbit of vertices in $\cals$,
and $T_1$, $T_2$ are dual to single geodesics.

\end{rem}

\begin{prop}\label{borused}  
 Under the assumptions of Proposition \ref{prop_RN}, suppose  furthermore that $T_1$ is minuscule and $G$ is not slender. 
If $v\in\cals$, then $G_v$ is not slender; it is QH, and every boundary component of the underlying orbifold   is used (see Definition \ref{dfn_used}).
\end{prop}

\begin{proof}
Recall that $F$ fixes an edge $e_i$ in $T_i$. Since $T_1$ is fully hyperbolic \wrt $T_2$, we have $F\ll G_{e_1}$. If $G_v$ is slender  but $G$ is not, 
  then $R\neq\{v\}$, so there exists an edge in  $R$ incident to $v$. Its stabilizer is a finite extension of $F$, so $G_{e_1}$ cannot be minuscule. This contradiction implies that $G_v$ is QH.
  
By Lemma \ref{lem_arcbis}, if some boundary component is unused, $G$ splits over  a group containing $F$ with index $\le2$ (hence in $\cala$ by the stability condition), again a contradiction.
\end{proof}

\subsection{Constructing a filling pair of splittings} \label{tfg}

 In this subsection, as well as the next one, we assume that all trees are minuscule. 
This guarantees symmetry of the core (Proposition \ref{prop_minuscule_hh}), so any two trees fully hyperbolic \wrt each other have a regular neighborhood. It  also implies  that ellipticity is a symmetric relation among trees (Lemma \ref{lem_minuscule_minimal}).

The goal of this subsection is the following result showing the existence of a pair of splittings $U$, $V$ that \emph{fill} $G$.

\begin{prop} \label{prop_exist_hh}
Assume that $G$ is totally flexible and all  trees are minuscule.
Given a tree $U$, there exists a tree $V$ such that $U$ and $V$ are   fully hyperbolic \wrt each other.
\end{prop}

 In the next subsection we will apply this proposition to a maximal $U$ (given by Lemma \ref{lem_Zorn_minuscule}).

\begin{example} \label{pant}   In  the surface case (Examples \ref{pantalons} and \ref{pant2}), if $U$ is dual to a family $\call$ decomposing $\Sigma$ into pairs of pants, then $V$ is dual to a family meeting (transversely) every curve in $\call$.

 More generally, suppose that $G$ is the fundamental group of a compact hyperbolic 2-orbifold $\Sigma$,  and $\calh$ is the set of boundary subgroups, with $\cala$ consisting of all virtually cyclic subgroups. Then $U$ is dual to a family of geodesics $\call$  by Proposition \ref{dual}, and the proposition claims that there exists a family $\call'$ such that every geodesic in $\call$ meets $\call'$ and every geodesic in $\call'$ meets $\call$ (but no geodesic belongs to both $\call$ and $\call'$). To prove this, consider a maximal family $\call'$ such that each geodesic in $\call'$ meets $\call$ (transversely). Applying Lemma \ref{courbes} to the orbifold obtained by cutting $\Sigma$ along $\call'$ shows that every geodesic in $\call$ meets $\call'$.
\end{example}

\begin{proof}
The difficulty is to find $V$ with $U$ fully hyperbolic \wrt  $V$. 
Once this is done, we redefine $V$ by collapsing all edges   which do not belong to the axis of an edge stabilizer of $U$. This makes  $U$   hyperbolic with respect to each one-edge splitting underlying $V$.
By  Lemma \ref{lem_minuscule_minimal},  each one-edge splitting underlying $V$ is hyperbolic
\wrt $U$ and therefore $V$ is fully hyperbolic with respect to $U$.

  We argue by induction on the number $n$ of orbits of edges of $U$. For $n=1$, the existence of $V$
is just the fact that, $G$ being totally flexible, $U$ is not universally elliptic.

\begin{figure}[htbp]
  \centering
  \includegraphics{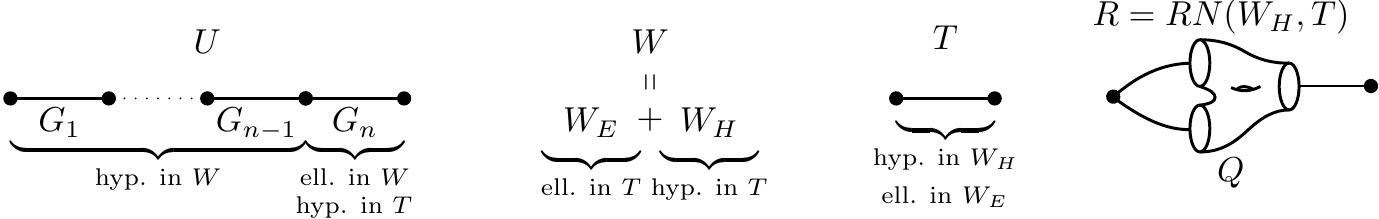}
\caption{The first players in the proof of Proposition \ref{prop_exist_hh}.}\label{fig_players}
\end{figure}

Denote by $G_1,\dots,G_n$ representatives of the edge stabilizers of $U$ (see Figure \ref{fig_players}, where we represent the quotient graphs of groups).
By induction,  there is a tree $W$   in which $G_1,\dots,G_{n-1}$
are hyperbolic. 
We can assume that $G_n$ is elliptic in $W$.
By the case $n=1$  there exists $T$ in which $G_n$ is hyperbolic, and we can assume that $T$ has
 a single  orbit of edges.
If $W$ is elliptic \wrt $T$, we   take for $V$ a standard refinement  of $W$ dominating $T$.

Otherwise, let $W_E,W_H$ be obtained from $W$ by collapsing   edges, keeping only those edges whose stabilizer is elliptic
or hyperbolic in $T$ respectively ($W_E$ may be trivial, but  $W_H$ is not). Note that $W_H$ is fully hyperbolic \wrt $T$.
By symmetry of the ellipticity relation 
among 
minuscule trees,
$T$ is (fully) hyperbolic \wrt $W_H$.  On the other hand,  $W_E$ and $T$ are elliptic with respect to each other,
and so are $W_E$ and $W_H$ (they are both collapses  of   $W$).

Let   $R=RN(W_H,T)$ be the regular neighborhood.
Recall that the set of vertices of $R$ is bipartite, with  
 $\calv$ the set of cut points of $\calc=\calc(W_H,T)\subset W_H\times T$, and $\cals$ 
  the set of connected components of $\calc\setminus \calv$.
Since $T$ has a single orbit of edges, there is   a single orbit
of vertices in $\cals$.
We fix $q\in \cals$, and denote by $Q$
its stabilizer. It is a QH vertex by Proposition \ref{borused}. 

\begin{lem}\label{QH_ell_minus}
  The QH vertex group $Q$ is elliptic in $W_E$.

\end{lem}

\begin{figure}[htbp]
  \centering
  \includegraphics{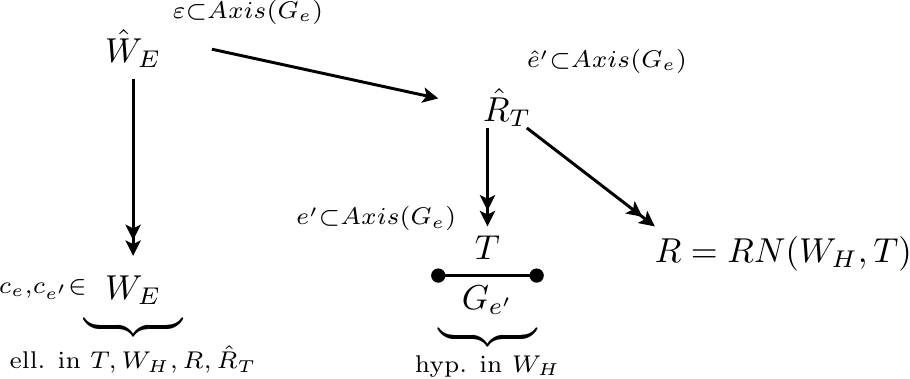}
  \caption{Proof that $Q$ is elliptic in $W_E$  ($\onto$ represents a collapse map).}
  \label{fig_Q_ell}
\end{figure}

  \begin{proof}
  In spirit,  the proof  is somewhat analogous to that of Theorem \ref{qhe}.

Let $e$ be an edge of $W_H$, and $e'$ an edge of $T$. 
Since $W_H$ is elliptic with respect to $W_E$, the stabilizer $G_e$ fixes a point $c_e\in W_E$.
This point is unique because otherwise  $G_e$ would fix an edge of $W_E$. But  $W_E$ is elliptic with respect to $T$, so   $G_e$ would be elliptic in $T$,
contradicting   the fact that $W_H$ is fully hyperbolic with respect to $T$.
Similarly, $G_{e'}$ fixes a unique point $c_{e'}\in W_E$.

As above, consider the core $\calc=\calc(W_H,T)\subset W_H\times T$. We claim  that $c_e=c_{e'}$ if the square $e\times e'$ is contained in $\calc$.
Assuming this, we associate to any square $e\times e'\inc\calc$ the point $c_e=c_{e'}\in W_E$.
Since adjacent squares are mapped to the same point, all squares in a given connected component of $\calc\setminus \calv$
are mapped to the same point. Since $Q$ is the stabilizer of such a component, it fixes a point in $W_E$, as required.

To conclude the proof, we have to prove our claim that $c_e=c_{e'}$ if $e\times e'\inc \calc$.
Let $\hat R_T$ be  a common refinement of $R$ and $T$ as in (\ref{compat})  (see Figure \ref{fig_Q_ell}).
By definition of the core, $e'$ lies in the axis (=minimal subtree) of $G_{e}$ in $T$.
Since  $\hat R_T$ collapses to $ T$, the edge $\hat e'$ of $\hat R_T$ projecting to $e'$ lies in the axis of $G_{e}$ in $\hat R_T$.
Note that $G_{\hat e'}=G_{e'}$.

Since $W_E$ is elliptic with respect to $W_H$ and $T$, it 
  is elliptic with respect to $R$ by (\ref{el1}),   hence also with respect to $\hat R_T$.
Let therefore 
$\hat W_E$ be 
a standard refinement  of $W_E$ dominating $\hat R_T$ as in Proposition \ref{prop_refinement}.

Let $M\subset \hat W_E$ be the  axis
of $G_{e}$.  Since $\hat R_T$ does not have to be a collapse of $\hat W_E$, we cannot claim that $M$ maps to the axis of $G_e$ in $\hat R_T$, but we can find 
an edge   $\eps\inc M$ whose image in $\hat R_T$ contains 
$\hat e'$, so that $G_\eps\subset G_{\hat e'}=G_{  e'}$.
Since $G_{e}$ fixes $c_{e}$ and only this point in $W_E$, the image of $M$ under the collapse map $\hat W_E\ra W_E$ is $\{c_{e}\}$.
In particular, $G_\eps$ fixes $c_{e}$. Since $G_\eps\subset G_{e'}$ fixes $c_{e'}$, it is enough to prove that $G_\eps$ fixes no 
edge in $W_E$. If it does, then $G_\eps$ is elliptic in $W_H$ (because $W_E$ is elliptic with respect to $W_H$),
but $G_{e'}$ is not (because $T$ is hyperbolic with respect to  $W_H$), so $G_\eps\ll G_{e'}$. This  contradicts the fact  that $T$ is minuscule and  proves the claim.
\end{proof}

\begin{figure}[htbp]
  \centering
  \includegraphics{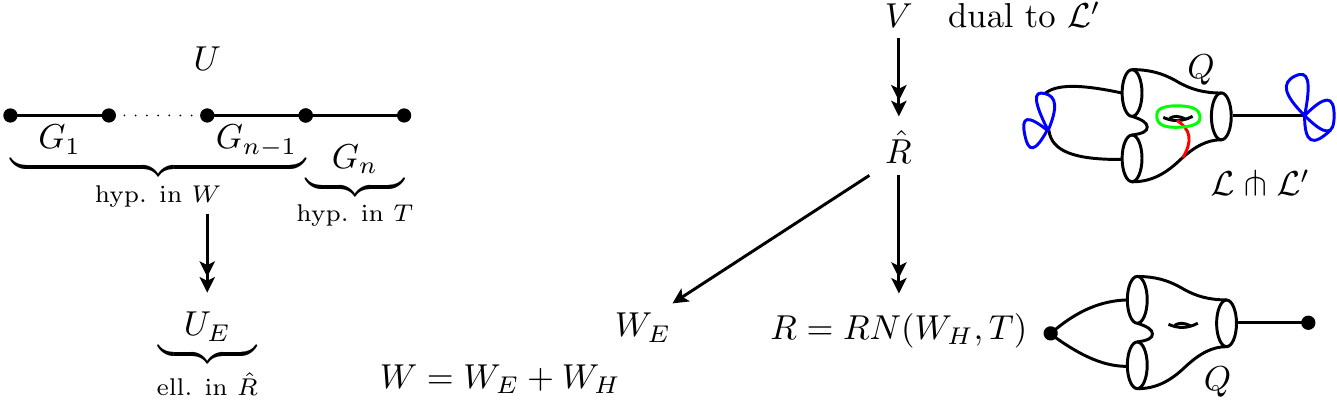}
  \caption{End of the proof of Proposition \ref{prop_exist_hh}.}
  \label{fig_endgame}
\end{figure}

Using the lemma, we can construct   a standard refinement $\hat R$ of $R$ dominating $W_E$  (see Figure \ref {fig_endgame})
without refining at vertices in the orbit of $q$.
In particular,  $q$ is still a QH vertex of $\hat R$.  
The tree  $V$ will be obtained by   refining $\hat R$ at  $q$. Recall that we want every edge stabilizer $G_i$ of $U$ to be hyperbolic in $V$. Since this will be automatic
if $G_i$ is hyperbolic in $\hat R$,
we consider the tree $U_E$ obtained by collapsing edges of $U$, keeping only those edges whose stabilizer is  elliptic in $\hat R$.

If $U_E$ is trivial,
we simply  take   $V=\hat R$, 
so we assume the contrary.
We claim that  $Q$ cannot be elliptic in $U_E$, nor in any collapse $\bar U_E$ which is not a point. Otherwise, the edge stabilizers of $T$ and $W_H$,  which are conjugate to subgroups of  $Q$ by (\ref{compat}), would be elliptic
in $\bar U_E$, so $\bar U_E$ would be elliptic with respect to  $T$ and $W_H$. This is a contradiction because  every $G_i$ is hyperbolic in $T$ or $W$, hence in $T$ or $W_H$ because it is elliptic in $W_E$  (which is dominated by $\hat R$).

 We now consider the action of $Q$ on its minimal subtree $\mu_{U_E}(Q)$. Note that $\mu_{U_E}(Q)$ meets every $G$-orbit of edges (otherwise $Q$ would be elliptic in some $\bar U_E$), so every $G_i$ which is elliptic in $\hat R$ contains (up to conjugacy) an edge stabilizer of the action of $Q$ on $\mu_{U_E}(Q)$.
 
 It follows from previous results that the action of  $Q$ on  $\mu_{U_E}(Q)$
  is (up to subdivision) dual to a family $\call$ of geodesics on the underlying orbifold $\Sigma$.  
To see this, view  $Q$  as a QH vertex stabilizer of $R$ or $\hat R$. Every boundary component of $\Sigma$ is used (Proposition \ref{borused}), the fiber is universally elliptic (Lemma \ref{uesle}), and incident edge groups are elliptic in $U_E$ because $U_E$ and $\hat R$ are  elliptic with respect to each other. The existence of $\call$ then follows from Lemma   \ref{dual2}.

 Now let $\call'$ be a family of geodesics transverse to $\call$ as in Example  \ref{pant}. Define $V$ by 
  refining $\hat R$  at $q$ 
  using the splitting of $Q$ dual to $\call'$, and observe  that every edge stabilizer $G_i$ of $U$ is hyperbolic in $V$, as required. This is clear  if $G_i$ is hyperbolic   in $\hat R$. Otherwise,   it contains (up to conjugacy) an edge stabilizer $J\inc Q$ of $\mu_{U_E}(Q)$, and $J$ is hyperbolic in $V$ because every curve in $\call$ meets $\call'$.
  \end{proof}

\subsection{ Flexible groups are QH when trees are minuscule}
\label{preuve}

We  can now prove Theorem \ref{thm_totally_flex} under the assumption    that all trees are minuscule. We will  prove in  the next subsection that this assumption is always fulfilled.

\begin{proof}[Proof of Theorem \ref{thm_totally_flex},  assuming that all trees are minuscule.]

Using finite presentability (which is assumed in Theorem \ref{thm_totally_flex}), we fix a maximal   splitting $U$ as in Lemma \ref{lem_Zorn_minuscule}.  
Proposition \ref{prop_exist_hh} yields $V$ such that
$U$ and $V$ are fully hyperbolic with respect to each other.
We may therefore consider the regular neighborhood $R$ of $U$ and $V$. We claim that $R$ is trivial (a point). This implies that $G$ is QH, as required.

Assume that $R$ is non-trivial. 
  Let   $\Hat R$ be a common refinement of $R$ and $U$ (it exists by Proposition \ref{prop_RN}(\ref{compat})). 
 Since $R$ is elliptic \wrt $V$, while $U$ is fully hyperbolic \wrt $V$, 
    each edge    of $\Hat R$ must be  collapsed in $R$ or in $U$ (possibly in both). 
 Being  non-trivial, $R$ contains an edge $\eps$, and we let $\hat \eps$ be the edge of $\hat R$ that is mapped to $\eps$.
Define $U'$ by collapsing all edges of $\Hat R$ which are not collapsed in $U$, except those in the orbit of $\eps$.

    Since $U'$ collapses to $U$, maximality of $U$ (Lemma \ref{lem_Zorn_minuscule}) implies that $U$ and $U'$ belong to  the same deformation space. Thus $\hat \varepsilon$ (viewed as an edge of $U'$) has an endpoint $v\in U'$ with  $G_v=G_{\hat\varepsilon}$, and its other endpoint   is in a different orbit. Since the action of $G$ on $U'$ is minimal, there is an edge $e$ of $U'$ with origin $v$ which does not belong to the orbit of  ${\hat\varepsilon}$. This edge has the same stabilizer as an edge of $U$, so $G_e$ is hyperbolic \wrt $V$. This is a contradiction since  $G_e\inc G_v=G_{\hat\varepsilon}=G_\varepsilon$, and $G_\varepsilon$ is elliptic in $V$ because $R$ is elliptic \wrt $V$ by (\ref{el1}).
 \end{proof}

\subsection{All splittings of a totally flexible group are minuscule}\label{sec_all_minus}

In this subsection, we complete the proof of Theorems \ref{thm_description_slender} and \ref{thm_description_slenderZ}
by showing that all   splittings of $G$ are minuscule
  if $G$ is totally flexible  (Proposition \ref{prop_all_minus}). As mentioned in Remark \ref{boncas}, this is sometimes not needed.
  
   In a few words, the proof goes as follows. Given an edge stabilizer $G_e$ of a  tree $T$, we first find a minuscule tree 
$T_1$ with an edge stabilizer $G_{e_1}$ commensurable with a subgroup of $G_e$. 
The goal is to show that $G_e$ and $G_{e_1}$ are in fact commensurable. 
This is  proved by showing that $G_{e_1}$ and $G_e$ are slender subgroups
of a QH vertex group, both containing the fiber. So ultimately, the argument relies on the fact that
two nested infinite slender subgroups of a  hyperbolic orbifold group are commensurable.
To embed  $G_{e_1}$ and $G_e$ into a QH group, we use  a tree $T_2$ in which $G_{e_1}$ is hyperbolic and we construct the core. Since we do not know that $T_2$ is minuscule, we have   to choose it carefully, to ensure that the core is symmetric (hence a surface).

\subsubsection{Complexity of edge groups}

  Recall that $\cala$ consists of slender groups; in particular, they are finitely generated.

 \begin{lem}\label{lem_access_tiny}
 Let $A_0\in\cala$. There is a bound, depending only on $A_0$, for the length $n$ of a chain $A_0\gg  A_1\gg A_2\dots \gg  A_n$. In particular, there is no infinite descending chain. 
 \end{lem}
 
 \begin{proof} 
   Since $A_{i-1}\gg A_{i}$, we can find a  tree $T_i$ 
     in which $A_{i}$ is elliptic, and $A_{i-1}$ is not.
     Let $B_0$ be the intersection of all subgroups of index 2 of $A_0$. For each $i$, the action of $A_0$ on $T_i$ has an invariant line, and 
  $B_0$ acts on this line  as a translation, given by a non-zero homomorphism $\varphi_i:B_0\to \Z$. 
   Letting $B_i=A_i\cap B_0$, we have $\varphi_i(B_j)=0$ if and only if $j\ge i$.
   This implies that the $\varphi_i$'s are linearly independent, and the lemma follows since $B_0$ is finitely generated.
  \end{proof} 

Using this lemma, we can associate a complexity $c(A_0)$ \index{0C0A@$c(A)$: complexity of an edge group $A$}
to any $A_0\in\cala$: it is 
 the maximal length $n$ of a chain 
$A_0\gg A_1\gg \dots \gg A_{n}$ where $A_1,\dots,A_n$ are commensurable with   edge stabilizers of  trees $T_1,\dots,T_n$.
 Thus $A_0$ is minuscule if and only if $c(A_0)=0$.
Note that $c(B)\leq c(A)$ if $B\subset A$, 
and $c(B)=c(A)$
  if $A,B$ are commensurable.

\begin{lem}\label{lem_contenu} 
Given an edge stabilizer $G_e$ of  a tree $T$, there exists a tree $T'$ with 
an edge stabilizer $G_{e'}$ such that  $G_{e'}$ is minuscule and $G_{e'}\cap G_e$ has finite index in $G_{e'}$. 
\end{lem}

\begin{proof}
The result is trivial if $c(G_e)=0$, with $T'=T$. If $c(G_e)>0$, we argue by induction. 
 By definition of the complexity, there exists $A\ll G_e$ with $A$ commensurable with an edge   stabilizer $G_{e_1}$. We have $c(G_{e_1})=c(A)<c(G_e)$, so by the induction hypothesis there exists a minuscule  edge stabilizer $G_{e'}$ such that
$G_{e'}\cap G_{e_1}$ has finite index in $G_{e'}$.
 Since $A$ and $G_{e_1}$  are commensurable, and $A\inc G_e$, the index of  $G_{e'}\cap G_e$   in $G_{e'}$ is finite.
\end{proof}

\subsubsection{Symmetry of the core} \label{symm}

 Recall (Proposition \ref{prop_minuscule_hh}) that the core of minuscule trees is symmetric. 

 In order to prove that all trees are minuscule,   we need to establish some weaker symmetry statements.
   We use the same notations as in Subsection \ref{fpc}.

\begin{lem}\label{lem_sym_cplete}
Let  $T_1,T_2$ be fully hyperbolic with respect to each other.
  If $ \check\calc\subset\calc$ and  $\calc^{(2)} \subset \check \calc^{(2)}$,
then $\check\calc= \calc$.
\end{lem}

\begin{proof}
Note that $ \check\calc\subset\calc$ implies $\check\calc^{(2)}\inc \calc^{(2)}$, so   we have $\check\calc^{(2)}= \calc^{(2)}$.
Assuming $\check\calc\ne\calc$, we get  $\calc^{(2)}=\check\calc^{(2)}\subseteq\check\calc\subsetneq\calc$.   
As pointed out in Subsection \ref{symmco},    $\calc\setminus\calc^{(2)}$ is a union of open vertical edges
 and each such edge disconnects $\calc$.
Our assumption $\check\calc\ne\calc$ says that $\calc\setminus\check\calc$ contains such an edge $\rond\eps=\{v\}\times \rond e_2$. 

The set of edges $e$ of $\mu_{T_2}(G_{v })$ such that $\{v\}\times e $ bounds a square in $\calc$ is   a non-empty $G_v$-invariant set which does not contain $e_2$. By minimality of the action of $G_v$ on $\mu_{T_2}(G_{v })$, this set intersects both components of 
$ \mu_{T_2}(G_{v })\setminus \rond e_2$. In particular, there exist $a,b\in\calc^{(2)}$ that are not in the same connected component
of $\calc\setminus \rond\eps$. But $a$ and $b$ may be joined by a path in $\check\calc$, a contradiction since $\check\calc\subset \calc\setminus \rond\eps$.
\end{proof}

\begin{lem}\label{lem_minimal_hh}
  Let $T_1$ be a minuscule one-edge splitting.
  Assume that $T_1$ is not universally elliptic.
Among all one-edge splittings
$T $ such that $T_1$ is not elliptic \wrt   $T $,
consider $T_2$ whose edge stabilizers have minimal complexity.

Then $\calc(T_1,T_2)=\check\calc(T_1,T_2)$.
\end{lem}

Note that $T_1$ and $T_2$ are fully hyperbolic \wrt each other by Lemma \ref{lem_minuscule_minimal}, so $\calc(T_1,T_2)$ and $\check\calc(T_1,T_2)$ are defined.

\begin{proof} 
Let $m$ be the midpoint of an edge $e_2$ of $T_2$. Since edge stabilizers of $T_2$ have minimal complexity,
the second assertion of Lemma \ref{lem_fibres_cnx}
shows that $Y_m^{(2)}=\calc^{(2)}\cap(T_1 \times \{m\})$ is connected.
In particular, $Y_m^{(2)}$ contains $\mu_{T_1}(G_{e_2})\times \{m\}$, and $\calc^{(2)}$ therefore contains all squares $e_1\times e_2$ with $e_1\inc \mu_{T_1}(G_{e_2})$. This shows $\check\calc^{(2)}\subset\calc^{(2)}$.

On the other hand,   $T_1$ being minuscule, Lemma \ref{lem_inclus} implies $\calc\inc\check\calc$. We conclude by Lemma \ref{lem_sym_cplete}.
\end{proof}

\subsubsection{All splittings are minuscule}

\begin{prop}\label{prop_all_minus}
If $G$ is totally flexible and not slender, then every tree $T$ is minuscule.
\end{prop}

\begin{proof} 
 We may assume that $T$ is  a one-edge splitting. Denoting by $G_e$ an edge stabilizer, 
  Lemma \ref{lem_contenu} yields   a minuscule one-edge tree  $T_1$ with an edge stabilizer $ G_{e_1}$ such that $G_{e_1}\cap G_e$ 
has finite index in $ G_{e_1}$.
Since $G$ is totally flexible, $T_1$ is not universally elliptic, so we can choose $T_2$ as in Lemma    \ref{lem_minimal_hh}. The core of $T_1$ and $T_2$ being symmetric, we can consider the regular neigbourhood decomposition $R=RN(T_1,T_2)$.

By (\ref{compat}), we know that $G_{e_1}$ fixes a 
unique vertex $v$ of $R$, and $v\in\cals$. Its finite index subgroup $G_{e_1}\cap G_e$ also has $v$ as its unique fixed point (it is hyperbolic in $T_2$, while edge stabilizers of $R$ are elliptic by (\ref{el1})).

First 
 assume that $G_v$ is not slender, so that it is QH with fiber $F$. 
Consider any slender group $A$ containing $G_{e_1}\cap G_e$, for instance $G_e$. We claim that $G_{e_1}\cap G_e$ has finite index in $A$.
First, $A$ is elliptic in $R$: otherwise $G_{e_1}\cap G_e$ acts on the axis of $A$, and (being elliptic) has  a subgroup  of index $\le2$ fixing the axis, a contradiction since $G_{e_1}\cap G_e$ fixes only $v$.  
We deduce $A\inc G_v$. 
Since $G_v$ is QH,  
its slender subgroups are contained in extensions of $F$ by virtually cyclic groups. As $G_{e_1}$ contains $F$ with infinite index, and $G_{e_1}\cap G_e$ has finite index in $G_{e_1}$, the index of $G_{e_1}\cap G_e$ in $A$ is finite, so $A$ is commensurable with $G_{e_1}$. Applied to $A=G_e$, this argument shows that $G_e$ is minuscule. 

If $G_v$ is slender,   there has to be at least one edge incident to $v$  since we assume that $G$ is not slender.
By (\ref{qh}), the stabilizer of this edge is commensurable with $F$.
On the other hand, $F$ is contained in $G_{e_1}$; it is elliptic in $T_2$, but $G_{e_1}$ is not, so $F \ll G_{e_1}$.
This shows that $G_{e_1}$ is not minuscule, a contradiction.
\end{proof}

\subsection{Slenderness in trees} \label{slt}

We have proved that non-slender flexible  vertex groups of JSJ decompositions are QH when $\cala$ satisfies a stability condition and consists of slender groups. We now consider edge groups which are only slender in trees. 

Recall that $A$ is 
slender in   $(\cala,\calh)$-trees\index{slender in $(\cala ,\calh )$-trees}
if,  whenever $G$ acts on a tree, $A$ fixes a point or leaves a    line invariant.

We now have the following generalization of Theorems \ref{thm_description_slender} and \ref{thm_description_slenderZ}:

\begin{thm}\label{thm_description_slender_arbres}
  Let $G$ be finitely presented relative to a finite family $\calh$ of finitely generated subgroups. Suppose that all groups in $\cala$ are slender in $(\cala,\calh)$-trees, and $\cala$ satisfies  ($SC$) or ($SC_\calz$), with fibers in a family $\calf$.

If $Q$ is a  flexible vertex group of a JSJ decomposition of $G$ over $\cala$ relative to $\calh$, and $Q$ is not slender in $(\cala,\calh)$-trees, then  (as in Theorems \ref{thm_description_slender} and \ref{thm_description_slenderZ}) 
$Q$ is QH with fiber in $\calf$ (and the orbifold has no mirror in the ($SC_\calz$) case). 
\end{thm}

This applies   whenever groups in $\cala$ are assumed to be slender or contained in a group of $\calh$ (up to conjugacy), for instance to  splittings of relatively hyperbolic groups over elementary subgroups relative to parabolic subgroups. 

 Theorem \ref{thm_description_slender_arbres} is proved as  in the slender case, with just a few changes  which we now describe.

 The arguments in Subsections \ref{fpc}
 through \ref{tfg} extend directly,   replacing ``slender'' by ``slender in $(\cala,\calh)$-trees'' and using Lemma \ref{extslen}. The proof of 
Proposition \ref{prop_exist_hh}
 uses Lemma \ref{uesle}, which we replace by the following statement:
 
 \begin{lem} \label{uesle2}
Let $Q$ be a QH vertex group.
  If $F$ and all groups in $\cala$ are slender in  $(\cala,\calh)$-trees, but $Q$ is not, then  $F$ is universally elliptic.  \end{lem}

  \begin{proof} If $F$ is not universally elliptic, then as in the proof of Lemma \ref{uesle}  $Q$ is an extension of   
 a group in $\cala$ by   a   virtually cyclic group, 
so is slender in  $(\cala,\calh)$-trees by Lemma \ref{extslen}.
\end{proof}

   In Subsection \ref{sec_all_minus}, Lemma \ref{lem_access_tiny}  requires $A_0$
to be  finitely generated, but the other groups $A_i$ may be arbitrary. Complexity may be infinite, 
but Lemmas \ref{lem_contenu} and \ref{lem_minimal_hh} remain valid because every tree is dominated by a tree with finitely generated stabilizers (see Corollary \ref{cor_fg_rel}).

 At the end of  the proof of Proposition \ref{prop_all_minus}  we need to know that a subgroup   $A\inc G_v$ which is slender in $(\cala,\calh)$-trees is contained in an extension of $F$ by a virtually cyclic group.  More generally, suppose that $Q$ is QH and the underlying orbifold $\Sigma$ contains an essential  simple geodesic. We claim that, \emph{if  $A\inc Q$   is slender in $(\cala,\calh)$-trees, then its image in $\pi_1(\Sigma)$ is virtually cyclic.} 
 
 To see this,  
 assume that $A$ is not an extended boundary subgroup. As in  Proposition \ref{prop_QHUE},  we consider a splitting of $Q$ dual to a\ geodesic $\gamma$,  in which $A$ is not elliptic. 
  The group  $A/(A\cap F)$ acts on a line with virtually cyclic edge stabilizers, so is virtually abelian, hence virtually cyclic because  it embeds into $\pi_1(\Sigma)$. This proves the claim.

 \subsection{Slender flexible groups}\label{pslflex}

\index{flexible vertex, group, stabilizer}\index{slender group}
In this subsection we consider  a slender flexible vertex group $Q$  of a JSJ tree. Whenever $G$ acts on a tree and $Q$ does not fix a point, there is a unique $Q$-invariant line and the action of $Q$ on this line gives rise to a map $\varphi:Q\to D_\infty\simeq \text{Isom}(\Z)$ whose image is isomorphic to $\Z$ (if orientation is preserved) or to $D_\infty$. 

A natural analogue of Theorem \ref{thm_description_slender} would be the following statement:
\begin{enumerate}
\item $Q$ maps onto the fundamental group of a compact Euclidean 2-orbifold $\Sigma$, with fiber in $\calf$; 
\item incident edge groups have finite image  in $\pi_1(\Sigma)$ (note that Euclidean 2-orbifolds whose fundamental group is not virtually cyclic have empty boundary);
\item if $G$ acts on a tree and $Q$ does not fix a point, the action of $Q$ on its invariant line factors through $\pi_1(\Sigma)$.
\end{enumerate}

Unfortunately,   Assertion (3) of  this statement is not correct (consider   splittings of $\Z^n$, with $n\ge3$).
We replace it by      two assertions.

\begin{prop} \label{slflex}
Let $G,\calh,\cala$ be as in Theorem \ref{thm_description_slender} or Theorem  \ref{thm_description_slenderZ}. Let $Q$ be a  slender flexible vertex group   of a JSJ tree. 

\begin{enumerate}
\item $Q$
is an extension $1\to F\to Q\to O\to 1$, with $F\in\calf$ and $O$ virtually $\Z^2$. Incident edge groups are contained in finite extensions of $F$ (they are finite if groups in $\cala$ are virtually cyclic). 

\item  there is a map $\psi:Q\to (D_\infty)^n$, for some $n\ge2$, whose image has finite index;  incident edge groups have finite image; 
if $G$ acts on a tree and $Q$ does not fix a point, the action of $Q$ on its invariant line factors through $\psi $.
\end{enumerate}

\end{prop}

In the first assertion, we do not claim that actions of $Q$ factor through the fundamental group of the 2-orbifold $O$ (the fiber does not have to be universally elliptic). In the second one, $Q$ maps to an $n$-dimensional Euclidean orbifold, but we do not claim that the kernel is in $\calf$. 

The first assertion only requires the first half of the stability condition,   the second one requires no stability condition.

The proposition remains true if $Q$ is only assumed to be slender in $(\cala,\calh)$-trees. 

\begin{example}  
The group $\bbZ^n$ has infinitely many  slender splittings, given by its morphisms to $\bbZ$.
The Klein bottle group  $\grp{a,t\mid tat\m=a\m}=\grp{t}*_{t^2=v^2}\grp{v}$
has exactly two  splittings,
corresponding to these two     presentations  (see Subsection \ref{sec_G_small}). They correspond to   morphisms to  $D_\infty$ and $\bbZ$ respectively.

In \cite{Beeker_abelian}, Beeker classified all flexible groups that occur in the abelian JSJ decomposition of the fundamental group of a graph of free abelian groups.
In particular, he exhibits a twisted Klein bottle group $K'=\grp{b_0,b_1\mid [b_0,b_1^2]=1,[b_0^2,b_1]=1}$
which splits as an amalgam $\grp{b_0,b_1^2}*_{\grp{b_0^2,b_1^2}}\grp{b_0^2,b_1}$
and also as an HNN extension (see Proposition 4.1 in \cite{Beeker_abelian}).
\end{example}

\begin{proof}
We sketch the proof of (1). Being flexible, $Q$ admits two different splittings relative to $\Inch_{ | Q}$. They are given by epimorphisms $\varphi_i:Q\to K_i$, with $K_i$ equal to $\Z$ or $D_\infty$. The image $O$ of $\varphi:Q\to K_1\times K_2$ is not virtually cyclic, so is virtually $\Z^2$. The kernel $F$ is in $\calf$ because it is also the kernel of $\varphi_{ | \ker \varphi_1}$, with $\ker \varphi_1\in\cala$. If $G_e$ is an incident edge group, its image by $\varphi_i$ has order at most 2 because $G_e$ is universally elliptic, hence  elliptic in the corresponding splitting of $Q$. 

For (2), we consider all non-trivial actions of $Q$ on the line which are restrictions of actions of $G$ on trees. We view them as non-trivial homomorphisms $\varphi_\alpha:Q\to D_\infty$ (there may be infinitely many of them). Let $J=\cap\ker \varphi_\alpha$ be the subgroup of $Q$ consisting of elements always acting as the identity, 
 and let $D=Q/J$.
 We shall show that $D$ embeds with finite index in some $ (D_\infty)^n$.
 
 The maps $\varphi_\alpha$ induce maps $\bar\varphi_\alpha:D\to D_\infty$, and $\cap\ker\bar\varphi_\alpha$ is trivial. 
 Let $D'$ be the subgroup of $D$ consisting of elements always acting as translations. It is abelian, torsion-free, has finite index (it contains  the intersection of all subgroups of index 2), and contains the commutator subgroup. Let $n$ be its rank.
 
 Each $\bar\varphi_\alpha$ induces $p_\alpha: D'\to \Z$. The subgroup of $\text{Hom}(D',\Z)$ generated by the $p_\alpha$'s has finite index because $\cap\ker p_\alpha$ is trivial, and we choose a finite family $\bar\varphi_{\alpha_i}$ ($1\le i\le n$) such that the corresponding $p_{\alpha_i}$'s are a basis over $\Q$. We show that the product map $$\Phi=\prod_{i=1}^n \bar\varphi_{\alpha_i}:D\to (D_\infty)^n$$ is injective and its image has finite index.
 
 Let $x\in D$. If it  has infinite order, there is a non-trivial $x^k$ in $D'$. This $x^k$ is mapped non-trivially by some $p_\alpha$, hence by some $p_{\alpha_i}$, hence by $\bar\varphi_{\alpha_i}$, hence by $\Phi$, so $x\notin\ker\Phi$. If $x$ has finite order and $x\ne1$, the order is 2 and some $\bar\varphi_{\alpha}$ maps it to a nontrivial reflection. We can find $y\in D$ such that $z=[x,y]$ is mapped by $\bar\varphi_{\alpha}$ to a non-trivial translation. The element $z$ is in $D'$, and as in the previous case $z\notin\ker\Phi$. It follows that $\Phi$ is injective. Its image has finite index because $D$ and $ (D_\infty)^n$ are both virtually $\Z^n$.

 Composing $\Phi$ with 
  the quotient map
 $  Q\to D$  yields a map $\psi:Q\to (D_\infty)^n$ whose image has finite index. As in Assertion (1), incident edge groups are universally elliptic, so their image by any $\varphi_\alpha$ has order at most 2. This implies that their image  by $\psi$ is finite. The last claim is obvious from the way $D$ was defined.
  \end{proof}
  
  \begin{rem} Similar arguments show that a finitely generated group is residually $D_\infty$ if and only if it is a finite index subgroup of a direct product whose factors are isomorphic to $D_\infty$ or $\Z/2\Z$.
  \end{rem}

\part{Acylindricity}
\label{partacyl}

In the previous chapters we have studied the JSJ deformation space.
Its existence was guaranteed by Dunwoody's accessibility,  which requires  finite presentability.
In this chapter we propose a construction based on the idea of acylindrical accessibility.

Sela \cite{Sela_acylindrical} defined a tree to be $k$-acylindrical if fixed point sets of non-trivial elements have   diameter bounded by $k$. Since we allow $G$ to have  torsion, the following definition is better adapted  \cite{Delzant_accessibilite}. 

\begin{dfn*}[Acylindrical] \index{acylindrical tree, splitting}
A tree $T$ is \emph{$(k,C)$-acylindrical}\index{0KC@$(k,C)$-acylindrical}\index{KC@$(k,C)$-acylindrical} if the pointwise stabilizer of every arc of length $\geq k+1$ is finite, of cardinality $\le C$.  
\end{dfn*}

\begin{example*} 
Let $\Sigma$ be a 2-orbifold as in Subsection \ref{orb}. The splitting of $\pi_1(\Sigma)$
 dual to a family $\call$ of geodesics is 2-acylindrical (1-acylindrical if all geodesics are 2-sided).  
If $\Sigma$ is the underlying orbifold of   a QH vertex group $Q$, the splitting of $Q$ dual to $\call$ is $(2,C)$-acylindrical if and only if the fiber $F$ is finite with order $\le C$.
\end{example*}

Acylindrical accessibility bounds the number of orbits of edges in any $(k,C)$-acylindrical tree, under the assumption that  $G$ is \emph{finitely generated} 
\cite{Sela_acylindrical,Delzant_accessibilite,Weidmann_accessibility}.

Since finite presentability is no longer required, this allows us for instance to construct JSJ decompositions of hyperbolic groups relative to 
  an infinite collection of infinitely generated subgroups, and   JSJ decompositions of finitely generated torsion-free CSA groups.

In order for our approach to work, we must  be able to  produce a $(k,C)$-acylindrical tree,
for some uniform constants $k$ and $C$, out of  an arbitrary tree.
We do so 
using the \emph{tree of cylinders} introduced in \cite{GL4}.

As an additional benefit, this construction (when available) produces a canonical tree\index{canonical tree} $T$ out of a   deformation space $\cald$.  If $\cald$ is invariant under automorphisms of $G$, so is $T$ 
  (see \cite{GL5,GL6} for applications of this construction, 
\cite{Bo_cut,ScSw_regular+errata,KhMy_effective,PaSw_boundaries} for other constructions of canonical decompositions,
and \cite{Lev_automorphisms,DaGr_isomorphism,BuKhMi_isomorphism,DG2,DaTo_isomorphism} 
for other  uses of such a canonical decomposition).
  A  general construction of another  invariant  tree, based on compatibility of splittings,
 will be given in Part \ref{chap_compat} (compatibility JSJ tree).

The tree of cylinders is introduced in Section \ref{sec_cyl}, and Section \ref{jsjac} describes our construction of JSJ decompositions
based on acylindricity.  Applications are given in Section \ref{exam}.

Unless indicated otherwise, we only assume that $G$ is finitely generated, and $\calh$ may be arbitrary.

\section{Trees of cylinders}\label{sec_cyl}

\index{tree of cylinders}
In the previous chapters we have defined and studied the JSJ deformation space, consisting of all JSJ trees.  
It is much better  to be able to find a   canonical
tree. 
A key example is provided by one-ended hyperbolic groups: Bowditch \cite{Bo_cut}\index{Bowditch} constructs a virtually cyclic JSJ tree using only the topology of the boundary $\partial G$; it is $\Out(G)$-invariant by construction. 

Unfortunately, it is not always possible to find an invariant JSJ tree. If for instance $G$ is free, the JSJ deformation space consists of  all  trees with a free action of $G$ (see Subsection \ref{free}), but (see Subsection \ref{freeg}) it is easy to check that the only $\Out(G)$-invariant tree is the trivial one (a point). 

 In this section we describe a construction, the \emph{(collapsed) tree of cylinders} \cite{GL4}, 
which,   under certain conditions,
associates   a new, nicer tree  to a given tree $T$.
 The first feature of this new tree is that it depends only on the deformation space $\cald$ of $T$.
Its second feature of interest to us here is that, under suitable assumptions, it is acylindrical with uniform constants, 
and   it lies in $\cald$ or at least in a deformation space not too different from $\cald$
 (the new tree is smally dominated  by  $T$  in the sense of Definition \ref{sm}).
Its third feature lies in its compatibility properties (in the sense of common refinements, see Subsection \ref{sec_morphisms});
this  will be used in Chapter \ref{chap_compat} to provide examples of compatibility JSJ trees.

  The results of this section and the next (where JSJ decompositions are constructed) will be  summed up in Corollary \ref{synthese}, which gives conditions ensuring that there is a canonical JSJ tree, and that its flexible vertices are QH with finite fiber.

\subsection{Definition}\label{defcyl}
We   recall the definition and some basic properties of the tree of cylinders (see  \cite{GL4} for details).

As usual, we fix 
  $\cala$ and $\calh$ and we restrict to $(\cala,\calh)$-trees. 
We let $\cale$\index{0AF@$\cale$: infinite groups in $\cala$} 
be the family of infinite  groups in $\cala$.  In applications we  assume  that $G$ is one-ended relative to $\calh$, so all trees have   edge stabilizers in $\cale$. The family $\cale$ is not stable under taking subgroups, 
  but it is \emph{sandwich closed}:\index{sandwich closed family of subgroups}
if $A,B\in\cale$  and $A< H< B$, then $H\in\cale$.

\begin{dfn}[Admissible equivalence relation]\label{eqrel}
An equivalence relation $\sim$ on $\cale$ is \emph{admissible}\index{admissible equivalence relation on $\cale$}  
(relative to $\calh$)  if the following axioms hold for any $A,B\in\cale$:
  \begin{enumerate}
\item If $A\sim B$ and $g\in G$, then $gAg\m\sim gBg\m$ (invariance under conjugation).
  \item If   $A\subset B$, then $A\sim B$ (nesting implies equivalence).
  \item \label{ax_3} Let $T$ be
  a tree  (relative to $\calh$) with infinite edge stabilizers. If $A\sim B$, and $A$, $B$ fix
$a,b\in T$ respectively, then for each edge $e\subset [a,b]$ one has $G_e\sim A\sim B$. 
  \end{enumerate}
  The equivalence class of $A\in\cale$ will be denoted by $\cla$.\index{0A@$[A]$: the equivalence class of $A\in \cala_\infty$}
   We let  $G$ act on $\cale/{\sim}$ by conjugation, and   the stabilizer of $A$  will be denoted by $G_\cla$.\index{0GA@$G_{[A]}$: the stabilizer of $[A]$}
\end{dfn}

Here are a few examples.
They will be studied in detail later (see Section \ref{exam}).
\begin{enumerate}
\item If $G$ is a torsion free CSA\index{CSA group} group (for instance a limit group, or more generally a toral relatively hyperbolic group), we can take for $\cale$ the set of infinite abelian subgroups,
and for $\sim$   the commutation relation: $A\sim B$ if $\grp{A,B}$ is abelian. The group $G_\cla$ is the maximal abelian subgroup containing $A$. 

\item If $G$ is  a relatively hyperbolic group\index{relatively hyperbolic group}  with small parabolic subgroups, we can take for $\cale$ the set of infinite elementary subgroups (a subgroup is elementary if it is virtually cyclic or parabolic; in this case, this is equivalent to being small).
The relation  $\sim$ is  co-elementarity:\index{co-elementary subgroups}
$A\sim B$ if and only if $\langle A,B\rangle$ is elementary.  The group $G_\cla$ is the maximal elementary subgroup containing $A$. 
We may also allow non-small parabolic subgroups, provided that we include them in $\calh$
 (i.e.\ we only consider   splittings relative
to the parabolic groups). 

\item
 $\cale$ is  the set of  infinite virtually cyclic\index{virtually cyclic} subgroups, and   $\sim$ is 
the commensurability relation ($A\sim B$ if and only if $A\cap B$ has finite index in $A$ and $B$). The group $G_\cla$ is the commensurator of $A$.\index{commensurable}\index{commensurator}
\end{enumerate}

Given  an admissible equivalence relation $\sim$ on $\cale$, we now associate a tree of cylinders $T_c$ to any tree $T$ with infinite edge stabilizers.

 We declare two edges  $e,f$ to be equivalent if $G_e\sim G_f$ (these groups are assumed to be in $\cala$, and they are in $\cale$ because they are infinite). 
The union of all edges in  an equivalence class is a subtree by axiom (\ref{ax_3}). 
Such a subtree $Y$ is called a  \emph{cylinder}\index{cylinder in a tree} of $T$; two distinct cylinders meet in at most one point. 
  The equivalence class in $\cale/{\sim}$ containing the stabilizers of edges of $Y$ will be denoted by $[Y]$.

\begin{dfn}[Tree of cylinders] \label{treecyl} \index{tree of cylinders}\index{0TC@$T_c$: tree of cylinders}
Given a tree $T$ with   edge stabilizers in $\cale$, its
  \emph{tree of cylinders} $T_c$  is the bipartite tree  
  with vertex set $V_0(T_c)\dunion V_1(T_c)$, where
  $V_0(T_c)$ is the set of vertices $v$ of $T$ which belong to at least two cylinders, 
$V_1(T_c)$ is the set  of cylinders $Y$ of $T$, and there is   an edge $\varepsilon=(v,Y)$ between $v$ and $Y$ in $T_c$ if and only if $v\in Y$ in $T$. 

Equivalently, one obtains $T_c$ from $T$ by replacing each cylinder $Y$ by the cone over the set of vertices $v\in Y$ belonging to another cylinder.

The    \emph{collapsed tree of cylinders} $T_c^*$ is the tree\index{0TC@$T_c^*$: collapsed tree of cylinders}\index{collapsed tree of cylinders}
obtained from $T_c$ by collapsing all edges whose stabilizer does not belong to $\cala$ \emph{(Warning: $T_c$ is not always an $\cala$-tree)}.
\end{dfn}

It is clear from the equivalent definition that $T_c$ is a tree, but 
 its edge stabilizers  do not always  belong to $\cala$; this is why we also consider $T_c^*$.  The  next lemma will say that 
  $T_c^*$ is  relative to $\calh$ (i.e.\ it is an $(\cala,\calh)$-tree). 
  
 The trees $T_c$ and $T_c^*$ are minimal.   
We claim that they are always irreducible or trivial (\ie they consist of a single point).
Indeed,
if $T_c$ or $T_c^*$ is non-trivial and not irreducible, then so is $T$ because it is compatible with them (Lemma \ref{compfac} below). 
  This implies that there is only one cylinder in $T$, so $T_c$ and $T_c^*$ are trivial,
a contradiction.

The stabilizer of a vertex $v\in V_0(T_c)$ is the stabilizer of $v$, viewed as a vertex of $T$;  it does not belong to $\cala$.
The stabilizer $G_Y$ of a vertex $Y\in V_1(T_c)$ is the stabilizer
of the equivalence class $[Y]$.
The stabilizer of an edge $\eps=(v,Y)$ of $T_c$ is  $G_\varepsilon=G_v\cap G_Y$; it is elliptic in $T$, and infinite because it contains $G_e$ if $e\inc T$ is any edge of $Y$ with origin $v$.  
If $G_\eps$ lies in $\cale$, it is a representative of $[Y]$.

\begin{lem} 
\ \label{lem_basicTc}
  \begin{enumerate}
  \item $T$ dominates $T_c$ and $T_c^*$, so $T_c^*$ is an $(\cala,\calh)$-tree.

  \item If $H<G$ is a vertex stabilizer of $T_c$ which is not elliptic in $T$, it is the stabilizer of the equivalence class $[Y]$ associated to some  cylinder $Y\subset T$.  
If $e$ is an edge of $Y$,   the equivalence class of $G_e$ is $H$-invariant
($hG_eh\m\sim G_e$ for $h\in H$).
  \item  $T_c$ and $T_c^*$ only depend  on the deformation space $\cald$ containing $T$ (we sometimes say that $T_c$ is the tree of cylinders of $\cald$).

\item  Suppose that
the stabilizer of every equivalence class $[A]$  belongs to   $\cala$. Then edge stabilizers of $T_c$ belong to $\cala$,  and therefore 
  $T_c=T_c^*$.
\item  The tree $T_c^*$ is equal to its  collapsed tree of cylinders:   $(T_c^*)_c^*=T_c^*$.

  \end{enumerate}

\end{lem}

Assertion (4) applies in particular to  Examples (1) and (2) above 
(CSA groups, relatively hyperbolic groups).

\begin{proof}
 Consider a   vertex $v$ of $T$. If $v$ belongs to two cylinders, it defines a vertex in $V_0(T_c)$, and this vertex is fixed by $G_v$. If $v$ belongs to a single cylinder $Y$, then $G_v$ fixes the   vertex of $V_1(T_c)$ corresponding to $Y$. This shows that $T$ dominates $T_c$, hence also its collapse $T_c^*$.

The second assertion follows from remarks made above: 
the stabilizer of a vertex   in $V_0(T_c)$ is a vertex stabilizer of $T$, the stabilizer of a vertex   in $V_1(T_c)$ is the stabilizer of an equivalence class $[Y]$. 
Also note that   (4) is clear since the stabilizer of an edge $\eps=(v,Y)$ of $T_c$ is contained in the stabilizer of $[Y]$.   Assertion (5) is  Corollary 5.8 of \cite{GL4}.

\cite{GL4} contains three proofs of the third assertion. We sketch one. A domination map $f:T\to T'$ induces a map $f_c:T_c\to T'_c$ mapping each edge onto a vertex or an edge. The image of $x\in V_0(T_c)$ by $f_c$ is the unique point of $T'$ fixed by $G_x$, viewed as a vertex of $V_0(T'_c)$;  the image of $Y\in V_1(T_c)$ is either the unique cylinder whose edge stabilizers are equivalent to those of $Y$, viewed as a vertex of $V_1(T'_c)$, or the unique point of $T'$ fixed by stabilizers of edges of $Y$, viewed as a vertex of $V_0(T'_c)$. If $T$ and $T'$ belong to the same deformation space, the map $g_c:T'_c\to T_c$ induced by a map $g:T'\to T$ is the inverse of $f_c$, so $T'_c= T_c$. 
\end{proof}

Assertion (3) implies that trees of cylinders are
  canonical  elements of their deformation space. In particular:
\begin{cor} \label{invaut}
 If $T$ is a JSJ tree over $\cala$ relative to $\calh$, then $T_c$ and $T_c^*$ are invariant under any automorphism of $G$ which preserves $\cala$ and $\calh$. \qed
\end{cor}

\begin{figure}[htbp]
 \centering
 \includegraphics{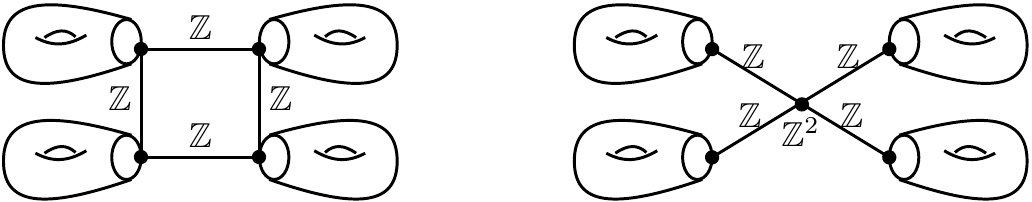}
 \caption{A tree and its tree of cylinders.}
 \label{fig_exemple_cyl}
\end{figure}

 In the following   examples, $\sim$ may be viewed at will as commutation, or co-elementarity, or commensurability. Indeed, two edge stabilizers  
will be equivalent if and only if they are equal. 
 
 \begin{example}
 The Bass-Serre tree of the graph of groups $\Gamma$ of Subsection \ref{exflex}   (three punctured tori glued along their boundaries, see Figure \ref{fig_3tores} in the introduction) 
is equal to its tree of cylinders (cylinders are tripods projecting bijectively onto $\Gamma$). 
 \end{example}
 
\begin{example} \label{z2}  
Now let $T$ be  the Bass-Serre tree of the graph of groups pictured on   the left hand side of Figure 
\ref{fig_exemple_cyl}  (reproducing Figure \ref{fig_ex}). 
There are four QH vertices $v_i$, which are fundamental groups of punctured tori $\Sigma_i$. There is an edge between $v_i$ and $v_{i+1}$ (mod 4), each carrying the same group $G_e$,  equal to the boundary subgroup of $\Sigma_i$.  Any cylinder $Y$ is a  line, with vertices $u_i$ ($i\in\Z)$ which are lifts of $v_i$ (mod 4). The setwise stabilizer of $Y$, which we shall call $P$,  is isomorphic to $\Z^2$, it acts on $Y$ by translations with   edge stabilizers equal to $G_e$ (up to   conjugacy). The group $G$ is hyperbolic relative to this   subgroup $P$ (it is a toral relatively hyperbolic group).
 
  In $T_c^* $ (which is equal to $T_c$), each line $Y$ is collapsed to a point  $u$  (a vertex of type $V_1$) and is replaced by edges joining that point  to each of the $u_i$'s. The stabilizer of $u$  is isomorphic to $P$, which is the centralizer (and the commensurator) of $G_e$. Unlike $T$, the tree $T_c^* $ is invariant under all automorphisms of $G$. 
 \end{example}

\subsection{Acylindricity}\label{sec_Tc_acyl}

\index{acylindrical tree, splitting}
Recall that $T$ is \emph{$(k,C)$-acylindrical} if the pointwise stabilizer  of every   arc   of length $\geq k+1$ has cardinality $\leq C$ (see the beginning of part \ref{partacyl}).

\begin{lem}\label{acl}
  Assume  that there exists $C>0$ such that,  if two groups of $\cale$ are inequivalent, then their   intersection has  order $\le C$. 
If $T$ is any $(\cala,\calh)$-tree, then $T_c^*$ is a $(2,C)$-acylindrical $(\cala,\calh)$-tree.
\end{lem}

This  applies in particular in Examples (1) and (2) above 
(CSA groups, relatively hyperbolic groups); note that in a relatively hyperbolic group there is a bound $C$ for the order of finite non-parabolic subgroups 
(see Lemma \ref{lem_coelem}).
In Example \ref{z2}, the tree $T$ is not acylindrical (cylinders are lines fixed by an infinite cyclic group) but $T_c^*$ is 2-acylindrical.

\begin{proof}
We have seen that $T_c^*$ is an $(\cala,\calh)$-tree.
Let $I$ be a segment of length 3 in $T^*_c$.  Its edges are images of edges $\varepsilon_i=(x_i,Y_i)$ of $T_c$. Let $e_i$ be an edge of $Y_i$ with origin $x_i$. The stabilizer $G_{\varepsilon_i}$ of $\varepsilon_i$ belongs to $\cala$ and contains $G_{e_i}$, so $G_{\varepsilon_i}$ belongs to $\cale$ and is equivalent to $G_{e_i}$.

Since  $T^*_c$ is a collapse of $T_c$, we can find $i,j$ with $Y_i\ne Y_j$. This implies $G_{e_i}\not\sim G_{e_j}$, so $G_{\varepsilon_i}\not\sim G_{\varepsilon_j}$. The stabilizer of $I$ fixes $\varepsilon_i$ and $\varepsilon_j$, so has  cardinality at most $C$.
\end{proof}

Virtually cyclic groups play a particular role in this context. 
 An infinite group $H$ is virtually cyclic if some finite index subgroup is infinite cyclic. Such a group maps onto $\Z$ or $D_\infty$ (the infinite dihedral group $\bbZ/2*\bbZ/2$) with finite kernel. 

\begin{dfn} [$C$-virtually cyclic] \label{dvc}\index{virtually cyclic@ $C$-virtually cyclic}\index{0CV@ $C$-virtually cyclic} Given $C\ge1$, we say that $H$ is \emph{$C$-virtually cyclic} if it maps onto $\Z$ or $D_\infty$ with kernel of order at most $C$. Equivalently, $H$ acts non-trivially on a simplicial line with edge stabilizers of order $\le C$. 
\end{dfn}

 An infinite virtually cyclic group   is $C$-virtually cyclic if    the order of its maximal finite normal subgroup is $\le C$.
  On the other hand, finite subgroups of a $C$-virtually cyclic group have cardinality bounded by $2C$.

 Recall (Subsection \ref{pti}) that $J<G $ is small in $T$  (resp.\ in $(\cala,\calh)$-trees) 
if it fixes  a point, or an end, or leaves a line invariant in $T$ (resp.\ in all $(\cala,\calh)$-trees on which $G$ acts). A group not containing $\F_2$, or contained in a group of $\calh$, is small in $(\cala,\calh)$-trees.

The following lemma is simple, but conceptually important. It says that subgroups which are small but not virtually cyclic are elliptic in acylindrical trees. 

\begin{lem}\label{vc}
If a subgroup $J\inc G$ is small in a $(k,C)$-acylindrical tree $T$,\index{small, small in $(\cala ,\calh )$-trees} but $J$ is not elliptic in $T$, then $J$ is    $C$-virtually cyclic. 
\end{lem}

\begin{proof} By smallness, $J$ preserves a line or fixes an end. The result is clear if it acts on a line, since edge stabilizers  for this action have order $\le C$.
If it fixes an end, the set  of   elliptic elements of $J$ is the kernel  $J_0$ of a homomorphism from $J$ to $\Z$. 
Every finitely generated subgroup of $J_0$ is elliptic. It fixes a ray, so has order $\le C$  by acylindricity. 
 Thus $J_0$ has order $\le C$.
\end{proof}

\subsection{Small domination}\label{sec_Tc_smally}

We have seen that $T$ dominates $T_c$ and $T_c^*$  (Lemma \ref{lem_basicTc}),    
  and that $T_c^*$ is $(2,C)$-acylindrical under the assumptions of Lemma \ref{acl}. This domination may be strict: in particular, as shown by Lemma \ref{vc}, small groups which are not virtually cyclic tend to be elliptic in $T_c$ and $T_c^*$.  For instance, in Example \ref{z2}, the subgroup $P=\Z^2$   becomes elliptic in $T_c^*$. This is unavoidable according to Lemma \ref{vc}, but $P$ (together with its subgroups and their conjugates) is the only group that is elliptic in $T_c^*$ and not in $T$.

We want the deformation space of $T_c^*$ to be  as close to that of $T$ as possible, as in this example, and this motivates the following definition; 
 it applies to an arbitrary pair of trees $(T,T^*)$.

\renewcommand{\SS}{\cals}
\newcommand{\SNVC}{\cals_{\mathrm{nvc}}}

\begin{dfn} [\smally dominates] \label{sm} \index{small domination, $\SS$-domination}
Let $T,T^*$ be $(\cala,\calh)$-trees.
We 
say that $T$     \emph{\smally dominates}   $T^*$
if:
  \renewcommand{\theenumi}{\roman{enumi}}
\begin{enumerate}
\item \label{it_dom}$T$ dominates $T^*$;
  \item \label{it_edge}
 edge stabilizers of $T^*$ are elliptic in $T$. 

\item \label{it_small} any group which is elliptic in $T^*$ but not in $T$
is small in $(\cala,\calh)$-trees.
 \end{enumerate}
 
More generally, let $\SS$ be a family of subgroups closed under conjugation and taking subgroups, with every $J\in \SS$ small in $(\cala,\calh)$-trees.  
If $T$ smally dominates $T^*$, and every group which is elliptic in $T^*$ but not in $T$ belongs to $\SS$, we say that $T$ \emph{$\SS$-dominates $T^*$.}
 \end{dfn}
\index{0SS@$\SS$: a family of small subgroups}
 
\begin{rem}
$\cals$-domination will be useful to describe flexible vertex stabilizers of JSJ trees
(as in the proof of Theorem \ref{JSJ_CSA}).
\end{rem}

 Note that  the first two conditions of the definition are always satisfied
if $T$ is a tree and    $T^*$  is its collapsed tree of cylinders $T_c^*$. 
The following proposition basically says that 
the third condition holds when the groups $G_{[A]}$ are small in $(\cala,\calh)$-trees. In this case we can make every tree acylindrical without changing its deformation space more than forced by Lemma \ref{vc}.

\begin{prop} \label{tcsmdom}
Let $\sim$ be an admissible equivalence relation on $\cale$, and let $C$ be an integer. Assume that:
\begin{enumerate}
\item if two groups of $\cale$ are inequivalent, their   intersection has  order $\le C$;
\item every stabilizer $G_{[A]}$ is small in $(\cala,\calh)$-trees;
\item one of the following holds:
{\begin{enumerate}
\item every stabilizer $G_{[A]}$  belongs to $\cala$;
\item if $A\inc A'$ has index 2, and $A\in\cala$, then $A'\in\cala$;
\item
no group $G_{[A]}$ maps onto $D_\infty$.
\end{enumerate} }
\end{enumerate}

 If $T$ is any $(\cala,\calh)$-tree with infinite edge stabilizers, then $T_c^*$ is a $(2,C)$-acylindrical  $(\cala,\calh)$-tree smally dominated by $T$. 
 
 Assume furthermore that all subgroups which are small in $(\cala,\calh)$-trees but not virtually cyclic are elliptic in $T$. Then $T_c^*$ belongs to the same deformation space as $T$.
\end{prop}

\begin{proof}
Acylindricity comes from Lemma \ref{acl}, and conditions (\ref{it_dom}) and (\ref{it_edge})  of small domination are 
always satisfied by pairs 
$(T,T_c)$ and $(T,T_c^*)$. 

 Vertex stabilizers of $T_c$ are vertex stabilizers of $T$ or equal to some $G_{[A]}
 $, so by Assumption (2) subgroups elliptic in $T_c$ satisfy condition (\ref{it_small}). In other words, $T$ would smally dominate $T_c$ if $T_c$ were an $\cala$-tree.  To conclude that $T_c^*$ is smally dominated by $T$, 
we show that
$T_c$ and $T_c^*$ belong to  the same deformation space under Assumptions (2) and (3) (see {\cite[Remark 5.11]{GL4}}). 

Note that $T_c^*=T_c$ when (3a) holds, so we may assume that (3b) or (3c) holds.
Let $\varepsilon =(x,Y)$ be an edge of $T_c$ such that $G_\eps\notin\cala$. 
 It suffices to prove that $G_\eps=G_Y$ (so that collapsing the orbit of $\varepsilon$ does not change the deformation space) and that $G_{\eps'}\in \cala$
for every   edge $\eps'=(x',Y)$ with $x'\ne x$ (so no further collapse occurs in the star of the vertex $Y$  of $T_c$). 

The group $G_Y$ is small in   $(\cala,\calh)$-trees by 
(2). We claim that it is elliptic in $T$.
Assume otherwise, and consider its subgroup $G_\varepsilon$, which is elliptic in $T$. If $G_Y$ fixes an end of $T$, then $G_\varepsilon$ fixes a ray, so $G_\varepsilon\in \cala$, a contradiction. The remaining possibility (which is ruled out if (3c) holds) is that $G_Y$ acts dihedrally on a line. In this case some subgroup of $G_\varepsilon$ of index at most 2 fixes an edge, so is in $\cala$, contradicting (3b).

We have proved  that $G_Y$ is elliptic in $T$, hence fixes   a (unique) vertex $v\in Y\inc T$. 
We claim $x=v$. If $x\ne v$, let $e$ be the initial edge of the segment $[x,v]$. 
Then $$G_\varepsilon= G_x\cap G_Y\subset G_x\cap G_v\subset G_e,$$ 
 contradicting $G_\varepsilon\notin \cala$.
Thus $x=v$, so $\eps=(v,Y)$ is the only edge of $T_c$  incident to $Y\in V_1(T_c)$ with stabilizer 
not in $\cala$.
Moreover, since $G_Y$ fixes $x$, we have $G_{\varepsilon}=G_x\cap G_Y=G_Y$.
  This proves that $T_c$ and $T_c^*$ lie in the same deformation space, and that $T$ smally dominates $T_c^*$.

For the furthermore, we need to show that every $G_Y$ is elliptic in $T$. By Assumption (2) it is small in  $(\cala,\calh)$-trees, so the only case to consider is when $G_Y$ is virtually cyclic. If $e$ is any edge of $T$ contained in $Y$, its stabilizer $G_e$ is infinite so is a finite index subgroup of $G_Y$. It follows that $G_Y$ is elliptic in $T$.
\end{proof}

\begin{rem}\label{tcsmdom2}
 Given $\SS$ as in Definition \ref{sm}, assume that every $G_{[A]}$ belongs to $\SS$. Then $T$ $\SS$-dominates $T_c^*$. If all groups of $\SS$ which are not virtually cyclic are elliptic in $T$, then $T_c^*$ belongs to the same deformation space as $T$. The proof is exactly as above.
\end{rem}

\subsection{Compatibility}
\label{sec_Tc_compat}

\index{tree of cylinders}\index{compatible trees} 
The tree of   cylinders has strong compatibility properties, which will be useful to construct  the compatibility JSJ tree (see  Theorem \ref{ouf3}).

The following fact is a general property of the tree of cylinders.
Recall that two trees are compatible if they have a common refinement.

\begin{lem}[{\cite[Proposition 8.1]{GL4}}] \label{compfac}
  $T_c$ and $T_c^*$ are compatible with any $(\cala,\calh)$-tree dominated by $T$. \qed
\end{lem}

We will also need the following more technical statement.

\begin{lem}\label{lraff}
Let $G$ be one-ended relative to $\calh$  (see Subsection \ref{AH}). Let $\sim$ be an admissible equivalence relation on $\cale$ such  that  every stabilizer $G_{[A]}$ (in particular, every group in $\cala$)  is small in $(\cala,\calh)$-trees.
Let $C$ be an integer, and 
 suppose that 
$\cala$ contains all $C$-virtually cyclic groups.  
  
Let $S$, $T$ be $(\cala,\calh)$-trees, with $S$ refining $T$. Assume that each vertex stabilizer $G_v$ of $T$ is small in $S$   (possibly elliptic) or   QH with
finite fiber   of cardinality $\leq C$. Then $S_c^*$ refines $T_c^*$.
\end{lem}

  The assumption  on vertex stabilizers of $T$ holds in particular if $T$ is a JSJ decomposition whose flexible vertices
are small or QH with   fiber of cardinality $\leq C$.

When $S$ dominates $T$, any map $f:S\ra T$ sends cylinder to cylinder and induces a cellular map $f{}_c^*$ from $S_c^*$ to  $T{}_c^*$ (it maps a vertex to a vertex, an edge to a vertex or an edge,   and is independent of $f$) \cite[Lemma 5.6]{GL4}.
The point of the  lemma is that $f{}_c^*$ is a collapse map if $f$ is.

\begin{proof} 
One passes from $S$ to $T$ by successively collapsing orbits of edges, in an arbitrary order. Starting from $S$, 
perform   collapses which do not change the deformation space as long as possible. Since trees in the same deformation space have the same tree of cylinders,
 this allows us to assume that no proper collapse of $S$  refining $T$ belongs to the same deformation space as $S$.  This ensures that, for each vertex $v$ of $T$, the action of $G_v$ on the preimage $S_v$ of $v$ in $S$ is minimal ($S_v$ may be a point).

Fix $v$, and consider the tree $S'$ 
obtained from $S$ by collapsing all edges mapped to a vertex of $T$ not in the orbit of $v$.
Thus $S_v$ embeds in $S'$, 
and we view $S_v$ as  the minimal subtree of $G_v$ in both $S$ or $S'$. We show that \emph{$S'_c{}^*$ refines $T_c^*$, and   vertex stabilizers  of $S'$  are small in $S$ or QH with fiber of cardinality $\le C$.} The lemma then follows by an inductive argument, considering $S$ and $S'$. We may assume that $S_v$ is not a point (otherwise $S'=T$).

First suppose that $G_v$ is small in $S$. 
This implies that $S_v$ is a line or a subtree with a $G_v$-fixed end, so all edges in $S_v$ belong to the same cylinder   of $S'$.
In particular, $G_v$ is elliptic in the tree of cylinders of $S'$, hence in $S_c'^*$.
We thus have  domination maps $S'\xra{f_1} T \xra{f_2} S_c'^*$.
As recalled above, by \cite[Lemma 5.6]{GL4}, these maps induce  (equivariant) cellular maps $S_c'^*\xra{f_{1c}^*} T_c^* \xra{f_{2c}^*} (S_c'^*)_c'^*=S_c'^*$ between
collapsed trees of cylinders (see Lemma \ref{lem_basicTc} for the equality $(S_c'^*)_c'^*=S_c'^*$).  The maps $f_{1c}^*$ and $f_{2c}^*$ cannot increase translation lengths, so it follows from Theorem \ref{thm_compat_length} that they are isomorphisms (one may also show directly that they are injective on any segment joing two vertices with stabilizer not in $\cala$). Thus $T_c^*=S_c'^*$. Vertex stabilizers of $S'$ are small in $S$ or QH because they are   vertex stabilizers of $T$ or contained in a conjugate of $G_v$.

 The second case is when $v$ is QH, 
 with finite fiber  $F$ and underlying orbifold $\Sigma$.  Consider the action of $G_v$ on $S_v$. By one-endedness, every component of $\bo \Sigma$ is used (Lemma \ref{lem_arcbis}), so by  Lemma  \ref{dual2} 
the action   is dominated by an action dual to a  family of geodesics $\call$ on $\Sigma$. 
It is in fact equal to such an action because edge stabilizers are small  in $(\cala,\calh)$-trees by assumption, 
hence virtually cyclic by Assertion (2) of Proposition \ref{prop_QHUE} (the assumption that $C$-virtually cyclic groups are in $\cala$ ensures that the groups $Q_\gamma$ are in $\cala$). In particular, vertices in $S_v$ are QH with fiber $F$.

We   claim that  \emph{any cylinder   of $S'$ containing an edge of the preimage $S_v$   of $v$ is entirely contained in $S_v$.} 
This implies that $S'_c$ refines $T_c$ by Remark 4.13 of \cite{GL4},  and therefore $S_c'^*$ refines $T_c^*$, thus completing the proof.   

 Let $e$ be an arbitrary edge in $S_v$. 
Given  an edge $f$  of $S'$ with $G_f\sim G_e$, we  have to prove   $f\subset S_v$. Suppose not.
Cylinders being connected,
we can assume that $f$ has an endpoint in $S_v$  (so is not collapsed in $T$).
Since $v$ is QH with finite fiber $F$, and   $G_f$ is infinite  by one-endedness, the image of $G_f$ 
in the  orbifold group $\pi_1(\Sigma)=G_v/F$ is infinite, and   therefore  contained with finite index in
  a boundary subgroup.

There is a geodesic $\gamma\in\call$ such that  $G_e$ contains the preimage $G_\gamma$  of the fundamental group of $\gamma$.
Since $G_e\sim G_f$, we have $\grp{G_\gamma,G_f}\subset \grp{G_e,G_f}\subset G_{[G_e]}$, so $\grp{G_\gamma,G_f}$ is small in $(\cala,\calh)$-trees.
This contradicts the 
second assertion of Proposition \ref{prop_QHUE} since $\grp{G_\gamma,G_f}$ is not virtually cyclic.
 This proves the claim, hence the lemma.
 \end{proof}

\section{Constructing JSJ decompositions using acylindricity}
\label{jsjac}

Using acylindrical accessibility we   show (Subsection  \ref{unac})  that  one may construct the JSJ deformation space of $G$ over $\cala$ relative to $\calh$, and describe its flexible subgroups, provided that every deformation space contains an acylindrical tree (with uniform constants). In Subsection \ref{sec_smally_acy} we prove the same results under the weaker assumption that every
tree  $T$ smally dominates an acylindrical tree $T^*$. 
In Section \ref{exam} we will combine this with Proposition \ref{tcsmdom}, which ensures the existence of  such trees $T^*$.

As already mentioned, acylindrical accessibility bounds the number of orbits of edges of acylindrical trees.
This does not prevent the existence of infinite sequences of refinements. 
For example, consider $G=A*B$ with $A,B$ torsion-free hyperbolic groups.
Fix $b\in B\setminus\{1\}$, and define $C_k=\grp{b^{2^k}}$ and $A_k=\grp{A,C_k}$.
Let $T_k$ be the Bass-Serre tree of the amalgam $G=A_k*_{C_k} B$.
It is $2$-acylindrical and can be refined into the splitting $G=A_{k+1}*_{C_{k+1}}C_k*_{C_k} B$.
This refinement is  not $2$-acylindrical but it belongs to the same deformation space as the $2$-acylindrical tree $T_{k+1}$. 
The trees $T_k$  all lie in distinct deformation spaces. Although this is true
in this example, it is not obvious in general that $G$ splits over the intersection of the groups $C_k$ (or some conjugates).

The  proof in Subsection \ref{exis} therefore uses not only acylindrical accessibility but also   arguments from Sela's proof \cite{Sela_acylindrical}, involving actions on $\bbR$-trees obtained by taking limits of splittings similar to the $T_k$'s considered above.
We refer the reader to the appendix
for some basic facts about \Rt s. In particular, we will use compactness of the space of projectivized length functions \cite{CuMo}, and Paulin's theorem   \cite{Pau_Gromov} that the Gromov topology and the axes topology agree on the space of irreducible \Rt s (and even on the space of semisimple trees, see Theorem \ref{thm_feighn}).

 As usual, 
 $\cala$ is a family stable under conjugation and taking subgroups. The family  $\calh$ may be  arbitrary, 
and $G$ is only assumed to be finitely generated
(it would be enough to assume
 that $G$ is   finitely generated relative to
a finite collection of subgroups, and these subgroups are in $\calh$).

\subsection{Uniform acylindricity }\label{unac}

Recall that a  tree $T$ is $(k,C)$-acylindrical if all   segments
of length $\geq k+1$ have stabilizer of cardinality $\leq C$. 
If $T'$ belongs to the same deformation space as $T$, it is $(k',C)$-acylindrical  for some $k'$ (see \cite{GL2}), but in general there is no control on $k'$.

  The main assumption of this section is the following.
We fix $k$ and $C$, and we assume that for each $(\cala,\calh)$-tree $T$ there is a $(k,C)$-acylindrical tree $T^*$ 
in the same deformation space.
Our goal is to  deduce 
the existence of the JSJ deformation space. 
This   section is a first step towards  Theorem \ref{thm_smallyacyl}, where
we allow $T^*$ not to lie in the same deformation space as $T$.

For example, suppose that $G$ is a one-ended torsion-free CSA group (for instance, a toral relatively hyperbolic group), and $\cala$ is the class of abelian subgroups. One can take   $T^*:=T_c=T_c^*$ to be the tree of cylinders
as in  the first example of Subsection \ref{defcyl}, provided that $G$ contains no non-cyclic abelian subgroups, or more generally that  these subgroups are contained  in a group of $\calh$ (up to conjugacy):
this guarantees that $T_c$ lies in the same deformation space as $T$  (see
Proposition \ref{tcsmdom}). The general case (no assumption on abelian subgroups) will be taken care of in   Subsection \ref{sec_smally_acy}.

 Recall (Definition \ref{dvc})  that a group is $C$-virtually cyclic  for some $C>0$ if it maps to $\bbZ$ or $D_\infty$ 
with kernel of cardinality $\leq C$.
Also recall that  $G$ is only assumed to be finitely generated, and  there is no restriction on $\calh$ (it can be any collection of   subgroups).

\begin{thm}\label{thm_JSJacyl}
 Given $\cala$ and $\calh$, suppose that there exist numbers $C$ and $k$ such that:
\begin{itemize}
\item  $\cala$ 
contains all $C$-virtually cyclic subgroups  of $G$, 
and all subgroups of order  $\leq 2C$;
\item for any $(\cala,\calh)$-tree $T$, there is an $(\cala,\calh)$-tree $T^*$ in the same deformation space
which is $(k,C)$-acylindrical.
\end{itemize}
Then the JSJ deformation space of $G$ over $\cala$ relative to $\calh$  exists.

Moreover, if all  groups in $\cala$  are small in $(\cala,\calh)$-trees, then the flexible vertices
are QH
  with fiber of cardinality at most $C$.
\end{thm}

See Subsections \ref{sec_slender}, \ref{jsjdf} and \ref{quah} for the definitions of small, flexible and QH.

  Theorem \ref{thm_JSJacyl} will be proved in  Subsections \ref{exis} and \ref{sec_JSJ_acyl_desc}. 
We start with  a general  lemma.

\begin{lem}\label{oef}
If a finitely generated group $G$ does not split over subgroups of order $\le C$ relative to a family $\calh$, 
there exists a finite  family $\calh'=\{H_1,\dots,H_p\}$, with $H_i$ a  finitely generated group contained in  a group of $\calh$, 
such that $G$ does not split over subgroups of order $\le C$ relative to  $\calh'$. 
\end{lem}

A special case of this lemma is proved in \cite{Perin_elementary}.

\begin{proof} All trees in this proof will have stabilizers in the family $\cala(C)$ consisting of all subgroups of order $\le C$. Note that over $\cala(C)$ all trees are universally elliptic, so having no splitting is equivalent to the  JSJ deformation space being trivial. 

 Let  $H_1,\dots,H_n,\dots$ be  an enumeration of all finitely generated subgroups of $G$ contained in  a group of $\calh$. 
We have  pointed out  in Subsection \ref{Dunw}   
that by Linnell's accessibility $G$ admits  JSJ decompositions over  $\cala(C)$,
so let $T_n$ be a JSJ tree   relative to $\calh_n=\{H_1,\dots,H_n\}$. 
We show the lemma by proving that $T_n$ is trivial for $n$ large. 

By Lemma \ref{lem_rafin} the tree $T_n$, which is relative to $\calh_{n-1}$, may be refined to a JSJ tree relative to $\calh_{n-1}$. If we fix  $k$, we may therefore find trees $S_1(k),\dots,S_k(k)$ such that $S_i(k)$ is a JSJ tree relative to $\calh_i$ and $S_i(k)$ refines $S_{i+1}(k)$. By Linnell's accessibility theorem, there is a uniform bound for the number of orbits of edges of $S_1(k)$ (assumed to have no redundant vertices). 
 As $T_i$ and $S_i(k)$ are in the same deformation space,
this number is an upper bound for the number of $i's$ such that $T_i$ and $T_{i+1}$ belong to different deformation spaces, so   for   $n$ large the trees $T_n$ all belong to the same deformation space $\cald$ 
(they have the same elliptic subgroups).

Since every $H_n$ is elliptic in $\cald$, so is every $H\in\calh$ (otherwise $H$ would fix a unique end, and edge stabilizers would increase infinitely many times along a ray going to that end). The non-splitting hypothesis made on $G$ then implies that $\cald$ is trivial. 
\end{proof}

We also note:

\begin{lem}\label{oneend} It suffices to prove Theorem \ref{thm_JSJacyl} under the additional hypothesis that 
 $G$ does not split over subgroups of order $\le 2C$ relative to  $\calh$ (``one-endedness condition''). 
\end{lem}

\begin{proof} Let $\cala (2C)\inc\cala$ denote the family of subgroups of order $\le 2C$. As mentioned  in the previous proof, Linnell's accessibility implies the existence of a JSJ tree over $\cala (2C)$ relative to $\calh$. We can now apply Proposition \ref{prop_JSJ_sommets}.
\end{proof}

 \subsubsection{Existence of the JSJ deformation space}\label{exis}
 
Because of Lemma \ref{oneend}, we assume from now on that $G$ does not split over subgroups of order $\le 2C$ relative to  $\calh$.

In this  subsection we prove the first assertion of Theorem \ref{thm_JSJacyl}  (flexible vertices will be studied in the next subsection). 
We have to construct a universally elliptic tree $T_J$ which dominates every universally elliptic tree   
(of course, all trees are $(\cala,\calh)$-trees and universal ellipticity is defined with respect to $(\cala,\calh)$-trees). 
Countability of $G$ allows us to choose a sequence of universally elliptic trees $U_i$ such that, if $g\in G$ is elliptic in every $U_i$, 
then it is elliptic in every universally elliptic tree.  
 By  Assertion (2) of  Lemma  \ref{cor_Zor}, it suffices  that  $T_J$ dominates  every $U_i$.  
 Inductively replacing each $U_i$   by its 
standard refinement
dominating $U_{i-1}$, as in the proof of Theorem \ref{thm_exist_mou},
we may assume that $U_{i}$ dominates $U_{i-1}$ for all $i$.  In particular, we are free to replace $U_i$ by  a subsequence when needed.

Let    $T_i$ be a $(k, C)$-acylindrical tree in the same deformation space as $U_i$, with distance function denoted by $d_i$  (we view $T_i$ as a metric tree, with all edges of length 1).
Let $\ell_i:G\to\Z$ be the translation length function of $T_i$ (see the appendix).   The proof   has two main steps.
First we assume that, for all $g$, the  sequence $\ell_i(g)$ is bounded, and  
 we   construct a universally elliptic $(\cala,\calh)$-tree $T_J$  which dominates every $U_i$. Such a  tree is a JSJ tree. In the second step, we deduce a contradiction from the assumption that the sequence $\ell_i$ is unbounded.

$\bullet$  If $\ell_i(g)$ is bounded for all $g$,  we pass to 
 a subsequence so that $\ell_i$ has a limit $\ell$ (possibly $0$). Since the set of length functions of trees is closed 
\cite[Theorem 4.5]{CuMo}, $\ell$ is the length function associated to  the action of $G$ on an  \Rt{}    $T$. 
 It  takes values in $\bbZ$, so the tree $T$ is simplicial  (possibly with edges of length $\frac12$, see Example \ref{arbsimp}). In general $T$  is not an  $\cala $-tree, but we will show that it is relative to $\calh$. 

We can assume that all trees $T_i$ are non-trivial.
By Lemma \ref{vc},
$T_i$ is irreducible except if $G$ is virtually cyclic, in which case the theorem is clear.
If $g\in G$ is hyperbolic in some $T_{i_0}$, then $\ell_i(g)\geq 1$ for all $i\geq i_0$  (because $T_i$ dominates $T_{i_0}$), so $g$ is hyperbolic in $T$.
Since every $T_i$ is irreducible, there exist $g,h\in G$ such that $g,h$,  and $ [g,h ]$  are hyperbolic in $T_i$ (for some $i$), hence in $T$, so  $T$ is also irreducible.
By \cite{Pau_Gromov}, $T_{i }$ converges to $T$ in the equivariant Gromov-Hausdorff topology\index{topology!equivariant Gromov-Hausdorff topology}\index{Gromov@(equivariant) Gromov-Hausdorff topology} (see Theorem \ref{thm_fred}).

 We will not claim anything about the edge stabilizers of $T$, but we study its vertex stabilizers.
We claim that \emph{a subgroup $H\inc G$ is elliptic in $T$ if and only if it is elliptic in every $T_i$}; in particular,  every $H\in \calh$ is elliptic in $T$, and $T$ dominates every $T_i$. 

The claim  is true if $H$ is finitely generated, since $g\in G$ is elliptic in $T$ if and only if it is elliptic in every $T_i$. 
If $H$ is infinitely generated,  fix   a finitely generated subgroup  $H'$   of cardinality  $>C$.
If   $H$ is elliptic in $T$, it is elliptic in $T_i$ because otherwise it fixes a unique  end and $H'$ fixes
an infinite ray of $T_i$, contradicting acylindricity of $T_i$.
  Conversely, if $H$ is   elliptic in every $T_i$ but not in $T$, the group $H'$ fixes   an infinite ray in $T$. 
By   convergence in the Gromov topology, $H'$ fixes a segment of length $k+1$ in $T_{i }$ for $i$ large. This contradicts acylindricity of $T_i$, thus proving the claim.

We now return to the trees $U_i$.  They are dominated by $T$ since the trees $T_i$ are.
Their edge stabilizers  are in $\cala_{\elli}$, 
the family of groups in $\cala$ which are universally elliptic. 
Since $\cala_{\elli}$ is stable under taking subgroups, any  equivariant map $f:T\ra U_i$ factors through the tree $T_J$ obtained from $T$ by collapsing all edges with stabilizer not in  $\cala_{\elli}$. 
The tree $T_J$ is a universally elliptic $(\cala,\calh)$-tree dominating every $U_i$, hence 
every universally elliptic tree. It is 
a JSJ tree.

  $\bullet $ We now suppose that $\ell_i(g)$ is unbounded for some $g\in G$, and we work towards a contradiction. Since the set of projectivized non-zero length functions  on  a finitely generated group is compact \cite[Theorem 4.5]{CuMo}, 
we may assume that ${\ell_i}/{\lambda_i}$ converges to the length function $\ell$ of a non-trivial  \Rt\ $T_\infty$, for some sequence $\lambda_i\to+\infty$. 
  By Theorem \ref{thm_feighn},  $T_i/\lambda_i$ converges to $T$ 
  in the Gromov topology (all $T_i$'s are irreducible, and  
  we   take   $T_\infty$ to be  a line if it is not irreducible).

\begin{lem} \label{stab}\ 
  \begin{enumerate}
  \item  
 Any subgroup $H<G$ of order $>C$ which is elliptic in every $T_i$ fixes a unique point in $T_\infty$. In particular, elements of $\calh$ are elliptic in $T_\infty$.  
  \item Tripod stabilizers of $T_\infty$ have cardinality $\leq C$.
  \item If $I\subset T_\infty$ is a non-degenerate arc,    its stabilizer $G_ I$ has order $\le C$ or is $C$-virtually cyclic.

  \item Let $J\subset I\subset T_\infty$ be two non-degenerate arcs. If 
    $G_ I$ is $C$-virtually cyclic, then $G_ I=G_ J$.
  \end{enumerate}
\end{lem}

\begin{proof}
 Recall that $T_{i+1}$ dominates $T_i$, so a subgroup acting non-trivially on $T_i$ also acts non-trivially on $T_j$ for $j>i$.

To prove (1), we  may assume that $H$ is finitely generated. 
It is elliptic in $T_\infty$ because all of its elements are. But it cannot fix an arc in $T_\infty$: 
otherwise, since $ {T_i}/{\lambda_i}$ converges to $T_\infty $ in the Gromov topology, 
$H$ would fix a long segment in $T_i$ for  $i$ large, contradicting 
acylindricity.

Using the Gromov topology, as in Example \ref{asept}, one sees that a finitely generated subgroup fixing a tripod of $T_\infty$ 
fixes a long tripod of   $T_i$ for $i$ large,  so has cardinality at most $C$
by acylindricity. This proves (2).

 To prove (3),  consider a group $H$ fixing a non-degenerate arc $I=[a,b]$ in $T_\infty$.   We write $ | I | $ for the length of $I$.   It suffices to show that $H$   is small:
 depending on whether $H$ is elliptic in every $T_i$ or not, Assertion (1) or  Lemma \ref{vc} then gives the required conclusion. 

If $H$ contains a non-abelian free group, choose elements $\{ h_1,\dots,h_n\}$ generating a free subgroup $\F_n\inc H$   of rank $n\gg C$, 
and choose $\varepsilon>0$ with  $\varepsilon \ll | I | $. 
For $i$ large, there exist  approximations  $a_i,b_i$ of $a,b$ in $T_i$ at distance at least $(|I|-\eps)\lambda_i$ from each other, and contained in
the characteristic set (axis or fixed point set) of each $h_j$. 
Additionally, 
the translation length of every $h_j$ in $T_i$ is at most $\eps\lambda_i$ if $i$ is large enough.
Then all commutators $[h_{j_1},h_{j_2}]$ fix most of  the segment $[a_i,b_i]$, contradicting acylindricity of $T_i$ if $n(n-1)/2>C$. 

We now prove (4).  By Assertion (1), we know that $G_I$ acts non-trivially on $T_i$ for $i$  large, so we can choose a hyperbolic element $h\in G_I$.
We suppose that  some $g\in G_J$ 
does not fix an endpoint, say  $a$, of $I=[a,b]$, and we argue towards a contradiction. 
Let $a_i,b_i\in T_i$ be in the axis of $h$ as above, with $d_i(a_i,ga_i)\geq \delta\lambda_i$ for some   $\delta>0$.
For $i$ large the translation lengths of $g$ and $h$ in $T_i$ are small compared to $\frac{  | J | \lambda_i}C$, and the elements
$[g,h],[g,h^2],\dots, [g,h^{C+1}]$ all fix a common long   arc in $T_i$.
By acylindricity, there exist  $j_1\neq j_2$ such that $[g,h^{j_1}]=[g,h^{j_2}]$, so $g$ commutes with $h^{j_1-j_2}$.
It follows that $g$ preserves the axis of $h$, and therefore moves $a_i $ by $\ell_i(g)$, a contradiction since  $\ell_i(g)/\lambda_i$ goes to 0 as $i\to\infty$.
\end{proof}

  Sela's proof of acylindrical accessibility\index{acylindrical accessibility}  \cite{Sela_acylindrical} now comes into play to describe the structure of $T_\infty$. 
We use the generalization given by Theorem 5.1 of \cite{Gui_actions}, which allows non-trivial tripod stabilizers.    
Lemma  \ref{stab} shows that stabilizers of unstable arcs and tripods have  cardinality at most $C$, and we have assumed that   $G$ does not split over a subgroup of cardinality at most $2C$ relative to $\calh$,
hence also relative to a finite family $\calh' $ by Lemma \ref{oef}. 
It follows  that  $T_\infty$ is a graph of actions as in Theorem 5.1 of \cite{Gui_actions}. In order to reach  the desired   contradiction, we have to rule out   several  possibilities.

First consider a vertex action $G_v \actson  Y_v$ of the decomposition of $T_\infty$ given by \cite{Gui_actions}. 
If $Y_v$ is a line on which $G_v$ acts  with dense orbits through a finitely generated group, 
then, by Assertion (3) of Lemma \ref{stab},  $G_v$ contains a finitely generated subgroup   $H$ mapping onto $\bbZ^2$ with  finite or virtually cyclic kernel, and acting non-trivially. 
The group $H$ acts non-trivially on $T_i$ for $i$ large,   contradicting Lemma \ref{vc}.

Now suppose that 
 $G_v\actson Y_v$ has kernel $N_v$, and the action of $G_v/N_v$  is dual to an   arational measured foliation
on a $2$-orbifold $\Sigma$ (with conical singularities). Then $N_v$ has order $\le C$ since it fixes a tripod. 
 Consider a one-edge splitting $S$ of $G$ (relative to $\calh$) dual to a simple closed 
curve  on $\Sigma$. 
 This splitting is over a
$C$-virtually cyclic group $G_e$.
In particular, $S$ is an $(\cala,\calh)$-tree.
Since $G_e$ is hyperbolic in $T_\infty$, it is also hyperbolic in $T_i$, hence in $U_i$, for $i$ large enough.
On the other hand,   being universally elliptic, $U_i$ is elliptic with respect to $S$. 
By  Remark 2.3 of \cite{FuPa_JSJ} (see Assertion (3) of Lemma \ref{cor_Zor}),  $G$ splits relative to $\calh$ over
an infinite index subgroup of $G_e$, \ie over a group of order $\le 2C$ (in fact $\leq C$ in this case), contradicting our assumptions.

 By  Theorem 5.1 of \cite{Gui_actions}, the only remaining possibility is that $T_\infty$ itself is a simplicial tree, and   all edge stabilizers are $C$-virtually cyclic.
Then $T_\infty$ is an $(\cala,\calh)$-tree, and its edge stabilizers are hyperbolic in $T_i$ for $i$ large.
This leads to a contradiction as in the previous case,   and concludes the proof of the first assertion of Theorem \ref{thm_JSJacyl}.

\subsubsection{Description of flexible vertices} \label{sec_JSJ_acyl_desc}

  Now that we know that the JSJ decomposition exists,  we prove the second assertion of Theorem \ref{thm_JSJacyl}:   if all groups in $\cala$ are small in $(\cala,\calh)$-trees,
then the flexible vertices are   QH with fiber of cardinality at most $C$. 

The arguments are similar to those used in Section \ref{Fujpap}, with one key difference: since we do not assume finite presentability of $G$, we do not have  Lemma \ref{lem_Zorn_minuscule}    constructing  a tree which is maximal for domination using  Dunwoody's accessibility.
We shall use acylindrical accessibility instead.

As in Subsection   \ref{sec_reduc} we may assume that $G$ is totally flexible, \ie that there exist non-trivial $(\cala,\calh)$-trees,
but none of them is universally elliptic.
Indeed, to prove that a flexible vertex $G_v$ of the JSJ decomposition of $G$ is QH, 
it is enough to study  $(G_v,\Inc_v^\calh)$ instead of $(G,\calh)$.
Note that we do not know in advance that JSJ trees have finitely generated edge stabilizers, so
we have to allow infinitely generated groups in $\Inc_v^\calh$ (even if groups in $\calh$ are finitely generated).

The fact that $G$ is totally flexible implies   that any edge stabilizer $A $  of any tree
  is $C$-virtually cyclic: since $A$ is small in $(\cala,\calh)$-trees, this follows 
from Lemma \ref{vc} applied to the action of $A$ on $T^*$, where $T$ is some $(\cala,\calh)$-tree in which $A$ is not elliptic. We may therefore replace $\cala$ by the family consisting of all $C$-virtually cyclic subgroups and all finite subgroups. It satisfies the stability condition $(SC)$ of Definition \ref{dfn_SC}, with $\calf$ the family of subgroups of order $\le C$, so we are free to use the results of Subsections \ref{fpc} and \ref{rn}, as well as Proposition \ref{prop_exist_hh} (splittings are clearly minuscule, so Subsection \ref{sec_all_minus} is not needed).

By acylindrical accessibility \cite{Weidmann_accessibility}, there is a bound on the number of orbits of edges of
a $(k,C)$-acylindrical tree. Among all $(k,C)$-acylindrical trees with no redundant vertex, consider a tree $U$ 
 whose number of  orbits of edges is maximal
(this is a substitute for the  maximal tree provided by
Lemma \ref{lem_Zorn_minuscule}).  Proposition \ref{prop_exist_hh} yields $V$ such that $U$ and $V$ are fully hyperbolic \wrt each other, and we let $R=RN(U,V)$ be their regular neighborhood (Proposition \ref{prop_RN}).  

 Recall (Proposition \ref{prop_RN} (\ref{qh})) that the stabilizer $G_v$ of a vertex $v\in\cals$ is QH or slender. If slender, it maps onto 
a group that is virtually $\bbZ^2$. On the other hand, 
 it must be virtually cyclic
by Lemma \ref{vc} because it acts hyperbolically on $U$. This contradiction shows that $G_v$ must be  QH (with fiber of cardinality $\le C$) if  $v\in\cals$. 
As in the proof of Theorem \ref{thm_totally_flex},  it thus suffices  to show that $R$ is a point.

By
(\ref{compat})
one obtains $U$ from $R$ by refining $R$ at vertices $v\in\cals$
and collapsing all other edges.   We let $S$ be a common refinement of $R$ and $U$ such that  no edge of $S$ is collapsed in both   $R$ and $U$
(it is the  lcm $S=R\vee U$ of  Subsection \ref{sec_arith}).
Let $S^*$ be a $(k,C)$-acylindrical tree with no redundant vertex in the deformation space of $S$.

We now consider commensurability classes of edge stabilizers.

Since $U$ is dual to geodesics in   QH vertices of $R$ by
(\ref{compat}), and   has no redundant vertex,   edges of $U$ in different  orbits have non-commensurable stabilizers. 
Every edge stabilizer of $U$ is also an edge stabilizer of $S$. Next observe that,
for any edge $e$ of $S$, there exists an  edge $e'$ of $S^*$ such that $G_e$ and $G_{e'}$ are commensurable.
Indeed, since $S^*$ dominates $S$, there exists $e'$ with  
$G_{e'}\subset G_e$. By one-endedness (Lemma \ref{oneend}), the cardinality of $G_{e'}$ is greater than $2C$.
Since $G_e$ is $C$-virtually cyclic, this implies that $G_{e'}$ is  infinite, hence commensurable to $G_e$.

 Given a tree $T$, denote 
by $ | T | $ the number of orbits of edges, and by $c(T)$ the number of equivalence classes of orbits of edges, where two orbits are equivalent if they contain edges with commensurable stabilizers. 
We have just proved $ | U | \le c(U)\le c(S)\le c(S^*)\le  | S^* | $.

The maximality property of $U$ now implies that these inequalities are equalities, so every edge stabilizer of $S$ is commensurable to an edge stabilizer of $U$. If $R$ is not a point, let $A$ be an edge stabilizer. It is elliptic in $V$ by (\ref{el1}). On the other hand, $A$ is an edge stabilizer of $S$, so is commensurable to an edge stabilizer of $U$. This    contradicts the  full  hyperbolicity of $U$  with respect to $V$, thus  completing the proof of Theorem \ref{thm_JSJacyl}.

\subsubsection{Splittings over virtually cyclic groups}

 Before generalizing Theorem \ref{thm_JSJacyl} in the next subsection, we give an application of the previous arguments. 
 In Section \ref{Fujpap} we proved that certain flexible   groups are QH when $\calh$ is a finite family of finitely generated subgroups and $G$ is finitely presented relative to $\calh$. Acylindricity will allow us to remove these assumptions for splittings over virtually cyclic groups. This is based on the following lemma. The stability conditions and total flexibility are defined in   Section \ref{Fujpap}.
  
\begin{lem} \label{acyli}
 Assume that all groups in $\cala$ are virtually cyclic, and $\cala$ satisfies one of the stability conditions\index{stability condition} $(SC)$ or  ($SC_\calz$).
 If $G$ is totally flexible, then   any tree $T$
 is $(2,C)$-acylindrical  for some $C$ (depending on $T$).
\end{lem}

\begin{proof}
Since groups in $\cala$ are virtually cyclic, we may assume (by making $\calf$ smaller if needed) that all groups in $\calf$ are finite.
 Total flexibility implies that $G$ is one-ended relative to $\calh$, so all trees are minuscule (of course, this also follows from Proposition \ref{prop_all_minus}). 
 By Proposition \ref{prop_exist_hh}, there exists a tree $T'$ such that $T$ and $T'$ are fully hyperbolic with respect to each other. Let $R=RN(T,T')$ be their regular neighborhood. By (\ref{compat}), the tree $T$ is dual to families of geodesics in the orbifolds underlying the QH vertex groups of $R$. It follows that $T$ is
 $(2,C)$-acylindrical, for $C$ the maximum order of the corresponding fibers (see the example in the beginning of Part \ref{partacyl}). 
\end{proof}

This lemma allows us to argue as in the previous subsection, provided that there is a bound for the order of   groups in $\calf$. In particular, we get the following strengthening of Theorem \ref{thm_RiSe}:

\begin{thm} 
\label{thm_RiSe2}
 Let $G$ be a finitely generated group, and let $\calh$ be an arbitrary family of subgroups.
 Let $\cala$ be the class of all finite or cyclic subgroups of $G$. Assume that there exists a JSJ tree $T$   over $\cala$ relative to $\calh$.
 
 If $Q$ is a  flexible vertex stabilizer of $T$,
 then $Q$ is virtually $\Z^2$ or QH with trivial fiber. Moreover, the underlying orbifold $\Sigma$ has no mirror, every boundary component of   $\Sigma$ is used, and $\Sigma$ contains an essential simple closed geodesic.
\end{thm}

\begin{proof}
The argument is the same as in   Subsection \ref{sec_JSJ_acyl_desc}.  One first reduces to the case when $G$ is totally flexible. Since $\calf=\{1\}$, all trees are 2-acylindrical by Lemma \ref{acyli}, so acylindrical accessibility applies.

\end{proof}

\subsection{Acylindricity up to small groups}\label{sec_smally_acy}

In this section we generalize Theorem \ref{thm_JSJacyl}. Instead of requiring that every deformation space contains a $(k,C)$-acylindrical tree, we require that 
every tree smally dominates some $(k,C)$-acylindrical tree   (Definition \ref{sm}).

 Recall that  $G$ is only assumed to be finitely generated, and  there is no restriction on $\calh$ (it can be any collection of   subgroups).

\begin{thm}\label{thm_smallyacyl} Given $\cala$ and $\calh$,
suppose that there exist numbers $C$ and $k$ such that:
\begin{itemize}
\item  $\cala$ 
contains all $C$-virtually cyclic subgroups, 
and all subgroups of cardinal $\leq 2C$;
\item every  $(\cala,\calh)$-tree $T$ \smally dominates some $(k,C)$-acylindrical\index{0KC@$(k,C)$-acylindrical}\index{KC@$(k,C)$-acylindrical} $(\cala,\calh)$-tree $T^*$.\index{small domination, $\SS $-domination}\index{acylindrical tree, splitting}
\end{itemize}
Then the JSJ deformation space of $G$  over $\cala$ relative to $\calh$  exists.

Assume further that all  groups in $\cala$  are small in $(\cala,\calh)$-trees. Then the flexible\index{flexible vertex, group, stabilizer} vertex groups  
 that
 are not small in $(\cala,\calh)$-trees are    QH\index{QH, quadratically hanging}
with fiber of cardinality at most $C$.

More generally,  if  $\cals$ is as in Definition \ref{sm} and $T$ always $\cals$-dominates $T^*$, then 
  flexible vertex groups  that do not belong to $\SS$ are    QH
with fiber of cardinality at most $C$.

\end{thm}

\begin{figure}[htbp] 
  \centering
  \includegraphics[width=\textwidth]{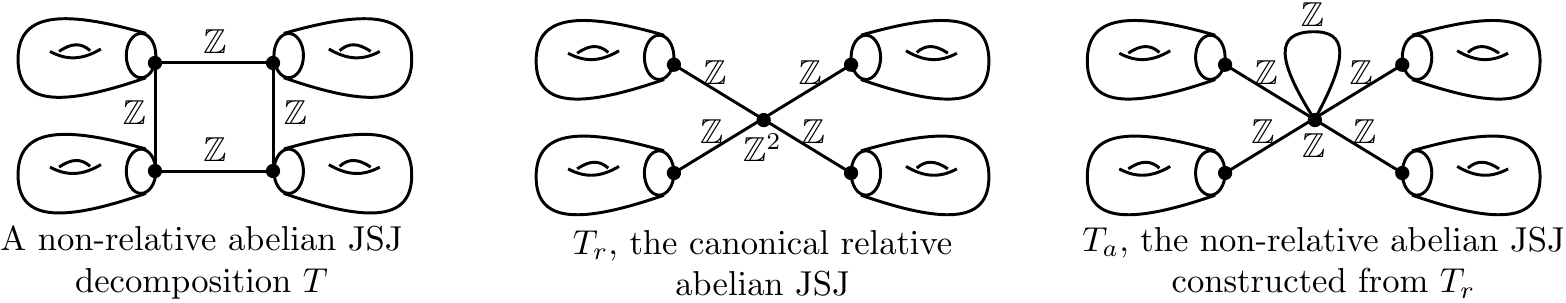}
  \caption{JSJ decompositions of a CSA group; relative means relative to non-cyclic abelian subgroups.}
  \label{fig_JSJCSA}
\end{figure}

 \begin{example} \label{asuiv}
 This applies for instance if 
  $G$ is a toral relatively hyperbolic group  (more generally, a torsion-free CSA group),\index{CSA group} $\cala$
   is the family of abelian subgroups,\index{abelian tree} and 
  we take for $T^*$ the tree of cylinders $T_c$ (equal to $T_c^*$), with $k=1$ and $C=2$   (see
Section \ref{exam} for more examples). 

In particular, let us return to 
    Example \ref{z2}, with $\calh=\es$. Figure \ref{fig_JSJCSA}
shows (the quotient graphs of groups of)   three trees.
The first one is an   abelian JSJ tree $T$; the next one is its tree of cylinders $T_c=T_c^*=T^*$,
which is also the JSJ tree $T_r$  relative to non-cyclic abelian subgroups constructed in the proof of Theorem \ref{thm_smallyacyl} (see Corollary \ref{Tr});
the last one is the JSJ tree $T_a$ constructed in the proof, obtained  by refining $T_r$ at vertices with stabilizer $\Z^2$ (Lemma \ref{lem_rel_vs_abs}).
It is another JSJ tree, in the same deformation space as $T$, and $T_r=T_c^*={(T_a)}_c^*$. 
 \end{example}

The rest of this section is devoted to the proof of Theorem \ref{thm_smallyacyl}. As above, we assume (thanks to Lemma \ref{oneend})   that $G$ does not split over groups of order $\leq 2C$ 
relative to $\calh$ (``one-endedness'' assumption). 

We write $\SS$ for the family of all subgroups of $G$ which are small in $(\cala,\calh)$-trees, or for a family as in Definition \ref{sm} if $T$   $\cals$-dominates $T^*$. 
We write  $ \SNVC$  for  the groups in $\SS$ which are not $C$-virtually cyclic. 
\index{0SS2@$\SNVC$: groups in $\SS$ not $C$-virtually cyclic}

\begin{lem}\label{lem_smallVC} 
Suppose that $T$ $\SS$-dominates a $(k,C)$-acylindrical   tree $T^*$. 
\begin{enumerate}
\item If a vertex stabilizer $G_v$ of $T^*$ is not elliptic in $T$, it belongs to $\SNVC$ (in particular, it is not $C$-virtually cyclic). 
\item 
 Let $H$ be a subgroup. It  is elliptic in $T^*$ if and only if it is elliptic in $T$ 
or   contained in a group $K \in \SNVC$.
In particular,  all  $(k,C)$-acylindrical trees $\SS$-dominated by $T$ belong to the same   deformation space.
\item Assume that $T^*$ is reduced. Then every  edge stabilizer  $G_e$  of $T^*$ has a subgroup of index at most $2$ fixing an edge in $T$. In particular, if $T$ is universally elliptic, so is $T^*$. 
\end{enumerate}
\end{lem}

  Recall (see Subsection \ref{comp}) that a tree $S$ is \emph{reduced}  if no proper collapse of $S$ lies in the same deformation space as $S$.    Equivalently, any edge $e=uv$ with $u,v$ in different orbits satisfies  $G_e\ne G_u$ and $G_e\ne G_v$.
  Since one may obtain a reduced tree in the same deformation space as $S$ by collapsing edges, there is no loss of generality in assuming that the smally dominated trees $T^*$ are reduced.

\begin{proof} Clearly $G_v\in\SS$ because $T$ $\cals$-dominates $T^*$.  Assume that it is $C$-virtually cyclic.
 Stabilizers of edges incident to $v$ have infinite index in $G_v$ since they are elliptic in $T$  and $G_v$ is not, so they have order $\le 2C$. 
This contradicts the ``one-endedness'' assumption (note that $T^*$ is not trivial because then $G=G_v$ would be $C$-virtually cyclic, also contradicting one-endedness).

The ``if'' direction of Assertion (2) follows from  Lemma \ref{vc}. 
Conversely, assume that $H$ is elliptic in $T^*$ but not in $T$. 
If     $v$ is a vertex of $T^*$ fixed by $H$, we have  $H\inc  G_v$ and  $G_v\in \SNVC$ 
  by Assertion (1).

For Assertion (3), let $u$ and $v$ be the endpoints of an edge $e$ of $T^*$.  
First suppose that $G_u$ is not elliptic in $T$. It is small in $T$, so it preserves a line or fixes an end. Since   $G_e$ is elliptic in $T$, some subgroup of index at most $2$ fixes an edge. 
 If $G_u$ fixes two distinct points of $T$, then $G_e$ fixes an edge. 
We may therefore assume that $G_u$ and $G_v$ each  fix  a unique point in $T$. 

If these fixed points are  different, $G_e=G_u\cap G_v$ fixes an edge of $T$. 
Otherwise,
  $\langle G_u,G_v\rangle$ fixes a point $x $ in $T$, and is therefore elliptic in $T ^*$. 
Since $T^*$ is reduced, some $g\in G$ acting hyperbolically on $T^*$ maps $u$ to $v$ and conjugates $G_u$ to $G_v$ 
 (unless there is  such a $g$, collapsing the edge $uv$ yields a tree in the same deformation space as $T^*$).
This element $g$ fixes $x$, so is elliptic in $T^*$, a contradiction.
\end{proof}

\newcommand{\calhs}{\calh\cup\SNVC}

\begin{cor} \label{cor_domin} \ 
  \begin{enumerate}\item $T^*$ is an $(\cala,\calhs)$-tree.

  \item \label{domin2}
    If $T_1$ dominates $T_2$, then $T_1^*$ dominates $T_2^*$. 

  \item  \label{domin3}
    If $T$ is an $(\cala,\calhs)$-tree, then $T^*$ lies in the same
    deformation space as $T$.  In particular,
    Theorem
    \ref{thm_JSJacyl} applies to $(\cala,\calhs)$-trees. \qed
  \end{enumerate}
\end{cor}

We also deduce:
\begin{cor} \label{ttpar}
\  
\begin{itemize}
 \item If a subgroup $J$ is small in $(\cala,\calhs)$-trees, it is also small in $(\cala,\calh)$-trees.
 \item If $J\notin \SS$ is elliptic in every   tree relative to $\calhs$, it is elliptic in every tree which is only relative to $\calh$.

\end{itemize}
\end{cor}

\begin{proof}
 If $J$ is not small in an $(\cala,\calh)$-tree $T$, it does not belong to $\SS$, and is not elliptic in $T^*$ by Assertion (2) of Lemma \ref{lem_smallVC}. By Lemma \ref{vc}, it is not small in the $(\cala,\calhs)$-tree $T^*$.
 
 If  $J\notin \SS$ is not elliptic in some $(\cala,\calh)$-tree $T$, then by Lemma \ref{lem_smallVC} it  is not elliptic in
 $T^*$.
\end{proof}

Thanks to the third assertion of Corollary \ref{cor_domin}, we may 
apply  Theorem \ref{thm_JSJacyl} to get:

\begin{cor} \label{Tr}
There exists  a JSJ tree $T_r$    over $\cala$ relative to $\calhs$.  \qed
   \end{cor}
   
We can assume that $T_r$ is reduced. 
We think of $T_r$ as  \emph{relative}, as it is relative to $\SNVC$ (not just to $\calh$). 
  In  Example \ref{asuiv},  $\SNVC$ is the class of all non-cyclic abelian subgroups and 
$T_r$ is  a JSJ tree 
over abelian groups relative to all non-cyclic abelian subgroups.

\begin{lem}\label{lem_Tr}
  If $T_r$ is reduced,  then  it is $(\cala,\calh)$-universally elliptic.
\end{lem}

\begin{proof}
We let $e$ be an edge of $T_r$ such that $G_e$ is not elliptic  in some $(\cala,\calh)$-tree $T$, and we argue towards a contradiction. We may assume that $T$   only has one orbit of edges. The first step is to show that $T_r$ dominates $T^*$.

Since $T^*$ is relative to $\calhs$,
 the group $G_e$ fixes a vertex  $u\in T^*$  by $(\cala,\calhs)$-universal ellipticity of $T_r$. This $u$ is unique because edge stabilizers of $T^*$ are elliptic in $T$. Also note that $G_u\in\SNVC$ by Assertion (1) of Lemma \ref{lem_smallVC}.

We may assume that $T^*$ is not a point (if it is, then $G\in \SNVC$  and  $T_r$ itself is a point), 
 so $G_u$ contains an edge stabilizer of $T^*$, hence the stabilizer of some edge $f\inc T$ since 
there is an equivariant map $T\to T^*$.
   The group $G_f$ is 
(trivially) $(\cala,\calhs)$-universally elliptic  because  $G_f \inc G_u\in \SNVC$. Since $T$ has a single orbit of edges, it is  $(\cala,\calhs)$-universally elliptic (but it is not relative to $\calhs$). On the other hand,  $T^*$ is an $(\cala,\calhs)$-tree, and it is $(\cala,\calhs)$-universally elliptic by Assertion (3) of Lemma \ref{lem_smallVC}.
By maximality of the JSJ, $T_r$ dominates $T^*$.

Recall that $u$ is the unique fixed point of $G_e$ in $T^*$. 
Denote by $a,b$ the   endpoints of $e$ in $T_r$. 
Since $T_r$ dominates $T^*$, the groups $G_a$ and $G_b$ fix $u$, so $\grp{G_a,G_b}\inc G_u\in\SNVC$ is elliptic in $T_r$. As in the   proof of Lemma \ref{lem_smallVC},
some $g\in G$ acting hyperbolically on $T_r$ maps $a$ to $b$ (because $T_r$ is reduced). 
This $g$ fixes $u$, so belongs to $G_u\in\SNVC$, a contradiction since $T_r$ is relative to $\calhs$.
\end{proof}

 We  shall   now construct a JSJ tree $T_a$ relative to $\calh$ by refining $T_r$ (a reduced JSJ tree relative to $\calhs$) at vertices with small stabilizer.
  This JSJ tree $T_a$ is thought of as \emph{absolute} as it is not relative to $\SNVC$.

\begin{lem}\label{lem_rel_vs_abs}
There exists a JSJ tree $T_a$ over $\cala$ relative to $\calh$.  It may be obtained by refining $T_r$ at vertices with stabilizer in $\SS$ (in particular, the set of vertex stabilizers not 
belonging to $\SS$ is the same for $T_a$ as for $T_r$). Moreover,  
  $T_a^*$ lies in the same deformation space as $T_r$.
\end{lem}

\begin{proof}
Let $v$ be a vertex of $T_r$. 
We shall  prove the existence of a JSJ tree $T_v$  for $G_v$
relative to the family
$\Inch_v$ consisting of incident edge groups and subgroups conjugate to a group of  $\calh$   (see   Definition \ref{indup}).
  By Proposition \ref{prop_JSJ_sommets},
  which applies because $T_r$ is $(\cala,\calh)$-universally elliptic (Lemma \ref{lem_Tr}), one then obtains a JSJ tree $T_a$ for $G$ relative to $\calh$ by refining $T_r$ using the trees $T_v$.

If  $G_v$  is elliptic in every $(\cala,\calh)$-universally elliptic tree $T$, its JSJ is trivial (see Lemma \ref{lem_passage}) and no refinement is needed at $v$. Assume therefore that $G_v$ is not elliptic in such a $T$. Consider   the $(\cala,\calhs)$-tree $T^*$. It  is $(\cala,\calh)$-universally elliptic by  Assertion (3) of  Lemma \ref{lem_smallVC}, hence $(\cala,\calhs)$-universally elliptic, so it is dominated by $T_r$. In particular, $G_v$ is elliptic in $T^*$, so 
belongs to $\SS$. 

Let $T_v\subset T$ be the minimal $G_v$-invariant subtree  (it exists by Proposition \ref{arbtf}   and Lemma \ref{relfg}, 
since the incident edge stabilizers of $v$ in $T_r$ are   elliptic in $T$).
Being small in $(\cala,\calh)$-trees,
 $G_v$ has 
 at most one non-trivial deformation space containing a universally elliptic tree (Proposition \ref{sma}).
Applying this   to splittings of $G_v$ relative to  $\Inch_v$,
we deduce that   $T_v$ is a
JSJ tree of $G_v$ relative to $\Inch_v$.
This shows the first two assertions of the lemma.

We    now show  the ``moreover''. 
 Since $T_a^*$ is an  $(\cala,\calhs)$-tree,  which is $(\cala,\calhs)$-universally elliptic by the third assertion of Lemma \ref{lem_smallVC}, 
it is dominated by $T_r$.
Conversely, $T_a$ dominates $T_r$  
and therefore $T_a^*$ dominates $T_r^*$ by  Corollary \ref{cor_domin}(\ref{domin2}), 
so $T_a^*$ dominates $T_r  $ since $T_r$ and $T_r^*$ lie in the same deformation space
by  Corollary \ref{cor_domin}(\ref{domin3}).
 \end{proof}

\begin{rem}
 By Corollary \ref{ttpar}, the type of vertex stabilizers not in $\SS$ (rigid or flexible) is the same in $T_a$ (relative to $\calh$) and $T_r$ (relative to $\calhs$).
\end{rem}

We can now conclude the proof of Theorem \ref{thm_smallyacyl}.
The JSJ deformation space relative to $\calh$ exists by Lemma \ref{lem_rel_vs_abs}. 
The description of flexible vertex groups follows from Theorem \ref{thm_JSJacyl},  since JSJ trees relative to $\calh$ and $ \calhs $ have the same vertex stabilizers not in $\SS$ by Lemma \ref{lem_rel_vs_abs}.

\section{Applications}\label{exam}

 Recall that $\cala_\infty$ is the family of infinite groups in $\cala$. 
Combining Proposition \ref{tcsmdom} and Theorem \ref{thm_smallyacyl}  yields: 
\begin{cor} \label{synthese}  Let $G$ be a finitely generated group. Given $\cala$ and $\calh$, let 
 $\sim$ be an admissible equivalence on $\cale$.  
Let  $\SS$ be the family of groups contained 
in some $G_{[A]}$, with $A\in \cala_\infty$.

 Assume that $G$ is one-ended relative to $\calh$, and  there exists an integer $C$ such that:
\begin{enumerate}
\item  $\cala$ 
contains all $C$-virtually cyclic subgroups, 
and all subgroups of cardinal $\leq 2C$;

\item if two groups of $\cale$ are inequivalent, their   intersection has  order $\le C$;
\item   every stabilizer $G_{[A]}$ is small in $(\cala,\calh)$-trees (hence so is every element of $\SS$);

\item one of the following holds:
{\begin{enumerate}
\item every stabilizer $G_{[A]}$  belongs to $\cala$;
\item if $A\inc A'$ has index 2, and $A\in\cala$, then $A'\in\cala$;
\item
no group $G_{[A]}$ maps onto $D_\infty$.
\end{enumerate} }
\end{enumerate}
Then:

\begin{enumerate}
\item 
 there is a JSJ tree $T_a$  over $\cala$ relative to $\calh$; its collapsed tree of cylinders $(T_a)_c^*$ is a JSJ tree relative to $\calh\cup\SNVC$ (with $ \SNVC$     the family of groups in $\SS$ which are not $C$-virtually cyclic);
  \item $T_a$ and $(T_a)_c^*$ have the same vertex stabilizers not in $\SS$; flexible vertex stabilizers   
 that do not belong to $\SS$  
 are    QH
with fiber of cardinality at most $C$;

 \item
  $(T_a)_c^*$ is a  
  canonical JSJ tree  relative to $\calh\cup\SNVC$; in particular, it is invariant under any automorphism of $G$ preserving $\cala$ and $\calh$;
  \item
 $(T_a)_c^*$ is compatible with every $(\cala,\calh)$-tree.
\end{enumerate}
\end{cor}

\begin{proof}
 By the one-endedness assumption, all $(\cala,\calh)$-trees $T$ have edge stabilizers in $\cale$, so the collapsed tree of cylinders $T_c^*$ is defined. By Proposition \ref{tcsmdom} and Remark  \ref{tcsmdom2}, it is   a $(2,C)$-acylindrical  $(\cala,\calh)$-tree $\cals$-dominated by $T$, and Theorem \ref{thm_smallyacyl} applies    taking $T^*:=T_c^*$ 
  (groups in $\cala$ are small in $(\cala,\calh)$-trees by Assumption (3)):  there exists a JSJ tree $T_a$, obtained as in Lemma \ref{lem_rel_vs_abs}, and its flexible vertex stabilizers are in $\SS$ or QH with fiber of cardinality at most $C$.

 Lemma \ref{lem_rel_vs_abs} states that  $(T_a)_c^*$
belongs to the deformation space of $T_r$ (a JSJ tree relative to $\calh\cup\SNVC$), and $T_a$ is obtained by refining $T_r$ at vertices with stabilizer in $\SS$. This proves the first two assertions of the corollary.
The third one follows from Corollary \ref{invaut}.

For the fourth assertion, let $T$ be any $(\cala,\calh)$-tree.
Since $T_a$ is universally elliptic,  there exists a refinement $S$ of $T_a$
dominating $T$  by Proposition  \ref{prop_refinement}. 
By
Lemma \ref{compfac},
$S_c^*$ and $T$ have a common refinement $R$.
Since $S_c^*$ is a refinement of $(T_a)_c^*$ by Lemma \ref{lraff}, the tree $R$ is a common refinement of $T$ and  $(T_a)_c^*$.
\end{proof}

\begin{rem} \label{gdss}   The result remains true if we enlarge $\SS$ (keeping it invariant under conjugating and taking subgroups), as long as all groups in $\SS$ are small in $(\cala,\calh)$-trees.
\end{rem}

\begin{rem}  By Corollary \ref{cor_unbout}, the one-endedness assumption is usually not necessary for the   assertions about $T_a$,
provided that the relative Stallings-Dunwoody space exists. 
\end{rem}

The first assumption of the corollary ensures that Propositions \ref{bu} and \ref{prop_QHUE} apply to  flexible vertex stabilizers $G_v$ of $T_a$ and $(T_a)_c^*$ that do not belong to $\SS$. In particular:

\begin{cor} \label{synthesep}
If $G_v$ is a flexible vertex stabilizer not belonging to $\cals$, then:
\begin{enumerate}
 \item  the underlying orbifold $\Sigma$ contains an essential simple geodesic;
 \item every boundary component of $\Sigma$ is used;
 \item every universally elliptic subgroup of $G_v$ is contained in an extended boundary subgroup;
 \item if $T$ is an $(\cala,\calh)$-tree in which $G_v$ does not fix a point,   the action of $G_v$ on its minimal subtree $\mu_T(G_v)$  is dual to a family of  geodesics of $\Sigma$.
\end{enumerate}
\end{cor}

\begin{proof} This follows directly from Propositions \ref{bu} and \ref{prop_QHUE}, 
 noting that edge stabilizers for the action of $G_v$ on $\mu_T(G_v)$ are virtually cyclic by the second assertion of Proposition \ref{prop_QHUE}.
  \end{proof}

In the following sections, we are going to describe examples where   Corollaries \ref{synthese} and \ref{synthesep} apply.
We first treat the case of abelian splittings of CSA groups.
To allow torsion, we introduce $K$-CSA groups in Subsection \ref{aKcsa}, and describe their JSJ decomposition
over virtually abelian groups.
We then consider elementary splittings   of relatively hyperbolic groups, 
and splittings
over virtually cyclic subgroups under the assumption that
these subgroups have small commensurators. We conclude by defining the $\Zmax$-JSJ decomposition of one-ended hyperbolic groups.

\subsection{CSA groups}\label{acsa}

In our first application, $G$ is a torsion-free CSA group, and we consider  abelian or cyclic splittings.
Recall that $G$ is \emph{CSA} if the commutation relation is transitive on $G\setminus\{1\}$, and maximal abelian subgroups
are malnormal.  Toral relatively hyperbolic groups, in particular limit groups  and torsion-free hyperbolic groups, are CSA.
 See Example \ref{asuiv} and Figure
\ref{fig_JSJCSA} for an illustration.

We let $\cala$ be either the family of abelian subgroups of $G$,  or the family of cyclic subgroups.  If $G$  is freely indecomposable  relative to $\calh$, commutation is an admissible equivalence relation on
 $\cala_\infty$ (see \cite{GL4}, or Lemma \ref{coadm} below),  and we can define trees of cylinders $T_c$.   The groups $G_{[A]}$ are maximal abelian subgroups, so are small in all trees. 
Over abelian groups   (\ie when $\cala$ is the class of abelian subgroups),  
all edge stabilizers of $T_c$  belong to $\cala$ since every $G_{[A]}$ is abelian, so $T_c^*=T_c$. 
Over cyclic groups,   $T_c$ may have non-cyclic edge  stabilizers, so we have to use $T_c^*$, 
obtained from $T_c$ by collapsing edges with non-cyclic stabilizers.

\begin{thm} \label{JSJ_CSA}
Let $G$ be a finitely generated torsion-free CSA group,\index{CSA group} and  $\calh$   any family of subgroups. 
 Assume that $G$ is freely indecomposable relative to $\calh$.
\begin{enumerate}
\item There is an  abelian (resp.\ cyclic)\index{abelian tree}\index{cyclic tree} JSJ tree $T_a$   relative to $\calh$. 
Its collapsed tree of cylinders $(T_a)_c^*$  (for commutation) is a JSJ tree relative to $\calh$ and   all non-cyclic abelian subgroups. 
\item 
   $T_a$ and $(T_a)_c^*$ have the same non-abelian vertex stabilizers; non-abelian flexible\index{flexible vertex, group, stabilizer} vertex stabilizers   
 are    QH (they are fundamental groups of compact surfaces).

\item
 $(T_a)_c^*$ is invariant under all automorphisms of $G$ preserving $\calh$.
 It is compatible with every $(\cala,\calh)$-tree.
\end{enumerate}
\end{thm}

\begin{proof} 
 We apply Corollary \ref{synthese}, with $\cala$ consisting of all abelian (resp.\ cyclic) subgroups, $\cals$ the family of abelian subgroups, and  $C=1$. Since $G$ is torsion-free, QH vertex groups have trivial fiber, and the underlying orbifold is a surface. 
\end{proof}

\subsection{$\Gamma$-limit  groups and $K$-CSA groups} \label{aKcsa}

 The notion of CSA groups is not well-adapted to groups with torsion. This is why we   shall introduce   $K$-CSA groups, where $K$ is an   integer. 
  Every hyperbolic group $\Gamma$
  is $K$-CSA for some   $K$. 
Being $K$-CSA is a universal property; in particular,   all $\Gamma$-limit groups are $K$-CSA.

We say that a group is \emph{$K$-virtually abelian}  \index{virtually abelian@$K$-virtually abelian}
if it contains an  abelian subgroup of index $\leq K$ 
(note that the infinite dihedral group $D_\infty$ is 1-virtually cyclic,  in the sense of Definition  \ref{dvc}, but only 2-virtually abelian). As usual, a group is locally  $K$-virtually abelian if its finitely generated subgroups are $K$-virtually abelian.

\begin{lem}\label{lem_Zorn2}
  If a countable  group $J$ is locally $K$-virtually abelian, then $J$ is $K$-virtually abelian.
\end{lem}

\begin{proof}
Let $g_1,\dots,g_n,\dots$ be a numbering of the elements of $J$.
Let $A_n\subset\grp{g_1,\dots,g_n}$ be an abelian   subgroup of index $\leq K$.
For a given  $k$, there are only finitely many subgroups of index $\leq K$ in $\grp{g_1,\dots,g_k}$,
so there is a subsequence $A_{n_i(k)}$ such that $A_{n_i(k)}\cap\grp{g_1,\dots g_k}$ is independent of  $i$. 
By a diagonal argument, one produces an   abelian subgroup $A$ of $J$ whose intersection with each $\grp{g_1,\dots,g_n}$ has index $\leq K$,
so $A$ has index $\leq K$ in $J$.
\end{proof}

\begin{dfn}[$K$-CSA] \label{kcsa}\index{0KA@$K$-CSA}\index{K-CSA}
Say that $G$ is $K$-CSA for some $K>0$ if:
  \begin{enumerate}
  \item Any finite subgroup has cardinality at most $K$ (in particular, any element of order $>K$ has infinite order).
\item Any element $g\in G$ of infinite order is contained in a \emph{unique} maximal virtually abelian group $M(g)$,\index{0MG@$M(g)$, $M(H)$: the maximal virtually abelian group containing $g$, $H$}
and $M(g)$ is $K$-virtually abelian.
\item $M(g)$ is its own normalizer.
  \end{enumerate}
\end{dfn}

A $1$-CSA group is just a torsion-free CSA group. The Klein bottle group is $2$-CSA but not $1$-CSA.
Any hyperbolic group $\Gamma$ is $K$-CSA for some $K$ since finite subgroups of $\Gamma$ have bounded order, 
and there are only finitely many isomorphism classes of virtually cyclic groups whose finite subgroups have bounded order (see Lemma 2.2 of \cite{GL_vertex} for a proof). 
Corollary \ref{lg} will say that $\Gamma$-limit groups also are $K$-CSA.

\begin{lem} \label{rem_KCSA}
Let $G$ be a $K$-CSA group.  
\begin{enumerate}
\item If   $g,h\in G$ have infinite order, the following conditions are equivalent:
\begin{enumerate}
\item $g $ and $h$ have non-trivial commuting powers;
\item $g^{K!}$ and $h^{K!}$ commute;
\item $M(g)=M(h)$;
\item $\langle g,h\rangle$ is virtually abelian.
\end{enumerate}
\item Any infinite virtually abelian subgroup $H$
is contained  in a  {unique} maximal virtually abelian group $M(H)$. \index{0MG@$M(g)$, $M(H)$: the maximal virtually abelian group containing $g$, $H$}
The group 
  $M(H)$  is $K$-virtually abelian and  almost malnormal: if $M(H)\cap M(H)^g$ is infinite, then $g\in M(H)$. 
\end{enumerate}
\end{lem}

\begin{proof}
  $(c)\Rightarrow (b)\Rightarrow (a)$ in Assertion (1) is clear since
  $g^{K!}\in A$ if $A\inc M(g)$ has index $\le K$. We prove
  $(a)\Rightarrow (c)$. If $g^m$ commutes with $h^n$, then $g^m$
  normalizes $M(h^n)$, so $M(g^m)=M(h^n)$ and
  $M(g)=M(g^m)=M(h^n)=M(h)$. Clearly $(c)\Rightarrow (d)\Rightarrow
  (a)$. This proves Assertion (1).

 Being virtually abelian, $H$ contains an element $h_0$ of infinite order, and we
  define $M(H)=M(h_0)$.   By Assertion (1), $M(H)$ does not depend on the choice of
$h_0$. To prove that $H\subset M(H)$, consider $h\in H$.
Since $hh_0 h\m\in H$ has infinite order, we have
  $M(h_0)=M(h h_0 h\m)=hM(h_0)h\m$, so $h\in M(h_0)$ because $M(h_0)$ equals its normalizer.
 A similar argument shows almost malnormality.
 There remains to prove uniqueness. If $A$ is any 
virtually abelian group  containing $H$, then $M(A)$ is defined and coincides with $M(h_0)$, so
$A\subset M(A)=M(h_0)=M(H)$.
\end{proof}

One easily checks that any subgroup of a $K$-CSA group is still $K$-CSA. This is in fact a consequence of the following proposition
saying that  $K$-CSA is a universal property.
We refer to \cite{CG_compactifying} for the topological space of marked groups, and its relation with universal theory.

\begin{prop}\label{prop_univ}
  For any fixed $K>0$, the class of $K$-CSA groups is defined by a set of  (coefficient-free) universal sentences   ($K$-CSA is a universal property).
In particular, the class of $K$-CSA groups is stable under taking subgroups, and closed in the space of marked groups.
\end{prop}

\begin{proof}
  For any finite group $F=\{a_1,\dots,a_n\}$, the fact that $G$ does not contain a subgroup isomorphic to $F$
is equivalent to a universal sentence saying that for any $n$-tuple $(x_1,\dots,x_n)$ satisfying the multiplication
table of $F$, not all $x_i$'s are distinct.
Thus, the first property of $K$-CSA groups is defined by (infinitely many) universal sentences.

\newcommand{\VA}{\mathrm{VA}}
Now consider the second property. 
We claim that, given $m$ and $n$, the fact that $\grp{g_1,\dots,g_n}$ is $m$-virtually abelian may be expressed by the disjunction $\VA_{m,n}$ of finitely many finite
systems of equations in the elements $g_1,\dots,g_n$. To see this, let $\pi:\F_n\to G$ be the homomorphism sending the $i$-th generator $x_i$ of   the free group $\F_n=\grp{x_1,\dots x_n}$ to $g_i$. If $A\inc \grp{g_1,\dots,g_n}$ has index $\le m$, so does $\pi\m(A)$ in $\F_n$. Conversely, if $B\inc \F_n$ has index $\le m$, so does $\pi(B)$ in $\grp{g_1,\dots,g_n}$. To define $\VA_{m,n}$, we then enumerate the subgroups of index $\le m$ of $\F_n$. For each subgroup, we choose a finite set of generators $w_i(x_1,\dots,x_n)$ and we write the system of equations
$[w_i(g_1,\dots,g_n),w_j(g_1,\dots,g_n)]=1$. This proves the claim.

By Lemma \ref{lem_Zorn2} (and Zorn's lemma), any $g$  is contained in a maximal $K$-virtually abelian subgroup. 
The second property of Definition \ref{kcsa} can be restated as follows: any finitely generated virtually abelian group is $K$-virtually abelian,
and if $\langle g,h\rangle$ and $\grp{g,g_1,\dots,g_n}$ are   $K$-virtually abelian,
with $g$ of order $>K$, then $\grp{g,h,g_1,\dots,g_n}$ is $K$-virtually abelian.
 This is defined by a set of universal sentences constructed using    the $\VA_{m,n}$'s.

If the first two properties of the definition hold, the third one is expressed by saying that, if $g$ has order $>K$ and $\grp{g,hgh\m}$ is $K$-virtually abelian, so is $\grp{g,h }$. This is a set of universal sentences as well.
\end{proof}

  Recall that  a $\Gamma$-limit group is defined as a limit of subgroups of $\Gamma$ in the space of marked groups. 
Proposition \ref{prop_univ}   implies  that, if $\Gamma$ is   $K$-CSA, then any $\Gamma$-limit group is $K$-CSA. In particular:

\begin{cor} \label{lg}
Let $\Gamma$ be a hyperbolic group. There exists $K$ such that any $\Gamma$-limit group\index{GA@$\Gamma$-limit group}  is $K$-CSA. 

Moreover, any subgroup of a $\Gamma$-limit group $G$ contains a non-abelian free subgroup or is $K$-virtually abelian. 
\end{cor}

\begin{rem}
We will not use the  ``moreover''. 
  There are additional restrictions on the virtually abelian subgroups. For instance, there exists $N\geq 1$ such that,
if $hgh\m=g\m$ for some
$g$ of infinite order, then $hg'^Nh\m=g'^{-N}$ for all $g'$ of infinite order in $M(g)$.
\end{rem}

\begin{proof}
  The first assertion is immediate from Proposition \ref{prop_univ}.

 Now let $H$ be an infinite subgroup of $G$ not containing $\F_2$.
By \cite[Proposition 3.2]{Koubi_croissance}, there exists a number $M$ such that, 
if $x_1,\dots, x_M$ are distinct elements of $\Gamma$, some element of the form $x_i$ or $x_ix_j$ has infinite order (\ie order $>K$). 
This universal statement also holds in $G$, so $H$ contains an element $g$ of infinite order. 
Recall that there exists a number $N$ such that, if  $x,y\in\Gamma$, 
then $x^N$ and $y^N$ commute or generate $\F_2$ (see \cite{Delzant_sous-groupes}). 
The same statement holds in $G$ since,
for each non-trivial word $w$, the universal statement $ 
 [x^N,y^N]\neq 1\Rightarrow w(x^N,y^N)\neq 1$
holds in $\Gamma$ hence in $G$.
Thus,  for all 
$h\in H$, the elements $g^N$ and $hg^Nh\m $ commute.
By Lemma \ref{rem_KCSA},
$H$ normalizes $M(g)$, so $H\subset M(g)$ and 
$H$ is $K$-virtually abelian.
\end{proof}

Let $G$ be a $K$-CSA group.
We now show how to define a tree of cylinders for virtually abelian splittings of $G$  (hence also for virtually cyclic splittings). 

\begin{dfn} [Virtual commutation]
Let $\cala$ be the family of all virtually abelian  
subgroups of $G$, and $\cale$ the family of  
{infinite}   subgroups in $\cala$.
Given $H,H'\in \cale$, define
$H\sim H'$ if $M(H)=M(H')$. This is an equivalence relation, which we call \emph{virtual commutation}.\index{virtual commutation}
\end{dfn}

Equivalently, $H\sim H'$ if and only if $\grp{H,H'}$ is virtually abelian. The stabilizer $G_{[H]}$ of the equivalence class   of any $H\in \cale$   (for the action of $G$ by conjugation) is   the virtually abelian group $M(H)$.

\begin{lem} \label{coadm}
  If $G$ is one-ended relative to $\calh$, the equivalence relation $\sim$  on $\cala_\infty$  is   admissible   (see Definition \ref{eqrel}).
\end{lem}

\begin{proof}
 By one-endedness, all $(\cala,\calh)$-trees have edge stabilizers in $\cale$. The first two properties of admissibility are obvious. 
Consider $A,B\in\cale$ with $A\sim B$, and an $(\cala,\calh)$-tree $T$ in which $A$ fixes some $a $ and $B$ fixes some $b $.
Since the group generated by two commuting elliptic groups is elliptic,
there are   finite index subgroups $A_0\subset A$ and  $B_0\subset B$ such that
 $\grp{A_0, B_0}$ fixes a point $c\in T$.  
Given  any edge $e$ in the segment $[a,b]$,
it is contained in $[a,c]$ or $ [c,b]$, so, say,   $A_0\inc G_e$ and $G_e\sim A_0\sim A$ as required. 
\end{proof}

\begin{thm} \label{JSJ_KCSA}
Let $G$ be a $K$-CSA group, and  $\calh$   any family of subgroups. 
Assume that $G$ is one-ended relative to $\calh$.
\begin{enumerate}
\item 
There is a  JSJ tree $T_a$   relative to $\calh$ over virtually abelian (resp.\ virtually cyclic) subgroups. 
Its collapsed tree of cylinders $(T_a)_c^*$  (for virtual commutation) is a JSJ tree relative to $\calh$ and   all virtually abelian subgroups which are not virtually cyclic.

\item 
   $T_a$ and $(T_a)_c^*$ have the same non-virtually abelian vertex stabilizers;   flexible vertex stabilizers   which are not virtually abelian 
 are    QH 
with finite fiber.

\item
 $(T_a)_c^*$ is invariant under all automorphisms of $G$ preserving $\calh$.  It is compatible with every $(\cala,\calh)$-tree.

\end{enumerate}
\end{thm}

 \begin{proof}

 We apply Corollary \ref{synthese}, with $\cala$  the family  of all virtually abelian (resp.\ virtually cyclic) subgroups,  $\sim$   virtual commutation, $\cals$ the family of virtually abelian subgroups, and  $C=K$.
\end{proof}

\subsection{Relatively hyperbolic groups} 
\label{sec_relh}

In this subsection we assume that  $G$ is hyperbolic\index{relatively hyperbolic group} relative to  a family of finitely generated subgroups $\calp=\{P_1,\dots, P_p\}$. Recall that a subgroup is \emph{parabolic}\index{parabolic} if it is conjugate to a subgroup of some $P_i$, \emph{elementary}\index{elementary subgroup} if it is 
  virtually cyclic  (possibly finite) or parabolic.   Any infinite elementary subgroup is contained in a unique maximal elementary subgroup. 

 The following lemma is folklore.
\begin{lem}\label{lem_coelem}

  Let $G$ be a relatively hyperbolic group. 
  \begin{enumerate}\item 
    There exists $C>0$ such that any elementary subgroup $A<G$ of
    cardinality $> C$ is contained in unique maximal elementary
    subgroup $E(A)$;  moreover, $E(A)$ is parabolic if $A$ is finite.
  \item If $A<G$ is  virtually cyclic
  but not parabolic, then it is finite of cardinality $\geq C$
    or $C$-virtually cyclic.

  \item 
    If $A,B<G$ are elementary subgroups such that
    $A\cap B$ has cardinality  $> C$, then  $E(A)=E(B)$.

  \end{enumerate}
\end{lem}

\begin{proof}
  The first assertion is contained in \cite[Lemma 3.1]{GL6}. 
If $A<G$ is virtually cyclic but not 
 $C$-virtually cyclic,
its maximal finite normal subgroup $F$ has cardinality $>C$. Thus
  $F$ is parabolic by (1), and so is $A$ since $E(A)=E(F)$.
This proves the second assertion.
The third assertion immediately follows from the first.
\end{proof}

\begin{dfn}[Co-elementary]\index{co-elementary subgroups}
 We say   that  two infinite elementary subgroups $A,B$ are \emph{co-elementary}
if   $\langle A,B\rangle$ is elementary, or equivalently if $E(A)=E(B)$. This is an equivalence relation $\sim$ on the set of infinite elementary subgroups.
  \end{dfn}
  
  We let $\cala$ be either the class of elementary subgroups, or the class of virtually cyclic groups. In both cases, co-elementarity is an equivalence relation on $\cale$.

  \begin{lem} \label{admi} 
Let $\cala$ be the family of elementary subgroups (resp.\ of virtually cyclic subgroups), and let $\calh$ be any family of subgroups. If every $P_i$ is small in $(\cala,\calh)$-trees, the co-elementarity equivalence relation on $\cale$ is  admissible (relative to $\calh$).  
  \end{lem}

\begin{proof} We fix an $(\cala,\calh)$-tree $T$ with infinite edge stabilizers. We assume that $A\sim B$, and that $A,B$ fix $a,b$ respectively in $T$. We must show $G_e\sim A$ for every edge $e\inc[a,b]$. This is clear if $\langle A,B\rangle$ is   not parabolic  (hence is virtually cyclic), since 
then $A\cap B$ is infinite (it has finite index in $A$ and $B$) and contained in $G_e$, so $A\sim (A\cap B)\sim G_e$.
Assume therefore that $\langle A,B\rangle$ is contained in some $P_i$. By assumption $P_i$ is small in $T$. We distinguish several cases.

If $P_i$ fixes a point $c$, the edge $e$ is contained in   $[a,c]$ or $[b,c]$, and $G_e$ contains $A$ or $B$ so is equivalent to $A$ and $B$.  The argument is the same if $P_i$ fixes an end of $T$ (with $c$ at infinity). The last case is when $P_i$ acts dihedrally on a line $L$. Let $a',b'$ be the projections of $a,b$ on $L$. They are fixed by $A$ and $B$ respectively. If $e$ is contained in $[a,a']$ or $[b,b']$, then $G_e$ contains $A$ or $B$, so we may assume $e\inc [a',b']$.
  Now a subgroup $A'\subset A$ of index at most $2$ fixes  $L$ pointwise, so $A'\subset G_e$ and  $G_e\sim A'\sim A$.
\end{proof}

 The lemma allows us to define trees of cylinders. 
 The stabilizer $G_{[A]}$ of the equivalence class   of any $A\in \cale$   (for the action of $G$ by conjugation) is $E(A)$, it is  small in $(\cala,\calh)$-trees if every $P_i$ is.  Note that $P_i$ is small in $(\cala,\calh)$-trees if it does not contain $\F_2$, or is contained in a group of $\calh$.

 When $\cala$ is the class of elementary subgroups, no collapsing is necessary: $(T_c)^*=T_c$. On the other hand, 
if  $\cala$ is the class of
  virtually  cyclic subgroups,  and some of the $P_i$'s are not virtually cyclic, one may have $G_{[A]}\notin \cala$ and $(T_c)^*$ may be  a proper collapse of $T_c$.

\begin{thm} \label{thm_JSJr}
Let $G$ be hyperbolic relative to  a family of finitely generated subgroups $\calp=\{P_1,\dots, P_p\}$,  with no $P_i$ virtually cyclic. 
Let $\cala$ be the class of all elementary subgroups of $G$ (resp.\ of all virtually cyclic subgroups).
Let $\calh$ be any family of subgroups.  

If $G$ is one-ended relative to $\calh$, and every $P_i$  
is small in $(\cala,\calh)$-trees,
then:
\begin{enumerate}
\item 
there is a  JSJ tree $T_a$   relative to $\calh$ over elementary (resp.\ virtually cyclic) subgroups; 
its collapsed tree of cylinders $(T_a)_c^*$  (for co-elementarity) is a JSJ tree relative to   $\calh\cup \calp$.

\item     $T_a$ and $(T_a)_c^*$ have the same non-elementary vertex stabilizers;   flexible\index{flexible vertex, group, stabilizer} vertex stabilizers   which are not elementary
 are    QH with finite fiber;

\item  
$(T_a)_c^*$ is invariant under all automorphisms of $G$ preserving $\calp$ and $\calh$, 
and 
is compatible with every $(\cala,\calh)$-tree.

\end{enumerate}
\end{thm}

When $G$ is hyperbolic and $\calh=\es$, the tree $(T_a)_c^*=(T_a)_c $ is the virtually cyclic JSJ tree constructed by Bowditch \cite{Bo_cut}\index{Bowditch} using the topology of $\partial G$. 

 Removing virtually cyclic groups from $\calp$ does not destroy relative hyperbolicity; the assumption that no $P_i$ is virtually cyclic makes statements simpler and   causes no loss of generality.

\begin{proof}
We apply Corollary \ref{synthese},  with $C$ as in  Lemma \ref{lem_coelem} and $\SS$ the family 
of all elementary subgroups (we use Remark \ref{gdss}  if we work over virtually cyclic groups and some $P_i$ is a torsion group).
All groups in $\SNVC$ are parabolic, and   every $P_i$ is in $\SNVC$.
In particular, a tree is relative to $\SNVC$ if and only if it is relative to $\calp$.
Automorphisms preserving $\calp$ preserve the set of elementary subgroups, so Corollary \ref
{invaut} applies.
\end{proof}

 The assumption that parabolic groups $P_i$ are small in $(\cala,\calp)$-trees is automatic 
   as soon as $\calh$ contains $\calp$, since we consider splittings relative to $\calh$. We therefore get:

\begin{cor} \label{thm_JSJrcourt2}
Let $G$ be hyperbolic relative to a finite family of finitely generated subgroups $\calp=\{P_1,\dots, P_p\}$. Let $\cala$ be the family of elementary subgroups of $G$. Let $\calh$ be any family of subgroups containing $\calp$.

   If $G$ is one-ended relative to $\calh$, there is a  JSJ tree   over $\cala$ relative to $\calh$ which is equal to its tree of cylinders, invariant under automorphisms of $G$ preserving $\calh$, and compatible with every $(\cala,\calh)$-tree. Its non-elementary flexible vertex stabilizers   are
   QH with finite fiber. \qed
\end{cor}

In particular:
 
\begin{cor} \label{thm_JSJrcourt}
Let $G$ be hyperbolic relative to a finite family of finitely generated subgroups $\calp=\{P_1,\dots, P_p\}$. Let $\cala$ be the family of elementary subgroups of $G$.

   If $G$ is one-ended relative to $\calp$, there is a      JSJ tree   over $\cala$ relative to $\calp$ which is equal to its tree of cylinders, invariant under automorphisms of $G$ preserving $\calp$, and compatible with every $(\cala,\calp)$-tree. Its non-elementary flexible vertex stabilizers   are
   QH with finite fiber. \qed
\end{cor}

Note that  $G$ is finitely presented relative to $\calp$, so 
existence of a JSJ tree also follows from Theorem \ref{thm_exist_mou_rel}.

\subsection{Virtually cyclic splittings}\label{sec_VC}
 
In this subsection  we consider splittings of $G$ over virtually cyclic groups,
assuming smallness of their  commensurators.

Let $\cala$ be the family of  virtually cyclic\index{virtually cyclic} (possibly finite) subgroups of $G$, and $\cale$   the family of all  infinite virtually cyclic subgroups. Recall that two subgroups $A$ and $B$ of $G$ are  commensurable if  $A\cap B$ has finite index
in $A$ and $B$. \index{commensurable}
 The commensurability relation $\sim$
is an admissible relation on $\cale$ (see \cite{GL4}), so one can define a tree of cylinders.

The stabilizer $G_{[A]}$ of the equivalence class   of a group $A\in \cale$ is its commensurator $\Comm(A)$,
 consisting of elements $g$ such that $gAg\m$ is commensurable with $A$.

Corollary \ref{synthese} yields:

\begin{thm}\label{thm_VC}  
Let $\cala$ be the family of   virtually cyclic
 subgroups,
and  let $\calh$ be  any set of subgroups of $G$, with $G$   one-ended relative to $\calh$.  Let 
$ \SS$ be the set of subgroups of commensurators of infinite virtually cyclic subgroups.
Assume that  there is a bound $C$ for the order of
finite subgroups of $G$,  
and that   all groups of $\SS$  are small in $(\cala,\calh)$-trees.
Then:

\begin{enumerate}
\item 
  There is a  virtually cyclic JSJ  tree $T_a$   relative to $\calh$.
  Its collapsed tree of cylinders $(T_a)_c^*$ (for commensurability)  is a virtually cyclic JSJ tree relative to $\calh$ and    the groups of $\SS$ which are not virtually cyclic.
  
  \item
   $T_a$ and $(T_a)_c^*$ have the same  vertex stabilizers not in $\SS$;  their flexible subgroups 
    commensurate some infinite virtually cyclic subgroup,
or are   QH  with finite fiber.

\item
$(T_a)_c^*$ is invariant under all automorphisms of $G$ preserving  $\calh$.  It is compatible with every $(\cala,\calh)$-tree.\qed
\end{enumerate} 

\end{thm}

\begin{rem}
  This applies  if $G$ is a
  torsion-free CSA group, or a $K$-CSA group, or any relatively
  hyperbolic group whose finite subgroups have bounded order as long
  as all parabolic subgroups  are small in $(\cala,\calh)$-trees.
 
  If $G$ is $K$-CSA, the   trees of cylinders  of a given $T$ 
   for commutation and
  for commensurability  belong to the same deformation space (this follows from Lemma \ref{rem_KCSA}).  
  
\end{rem}

We also have:  

 \begin{thm} \label{thm_JSJct}
 Let $G$ be torsion-free and commutative transitive.\index{commutative transitive} 
  Let $\calh$ be any family of subgroups.  If $G$ is freely indecomposable relative to $\calh$, then:
 \begin{enumerate}
 \item 
  There is a  cyclic JSJ tree  $T_a$ relative to $\calh$. Its  collapsed tree of cylinders ${(T_a)}_c^*$ (for commensurability) is a JSJ tree   relative to $\calh$ and all subgroups isomorphic to  a  solvable Baumslag-Solitar group\index{Baumslag-Solitar group} $BS(1,s)$. 
 \item $T_a$ and ${(T_a)}_c^*$ have the same non-solvable vertex stabilizers. Their flexible subgroups are QH (they are
surface groups), unless $G=\Z^2$.
\item $(T_a)_c^*$ is invariant under all automorphisms of $G$ preserving  $\calh$.  It is compatible with every $(\cala,\calh)$-tree.
\end{enumerate}
 \end{thm}

 Recall that $G$ is \emph{commutative transitive} if commutation is a transitive  relation on   $G\setminus\{1\}$.
 
 \begin{rem}\label{metab}
 We cannot apply Corollary \ref{synthese} directly, because we cannot claim that $G_{[A]}$, the commensurator of a cyclic subgroup $A=\langle a\rangle$, is a $BS(1,s)$. Note however that $G_{[A]}$ is metabelian: if $g$ commensurates $A$, there is a relation $g\  a^pg\m=a^q$; mapping $g$ to $p/q$ defines a
 map from $G_{[A]}$ to $\Q^*$  whose     kernel is the centralizer of $A$, an abelian group.
\end{rem}

 \begin{proof}
   By Proposition 6.5 of \cite{GL4}, if $T$ is any tree with cyclic edge stabilizers, a vertex
   stabilizer  of its collapsed tree of cylinders    $T^*_c$
   which is not elliptic in $T$ is a  solvable
   Baumslag-Solitar group $BS(1,s)$ (with $s\ne-1$
   because of commutative transitivity).   In particular, $T$ $\SS$-dominates 
   $T^*_c$, with $\SS$ consisting  of
   all groups contained in a $BS(1,s)$ subgroup.

   We may therefore apply     Theorem \ref{thm_smallyacyl},    taking $T^*:=T^*_c$,  and argue as in the proof of Corollary \ref{synthese}, using Lemmas \ref{lem_rel_vs_abs} and Lemma \ref{lraff} (which applies since $G_{[A]}$ cannot contain $\F_2$).
    No
   group in $\cals$ can be flexible,
   except if $G_v=G\simeq \Z^2$, so all flexible groups are QH surface
   groups. 
 \end{proof}

\subsection{The $\Zmax$-JSJ decomposition} \label{zmax}

In this section,  $G$ is a one-ended hyperbolic group. We
consider splittings of $G$ over virtually cyclic subgroups (necessarily infinite), and for simplicity we assume $\calh=\es$.

Theorem \ref{thm_JSJr}
yields  a tree $T=(T_a)^*_c=(T_a) _c$, which is the tree of cylinders of any JSJ tree. This tree is itself a JSJ tree,  it is canonical  (in particular, invariant under automorphisms), and its flexible vertex stabilizers are QH with finite fiber. It is in fact the tree constructed   by Bowditch\index{Bowditch} in \cite{Bo_cut}.

It has been noticed by several authors \cite{Sela_structure,DaGr_isomorphism,DG2} 
that it is sometimes useful to replace $T$ by a slightly different tree whose edge stabilizers are  {maximal} virtually cyclic subgroups with infinite center.

To motivate this, recall that there is a strong connection between splittings and automorphisms. By Paulin's theorem  \cite{Pau_arboreal}  combined with Rips's theory of actions on \Rt s \cite{BF_stable}, $G$ splits over a virtually cyclic subgroup whenever $\Out(G)$ is infinite.

Conversely, suppose $G=A*_CB$, with $C$ virtually cyclic (there is a similar discussion for HNN extensions). If $c$ belongs to the center of $C$, it defines a Dehn twist $\tau_c$: the automorphism of $G$ which is conjugation by $c$ on $A$ and the identity on $B$.

But this does not always imply that $\Out(G)$ is infinite (see \cite{MNS_downunder}). 
There are two reasons for this. First, even though $C$ is infinite by one-endedness, its center  may  be finite, 
for instance if $C$ is infinite dihedral (see Subsection \ref{basic}). Second, if some power of $c$ centralizes $A$ or $B$,
then the image of $\tau_c$ in $\Out(G)$ has finite order.

\index{0Z@$\Zc$: virtually cyclic subgroups with infinite center}
\index{0Zmax@$\Zmax$: maximal virtually cyclic subgroups with infinite center}
We therefore consider  the set $\Zc$ of subgroups $C<G$ 
that are virtually cyclic with infinite center, and the family $\Zmax$ consisting of the maximal elements of $\Zc$ (for inclusion).
It is now true that \emph{$\Out(G)$ is infinite if and only if $G$ splits over a group $C\in\Zmax$}    (\cite{Sela_structure,DG2,Carette_automorphism,GL4}).

We say that a subgroup is a $\Zc$-subgroup or a $\Zmax$-subgroup if it belongs to $\Zc$ or  $\Zmax$.
 For   $C\in\Zc$, we denote by $\hat C$ the unique $\Zmax$-subgroup of $G$ containing $C$
(it is the pointwise stabilizer of the pair of points of $\bo G$ fixed by $C$).

A tree is a $\Zc$-tree, or a $\Zmax$-tree, if its edge stabilizers are in $\Zc$  or  $\Zmax$. 
\begin{dfn}[$\Zmax$-JSJ tree]    A $\Zmax$-tree is a \emph{$\Zmax$-JSJ tree} if  it is   elliptic with respect to every $\Zmax$-tree,
and maximal (for domination) for this property.
\end{dfn}

Beware that $\Zc$  and  $\Zmax$ are not stable under taking subgroups, so this does not fit in our usual setting.

All $\Zmax$-JSJ trees belong to the same deformation space, the \emph{$\Zmax$-JSJ deformation space}.   As in Subsection \ref{jsjdf},  this follows from Proposition \ref{prop_refinement}, but we need to know that, \emph{given $\Zmax$-trees $T_1$, $T_2$, any standard refinement $\hat T_1$ of $T_1$ dominating $T_2$ is a $\Zmax$-tree.} 

To see this, consider an edge $e$ of $\hat T_1$. If its image in $T_1$ is an edge, then $G_e$ is an edge  stabilizer of $T_1$ so is $\Zmax$. If not, $e$ is contained in the preimage $Y_v$ of a vertex $v$ of $T_1$. The group $\hat G_e$ is elliptic in $T_1$ (because it contains $G_e$ with finite index). If it fixes some $w$, the segment between $v$ and $w$ is fixed by $G_e$, hence by $\hat G_e$ because $T_1$ is a $\Zmax$-tree. We deduce that $\hat G_e$ fixes $v$, and therefore leaves $Y_v$ invariant. The fact that $Y_v$ maps injectively into the $\Zmax$-tree  $T_2$ now implies that $\hat G_e$ fixes $e$,  since it fixes every edge in the image of $e$ in $T_2$.

\begin{rem} \label{pointf}
This argument is based on the following useful fact: if a subgroup $H<G$ fixes a vertex $v$ in  a $\Zmax$-tree, and $H$ contains a  group $C\in\Zc$, then $\hat C$ fixes $v$. 
\end{rem}

In this section we shall construct    and describe a canonical  $\Zmax$-JSJ tree.
Because of the relation between $\Zmax$-splittings and  infiniteness of $\Out(G)$   mentioned above, it  is algorithmically computable 
 (for computability of the usual JSJ decomposition, see \cite{Touikan_finding} and \cite{Barrett_JSJ}).

\begin{lem}\label{lem_TZmax}
  Given any $\Zc$-tree $T$, one can construct a $\Zmax$-tree $T_{\Zmax}$ with the following properties:
  \begin{enumerate}
  \item $T$ dominates $T_{\Zmax}$;
  \item every $\Zmax$-tree  $S$ dominated by $T$ is dominated by $T_{\Zmax}$;
  \item every edge stabilizer of $T_{\Zmax}$ has finite index in some edge stabilizer of $T$.
  \end{enumerate}
\end{lem}

\begin{example}
  It may happen that $T_{\Zmax}$ is trivial even though $T$ is not. This occurs for instance for $T$ corresponding to a splitting
of the form $A*_{a=c^k} \grp{c}$.
\end{example}
\begin{proof}
  Let $T'$ be the quotient of $T$ by the smallest equivalence relation such that,
for all edges $e\subset T$ and all $h\in \hat G_e$, we have $h.e\sim e$.
We shall give an alternative description of $T'$, which shows that it is   a $\Zmax$-tree satisfying all the required properties;
however, it may happen that $T'$ is not minimal (it may even be trivial), so we
  define  $T_{\Zmax}$ as the minimal $G$-invariant subtree of $T'$, and the lemma follows.
  
  We construct $T'$ by folding. We argue by induction on the number of $G$-orbits of edges $e$ 
with $G_e\neq \hat G_e$. Let $e$ be such an edge  (if there is none,   $T'=T$ is a $\Zmax$-tree).
Since $[\hat G_e:G_e]<\infty$, the group $\hat G_e$ is elliptic in $T$.
If one of the endpoints of $e$ is fixed by $\hat G_e$, let $T_1$ be the tree obtained by  folding together all the edges in the  $\hat G_f$-orbit of $G_f$, for every edge $f$ in the $G$-orbit of $e$. If not,   let $e'$ be the first edge in the shortest path joining $\Fix \hat G_e$ to $e$.
Since this path is fixed by $G_e$, one has $G_e\subset G_{e'}\subsetneq \hat G_{e}=\hat G_{e'}$,
and one can   fold the edges in the $\hat G_{e'}$-orbit of $e'$.
In both cases we obtain a tree $T_1$ having  fewer orbits of edges with $G_e\neq \hat G_e$, and  any map $T\ra S$ 
factors through $T_1$. The lemma now follows by induction. 
\end{proof}

\begin{rem}
 A similar construction is used in Section 5 of \cite{GL_McCool}: $G$ is a toral relatively hyperbolic group, $T$ is a tree with abelian edge stabilizers, and it is   replaced by a tree whose edge stabilizers  are abelian and stable under taking roots.
\end{rem}

We denote by $\cala$ the family consisting of all subgroups   $C\in\Zc$ and all their finite subgroups (alternatively, one could include all   finite subgroups of $G$).
It is stable under taking subgroups, and
since $G$ is one-ended  $\cala$-trees have edge stabilizers in $\Zc$.
Let $T=(T_a)^*_c$ be the canonical JSJ tree over $\cala$ provided by Theorem \ref{thm_JSJr}. Its flexible vertex stabilizers are QH with finite fiber, and the underlying orbifold has no mirrors (see Theorem \ref
{thm_description_slenderZ}).

Let $T_{\Zmax}$ be the tree associated to $T$ by Lemma \ref{lem_TZmax}. It is a $\Zmax$-tree, which is   elliptic  with respect to every $\Zmax$-tree by Assertion (3) of the lemma. Unfortunately, it is not always a $\Zmax$-JSJ tree. We illustrate  this on the Klein bottle group $K$ (of course $K$ is not hyperbolic, see   hyperbolic examples below).

As pointed out in Subsection \ref{sec_G_small}, the cyclic JSJ decomposition of $K$ is trivial, and $K$ is flexible because it has   two cyclic splittings, corresponding to the presentations 
 $K=\grp{a,t\mid tat\m=a\m}=\grp{t}*_{t^2=v^2}\grp{v}$.   Thus $T$ and $T_{\Zmax}$ are trivial trees. 
  Note, however, that
  the amalgam is not over a $\Zmax$-subgroup (the group generated by $t^2$ is not a maximal cyclic subgroup). It follows that the Bass-Serre tree of the HNN extension is elliptic with respect to every $\Zmax$-tree, and is  a $\Zmax$-JSJ tree.
 
 Geometrically, let $\Sigma$ be a flat Klein bottle, or a compact hyperbolic surface, or more generally a hyperbolic  orbifold as in Theorem \ref
{thm_description_slenderZ} (we allow cone points, but not   mirrors).
An essential   simple closed geodesic $\gamma$ on $\Sigma$  defines a splitting over a cyclic group, which is in $\Zmax$ if and only if $\gamma$ is 2-sided (see Subsection \ref{fmg}).
 The $\Zmax$-JSJ tree is trivial if and only if every essential 2-sided simple closed geodesic crosses (transversely) some other 
 2-sided   geodesic. 
 
 This happens  in almost all  cases, but there are exceptions:
 the flat Klein bottle, 
 the Klein bottle with one conical point, and   the Klein bottle with one open disc removed.   On a flat Klein bottle, all essential 2-sided simple closed geodesics are isotopic. In the other two exceptional cases (which are hyperbolic),    there is a unique essential 2-sided simple closed geodesic.
 
 We can now construct  a canonical\index{canonical tree}
 $\Zmax$-JSJ tree $\hat T_{\Zmax}$.
 If no orbifold $\Sigma_v$ underlying a QH vertex $v$ of the canonical JSJ tree $T$ is an exceptional one,
 we let $\hat T_{\Zmax}$ be the tree obtained by applying Lemma \ref{lem_TZmax} to $T$. 
 Now suppose that there are vertices $v$ with   $\Sigma_v$ a Klein bottle with one conical point (in which case $G_v=G$ because $\bo\Sigma=\es$), or a Klein bottle with one open disc removed.
 We then  refine $T$ at these vertices, 
 using the splitting of $\Sigma_v$ dual to the unique essential 2-sided geodesic, and we   apply the construction of Lemma  
\ref{lem_TZmax} to the tree $\hat T$ thus obtained.

\begin{prop}
  Let $G$ be a one-ended hyperbolic group. The tree $\hat T_{\Zmax}$ constructed   above is a canonical  $\Zmax$-JSJ tree.
\end{prop}

\begin{proof}
 First suppose that $T$ has no exceptional QH vertex.  Then  $\hat T_{\Zmax}= T_{\Zmax}$, and by Assertion (3) of Lemma \ref{lem_TZmax}  it is elliptic  with respect to every $\Zmax$-tree. It is canonical because $T$ is canonical  and the definition of $T_{\Zmax}$ given in the first paragraph of the proof  of Lemma  \ref{lem_TZmax} does not involve choices. We only need to prove maximality:   any vertex stabilizer    of $T_{\Zmax}$ is elliptic in every tree $S$ which  is   elliptic with respect to every $\Zmax$-tree.
 
 Recall that $T$ is a tree of cylinders, so is bipartite (see Subsection 
 \ref{defcyl}).  If $v\in V_1(T)$, its stabilizer in $T_{\Zmax}$ is the same as in $T$ (a maximal virtually cyclic subgroup). It  is elliptic in $S$ because it contains an edge stabilizer  of $T$ with finite index.

 If two edges of $T$ are folded when passing from $T$ to $T_{\Zmax}$, they have commensurable stabilizers, so belong to the same cylinder of $T$ (which is the star of a vertex in $V_1$). This implies that,  if $v\in V_0(T)$, its stabilizer in $T_{\Zmax}$ is
  a multiple amalgam (\ie a tree of groups) $\hat G_v=
  G_v*_{B_1} \hat B_1\dots *_{B_k} \hat B_k 
  $, 
where $B_1,\dots,B_k$ are representatives of
  conjugacy classes of incident edge stabilizers.    The group $G_v$ is clearly elliptic in $S$ if it is rigid,
and also   if it is QH with non-exceptional underlying orbifold 
because no $\Zmax$-splitting of $G_v$ relative to its incident edge group is universally elliptic. 
It follows that $\hat G_v$ is elliptic by Remark \ref{pointf}.

  The argument when there are exceptional QH vertices is similar. Refining $T$ replaces the exceptional QH vertices by QH vertices whose underlying orbifold is a pair of pants or an annulus with a conical point; their stabilizers do not split over a $\Zmax$-subgroup relative to the boundary subgroups.
\end{proof}

\begin{rem}
The proof shows that flexible  vertex groups of $\hat T_{\Zmax}$ are \emph{QH  with sockets}\index{socket} 
(also called  sockets \cite{Sela_structure},  or orbisockets \cite{DG2}),
\ie QH groups with roots added to the boundary subgroups.
More precisely, they are
 of the form $\hat G_v=
  G_v*_{B_1} \hat B_1\dots *_{B_k} \hat B_k 
  $
  with $G_v$ a QH vertex group of $T$ and $B_1,\dots,B_k$   representatives of
the conjugacy classes of boundary subgroups of $G_v$. The incident edge groups of $\hat G_v$ in $\hat T_{\Zmax}$ are the $\hat B_i$'s (some of them may be missing when the tree $T'$ defined in the proof of Lemma \ref{lem_TZmax} is not minimal).
\end{rem}

\begin{rem}
In applications to model theory,
  hyperbolic surfaces $\Sigma$ which
  do not carry a pseudo-Anosov diffeomorphism play a special role.
  There are four of them: the pair of pants, the twice-punctured projective plane, the once-punctured Klein bottle, the closed non-orientable surface of genus 3. The first two have finite mapping class group, but the   other two do not
 and this causes   problems (see the proof of Proposition 5.1 in \cite{GLS_finiteindex}).

  When a once-punctured Klein bottle appears in a QH vertex, one may refine the splitting as explained above, using the unique   essential 2-sided simple closed geodesic $\gamma$ (this creates a QH vertex based on  a pair of pants). The Dehn twist $T_\gamma$ around $\gamma
$ generates a finite index subgroup of the mapping class group of $\Sigma$. It  acts trivially on every vertex group of the refined splitting.

Similarly, if $\Sigma$ is a closed non-orientable surface of genus 3, there is  a unique 1-sided simple geodesic $\gamma$ whose complement is orientable (see  Proposition 2.1 of \cite{GoMa_homeotopy}). It is the core of a M\"obius band whose complement is a once-punctured torus $\Sigma'$ (unlike $\Sigma$, it carries a pseudo-Anosov diffeomorphism).  The mapping class group of $\Sigma$ leaves $\gamma$ invariant  and preserves the cyclic splitting of $\pi_1(\Sigma)$ given by decomposing $\Sigma$ as the union of $\Sigma'$ and a M\"obius band. It is isomorphic to $GL(2,\Z)$, the mapping class group of  $\Sigma'$,

These refinements give  a canonical way of modifying the JSJ decomposition of a torsion-free one-ended hyperbolic group (described in Subsection \ref{acsa})
so that all surfaces appearing in QH vertices have a mapping class group which is   finite or contains a pseudo-Anosov map.
\end{rem}

\part{Compatibility}\label{chap_compat}

As usual, we fix a family $\cala$ of subgroups which is stable under conjugating and   taking subgroups, and another family $\calh$, and we only consider $(\cala,\calh)$-trees.  We work with simplicial trees, but we often view them as metric trees (with every edge of length 1) in order to apply the results from the appendix   (for instance, the fact that compatibility  passes to the limit). 
We will  freely use some concepts from the appendix, in particular   the arithmetic of trees (Subsection  \ref{sec_arith}).

 In    Section \ref{de}
we have defined  a JSJ tree (of $G$ over $\cala$ relative to $\calh$) as a tree which is universally elliptic  and dominates every universally elliptic tree. Its deformation space is the JSJ deformation space $\cald_{JSJ}$. In the next section we  define the compatibility JSJ deformation space $\Dco$ and the compatibility JSJ tree $\Tco$.
   The deformation space $\Dco$ contains a universally compatible tree and dominates every universally compatible tree. The tree $\Tco$ is a preferred universally compatible tree in $\Dco$. In particular, it is invariant under automorphisms of $G$ preserving $\cala$ and $\calh$. In Section \ref{exemp} we give examples, provided in particular by   trees of cylinders  (see Section \ref{sec_cyl}).

 \section{The  compatibility JSJ tree}\label{compt}

Recall (Subsection \ref{sec_morphisms}) that two trees $T_1$ and $T_2$ are compatible if they have a common refinement. In other words, there exists a tree 
$T$ with collapse maps $T\to T_i$.\index{compatible trees}

\begin{dfn}[Universally compatible]\index{universally compatible tree}
A tree $T$ is \emph{universally compatible}  (over $\cala$ relative to $\calh$) if it is compatible with every   tree.    
\end{dfn}

In particular, this means that any tree $T'$ can be obtained from $T$ by refining and collapsing.
When $T'$ is a one-edge splitting, either $T'$ coincides with the splitting associated to one of the edges of $T/G$,
or one can obtain $T'/G$ by refining $T/G$ at some vertex $v$ using a one-edge splitting of $G_v$ relative to the incident edge groups, and collapsing
all the original edges of $T/G$.

\begin{dfn}[Compatibility JSJ deformation space]\index{compatibility JSJ deformation space}
  If, among deformation spaces containing a universally compatible
  tree, there is one which is maximal for domination, it is unique.
  It is denoted by $\Dco$\index{0DC@$\Dco$: compatibility JSJ space} and it is called the \emph{compatibility JSJ
    deformation space} of $G$ over $\cala$ relative to $\calh$.
\end{dfn}

 To prove uniqueness, consider   universally compatible trees $T_1,T_2$. By Corollary \ref{coll} we may assume that they are irreducible. They are compatible with each other, and $T_1\vee T_2$ (see
 Lemma \ref{ppcm} in the appendix)
is universally compatible by
Proposition \ref{prop_lcm} (2).
If $T_1$ and $T_2$ belong to maximal deformation spaces, 
we get that $T_1,T_1\vee T_2$, and $T_2$ lie in the same deformation space,    proving uniqueness.

 Clearly, a universally compatible tree is universally elliptic.
This implies that $\Dco$ is dominated by $\cald_{JSJ}$. Also note that, if $T$ is universally compatible and $J$ is an  edge stabilizer in an arbitrary tree, then $J$ is elliptic in $T$ (i.e.\ any tree is elliptic with respect to $T$).

\subsection{Existence of the compatibility JSJ space}\label{exdur}

\begin{thm}\label{thm_exists_compat}
If $G$ is finitely presented relative to a family $\calh=\{H_1,\dots,H_p\}$ of finitely generated subgroups, the  {compatibility JSJ space} $\Dco$ of $G$ (over $\cala$ relative to $\calh$) exists. 
\end{thm}

The heart of the proof of Theorem \ref{thm_exists_compat} is the following proposition.

\begin{prop}\label{prop_lim_compat}   Let $G$ be finitely presented relative to a family $\calh=\{H_1,\dots,H_p\}$ of finitely generated subgroups.
  Let $T_0\leftarrow T_1\dots \leftarrow T_k \leftarrow\cdots$ be a sequence of   refinements of
irreducible universally compatible  trees.  There exist collapses $\ol T_k$ of
$T_k$, in the same deformation space as
$T_k$, such that the sequence $\ol T_k$ converges to a universally compatible 
  simplicial $(\cala,\calh)$-tree
$T$ which dominates every $T_k$.
\end{prop} 

We view trees as metric,  with each edge of length 1, and convergence is in the space of \Rt s (unlike the proof of Theorem \ref{thm_JSJacyl}, no  rescaling  of the metric is necessary here).   The tree $T$ may have redundant vertices.  

\begin{proof}[Proof of Theorem \ref{thm_exists_compat} from the proposition]   
We may assume that there is a non-trivial  universally compatible  tree. We may also assume that all such trees are irreducible: otherwise it follows from Corollary \ref{coll} (see the appendix) that there is only one deformation space of   trees, and the theorem is trivially true.

Let  $(S_\alpha )_{\alpha \in
A}$ be the  set of   
(isomorphism
 classes of) universally compatible   trees. We have to find a
universally compatible  tree $T$ which dominates every
$S_\alpha $. By
Lemma \ref{lem_Zorn}, we only need $T$ to dominate all trees in a countable set   $S_k$, $k\in  \N$. 
 
 Let  $T_k$ be  the lcm $S_0\vee\dots\vee S_k$ (see Definition \ref{lcm});
 it is universally compatible by Assertion (2)
of Proposition \ref{prop_lcm},  and refines $T_{k-1}$.
Proposition \ref{prop_lim_compat} yields the desired tree $T$: it dominates every $T_k$, hence every $S_k$,
and it is universally compatible by Corollary \ref{cofer}.
\end{proof}

\begin{proof}[Proof of Proposition \ref{prop_lim_compat}]
By    Dunwoody's accessibility (see 
Proposition \ref{prop_accessibility_rel}), there exists a tree $S$  
which dominates every $T_k$ (this is where we use finite presentability). But of course we cannot claim that it is  universally
compatible.

 We may assume that $T_k$ and $S$ are minimal, that $T_{k+1}$ is different from $T_k$, and that the gcd $S\wedge T_k$ (see Definition \ref{gcd}) is independent of
$k$. We define $S_k=S\vee T_k$;  it has no redundant vertices  (see Remark \ref{redun}). We denote by $\Delta_k,\Gamma,\Gamma_k$ 
the
quotient graphs of groups of
$T_k ,S ,S  _k$, and we let $\pi _k:\Gamma _k\to \Gamma $ be the collapse map  (see Figure \ref{fig_raff}).
We
denote by $\rho _k$ the collapse map $\Gamma _{k+1}\ra \Gamma _k$.

\begin{figure}[htbp]
 \centering
 \includegraphics{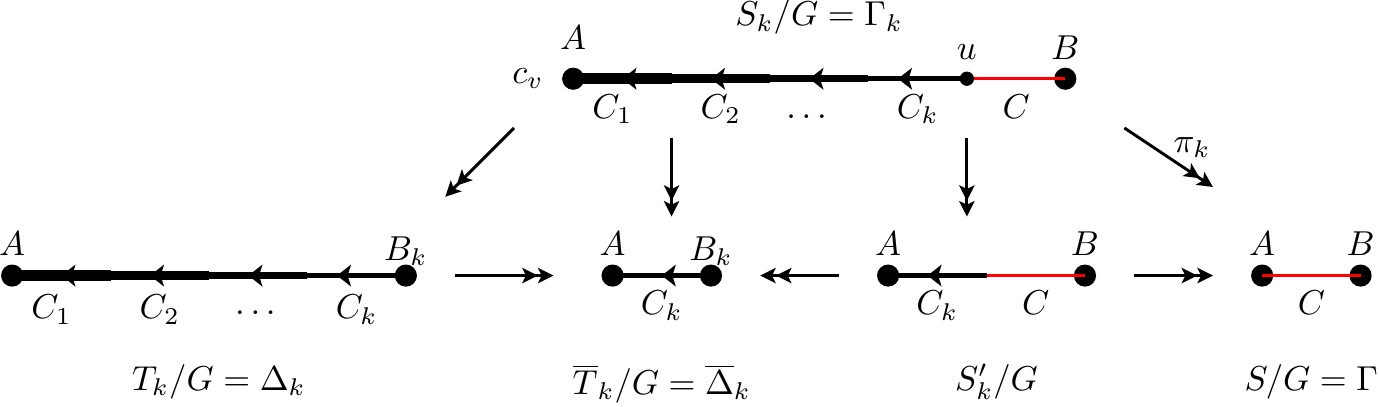}
 \caption{Segments in the quotients of the trees  $T_k$, $S$, $S_k$, $\ol T_k$, $S'_k$, with $C_1\supset C_2\supset \dots\supset C_k \supset C$,
and $B_k=B*_C C_k$.}
 \label{fig_raff}
\end{figure}

The trees $S_k$ all belong to the deformation space of $S$, and
$S_{k+1}$ strictly refines $S_k$. In particular, the number of edges of   $ \Gamma_
k $ grows.  The idea now is the following. Accessibility
holds within a given deformation space (see \cite{GL2} page 147; this is an easy form of accessibility, which requires no smallness or finite presentability hypothesis,  and in any case is not used  directly here). This implies that  the growth of $\Gamma_k$ occurs through the creation of a bounded number of
long segments whose interior  vertices   have  valence 2, with one of the incident edge groups equal to
the vertex group (but the other edge group is smaller than the vertex group, since $S_k$ has no redundant vertices). We now make this precise.

Fix $k$.
For each vertex $v\in \Gamma $, define $Y_v=\pi_k\m(\{v\})\subset
\Gamma _k$. The $Y_v$'s are disjoint, and edges of $\Gamma _k$ not contained in $ \cup_v Y_v$
correspond to edges of $\Gamma $.

Since $S_k$ and $S$ are in the same deformation
space, 
  $Y_v$ is a tree of groups, 
and  it contains a vertex $c_v$
whose vertex group equals the fundamental group of $Y_v$ (which is the vertex group of $v$ in $\Gamma$). This $c_v$ may fail to
be unique, but we can choose one for every $k$  in a way which is compatible with the   maps
$\rho _k$. We orient edges of $Y_v$ towards $c_v$. The group carried by
such an edge is then equal to the group carried by its initial vertex. 

Say that a vertex $u\in Y_v$ is    \emph{peripheral} if $u=c_v$ or $u$   is  adjacent to an  edge
of
$\Gamma _k$ which is not in $Y_v$ (\ie is mapped onto an edge of $\Gamma $  by $\pi _k$).  By
minimality of
$S_k$, each terminal vertex
$u_0$ of
$Y_v$ is peripheral (because it carries the same group as the initial edge of the segment
$u_0c_v$).  

In each $\Gamma _k$, the total  number of peripheral vertices is bounded by $2|E(\Gamma) |+ | V(\Gamma) | $. 
It follows that the number of points of valence $\ne2$ in $\cup_v Y_v$ is   bounded.
Cutting each $Y_v$   at its peripheral vertices 
and   its points of valence $\ge3$ produces the \emph{segments} of $\Gamma _k$ mentioned
earlier.   On the example of Figure \ref{fig_raff}, there is one segment $c_vu$ in $\Gamma_k$, 
corresponding to the edges labelled $C_1,\dots,C_k$.
The point $c_v\in \Gamma_k$ is the vertex labelled by $A$.  
The vertex $v$ of $\Gamma$ to which the segment corresponds is the vertex of $\Gamma$ labelled $A$.
The vertex $u$ is peripheral.
 
Having defined segments for each $k$, we now let $k$ vary. 
The preimage of a segment of $\Gamma _k$ under the map $\rho 
_k$ is a union of segments of $\Gamma _{k+1}$. Since the number of segments is bounded
independently of $k$, we may assume that $\rho _k$ maps every segment of $\Gamma _{k+1}$ onto
a segment of $\Gamma _k$. In particular, the number of segments is independent of $k$.

 Recall that we have oriented the edges of $Y_v$ towards
  $c_v$. Each edge  contained in   $\cup _v Y_v$ carries the same group as its initial
vertex, and edges in a given segment are coherently oriented. Segments are therefore oriented.

There are various ways of performing collapses on $\Gamma _k$. Collapsing all edges contained
in segments yields $\Gamma $ (this does not change the deformation space). On the other hand, one
obtains $\Delta_k=T_k/G$ from $\Gamma _k$ by collapsing some of the edges which are not contained in any
segment (all of them if $S\wedge T_k$ is   trivial).

The segments of $\Gamma _k$ may be viewed as segments in   
$\Delta_k$, but collapsing
the initial edge of a segment of $\Delta_k$ may now change the deformation space (if the
group  carried by the initial  point  of the segment  has increased when $\Gamma_k$ is collapsed to $\Delta_k$).

We define a graph of groups $\ol\Delta_k$ by collapsing, in each segment of
$\Delta_k$, all edges but the initial one. The corresponding tree $\ol T _k$ is a collapse of $T_k$ which
belongs to the same deformation space as $T_k$. Moreover, the number of edges of $\ol\Delta 
_k$ (prime factors of $\ol T _k$)  is constant: there is one per segment, and one for each common prime factor of $T_k$ and $S$.
 
  Let $\ell_k:G\to \Z$ be the length function of $\ol T _k$.

  \begin{lem}
The sequence $\ell_k$  is non-decreasing  (\ie every sequence $\ell_k(g)$ is non-decreasing) and
converges. 
\end{lem} 

\begin{proof}

The difference between
$\ell_k
$ and $\ell_{k -1}$ comes from the fact that initial edges of segments of $\Delta  _{k }$ may be
collapsed in $\Delta  _{k -1}$. Fix a segment $L$ of $\Delta _k$. Let $e_k$ be its initial
edge. We assume that $e_k$ is distinct from  the
 edge $f_k$  mapping onto the initial edge of the image of $L$ in $ \Delta
_{k -1}$.  

Assume for
simplicity that $e_k$ and $f_k$ are adjacent (the general case is similar). The group carried by
$f_k$ is equal to the group carried by its initial vertex $v_k$. A given lift $\tilde v_k$  of
$v_k$ to
$  T _k$  is therefore adjacent to only  one lift of $f_k$ (but to several lifts of $e_k$).
On any translation axis   in $ T _k$, every occurrence of a lift of $f_k$ is immediately preceded by
an occurrence of a lift of $e_k$.  The length function of the prime factor of $  T _k$ and $\ol T_{k-1}$ corresponding
to
$f_k$ is therefore bounded from above by that of the prime factor of $  T _k$ and $\ol T_{k}$ corresponding to $e_k$. Since this is true for every
segment, we get $\ell_{k -1}\le \ell_k$ as required.

Let $S'_k=S\vee\ol T _k$. It   collapses to  $S$, belongs to the same deformation space as $S$
(because it is a collapse of $S_k$), and the number of edges of $S'_k/G$ is bounded. By an
observation due to Forester (see \cite[p.\ 169]{GL2}), this implies an inequality 
$ \ell(S'_k)\le C \ell (S)$, with $C$ independent of $k$. Since $\ell_k\le \ell(S'_k)$, we get convergence.
\end{proof}

We call $\ell$ the limit of $\ell_k$. It is the length function of a tree $T$ because the set of length
functions of trees is closed  \cite 
{CuMo}.   This tree is simplicial
  because
$\ell$ takes values in $\Z$  (see Example \ref{arbsimp}),
and   irreducible because $\ell_k$ is  non-decreasing. 
It  is universally compatible as a limit of universally compatible trees, by  Corollary \ref{cofer}.
Since $\ell\geq \ell_k$, every $g\in G$ elliptic in $T$ is elliptic in $T_k$,
and $T$ dominates $T_k$ by 
Lemma \ref{cor_Zor}. Each $H_i$ is elliptic in $T$ because it is finitely generated and its elements are elliptic.

There remains to prove that every edge stabilizer  $G_e$ of $T$ belongs to $\cala$. If $G_e$ is
finitely generated, there is a simple argument using the  equivariant Gromov topology. In general,
we argue as follows. We may find hyperbolic
elements $g,h$ such that $G_e$ is the stabilizer of the bridge between $A(g)$ and $A(h)$ 
(the bridge might be $e\cup e'$  as in the proof of Lemma \ref{po} if an endpoint of $e$ is a valence 2 vertex). Choose $k$   so
that the values of $\ell_k$ and $\ell$ coincide on $g,h,gh$. In particular, the axes of $g$ and $h$ in
$\ol T_k$ are disjoint.

Any $s\in G_e$ is elliptic in $\ol T_k$ since $\ell_k \leq \ell $.
Moreover, $\ell_k(gs)\leq \ell(gs)\le\ell(g)=\ell_k(g)$. The fixed point set of $s$ in $\ol T_k$ must
  therefore
intersect the axis of
$g$, since otherwise $\ell_k(gs)>\ell_k(g)$. Similarly, it intersects the axis of $h$.
It follows that $G_e$ fixes the bridge between the axes of $g$ and $h$ in $\ol T_k$, so $G_e\in\cala$.
  This concludes the proof of Proposition \ref{prop_lim_compat}.
\end{proof}

\subsection{The compatibility JSJ tree $\Tco$}\label{adur}

 We  shall now deduce from \cite{GL2}   that  $\Dco$, if irreducible,    contains a canonical\index{canonical tree}
 tree $\Tco$, which we call the \emph{compatibility JSJ tree.} 
It is fixed under any automorphism of $G$ that leaves $\Dco$ invariant.
Note that  $\tco$  may be refined to a JSJ tree (Lemma \ref{lem_rafin}).

\begin{lem}  An irreducible\index{irreducible tree, deformation space}  deformation space 
$\cald$ can only contain  finitely many reduced
 universally compatible  trees. 
\end{lem}

  Recall (see Subsection \ref{comp}) that $T$ is \emph{reduced}\index{reduced tree}  if no proper collapse of $T$ lies in the same deformation space as $T$.   If $T$ is not reduced, one may perform collapses so as to obtain a reduced tree $T'$ in the same deformation space (and $T'$ is
universally   compatible if $T$ is).

\begin{proof} This follows from results in \cite{GL2}. We refer to \cite{GL2} for definitions  not given here. Suppose that there are infinitely many reduced  universally compatible trees $T_1,T_2,\dots$. Let $S_k=T_1\vee T_2\vee\dots\vee T_k$. It is an $(\cala,\calh)$-tree and belongs to $\cald $ by Assertion (3) of Proposition \ref{prop_lcm} . 

As pointed out on page 172 of \cite{GL2}, the tree $S_k$ is BF-reduced, \ie reduced in the sense of \cite{BF_bounding}, because all its edges are surviving edges (they survive in one of the $T_i$'s),  and the space $\cald$ is non-ascending 
by Assertion (4) of Proposition 7.1 of \cite{GL2}. 
 Since  there is a bound $C_\cald$ for the number of orbits of edges of a BF-reduced tree in  $\cald$ 
(\cite[Proposition 4.2]{GL2}), 
  the sequence $S_k$ is eventually constant. 
\end{proof}      

\begin{rem}  The proof shows that $\cald$   only contains  finitely many reduced
  trees which are compatible with every tree in $\cald$.
\end{rem}

\begin{cor} \label{corpref}
 If $\cald$ is irreducible and contains a  universally compatible tree, it   has a \emph{preferred element:} the lcm\index{lcm of pairwise compatible trees} of its reduced universally compatible trees.  
\qed 
\end{cor}

 This preferred element is universally compatible by Assertion (2) of Proposition \ref{prop_lcm}.

\begin{dfn} [Compatibility JSJ tree $\Tco$] \label{deftco}
\index{compatibility JSJ tree} 
If the compatibility JSJ deformation space $\Dco$ exists and is irreducible, its preferred element is called  the \emph{compatibility JSJ tree $\Tco$  of $G$ (over $\cala$ relative to $\calh$)}.\index{0TC@$\Tco$: compatibility JSJ tree} If $\Dco$ is trivial, we define $\Tco$ as the trivial tree (a point). 
\end{dfn}

 It may happen that $\Dco$ is neither trivial nor irreducible. It   then follows from Remark \ref{coll} that it is the   only  non-trivial   deformation space of  trees. If there is a unique reduced tree $T$ in $\Dco$ (in particular, if 
$\Dco$   consists of actions on a line), we define $\Tco=T$. Otherwise we do not define $\Tco$. See Subsection \ref{gbsd} for an example where $\Dco$ consists of trees with exactly one fixed end.

\section{Examples}\label{exemp}

We start with various examples,  and   we then explain in Subsection \ref{ac} that the tree of cylinders  $(T_a)_c^*$ of Section 
\ref{sec_cyl} belongs to the compatibility deformation space if all groups $G_{[A]}$ belong to $\cala$. 

 For simplicity we assume   $\calh=\es$  in Subsections \ref{freeg} through \ref{sec_PD}.

\subsection{Free groups}  \label{freeg}
When $\cala$ is $\Aut(G)$-invariant, the compatibility JSJ tree $\tco$ is
$\Out(G)$-invariant. This sometimes forces it to be trivial. Suppose for instance that
$G$ has a finite generating set $a_i$ such that all elements $a_i$ and $a_ia_j^{\pm1}$
($i\ne j$) belong to the same $\Aut(G)$-orbit. Then the only $\Out(G)$-invariant length
function $\ell$ is the trivial one. This follows from Serre's lemma (see Subsection \ref{tre})
 if the generators
are elliptic,   
from the inequality $\max(\ell(a_ia_j),
\ell(a_ia_j^{-1}))\ge\ell(a_i)+\ell(a_j)$ (see Lemma \ref{calcul}) if they are hyperbolic. In particular:

\begin{prop} If $G$ is  a free group and $\cala$ is $\Aut(G)$-invariant, then $\tco$ is trivial. \qed
\end{prop}

\subsection{Algebraic rigidity}

 The following result   provides simple examples with
$\tco$ non-trivial.

\begin{prop} \label{prop_rig}
Assume that there is only one reduced JSJ tree  $T_J\in\cald_{JSJ}$, and that
$G$ does not split over a subgroup contained with infinite index in a group of $\cala$.
Then $\tco$ exists and equals $T_J$. 
\end{prop}

\begin{proof}   Let $T$ be any $\cala$-tree. The second assumption implies that
$T$ is elliptic with respect to  $T_J$ by  Remark 2.3 of \cite{FuPa_JSJ} or  Assertion (3) of Lemma \ref{cor_Zor},
so  we can consider a standard refinement 
$\hat T$ of $T$   dominating $T_J$  as in  Proposition \ref 
{prop_refinement}.
Any equivariant map $f:\hat T\ra T_J$ must
 be constant on any edge whose stabilizer is not universally elliptic,
hence factors through the tree $T'$ obtained from $\hat T$ by collapsing these edges.
In particular, $T'$ dominates $T_J$ hence is a JSJ
tree because it is universally elliptic. Since $T_J$ is the unique reduced JSJ tree, $T'$
is a refinement of $T_J$, so $T_J$ is compatible with $T$. This shows that $T_J$ is universally compatible.  Thus $\tco=T_J$. 
\end{proof}

A necessary and sufficient condition for a tree to be the unique reduced
tree in its deformation space is given in \cite{Lev_rigid} (see also \cite{ClFo_Whitehead}). 

The proposition applies for instance  to free splittings  and splittings over finite groups,
whenever there is a JSJ tree with only one orbit of edges. This provides examples of
virtually free groups with $\tco$ non-trivial:  any amalgam $F_1*_F F_2$ with $F_1, F_2$ finite and $F\ne F_1, F_2$ has this property (with $\cala$ the set of finite subgroups).

\subsection{Free products}
\index{free splitting}
 Let $\cala$ consist only of the trivial group. Let
$G=G_1*\dots*G_p*\F _q$ be a Grushko decomposition\index{Grushko decomposition, deformation space} ($G_i$ is non-trivial, not $\Z$, and freely indecomposable,  $\F_q$ is free of rank $q$). 
 If
$p=2$ and $q=0$, or $p=q=1$,   there is a JSJ tree with   one
orbit of edges and $\tco $ is a
one-edge splitting as explained above. 
We now show that \emph{$\tco$ is trivial if $p+q\ge3$ } (of course it is
trivial also if   $G$ is freely indecomposable or free of rank $\ge2$).

Assuming $p+q\ge3$, we actually show that there is no non-trivial 
tree $T$ with trivial edge stabilizers which is invariant under a finite index subgroup of
$\Out(G)$. By collapsing edges, we may assume that
$T$   only has one orbit of edges. Since $p+q\ge3$, we can write $G=A*B*C$ where
$A,B,C$ are non-trivial and $A*B$ is a vertex   stabilizer of $T$. Given a non-trivial $c\in
C$ and $n\ne 0$, the subgroup $c^nAc^{-n}*B$ is the image of $A*B$ by an automorphism
but is not conjugate to $A*B$. This contradicts the invariance of $T$. 

\subsection{(Generalized) Baumslag-Solitar groups}\label{gbsd}\index{generalized Baumslag-Solitar group}
We consider cyclic splittings of generalized Baumslag-Solitar groups\index{generalized Baumslag-Solitar group} (see Subsection  \ref{gbs}).

First consider a solvable  Baumslag-Solitar group\index{Baumslag-Solitar group}   $BS(1,s)$, with $ | s | \ge2$. In this case $\Dco$ is trivial if $s$ is not a prime power. If $s$ is  a prime power, $\Dco$ is the JSJ deformation space (it is   not irreducible).

When   $G=BS(r,s)$ with none of $r,s$ dividing the
other, Proposition \ref{prop_rig} applies   by \cite{Lev_rigid}. 
In particular, $\Dco$ is non-trivial. This holds, more generally, when $G$ is a generalized Baumslag-Solitar group defined by a labelled graph with no   label dividing another label at the same vertex. See \cite{Beeker_compatibility} for a    systematic study of $\Dco$ for generalized Baumslag-Solitar groups.

\subsection{The canonical decomposition of Scott  and Swarup}
\label{scoswa}

\index{Scott-Swarup}\index{canonical tree} Recall that a group is $\VPC_{n}$\index{VPC} (resp.\ $\VPC_{\le n}$)  if it is 
virtually polycyclic of Hirsch length $n$  (resp.\  $ {\le n}$).
 Let $G$ be a finitely presented group, and $n\ge1$. Assume that $G$ does not split over a
    $\VPC_{n-1}$ subgroup, and that $G$ is not $\VPC_{n+1}$.  Let $\cala $ consist of all subgroups of $\VPC_n$ subgroups.
    
    We have shown in \cite{GL5} that the tree of cylinders (for commensurability) of the JSJ deformation space is (up to subdivision) 
    the  Bass-Serre tree $T_{SS}$ of the regular neighbourhood\index{regular neighbourhood} $\Gamma_n=\Gamma(\calf_n:G)$ constructed by Scott-Swarup in Theorem 12.3 of \cite{ScSw_regular+errata}.

    Since $T_{SS}$ is universally compatible 
(\cite[Definition 6.1(1)]{ScSw_regular+errata}, or \cite[Corollary 8.4]{GL4} and \cite[Theorem 4.1]{GL5}), it is dominated by the compatibility deformation space $\Dco$.
        The domination may be strict: if $G=BS(r,s)$, the tree $T_{SS}$ is always trivial but,  as pointed out above, $\Dco$ is non-trivial when none of $r,s$ divides the
other.

\subsection{Poincar\'e duality groups}\label{sec_PD}
 Let  $G$ be a Poincar\'e duality group of dimension $n$ (see also work by Kropholler on this subject \cite{Kro_JSJ}).
Although such a group is not necessarily finitely presented, it is almost finitely presented   \cite[Proposition  1.1]{Wall_PoincareGT},
which is sufficient for Dunwoody's accessiblity, so   the JSJ  deformation space  and the compatibility JSJ deformation space exist.
By \cite[Theorem A]{KroRol_splittings3}, if $G$ splits over a virtually solvable subgroup $H$, then $H$
is 
$\VPC_{n-1}$. We therefore  consider the family $\cala$ consisting of 
$\VPC_{\leq n-1}$-subgroups.

By \cite[Corollary 4.3]{KroRol_relative}, for all  $\VPC_{n-1}$ subgroups $H$, the number of coends $\Tilde e(G,H)$ is $2$.
By \cite[Theorem 1.3]{KroRol_relative}, if $G$ is not virtually polycyclic,
then $H$ has finite index in its commensurator. 
 By Corollary 8.4(2) of \cite{GL4},
this implies that   
the JSJ deformation space 
contains a universally compatible tree (namely its tree of cylinders for commensurability),
 so equals $\Dco$.

But one has more in this context:
any universally elliptic tree is universally compatible.
Indeed, since $\VPC_{n-1}$-subgroups of $G$ have precisely $2$ coends, Proposition 7.4 of
\cite{ScSw_regular+errata} implies that any two one-edge splittings $T_1,T_2$ of $G$ over $\cala$ 
with edge stabilizers of $T_1$ elliptic in $T_2$ are compatible.
Indeed, strong crossing of almost invariant subsets corresponding to $T_1$ and $T_2$
occurs if and only  edge stabilizers of $T_1$ are not elliptic in $T_2$ (\cite[Lemme 11.3]{Gui_coeur}),
and the absence of (weak or strong) crossing is equivalent to compatibility of $T_1$ and $T_2$ \cite{ScSw_splittings}.
To sum up, we have:

\begin{cor}
  Let $G$ be a Poincar\'e duality group of dimension $n$, with $G$  not virtually polycyclic. Let $\cala$ the family of 
$\VPC_{\leq n-1}$-subgroups. Then $\tco$ exists  
 and lies in the JSJ deformation space 
of $G$ over $\cala$. \qed
\end{cor}

In particular, $G$ has a canonical JSJ tree over $\cala$.

\subsection{Trees of cylinders} \label{ac}

In Sections \ref{jsjac} and \ref{exam} we have used trees of cylinders\index{tree of cylinders} to construct a universally compatible tree $(T_a)^*_c$ (see  the last assertion of Corollary \ref{synthese}); we always denote by $T_a$ a JSJ tree over $\cala$ relative to $\calh$,  as in Lemma \ref{lem_rel_vs_abs}. We now show that $(T_a)^*_c$ belongs to the compatibility JSJ deformation space under the additional assumption that  every stabilizer $G_{[A]}$ belongs to $\cala$  (this implies that trees of cylinders have edge stabilizers in $\cala$, so   (see Subsection \ref{defcyl}) collapsing is not needed:   $T_c=T_c^*$).

  \begin{thm} \label{ouf3}
Given $\cala$ and $\calh$, let 
 $\sim$ be an admissible equivalence on $\cale$.  Assume that $G$ is one-ended relative to $\calh$, and  there exists an integer $C$ such that:

 \begin{enumerate}
\item  $\cala$ 
contains all $C$-virtually cyclic subgroups, 
and all subgroups of cardinal $\leq 2C$;

\item if two groups of $\cale$ are inequivalent, their   intersection has  order $\le C$;
\item every stabilizer $G_{[A]}$ belongs to $\cala$, but $G\notin\cala$;
\item  every  $G_{[A]}$ is small in $(\cala,\calh)$-trees.
\end{enumerate}
Then  the compatibility JSJ deformation space  $\Dco$  exists and contains $(T_a)_c $, the   
tree of cylinders of JSJ trees (see Corollary \ref{synthese}). 
It is trivial or irreducible, so the JSJ compatibility tree $\Tco$ is defined.     Flexible vertex stabilizers   of $\Tco$ belong to $\cala$ or are   QH subgroups with finite fiber.

  \end{thm}

\begin{proof}  
  Corollary \ref{synthese}, applied with $\SS=\cala$, shows that $(T_a)_c$ is universally compatible  and its flexible vertex stabilizers are in $\cala$ or QH. 
 The point is to    show that $(T_a)_c$
 is maximal (for domination) among universally compatible trees. 
 This will prove that $\Dco$  exists and contains $(T_a)_c$. It is  trivial or irreducible because 
 any tree of cylinders is trivial or irreducible (see Subsection \ref{defcyl}).

We consider a  universally compatible tree $T$, and we show that $(T_a)_c$ dominates $T$. 
Replacing  $T$ by  $T\vee (T_a)_c$ (which is universally compatible by  Assertion (2) of Proposition \ref{prop_lcm}),   
 we can assume that   $T$ refines $(T_a)_c$.

We have to show that each vertex stabilizer $G_v$ of $(T_a)_c$ 
is elliptic in $T$.
If $G_v$ is not a  $G_{[A]}$, then it is elliptic in $T_a$, hence in $T$  because $T$,  being universally elliptic, is dominated by $T_a$.
      By  Assumptions (3) and (4), we can therefore assume  that $G_v\in\cala$, and also that $G_v$   is small in $T$.

  Since $G\notin\cala$, the quotient graph $(T_a)_c/G$ is not a point (equivalently, there are at least two cylinders in $T_a$). 
There are two cases. If the image of $v$ in   $(T_a)_c/G$ has valence at least 2,  we can refine $(T_a)_c$ to a minimal tree $T'$ (in the same deformation space) 
having $G_v$ as an edge stabilizer. 
Since $G_v\in\cala$, this is an $(\cala,\calh)$-tree. 
  Its edge group $G_v$ is elliptic in $T$ because $T$ is universally compatible. 

The remaining case is when  the image 
 of $v$ in $(T_a)_c/G$ has valence 1. 
  We assume   that $G_v$ is not elliptic in $T$, and we argue towards a contradiction. 

 Let  $e$ be an edge of $(T_a)_c$ containing $v$.
We are going to prove that $G_v$ contains a 
 subgroup $G_0$ of index 2
 with  $G_{e}\inc G_0$.
Assuming this fact, we can refine $(T_a)_c$ to a minimal   $(\cala,\calh)$-tree $T'$ in which $G_0$ is an edge stabilizer.
 As above, $G_0$ is elliptic in $T$, and so is $G_v$. This is the required contradiction,   proving that $(T_a)_c$ dominates $T$.

 We now construct $G_0$.
Since we have assumed that $T$ refines $(T_a)_c$, 
 Proposition \ref{arbtf} and Lemma \ref{relfg} 
imply that $G_v$ contains a hyperbolic element.
 We know that 
  $G_v$
is small in $T$, so there are only two possibilities. 

If there is a fixed end,   the    action 
 defines a homomorphism $\chi:G_v\ra \bbZ$ (see Subsection \ref{tre}). This homomorphism vanishes on $G_e$, which is elliptic in $T$, but   is     non-trivial  because there is a hyperbolic element, so we define $G_0$ as the preimage of the index 2 subgroup of the image of $\varphi$.
 
  If the action is dihedral, we get an epimorphism $\chi:G_v\ra \bbZ/2*\bbZ/2$. 
Since $G_{e}$ is elliptic in $T$, its image under $\chi$ is trivial or contained in a conjugate of a $\bbZ/2$ factor. One constructs $G_0$ as the preimage of a suitable index 2 subgroup of the image of $\chi$.
\end{proof}

This theorem applies directly to abelian splittings\index{abelian tree} of CSA\index{CSA group} groups (Subsection \ref{acsa}), virtually abelian splittings of $K$-CSA groups (Subsection \ref{aKcsa}), elementary splittings of relatively hyperbolic groups\index{relatively hyperbolic group} (Subsection \ref{sec_relh}), because the condition $G_{[A]}\in\cala$ is satisfied. In each of these cases we conclude that $(T_a)^*_c=(T_a)_c$ belongs to $\Dco$. 
  We get for instance:

\begin{cor}
  Let $G$ be a finitely generated one-ended torsion-free CSA group, $\cala$ the class of abelian subgroups, and $\calh$ a family of subgroups.
Let $\sim$ be the commutation relation among infinite abelian subgroups.

Then the compatibility JSJ deformation space $\Dco$ exists and contains the tree of cylinders of any JSJ tree.\qed
\end{cor}

\begin{cor}
 
  Let $G$ be hyperbolic relative to  a family of finitely generated subgroups $\calp=\{P_1,\dots, P_p\}$.   
Let  $\cala$ be the class of elementary subgroups, and let $\calh$ any family of subgroups
containing  all $P_i$'s which contain $\F_2$.
Let $\sim$ be the co-elementary relation on $\cala_\infty$.

 If $G$ is one-ended relative to $\calh$, then the compatibility JSJ deformation space $\Dco$ exists and contains the tree of cylinders of any JSJ tree.\qed
\end{cor}

We now consider    cyclic splittings,  with $\sim$ the commensurability relation.\index{cyclic tree}\index{commensurable}

\begin{example*}  Let $H$ be a torsion-free hyperbolic group with
  property (FA) (it has no non-trivial action on a tree), and $\langle a \rangle$ a maximal cyclic
  subgroup. Consider the HNN extension $G=\langle H,t\mid
  tat\m=a\rangle$, a one-ended torsion-free CSA group. 
The Bass-Serre  tree $T_0$ is a JSJ tree over abelian groups. 
Its tree of cylinders $T_1$ is  the Bass-Serre tree of the amalgam $G=H*_{\langle a\rangle}\langle
  a,t\rangle$, it is also the compatibility JSJ tree over abelian groups by    Theorem \ref{ouf3}.
Over cyclic groups, $T_0$ is a JSJ tree, its (collapsed) tree of cylinders is $T_1$,
 but the compatibility JSJ tree is $T_0$ (this follows from Proposition \ref{prop_rig} and \cite{Lev_rigid}; 
the non-splitting assumption   of Proposition \ref{prop_rig} holds over cyclic groups, but not over abelian groups). 
\end{example*}

In this example $\Dco$   strictly dominates $(T_a)_c^*$. One obtains a tree in $\Dco$ by refining $T_1=(T_a)^*_c$ at vertices with group $\Z^2$. This is a general fact.

\begin{thm} 
Let $G$ be a finitely generated  torsion-free group.
Let $\cala$ be the family of     cyclic
 subgroups,
and  let $\calh$ be  any family of subgroups of $G$, with $G$   one-ended relative to $\calh$.  
Assume that commensurators\index{commensurator} of infinite cyclic subgroups are small in $(\cala,\calh)$-trees.

Then the   cyclic compatibility JSJ space $\Dco$ relative to $\calh$ exists. If furthermore $G$ is not a solvable Baumslag-Solitar group $BS(1,s)$,
one obtains a tree in $\Dco$ by
  refining $(T_a)^*_c$ (the collapsed tree of cylinders for commensurability) at   some vertices $v$ with $G_v$ virtually  $\bbZ^2$.
 
\end{thm}

\begin{rem}  
If $G$ has torsion, but there is a bound  for the order of
finite subgroups,
a similar  theorem   holds for virtually cyclic splittings (with the same proof). In the furthermore, one must assume that $G$ is not  virtually $BS(1,s)$; one has to prove that 
an ascending HNN extension of an infinite virtually cyclic group is virtually $BS(1,s)$, we leave this as an exercise to the reader (compare 
\cite{FaMo_QI2}).

\end{rem}

\begin{proof}  
By Theorem \ref{thm_VC}, $(T_a)^*_c$  is universally compatible, so $\Dco$, if it exists, dominates $(T_a)^*_c$ and is dominated by $T_a$. 
By \cite[Remark 5.11]{GL4}, smallness of commensurators implies that $(T_a)_c$ and $(T_a)^*_c$ are in the same deformation space. It follows that any group elliptic in $(T_a)^*_c$ 
but not in   $T _a$  is contained in the commensurator of some $A\in\cale$,
hence  is small in $(\cala,\calh)$-trees. 

 We shall now show that universally elliptic trees $S$ dominating $(T_a)^*_c$ and   dominated by $T_a$ belong to only finitely many deformation spaces. This will prove the existence of $\Dco$. 
To determine 
such a tree $S$ up to deformation, one needs to know the action on $S$ of vertex stabilizers $G_v$ of $(T_a)^*_c$ (up to deformation). This action has universally elliptic edge stabilizers by Lemma \ref{lem_passage},  and $G_v$ is small in $(\cala,\calh)$-trees or elliptic in $T_a$ (hence in $S$),
so only two deformation spaces are possible for the action of a given  $G_v$ on $S$ by Corollary \ref{sma2}. This shows the required finiteness, hence the existence of $\Dco$.

 One may obtain a tree  $T'\in \Dco$ by refining $(T_a)^*_c$ at vertices $v$ with $G_v$ not elliptic in   $T' $. 
As explained above such a $G_v$ 
is small in $T'$.  There are two possibilities. If $G_v$  fixes exactly one end, then $\Dco$ is an ascending deformation space (as defined in \cite[Section 7]{GL2}). By  Proposition 7.1(4) of \cite{GL2}, $\Dco$ cannot be  irreducible, so $G=G_v$ is an ascending HNN extension of a   cyclic group, hence isomorphic to some $BS(1,s)$.
The other possibility is that $G_v$ acts on a line, hence is virtually $\Z^2$ 
because edge stabilizers are   cyclic ($G_v$ is isomorphic to $\Z^2$ or the Klein bottle group). 
 \end{proof}

\begin{cor}
Let $G$ be torsion-free and commutative transitive.\index{commutative transitive}
  Let $\calh$ be any family of subgroups.  If $G$ is freely indecomposable relative to $\calh$, and is not a solvable Baumslag-Solitar group,
   the  cyclic  compatibility JSJ deformation space  $\Dco$ relative to $\calh$ exists and may be obtained by  (possibly) refining     
$(T_a)_c^*$ at vertices with stabilizer isomorphic to $\Z^2$. 
\end{cor}

\begin{proof}
The theorem applies because 
 commensurators of  non-trivial  cyclic subgroups are metabelian (see Remark \ref{metab})  hence small.
\end{proof}

\appendix

\section{$\bbR$-trees, length functions, and compatibility}\label{arbr}

 In this appendix we view a simplicial tree as a  metric space (by giving length 1 to every edge), and more generally we consider \Rt s.  
  See \cite{Sh_dendrology,Chi_book}
for basic facts on \Rt s.
 
 Recall that two simplicial trees $T_1,T_2$ are compatible if there is a tree that collapses onto both $T_1$ and $T_2$. In the context of 
   $\bbR$-trees (simplicial or not),   collapse maps have a natural generalisation
      as  maps
preserving alignment: the image of an arc $[a,b]$ is the  segment  $[f(a), f(b)]$ (possibly  a point). 
Compatibility of $\bbR$-trees thus makes sense.

The length function of an \Rt{} $T$  with an isometric action of $G$  is the map $\ell:G\ra\bbR$ defined by $\ell(g)=\min_{x\in T}d(x,gx)$.

The first main result of this appendix is  Theorem \ref{prop_comp}, saying that two $\bbR$-trees are compatible
if and only if the sum of their length functions is again a length function.
This has a nice consequence: the set of $\bbR$-trees compatible with a given tree  is closed.

As a warm-up, we give a proof of the following classical facts: a minimal  irreducible \Rt{}  is determined by its length function; 
the equivariant Gromov-Hausdorff  topology and  the  axes topology (determined by length functions)  agree on the space of irreducible $\bbR$-trees 
(following a suggestion by M.\ Feighn, we extend this to the space of semi-simple trees). 
Our proof does not use based length functions and extends to a proof of Theorem \ref{prop_comp}.

After proving Theorem \ref{prop_comp}, we  show that pairwise compatibility for a finite set of $\bbR$-trees implies the existence of a common refinement.
We then define prime factors, greatest common divisors (gcd's), least common multiples (lcm's)  for irreducible simplicial trees.
We conclude by explaining how to obtain actions on \Rt s by blowing up vertices of JSJ trees.

We assume that $G$ is finitely generated, but  
 Subsections
 \ref{marb} 
 to \ref{cr} apply to any  
infinitely generated group (hypotheses such as irreducibility
ensure that $G$ contains enough hyperbolic elements). We leave details to the reader.

\subsection{Metric trees and length functions} \label{marb}

When endowed with a path metric making each edge isometric to a closed interval, a simplicial tree becomes an $\bbR$-tree (we usually declare each edge to have length 1). 
An \emph{$\bbR$-tree} is a geodesic metric space $T$ in which any two distinct points are connected by a unique topological arc (which we  often call a segment).
\index{R-tree@$\bbR$-tree}
Most  considerations of the preliminary section  apply to
\Rt s as well as simplicial trees.

 We denote by $d$, or $d_T$,  the distance in a tree $T$. All \Rt s are equipped with an isometric action of $G$, and considered equivalent if they are equivariantly isometric.   If $\lambda>0$, we denote by $T/\lambda$ the tree $T$ equipped with the distance $d/\lambda$.

A \emph{branch point} is a point $x\in T$ such that $T\setminus\{x\}$ has at least three components.  A non-empty subtree is \emph{degenerate} if it is a single point, \emph{non-degenerate}\index{degenerate segment} otherwise. If $A,B$ are disjoint   closed subtrees, the \emph{bridge}\index{bridge} between them is the unique arc $I=[a,b]$ such that $A\cap I=\{a\}$ and $B\cap I=\{b\}$. 
 
 We say that a map $f:T\to T'$ \emph{preserves alignment},\index{alignment preserving map} or is a \emph{collapse map}, \index{collapse map}
if the  image of any segment $[a,b]$ is the  segment  $[f(a), f(b)]$ (possibly  a point). The restriction of $f$ to $[a,b]$ is then continuous.
A map  is a collapse map if and only if its restriction to any segment is continuous, and the preimage of any point is a subtree.  
Two trees $T_1$ and $T_2$ are \emph{compatible}\index{compatible trees} if there exists a tree $T$ with collapse maps $f_i:T\to T_i$.

If $g\in G$, we denote by $\ell(g)$ its \emph{translation length}
\index{translation length}\index{0LG@$\ell(g)$: the translation length of $g$}
 $\ell(g)=\min_{x\in T}d(x,gx)$.  There is no parabolic isometry in an $\bbR$-tree, so the minimum is achieved on  a non-empty subset of $T$, 
the
  \emph{characteristic set} \index{0AG@$A(g)$: characteristic set of $g$}\index{characteristic set}
$A(g)$: the fixed point set if $g$ is elliptic  ($\ell(g)=0$), the axis if  $g $ is hyperbolic ($\ell(g)>0$).

The map $\ell:G\to\R$ is   the \emph{length function}\index{length function} of $T$; we denote it by $\ell_T$ if there is a risk of confusion.   We say that a map $\ell:G\to\R$ is a length function if there is a tree $T$ such that $\ell=\ell_T$. \index{0LT@$\ell$, $\ell_T$: the length function of $T$}

The action is \emph{minimal}\index{minimal action} if there is no proper $G$-invariant subtree.
  If $G$ contains a hyperbolic element, there is a unique minimal  subtree:  the union of translation axes of hyperbolic elements.

As in Proposition \ref{cinq}, if a group $H$ (possibly infinitely generated) acts on an \Rt\ $T$, one of the following must occur:  

$\bullet$  the action is irreducible \index{irreducible tree, deformation space}
(there are two hyperbolic elements $g,h$ whose axes have intersection of finite length;  for $n$ large,  $g^n$ and $h^n$ generate an $\F_2$ acting freely and discretely);

$\bullet$
 there  is a fixed point   in $T$, or only in its metric completion  (trivial action); 

$\bullet$ there is an invariant line;

$\bullet$ there is a    fixed end \index{end of a tree}
(an end of an $\bbR$-tree is defined as an equivalence class of geodesic rays up to finite Hausdorff distance).

  If  $H$ is finitely generated (or finitely generated relative to finitely many elliptic subgroups), and fixes a point in the metric completion of $T$, then it fixes a point in $T$.
When there is a fixed end, the length function is the absolute value of a homomorphism $\chi:H\to\R$ (such length functions are usually called abelian; we do not use this terminology, as it may cause confusion).   

 As in \cite{CuMo}, we say that  a minimal $T$ is \emph{semi-simple}\index{semi-simple action on a tree} if there is a hyperbolic element, and either there is  an  invariant line in $T$, or the action is irreducible.

We will use  the following facts, with $T$   an \Rt{} with a minimal action of $G$.

\begin{lem}[\cite{Pau_Gromov}] \label{calcul}
Let $g,h$ be hyperbolic elements. 
\begin{enumerate}
 \item If their axes  $A(g)$, $A(h)$  are disjoint, then
$$\ell(gh)=\ell(g\m h)=\ell(g)+\ell(h)+2d(A(g),A(h))>\ell(g)+\ell(h).$$ The intersection between $A(gh)$
and $A(hg)$ is the bridge between $A(g)$ and $A(h)$. 
\item 
  If their axes   meet, then
$$\min(\ell(gh),\ell(g\m h))\le\max(\ell(gh),\ell(g\m h))=\ell(g)+\ell(h).$$ The   inequality is an equality
if and only if   the axes meet in a single point. \qed
\end{enumerate}
\end{lem}

\begin{lem}[{\cite[Theorem 2.7]{CuMo}}] \label{irrt}
  $T$ is irreducible if and only if there exist hyperbolic elements $g,h$ with $[g,h]$ hyperbolic.
   \qed
\end{lem}

\begin{lem}[{\cite[Lemma 4.3]{Pau_Gromov}}] \label{seg}
If $T$ is irreducible,  any arc $ [a,b]$ is contained in the axis   of some $g\in G$. \qed
\end{lem}

\begin{example} \label{arbsimp}
We use these lemmas to show that, 
\emph{if $T$ has a minimal  irreducible action of $G$, and   $\ell$  takes values in $\Z$, then $T$ is a simplicial tree.}
It suffices to prove that the distance between any two branch points lies in $\frac12\bbZ$. Given two branch points $a,b\in T$,
by Lemma \ref{seg}, one can find hyperbolic elements $g,h\in G$ with disjoint axes   such that the bridge between the axes of $g$ and $h$
is precisely $[a,b]$. By Lemma \ref{calcul}(1), $d(a,b)=\frac12(\ell(gh)-\ell(g)-\ell(h))$.
\end{example}

\subsection{From length functions to trees}\label{fred}

\newcommand{\Tirr}{\calt_{\mathrm{irr}}}
 Let $G$ be a finitely generated group. \index{0Tz@$\calt$: set of actions on $\bbR$-trees}\index{0Tzi@$\Tirr$: set of irreducible actions on $\bbR$-trees}
Let $\calt$ be the set of minimal isometric actions of $G$ on $\bbR$-trees modulo equivariant isometry.
 Let $\Tirr\subset \calt$ 
be the set of irreducible $\bbR$-trees.
The following are classical results:

\begin{thm}[\cite{AlperinBass_length,CuMo}] \label{thm_compat_length}
  Two minimal irreducible \Rt s $T,T'$ with the same length function are equivariantly isometric. 
\end{thm}

\begin{thm}[\cite{Pau_Gromov}] \label{thm_fred}
  The equivariant Gromov-Hausdorff topology and the axes topology agree on $\Tirr$. 
\end{thm}

By Theorem \ref{thm_compat_length}, the assignment $T\mapsto\ell_T$ defines an embedding $\Tirr\to\R^G$. 
The \emph{axes topology} is the topology induced by this embedding. \index{topology!axes topology}\index{axes topology}
  The set of length functions  is closed in $\R^G$  (even if $G$ is not finitely generated),
and when $G$ is finitely generated 
it is projectively compact
\cite[Theorem 4.5]{CuMo}.

\index{topology!equivariant Gromov-Hausdorff topology}\index{Gromov@(equivariant) Gromov-Hausdorff topology}  
The \emph{equivariant Gromov-Hausdorff topology} (or just Gromov topology) on $\calt$  is defined by the following  neighbourhood basis.
Given $T\in\calt$, a number $\eps>0$, a finite subset $A\subset G$, and $x_1,\dots, x_n\in T$,
define $N_{\eps,A,\{x_1,\dots x_n\}}(T)$ as the set of trees $T'\in \calt$ such that there exist
$x'_1,\dots,x'_n\in T'$ with 
 $$|d_{T'}( x'_{i },ax'_{j})-d_T( x_{i }, ax_{j })|\leq \eps$$  for all $a\in A$ 
 and $i,j\in \{1,\dots,n\}$. We call $x'_i$ an \emph{approximation} of $x_i$ in $T'$.\index{approximation point}

\begin{example} \label{asept}
 As an illustration  of this definition, let us explain why, for a given $g\in G$,  \emph{the set of trees  $T\in\calt$ such that $g$ fixes a tripod in $T$ is open in the Gromov topology.}
Recall that a tripod is the convex hull of 3 points that do not lie in a segment.
Let $x_1,x_2,x_3$ (with indices modulo $3$) be the endpoints of a tripod fixed by $g$, and $p$ the center of this tripod.
Let $\eps$ be very small compared to the distances $d_T(p,x_i)$, and take $A=\{1,g\}$. If $T'$ lies in $N_{\eps,A,\{x_1,x_2,x_3\}}$,
consider approximation points $x'_i$ of $x_i$.
Since $d_T(x_i,gx_i)=0$, we have $d_{T'}(x'_i,gx'_i)<\eps$. It follows that the midpoint  $m'_i$ of $[x'_i,gx'_i]$ in $T'$
is at distance at most $\eps/2$ from $x'_i$. 
Now $$d_{T}(x_{i-1},x_i)+d_T(x_i,x_{i+1})-d_T(x_{i-1},x_i)=2d_T(p,x_i)\gg \eps,$$ 
so $$d_{T'}(x'_{i-1},x'_i)+d_{T'}(x'_i,x'_{i+1})-d_{T'}(x'_{i-1},x'_i)\gg\eps$$
and  $$d_{T'}(m'_{i-1},m'_i)+d_{T'}(m'_i,m'_{i+1})-d_{T'}(m'_{i-1},m'_i)>0.$$
It follows that the three points $m'_1,m'_2,m'_3$ do not lie in a segment.  But the midpoint of $[x,gx]$ always belongs   to the characteristic set of $g$, 
so the characteristic set of $g$ cannot be a line. Thus  $g$ is elliptic, and therefore fixes $m'_1,m'_2,m'_3$.
\end{example}

The fact that the axes topology is finer   than the Gromov  topology (which is the harder half of   Theorem \ref{thm_fred})    should be viewed as a version with parameters of Theorem \ref{thm_compat_length}: the length function
determines the tree, in a continuous way. 
 As a preparation for the next subsection, we now give   quick proofs of these theorems. 
Unlike previous proofs, ours does not use based length functions.

\begin{proof} [Proof of Theorem \ref{thm_compat_length}]
Let $T,T'$ be minimal irreducible \Rt s with the same length function $\ell$. We denote by $A(g)$ the
axis of a hyperbolic element $g$ in $T$, by $A'(g)$ its axis in $T'$. By Lemma \ref{calcul},  $A(g)\cap A(h)$ is empty if
and only if $A'(g)\cap A'(h)$ is empty.

We define an isometric equivariant map $f$ from the set of branch points of $T$ to $T'$, as follows. Let  $x $ be a
branch point  of $T$, and $y\ne x$ an auxiliary branch point. 
By Lemmas \ref{calcul} and \ref{seg}, there exist hyperbolic elements $g,h$   whose axes in $T$ do
not intersect,   such that $[x,y]$ is   the bridge between $A(g)$
and $A(h)$, with $x\in A (g)$ and $y\in A (h)$. 
  Then $\{x\}=A(g)\cap A({gh})\cap A({hg})$. The axes of $g$ and $h$ in $T'$ do not intersect,
so 
$ A'(g)\cap A'({gh})\cap A'({hg})$ is a single point which we call $f(x)$.

Note that $f(x)=\cap_k A'(k)$, the intersection being over all
hyperbolic elements $k$ whose axis in $T$ contains $x$: if $k$ is such an element,
  its axis in $T'$ meets all three sets $ A'(g)$, $A'({gh})$, $A'({hg})$, so
contains $f(x)$. This gives an intrinsic definition of $f(x)$, independent of   the choice of $y$,
$g$, and
$h$.
In particular, $f$ is $G$-equivariant. It is isometric because $d_{T'}(f(x),
f(y)) $ and $d_T(x,y)$ are both  equal to $1/2\bigl(\ell(gh)-\ell(g)-\ell(h)\bigr)$.

We then  extend  $f$ equivariantly and isometrically first to the closure
of the set of branch points of $T$,  and then  to each   complementary interval. The resulting map from
$T$ to $T'$ is onto because $T'$ is minimal.
\end{proof}  

\begin{proof} [Proof of Theorem \ref{thm_fred}]  Given $g\in G$, the map $T\mapsto \ell_T(g)$, from $\Tirr $ to $\R$, is
continuous in the Gromov topology: this follows from the formula 
 $\ell(g)=\max(d(x,g^2x)-d(x,gx),0)$.
This shows that the Gromov topology is finer   than the axes topology.

For the converse, we fix $\varepsilon >0$, a finite set of points $x_i\in T$, and a finite set of elements
$a_k\in G$. We have to show that, if the length function $\ell '$ of $T'$ is close enough to $\ell$ on 
a suitable finite subset of $G$, there exist points $x'_i\in T'$ such that $$|d_{T'}(x'_i,a_kx'_j)
-d_T(x_i,a_kx_j)|<\varepsilon $$ for all $i,j,k$. 

First assume that each $x_i$ is a branch point.
For each $i$, choose elements $g_i,h_i$ as in the previous proof, 
with $x_i$ an endpoint of the bridge between $A(g_i)$ and $A(h_i)$. If $\ell'$ is close to $\ell$, the axes of
$g_i$ and $h_i$ in $T'$ are disjoint and we can define $x'_i$ as $ A'(g_i)\cap A'({g_ih_i})\cap A'({h_ig_i})$.
A different choice $\tilde g_i,\tilde  h_i$ may lead to a different point $\tilde x'_i$. But the distance
between $  x'_i$ and $\tilde x'_i$ goes to $0$ as $\ell'$ tends to $\ell$ because all pairwise distances
between $ A'(g_i),A'({g_ih_i}), A'({h_ig_i}),A'(\tilde g_i),A'({\tilde g_i\tilde h_i}),A'({\tilde h_i\tilde
g_i})$ go to 0. It is then easy to complete the proof. 

If some of the $x_i$'s are not branch points, one can add new points so that each such $x_i$ is contained in an arc bounded by branch points $x_{b_i}$, $x_{c_i}$. One then defines $x'_i$ as the point dividing  $[x'_{b_i},x'_{c_i}]$ in the same way as $x_i$ divides $[x_{b_i},x_{c_i}]$.
 \end{proof}

\newcommand{\Tss}{\calt_{\mathrm{ss}}}
As suggested by M.\ Feighn, one may extend the previous results to reducible trees. Let $\Tss$ consist of all minimal trees which are either 
irreducible or isometric to $\R$ (we only rule out trivial trees and trees with exactly one fixed end).
 
Every non-zero length function is the length function of a tree in $\Tss$. 

\begin{thm} \label{thm_feighn}  
  Two minimal    trees $T,T'\in \Tss$ with the same length function are equivariantly isometric. 
   The equivariant Gromov-Hausdorff topology and the axes topology agree on $\Tss$. 
\end{thm}

In other words, the assignment $T\mapsto \ell_T$ induces a homeomorphism between $\Tss$, 
equipped with the   Gromov topology, and the space of non-zero length functions. 

We note that the results of \cite{CuMo} are  stated for   all trees in $\Tss$, those of \cite{Pau_Gromov} for  trees which are irreducible or dihedral.

\begin{proof}  We refer to \cite[page 586]{CuMo} for a proof of the first assertion 
when the actions are not irreducible. 
Since the set of irreducible length functions is open,  it suffices to show the following fact: 

\begin{claim} \label{cvg} {If $T_n$ is a sequence of trees in $\Tss$ whose length functions $\ell_n$ converge to the length function $\ell$ of an action of $G$ on $T=\R$, then $T_n$ converges to $T$ in the Gromov topology.}
\end{claim} 

To prove the claim, we denote by $A_n(g)$ the characteristic set of $g\in G$ in $T_n$, and we fix $h\in G$ hyperbolic in $T$ (hence in $T_n$ for $n$ large). We denote by  $I_n(g)$ the (possibly empty or degenerate) segment  $A _n(g)\cap A_n(h)$.

The first case is when $G$ acts on $T$ by translations. To show that $T_n$ converges to $T$, it suffices to show that, given   elements $g_1,\dots, g_k$ in $G$,  the length of    $ \bigcap _i I_n(g_i)$  goes to infinity with $n$. By a standard argument using Helly's theorem, we may assume $k=2$. 

We first show that, for any $g$,  the length   $ | I_n(g)  | $ goes to infinity. Let $N\in \N$ be arbitrary.
Since $g\m h^Ngh^{-N}$ is elliptic in $T$, the distance between $I_n(h^Ngh^{-N})$ and $I_n(g)$ goes to $0$ as $n\to\infty$. But $I_n(h^Ngh^{-N})$ is the image of  $I_n(g)$ by $h^N$, so $\liminf_{n\to\infty}  | I_n(g) | \ge N\ell(h)$.

 To show that the overlap  between $I_n(g_1) $ and $I_n(g_2) $ goes to infinity,  we can assume  that the relative position of $I_n(g_1) $ and $I_n(g_2) $ is the same for all $n$'s. 
  If they are disjoint, $I_n(g_1g_2g_1\m)$ or $I_n(g_2g_1g_2\m)$ is empty, a contradiction. 
Since  every $|I_n(g )|$ goes to infinity, the result is clear if $I_n(g_1)$ and $I_n(g_2)$ are   nested. 
In the remaining case, up to changing $g_i$ to its inverse, we can assume that $g_1,g_2$ translate in the same direction along  
$A_n(h)$ if they are both hyperbolic.
Then $I_n(g_1) \cap I_n(g_2) $ equals $I_n(g_1g_2)$ or $I_n(g_2g_1)$, so its length goes to infinity.

Now suppose that the action of $G$ on $T$ is dihedral. Suppose that 
$g\in G$ reverses orientation on $T$. 
For $n$ large, the axes of $h, g\m hg, ghg\m $ in $T_n$ have a long overlap by the previous argument. On this overlap $h $ translates in one direction, $g\m hg$ and $ghg\m $ in the other (because $\ell_n( h   g\m hg)$ is close to 0 and $\ell_n( h\m   g\m hg)$ is not). It follows that $g$ acts as a central symmetry on  a long subarc of $A_n(h)$. Moreover, if $g,g'$ both reverse orientation, the distance between their fixed points on $A_n(h)$ is close to $2\ell(gg')$. The convergence of $T_n$ to $T$ easily follows from these observations. This proves the claim, hence the theorem.
\end{proof}  

\subsection{Compatibility and length functions}
\label{sec_compatible_via_longueurs}

Recall that two  \Rt s $T_1,T_2$ are \emph{compatible} if they have a \emph{common refinement}: there exists an   \Rt  {} $\hat T$
with  (equivariant)  maps $g_i:\hat T\to T_i$   preserving alignment (the image of a segment is a segment,  possibly  a point, see  Subsections \ref{comp}  and \ref{marb});
we call such maps collapse maps.\index{compatible trees}

If $T_1$ and $T_2$ are compatible, they have a \emph{standard common refinement $T_s$} constructed as
follows.\index{standard common refinement} 

We denote by $d_i$ the distance in $T_i$, and by $\ell_i$ the length function. 
Let $\hat T$ be any common refinement.  Given $x,y\in \hat T$, define $$\delta
(x,y)=d_1(g_1(x),g_1(y))+d_2(g_2(x),g_2(y)).$$ This is a pseudo-distance satisfying $\delta
(x,y)=\delta (x,z)+\delta (z,y)$ if $z\in[x,y]$ (this is also a length measure, as defined in \cite{Gui_dynamics}).
The associated metric space $(T_s,d)$, obtained by identifying $x,y$ when $\delta(x,y)=0$, is an \Rt{} which
refines $T_1$ and $T_2$, with maps
$f_i:T_s\to T_i$ satisfying $d
(x,y)=d_1(f_1(x),f_1(y))+d_2(f_2(x),f_2(y))$.
The $f_i$'s are $1$-Lipschitz, hence continuous.

The length function of $T_s$ is $\ell=\ell_1+\ell_2$ (this follows from  the formula    $\ell(g)=\lim_{n\to\infty}\frac1 nd(x,g^nx) $). In particular,
$\ell_1+\ell_2$ is a length function. We now prove the converse.

\begin{thm} \label{prop_comp}  
Two minimal irreducible \Rt s $T_1$, $T_2$   with an action of $G$ 
  are compatible if and only if the sum $\ell=\ell_{ 1}+\ell_{ 2}$ of their length functions
  is a length function.   
\end{thm}

\begin{rem}
  If $T_1$ and $T_2$ are compatible, then 
$\lambda_1 l_{ 1}+\lambda_2 l_{ 2}$ is   a length function for all $\lambda_1,\lambda_2\geq 0$.
\end{rem}

\begin{cor} \label{cofer}
  Compatibility is a closed relation on $\Tirr\times \Tirr$.
In particular, the set of irreducible \Rt s compatible with a given $T_0$ is closed in $\Tirr$.
\end{cor}

\begin{proof} 
This follows from the fact that the set of length functions
is a closed subset of $\bbR^G$ \cite{CuMo}.
\end{proof}

\begin{proof}[Proof of Theorem \ref{prop_comp}.]
We   have to prove the ``if'' direction.  Let $T_1,T_2$ be irreducible minimal \Rt s with  length  functions
$\ell_1,\ell_2$, such that $\ell=\ell_1+\ell_2 $ is the length function of a minimal  \Rt {}
$T$. We denote  by $ A(g)$, $A_1(g)$, $A_2(g)$ axes in $T$, $T_1$, $T_2$ respectively.

By Lemma \ref{irrt}, $T$ is irreducible: hyperbolic elements $g,h$ with  $[g,h]$  hyperbolic exist in $T$ since they exist in $T_1$.  
 We want to prove that $T$ is a common refinement of $T_1$ and $T_2$. In fact, we show that $T$ is the
standard refinement $T_s$ mentioned earlier (which is unique by Theorem \ref{thm_compat_length}). 
The proof is similar to that of Theorem \ref{thm_compat_length}, but we first need a few   lemmas. 

\begin{lem}\label{lem_semigroup}[{\cite[lemme 1.3]{Gui_coeur}}] 
 Let $S\subset G$ be a finitely generated semigroup   
such that no point or line in $T$ is   invariant under the subgroup $\langle S\rangle$ generated by $S$.  
Let
$I$ be an arc contained in the axis of a hyperbolic element $h\in S$.

Then there exists a finitely generated semigroup $S'\subset S$ with $\langle S'\rangle=\langle
S\rangle$ such that
 every element $g\in S'\setminus\{1\}$ is hyperbolic in $T$, its axis contains $I$,
and $g$ translates in the same direction as $h$ on $I$.
\qed
\end{lem}

\begin{lem}\label{lem_hyperbolic}  Let $T_1,T_2, T$ be arbitrary irreducible minimal  trees. 
  Given an  arc $I\subset T$, there exists $g\in G$ which is hyperbolic in $T_1$, $T_2$ and $T$,
and whose axis in $T$ contains $I$.
\end{lem}

\begin{proof} 
Apply Lemma \ref{lem_semigroup} with $S=G$ and any $h$   whose axis in $T$ contains $I$ (if $G$ is not finitely generated, one takes for $S$ the group generated by $s_1,t_1,s_2,t_2,h$, with $s_i, t_i, [s_i,t_i]$   hyperbolic in $T_i$).
Since $S'$   generates $S$ as a group, it must contain  an element  $h'$ which is hyperbolic
in
$T_1$: otherwise   $S$ would have a global fixed point in $T_1$ by Serre's lemma (see Subsection \ref{tre}). 
Applying Lemma
\ref{lem_semigroup} to the action of $S'$ on $T_1$, we get a semigroup
$S''\subset S'$ whose non-trivial elements are hyperbolic in $T_1$. Similarly, $S''$ contains an element
$g$ which is hyperbolic in $T_2$. This element $g$ satisfies the conclusions of the lemma.
\end{proof}

\begin{rem} 
More generally, one may require that $g$ be hyperbolic in finitely many trees $T_1,\dots, T _n$.
\end{rem}

We again assume that $T_1$, $T_2$, $T$ are as described at the beginning of the proof of Theorem \ref{prop_comp}.

\begin{lem}\label{lem_interaxe}
  Let $g,h$ be hyperbolic in $T_1$ and $T_2$ (and therefore in $T$).
\begin{itemize}
\item If their axes in $T$ meet, so do their axes in $T_i$. 
\item If their axes in $T$ do not meet, their axes in $T_i$ meet in at most one point. In particular, the
elements
$gh$ and $hg$ are hyperbolic in $T_i$. 
\end{itemize}
\end{lem}

\begin{proof} Assume that $A(g)$ and $A(h)$ meet, but $A_1(g)$ and $A_1(h)$ do not.  Then  $\ell_1(gh)>\ell_1(g)+\ell_1(h)$. Since
$\ell (gh)\le\ell (g)+\ell (h)$, we get $\ell_2(gh)< \ell_2(g)+\ell_2(h)$. Similarly, $\ell_2(g\m h)<
\ell_2(g)+\ell_2(h)$. But these inequalities are incompatible by Lemma
\ref{calcul}.

Now assume that $A(g)$ and $A(h)$ do not meet, and $A_1(g),A_1(h)$   meet in a non-degenerate arc. We may assume 
$$ \ell_1(gh)<\ell_1(g\m h)=\ell_1(g)+\ell_1(h).$$ Since $$ \ell (gh)=\ell (g\m h)>\ell (g)+\ell (h),$$ we have 
$$ \ell_2(gh)>\ell_2(g\m h)>\ell_2(g)+\ell_2(h),$$   contradicting Lemma
\ref{calcul}. 
\end{proof}

We can now complete the proof of Theorem \ref{prop_comp}.
It suffices to define maps $f_i:T\ra T_i$ such that 
 $$d (x,y)=d_1(f_1(x),f_1(y))+d_2(f_2(x),f_2(y)).$$ Such maps are collapse maps (they are $1$-Lipschitz, and if three points 
satisfy a triangular equality in $T$,
 then their images under $f_i$ cannot satisfy a strict triangular inequality), so $T$ is the standard
common refinement $T_s$. 

The construction of $f_i$ is
the same as that of $f$ in the proof of Theorem \ref{thm_compat_length}. Given branch
points    $x,y\in T$,  we use Lemma
\ref{lem_hyperbolic} to get elements $g$ and $h$ hyperbolic in all three trees, 
and such that the bridge between $A(g)$ and $A(h)$ is $[x,y]$. 
Then Lemma \ref{lem_interaxe} guarantees that $A_i(g)\cap A_i({gh})\cap
A_i({hg})$ is a single point of $T_i$, which we define as $f_i(x)$; the only new phenomenon is that $A_i(g)$
and
$ A_i({h})$ may now intersect in  a single point. 

The relation between $d,d_1,d_2$ comes from the
equality
$\ell=\ell_1+\ell_2$, using the formula  $$d_i(f_i(x),
f_i(y)) =1/2(\ell_i(gh)-\ell_i(g)-\ell_i(h)).$$
Having defined $f_i$ on branch points,  we extend it by continuity to the closure
of the set of branch points of $T$ (it is $1$-Lipschitz) and then linearly to each   complementary
interval. The relation between $d,d_1,d_2$ still holds.
\end{proof}

\subsection{Common refinements}
\label{cr}
 
The following result is proved for almost-invariant sets in \cite[Theorem 5.16]{ScSw_regular+errata}.

\begin{prop}\label{prop_flag}
  Let $G$ be a finitely generated group, and  let $T_1,\dots,T_n$ be irreducible minimal $\bbR$-trees such that $T_i$ is compatible with $T_j$ for $i\neq j$. 
Then there exists  a common refinement\index{common refinement} $T$ of all $T_i$'s.
\end{prop}

\begin{rem}
  This statement may be interpreted as the fact that the set of  projectivized trees satisfies the \emph{flag} condition
for a simplicial complex: whenever one sees the $1$-skeleton of an $n$-simplex, there is indeed an $n$-simplex.
Two compatible trees  $T_i$, $T_j$ define a $1$-simplex $t\ell_i+(1-t)\ell_j$ of length functions.
If there are segments joining any pair of length functions $\ell_i, \ell_j$,
the proposition says that there is an $(n-1)$-simplex $\sum t_i \ell_i$ of length functions.
\end{rem}

 To prove Proposition \ref{prop_flag}, we
  need some terminology from \cite{Gui_coeur}. 
A \emph{direction} in an \Rt\ $T$ is a connected component $\delta$ of $T\setminus\{x\}$
for some $x\in T$. A \emph{quadrant} in $T_1\times T_2$ is a product $Q=\delta_1\times\delta_2$ of a direction of   $T_1$
by a direction of $T_2$.
A quadrant $Q=\delta_1\times\delta_2$ is \emph{heavy} if there exists $h\in G$ hyperbolic in $T_1$ and $T_2$ such that 
$\delta_i$ contains a positive semi-axis of $h$ (equivalently, for all $x\in T_i$ one has  $h^n(x)\in \delta_i$ for $n$ large).
 We say that $h$ \emph{makes $Q$ heavy}.
The  \emph{core} $\calc(T_1\times T_2)\subset T_1\times T_2$ is the complement of the union of quadrants which are not heavy  (this is not the same core as in Subsection \ref{fpc}). \index{core in the product of two trees}

By \cite[Th\'eor\`eme 6.1]{Gui_coeur}, $T_1$ and $T_2$ are compatible if and only if $\calc(T_1\times T_2)$ contains no 
non-degenerate rectangle  (a product $I_1\times I_2$ where each $I_i$ is an arc not reduced to a point).

We first prove a technical lemma.
\begin{lem}\label{lem_qdt} Let $T_1,T_2$ be irreducible and minimal.
Let $f:T_1\ra T'_1$ be a collapse map, with $T'_1$ irreducible.
Let $\delta'_1\times \delta_2$ be a quadrant in $T'_1\times T_2$,
and $\delta_1=f\m(\delta'_1)$.
If the quadrant  $\delta_1\times \delta_2\subset T_1\times T_2$  is heavy,
then so is $\delta'_1\times \delta_2$.
\end{lem}
 
  Note that    $\delta_1$ is a direction because $f$ preserves alignment.

\begin{proof}
  Consider an element $h$ making $\delta_1\times\delta_2$ heavy.
If $h$ is hyperbolic in  $T'_1$, then $h$ makes  $\delta'_1\times\delta_2$ heavy and we are done.
 If not, assume that we can find some $g\in G$, hyperbolic in  $T'_1$ and $T_2$ (hence in $T_1$), such that for $i=1,2$ 
the axis  $A_i(g)$ of $g$ in $T_i$  intersects 
  $A_i(h)$ in a compact set.
Then for $n>0$ large enough the element $h^ngh^{-n}$ makes $\delta_1\times \delta_2$ heavy.
Since this element is hyperbolic in $T_1'$ and $T_2$, it makes $\delta'_1\times \delta_2$ heavy.

We now prove the existence of $g$.
Consider a line $l$ in $T_2$, disjoint from $A_2(h)$, and  the bridge $[x,y]$ between $l$ and $A_2(h)$.
Let $I\subset l$ be an arc containing $x$ in its interior.
By  Lemma \ref{lem_hyperbolic}, there exists $g$ hyperbolic in $T'_1$ and $T_2$   whose axis in $T_2$ contains $I$,  hence is disjoint from $A_{2}(h)$.
Being hyperbolic in $T'_1$, the element $g$ is hyperbolic in $T_1$. Its axis intersects $A_{ 1}(h)$ in a compact set  because $A_{ 1}(h)$ is mapped to a single point in $T'_1$ (otherwise, $h$ would be hyperbolic in $T'_1$).
\end{proof}

\begin{proof}[Proof of Proposition \ref{prop_flag}]
We assume $n=3$, as the general case then follows by a straightforward induction.
Let $T_{12}$ be  the  standard common  refinement of $T_1,T_2$ (see Subsection \ref{sec_compatible_via_longueurs}).
Let $\calc$ be the core of $T_{12}\times T_3$. 
By \cite[Th\'eor\`eme 6.1]{Gui_coeur}, it is enough to prove that $\calc$ 
does not contain a product of non-degenerate arcs $[a_{12},b_{12}]\times[a_3,b_3]$.
Assume otherwise.
Denote by $a_1,b_1,a_2,b_2$ the images of $a_{12},b_{12}$ in $T_1,T_2$.
Since $a_{12}\neq b_{12}$, at least one inequality $a_1\neq b_1$ or $a_2\neq b_2$ holds.
Assume for instance $a_1\neq b_1$.

We claim that $[a_1,b_1]\times [a_3,b_3]$ is contained in the core of 
$T_1\times T_3$,  giving a contradiction.  
We have to show that any quadrant $Q=\delta_1\times\delta_3$ of $T_1\times T_3$ intersecting $[a_1,b_1]\times[a_3,b_3]$ is heavy. 
Denote by $f:T_{12}\ra T_1$ the  collapse map. 
The preimage $\Tilde Q=f\m(\delta_1)\times\delta_3$ of $Q$ in $T_{12}\times T_3$ is a quadrant intersecting
$[a_{12},b_{12}]\times [a_3,b_3]$.
Since this rectangle in contained in $\calc(T_{12}\times T_3)$, the quadrant $\Tilde Q$ is heavy, and so is $Q$ by Lemma \ref{lem_qdt}.
\end{proof}

\subsection{Arithmetic of trees}\label{sec_arith}

\newcommand{\Sirr}{\cals_{\mathrm{irr}}}

In this subsection, we work with simplicial trees. \index{0S@$\Sirr$: set of irreducible simplicial trees}
We let  $\Sirr$ be the set of simplicial trees $T$ which
are minimal, irreducible, with no redundant vertices and no inversion.  We also view such a $T$ as a metric tree, by
declaring each edge to be of length 1. This makes $\Sirr$ a subset of $\Tirr$.  
  By Theorem \ref{thm_compat_length}, a tree $T\in \Sirr$ is determined by its length function $\ell $.

\index{prime factors of a tree}
\begin{dfn}[Prime factors] The \emph{prime factors} of $T$ are the one-edge splittings $T_i$ obtained from $T$ by collapsing
edges in all orbits but one. Clearly $\ell=\sum_i\ell_i$, where $\ell_i $ is the length function of  $T_i$. 
\end{dfn} 

We may view a prime factor of
$T$  as an orbit of edges of $T$, or
as an edge of the quotient graph of groups $\Gamma =T/G$.
 Since    $G$ is assumed to be  finitely generated, there are  finitely many prime factors (by Proposition \ref{arbtf}, this remains true if $G$ is only   finitely generated relative   to a finite collection of elliptic subgroups).

\begin{lem} \label{po}
Let $T\in\Sirr$.
\begin{enumerate} 
 \item Any non-trivial tree $T'$ obtained from $T$ by collapses (in particular, its prime factors) belongs to
$\Sirr$. 
\item
The prime
factors of $T$  are distinct ($T$ is ``squarefree''). 
\end{enumerate}
\end{lem}

\begin{proof}  Let $e$ be any edge of $T$ which is not collapsed in $T'$. 
 Since $T$ has no redundant vertex and is not a line, either  the endpoints of $e$ are branch points $u,v$,
or there are branch points $u,v$ such that $[u,v]=e\cup e'$ with 
  $e'$ in the same orbit as $e$.
Using Lemma \ref{seg},  we can find elements $g,h$ hyperbolic in $T$, 
whose axes are not collapsed to points in $T'$, and such that $[u,v]$ is the bridge between their axes. 
Since $g,h$ are hyperbolic with disjoint axes in $T'$, the tree $T'$ is irreducible. 
It is easy to check that collapsing cannot create redundant vertices, so $T'\in\Sirr$. 
 
Now suppose that $e$, hence $[u,v]$, gets collapsed in some prime factor $T''$. Then $\ell'(gh)>\ell'(g)+\ell'(h)$ holds in $T'$ but not in $T''$, so $T'\ne T''$. 
\end{proof}

 Because of this lemma, a tree of $\Sirr$ is determined by its prime factors. In particular, $T$ refines
$T'$ if and only if every prime factor of $T'$ is also a prime factor of $T$.

\begin{cor}\label{coll} Assume that $T$ and $T'$ are compatible non-trivial trees. If $T$ is irreducible, so is $T'$. If $T$ is not irreducible, $T'$ belongs to the same deformation space as $T$.
\end{cor}

\begin{proof}
 The lemma shows that performing a collapse on  an irreducible simplicial tree yields an irreducible tree (or a point). If $T$ is not irreducible, the quotient graph of groups is a circle and every edge $e$ has an endpoint $v$ such that the inclusion $G_e\to G_v$ is onto (see Subsection \ref{tre}).
This implies  that performing a collapse on  a non-irreducible tree   yields a  minimal non-irreducible tree belonging to the same deformation space (or a point). The lemma follows.
\end{proof}

If $T_1$ and $T_2$ are compatible, the standard refinement $T_s$   with $\ell_{T_s}=\ell_{T_1}+\ell_{T_2}$ 
constructed in Subsection  \ref{sec_compatible_via_longueurs} 
is a metric tree which should be viewed as
the  ``product''  of $T_1$ and $T_2$. We shall now define the lcm $T_1\vee
T_2$ of simplicial trees $T_1$ and
$T_2$. To understand the difference between the two, suppose $T_1=T_2$.  Then
$T_s$ is obtained from
$T_1$ by subdividing each edge, its length function is $2\ell_1$. On the other hand, $T_1\vee
T_1=T_1$.

\begin{dfn} [gcd] \label{gcd} \index{gcd of two trees}\index{0LT2@$\ell_1\wedge\ell_2$: gcd of two length functions}
\index{0Tv@$T_1\wedge T_2$: gcd of two trees}
Consider two trees $T_1,T_2\in\Sirr$, 
with length functions $\ell_1,\ell_2$. 

We define $\ell_1\wedge\ell_2$ as the sum of all length functions which appear as prime
factors in both
$T_1$ and $T_2$. It is the length function of a tree $T_1\wedge T_2$ (possibly a point) which is a
collapse of both $T_1$ and $T_2$. We call $T_1\wedge T_2$ the  \emph{gcd} of $T_1$ and $T_2$. 
\end{dfn} 

We define $\ell_1\vee\ell_2=\ell_1+\ell_2-\ell_1\wedge\ell_2$ as the sum of all length functions
which appear as prime factors in  
$T_1$ or $T_2$ (or both). 
\index{0LT3@$\ell_1\vee\ell_2$: lcm of two length functions}

\begin{lem} \label{ppcm}\index{0Tv2@$T_1\vee T_2$: lcm of compatible trees}
Let $T_1$ and $T_2$ be compatible trees in $\Sirr$. There is a  tree $T_1\vee T_2\in\Sirr$  whose
length function is $\ell_1\vee\ell_2$. It is a common refinement of $T_1$ and
$T_2$, and   no edge of $T_1\vee T_2$ is
collapsed  in both
$T_1$ and
$T_2$.    
\end{lem}

\begin{proof}
Let $T$ be any common refinement. We modify it as follows. We collapse any edge which is collapsed  
in both $T_1$ and $T_2$. We then  restrict to the minimal subtree and remove redundant vertices. The
resulting  tree $T_1\vee T_2$ belongs to $\Sirr$ (it is irreducible because it refines $T_1$). It is a common refinement of
$T_1$ and
$T_2$, and   no edge   is collapsed  in both
$T_1$ and
$T_2$.     

We check that $T_1\vee T_2$ has the correct  length function  by finding its
prime factors.  Since no
edge   is collapsed  in both
$T_1$ and
$T_2$, a prime factor of $T_1\vee T_2$ is a   prime factor of   $T_1$ or $ T_2$. Conversely, a prime
factor of $T_i$ is associated to an orbit of edges of $T_i$, and this orbit lifts to $T_1\vee T_2$. 
\end{proof}

 \begin{rem} \label{redun}
 Unlike the standard refinement $T_s$, the tree $T_1\vee T_2$ does not have redundant vertices.
\end{rem}

 \begin{prop} \label{prop_lcm} 
Let $T_1,\dots, T_n$ be pairwise compatible trees of   $ \Sirr$. There exists a tree
$T_1\vee\dots\vee T_n$ in $\Sirr$ whose length function is the sum of all length functions
which appear as a prime factor  in  
some $T_i$. Moreover:
\begin{enumerate}
\item A tree $T\in \Sirr$ refines $T_1\vee\dots\vee T_n$ if and only if it refines each $T_i$.
\item A tree $T\in \Sirr$ is compatible with  $T_1\vee\dots\vee T_n$ if and only if it is compatible with
each $T_i$.
\item A subgroup $H$ is elliptic in  $T_1\vee\dots\vee T_n$ if and only if it is elliptic in 
each $T_i$. If $T_1$ dominates each $T_i$, then $T_1\vee\dots\vee T_n$ belongs to the deformation
space of $T_1$. 
\end{enumerate}
\end{prop}

\begin{proof} First suppose $n=2$. We show that $T_1\vee T_2$ satisfies the additional conditions. 

If $T$ refines   $T_1$ and $T_2$, it refines $T_1\vee  T_2$ because
every prime factor of $T_1\vee  T_2$ is a prime factor of $T$. This proves Assertion (1).

If $T_1,T_2, T$ are pairwise compatible, they have a common refinement $\hat T$ by 
 Proposition \ref{prop_flag} or \cite[Theorem 5.16]{ScSw_regular+errata} (where one should exclude ascending HNN extensions).
This $\hat T$ refines $T_1\vee T_2$ by Assertion (1), so $T$ and $T_1\vee T_2$ are compatible. 

Assertion (3) follows from the fact that no edge of  $T_1\vee T_2$ is collapsed  in both   $T_1$ and $T_2$, 
as in the proof of Proposition  \ref 
{prop_refinement}: if $H$ fixes a point $v_1\in T_1$ and is elliptic in $T_2$, it fixes a
point in the preimage of $v_1$ in $T_1\vee T_2$.

The case $n>2$ now follows easily by induction. By Assertion (2), the tree $T_1\vee\dots\vee T_{n-1}$ is
compatible with $T_{n}$, so we can define $T_1\vee\dots\vee T_{n}=
(T_1\vee\dots\vee T_{n-1})\vee T_{n}$.
\end{proof}

 \begin{dfn} [lcm] \label{lcm}   We call $T_1\vee\dots\vee T_n$   the \emph{lcm} of the compatible trees $T_i$. \index{lcm of pairwise compatible trees}
\end{dfn}

 \subsection{Reading  actions on $\bbR$-trees}\label{sec_read}

 Rips theory gives a way to understand stable actions on $\bbR$-trees,  by relating 
 them to actions on simplicial trees. Therefore, they are closely related to JSJ decompositions.
We consider first the JSJ deformation space, then the compatibility JSJ tree.

For simplicity, we assume $\calh=\es$.
\subsubsection{Reading $\bbR$-trees from the JSJ deformation space}

 \begin{prop}  \label{sell}
 Let $G$ be finitely presented. 
 Let $T_J$ be a   JSJ tree  over the family of slender subgroups.\index{slender group} 
 If  $T$ is an $\bbR$-tree with a stable action of $G$ whose arc stabilizers are slender, 
then edge stabilizers of $T_J$ are elliptic in $T$.
 \end{prop}
 
Recall \cite{BF_stable} that an arc $I\inc T$  is \emph{stable} if any $g\in G$ that   pointwise fixes some subarc  of $I$ also fixes $I$. The action of $G$ on $T$ is stable if every arc is stable. \index{stable action on an $\bbR$-tree}

 \begin{proof}
   By \cite{Gui_approximation}, $T$ is a limit of simplicial  trees $T_k$
 with slender edge stabilizers. 
 Since $T_J$ is universally elliptic,   each edge stabilizer $G_e$ of $T_J$ is elliptic in every $T_k$.
Passing to the limit, we deduce that each element of $G_e$ is elliptic in $T$.
Since $G_e$ is finitely generated, $G_e$ is elliptic in $T$. 
 \end{proof}

 \begin{rem} More generally, suppose that $G$ is finitely presented and  $\cala$   
 is stable under extension by finitely generated free abelian groups:
 if $H<G$ is such that $1\ra A\ra H\ra \bbZ^k\to 1$, with $A\in \cala$, then
 $H\in \cala$.  Let $T_J$ be a JSJ tree over $\cala$ with finitely generated edge stabilizers (this exists by Theorem \ref{thm_exist_mou}).
 If $T$ is a stable 
 $\bbR$-tree with arc stabilizers in $\cala$, then edge stabilizers of $T_J$ are elliptic in $T$.
 \end{rem}
 
  Recall 
  (Theorem \ref{thm_description_slender}) that, when $G$ is finitely presented, 
flexible vertices of the slender JSJ deformation space are either slender or QH with slender fiber.

 \begin{prop}\label{prop_read_mou} Let $G, T_J,T$ be as in the previous proposition.
 There exists an $\bbR$-tree $\Hat T$ obtained by blowing up each flexible vertex $v$ of $T_J$
 into
 \begin{enumerate}
 \item an action by isometries on a line if $G_v$ is slender,
 \item an action dual to a measured foliation on the  underlying $2$-orbifold of $G_v$ if $v$ is QH,
 \end{enumerate}
 which resolves (or dominates) $T$ in the following sense:
 there exists a $G$-equivariant map $f:\Hat T\ra T$ which is piecewise linear:  
every segment of $\hat T$ can be decomposed into finitely many subsegments, in restriction to which $f$ preserves alignment.
 \end{prop}

 \begin{proof}   Using ellipticity of $T_J$ with respect to $T$, we argue as in the proof of Proposition \ref{prop_refinement}, with $T_1=T_J$ and $T_2=T$. If 
  $v\in V(T_J)$ and $G_v$ is elliptic in $T$, we let $Y_v\inc T$ be a fixed point. If $G_v$ is not elliptic in $T$, we let $Y_v$ be its minimal subtree. It is a line if  $G_v$ is slender. If $v$ is a QH vertex, then $Y_v$ is dual to a measured foliation 
  of the underlying orbifold by Skora's theorem \cite{Skora_splittings} (applied to a covering surface $\Sigma_0$).
  \end{proof}
  
    \begin{rem} 
The arguments given above may be applied in more general situations. 
For instance, assume that $G$ is finitely generated, that all  subgroups of $G$ not containing $\F_2$ are slender, 
and that $G$ has a JSJ tree $T_J$ over slender subgroups
 whose flexible subgroups are QH.
Let $T$ be an $\bbR$-tree with slender arc stabilizers such that $G$ does not split over a subgroup of the stabilizer of an unstable arc or of a tripod in $T$.
Then, applying \cite{Gui_actions} and the techniques of \cite{Gui_approximation}, we see that $T$ is a limit of slender trees,
so 
  Propositions \ref{sell} and \ref{prop_read_mou} apply. 
  \end{rem}

\subsubsection{Reading $\bbR$-trees from the compatibility JSJ tree} \label{readc}

In \cite{Gui_reading}, the first author explained how to obtain all small actions of a one-ended hyperbolic group $G$ on \Rt s from a JSJ tree. 
The proof was based on Bowditch's\index{Bowditch} construction of a JSJ tree from the topology of $\bo G$. 
Here we give a different, more general, approach, based on  Corollary \ref{cofer} (saying that compatibility is a closed condition)
and  results of  Subsection \ref{ac}
(describing the compatibility JSJ space).
Being universally compatible, 
$\Tco$ 
is compatible with any \Rt\ which is a   limit of simplicial $\cala$-trees. 
We illustrate this idea in a simple case.

  Let $G$ be a one-ended finitely presented  torsion-free CSA group.\index{CSA group}
    Assume that $G$ is not   abelian, 
and let $\Tco$ be its compatibility JSJ tree over the class $\cala$ of   abelian groups\index{abelian tree}  (see Definition \ref{deftco}  and Subsection \ref{ac}).
Let $G\actson T$ be an action on an $\bbR$-tree with trivial tripod stabilizers, and   abelian arc stabilizers.
By \cite{Gui_approximation}, $T$ is a limit of  simplicial $\cala$-trees. Since $\Tco$ is compatible with all $\cala$-trees,
it is compatible with $T$ by  Corollary  \ref{cofer}.
Let $\Hat T$ be the standard common refinement of $T$ and $\Tco$ with length function
$\ell_T+\ell_{\Tco}$ (see Subsection \ref{sec_compatible_via_longueurs}). 
Let $f_{co}:\Hat T\ra \Tco$ and $f:\Hat T\ra T$ be  maps preserving alignment such that $$d_{\Hat T}(x,y)=d_{\Tco}(f_{co}(x),f_{co}(y))+
d_{T}(f(x),f(y)).$$

To each vertex $v$ and each edge $e$ of $\Tco$  there correspond  closed subtrees  $\Hat T_v=f_{co}\m(v)$
and $\Hat T_e=\ol{f\m(\rond{e})}$ of $\Hat T$.
By minimality, $\Hat T_e$ is an arc of $\Hat T$ containing no branch point except maybe at its endpoints. 
The relation between $d_{\Hat T},d_{\Tco}$, and $d_T$ shows that the restriction of $f$ to $\Hat T_v$ is an isometric embedding.
In particular, $T$ can be obtained from  $\Hat T $ by changing the length of the arcs  $\Hat T_e$  (possibly making the length $0$).

 We shall now describe   the action of $G_v$ on $\Hat T_v$.
  Note that $G_v$ is infinite. Its action on $ \Hat T_v$ need not be minimal, but it is finitely supported, see \cite{Gui_actions}.
Given an edge $e$ of $\Tco$ containing $v$, we denote by $x_e$ the endpoint of $\Hat T_e$ belonging to $\Hat T_v$.  If  $\Hat T_v$ is not minimal,
  it is the convex hull of the set of points $x_e$ which are extremal (\ie $\hat T_v\setminus x_e$ is connected).

 First suppose that  $G_v$ fixes some $x\in\Hat T_v$ and that $\Hat T_v$ is not a point. Note that, if $x_e\ne x$ is extremal, the stabilizer of the arc $[x,x_e]$ contains $G_e$, so is infinite. We claim that, if $e$ and $f$ are edges of $\Tco$ containing $v$ with $x_e\ne x_f$ both extremal, then $[x,x_e]\cap [x,x_f]=\{x\}$. If not, the intersection is an arc $[x,y]$. The stabilizer of 
$[x,y]$ contains  $\grp{G_e,G_{f}}$ and is abelian. It follows that any point in the $G_{f}$-orbit of $x_e$ is fixed by $G_e$. 
Since tripod stabilizers are trivial, we deduce that $G_{f}$ fixes $x_e$, 
a contradiction.
We have thus proved that  $\Hat T_v$ is a cone on a finite number of orbits of points.

 If $G_v$ does not fix a point in $\Hat T_v$, then it is flexible. If it is   abelian, triviality of tripod stabilizers implies that 
 $\Hat T_v$ is a line.
 If $G_v$ is not abelian, it is a surface group by 
 Theorems  \ref{JSJ_CSA} and \ref{ouf3}.
Skora's theorem \cite{Skora_splittings} asserts that 
the minimal subtree   $\mu_{T_v}(G_v)$ is dual to a measured lamination on a compact surface.  By triviality of tripod stabilizers,     $T_v\setminus \mu_{T_v}(G_v)$ is a disjoint union of segments, and the pointwise  stabilizer of each such segment 
has index at most $2$ in a boundary subgroup of $G_v$.

It follows in particular  from this analysis that $\Hat T$ and $T$ are geometric (see \cite{LP}).

\bibliographystyle{alpha2}
\bibliography{published,unpublished}

\begin{thebibliography}{GAnMB06}

\bibitem[AB87]{AlperinBass_length}
Roger Alperin and Hyman Bass.
\newblock Length functions of group actions on ${\Lambda}$-trees.
\newblock In {\em Combinatorial group theory and topology (Alta, Utah, 1984)},
  pages 265--378. Princeton Univ. Press, Princeton, NJ, 1987.

\bibitem[Bar16]{Barrett_JSJ}
Benjamin Barrett.
\newblock Computing {JSJ} decompositions of hyperbolic groups, 2016.
\newblock arXiv:1611.00652.

\bibitem[Bee13]{Beeker_compatibility}
Benjamin Beeker.
\newblock Compatibility {JSJ} decomposition of graphs of free abelian groups.
\newblock {\em Internat. J. Algebra Comput.}, 23(8):1837--1880, 2013.

\bibitem[Bee14]{Beeker_abelian}
Benjamin Beeker.
\newblock Abelian {JSJ} decomposition of graphs of free abelian groups.
\newblock {\em J. Group Theory}, 17(2):337--359, 2014.

\bibitem[BF91]{BF_bounding}
Mladen Bestvina and Mark Feighn.
\newblock Bounding the complexity of simplicial group actions on trees.
\newblock {\em Invent. Math.}, 103(3):449--469, 1991.

\bibitem[BF95]{BF_stable}
Mladen Bestvina and Mark Feighn.
\newblock Stable actions of groups on real trees.
\newblock {\em Invent. Math.}, 121(2):287--321, 1995.

\bibitem[Bow98]{Bo_cut}
Brian~H. Bowditch.
\newblock Cut points and canonical splittings of hyperbolic groups.
\newblock {\em Acta Math.}, 180(2):145--186, 1998.

\bibitem[Bow01]{Bo_peripheral}
Brian~H. Bowditch.
\newblock Peripheral splittings of groups.
\newblock {\em Trans. Amer. Math. Soc.}, 353(10):4057--4082 (electronic), 2001.

\bibitem[Bow12]{Bow_relhyp}
Brian~H. Bowditch.
\newblock Relatively hyperbolic groups.
\newblock {\em Internat. J. Algebra Comput.}, 22(3):1250016, 66, 2012.

\bibitem[BKM07]{BuKhMi_isomorphism}
Inna Bumagin, Olga Kharlampovich, and Alexei Miasnikov.
\newblock The isomorphism problem for finitely generated fully residually free
  groups.
\newblock {\em J. Pure Appl. Algebra}, 208(3):961--977, 2007.

\bibitem[Car11]{Carette_automorphism}
Mathieu Carette.
\newblock The automorphism group of accessible groups.
\newblock {\em J. Lond. Math. Soc. (2)}, 84(3):731--748, 2011.

\bibitem[CG05]{CG_compactifying}
Christophe Champetier and Vincent Guirardel.
\newblock Limit groups as limits of free groups.
\newblock {\em Israel J. Math.}, 146:1--75, 2005.

\bibitem[Cha07]{Charney_introduction}
Ruth Charney.
\newblock An introduction to right-angled {A}rtin groups.
\newblock {\em Geom. Dedicata}, 125:141--158, 2007.

\bibitem[Chi01]{Chi_book}
Ian Chiswell.
\newblock {\em Introduction to {$\Lambda$}-trees}.
\newblock World Scientific Publishing Co. Inc., River Edge, NJ, 2001.

\bibitem[Cla05]{Clay_contractibility}
Matt Clay.
\newblock Contractibility of deformation spaces of {$G$}-trees.
\newblock {\em Algebr. Geom. Topol.}, 5:1481--1503 (electronic), 2005.

\bibitem[Cla09]{Clay_deformation}
Matt Clay.
\newblock Deformation spaces of {$G$}-trees and automorphisms of
  {B}aumslag-{S}olitar groups.
\newblock {\em Groups Geom. Dyn.}, 3(1):39--69, 2009.

\bibitem[Cla14]{Clay_when}
Matt Clay.
\newblock When does a right-angled {A}rtin group split over {$\Bbb{Z}$}?
\newblock {\em Internat. J. Algebra Comput.}, 24(6):815--825, 2014.

\bibitem[CF09]{ClFo_Whitehead}
Matt Clay and Max Forester.
\newblock Whitehead moves for {$G$}-trees.
\newblock {\em Bull. Lond. Math. Soc.}, 41(2):205--212, 2009.

\bibitem[Coh89]{Cohen_combinatorial}
Daniel~E. Cohen.
\newblock {\em Combinatorial group theory: a topological approach}, volume~14
  of {\em London Mathematical Society Student Texts}.
\newblock Cambridge University Press, Cambridge, 1989.

\bibitem[CL83]{CollinsLevin_automorphisms}
Donald~J. Collins and Frank Levin.
\newblock Automorphisms and {H}opficity of certain {B}aumslag-{S}olitar groups.
\newblock {\em Arch. Math. (Basel)}, 40(5):385--400, 1983.

\bibitem[CM87]{CuMo}
Marc Culler and John~W. Morgan.
\newblock Group actions on {$\mathbb{R}$}-trees.
\newblock {\em Proc. London Math. Soc. (3)}, 55(3):571--604, 1987.

\bibitem[CV86]{CuVo_moduli}
Marc Culler and Karen Vogtmann.
\newblock Moduli of graphs and automorphisms of free groups.
\newblock {\em Invent. Math.}, 84(1):91--119, 1986.

\bibitem[DG08]{DaGr_isomorphism}
Fran{\c{c}}ois Dahmani and Daniel Groves.
\newblock The isomorphism problem for toral relatively hyperbolic groups.
\newblock {\em Publ. Math. Inst. Hautes \'Etudes Sci.}, 107:211--290, 2008.

\bibitem[DG11]{DG2}
Fran{\c{c}}ois Dahmani and Vincent Guirardel.
\newblock The isomorphism problem for all hyperbolic groups.
\newblock {\em Geom. Funct. Anal.}, 21(2):223--300, 2011.

\bibitem[DT13]{DaTo_isomorphism}
Fran{\c{c}}ois Dahmani and Nicholas Touikan.
\newblock Isomorphisms using {D}ehn fillings: the splitting case, 2013.
\newblock arXiv:1311.3937 [math.GR].

\bibitem[Del96]{Delzant_sous-groupes}
Thomas Delzant.
\newblock Sous-groupes distingu\'es et quotients des groupes hyperboliques.
\newblock {\em Duke Math. J.}, 83(3):661--682, 1996.

\bibitem[Del99]{Delzant_accessibilite}
Thomas Delzant.
\newblock Sur l'accessibilit\'e acylindrique des groupes de pr\'esentation
  finie.
\newblock {\em Ann. Inst. Fourier (Grenoble)}, 49(4):1215--1224, 1999.

\bibitem[DD89]{DicksDunwoody_groups}
Warren Dicks and M.~J. Dunwoody.
\newblock {\em Groups acting on graphs}, volume~17 of {\em Cambridge Studies in
  Advanced Mathematics}.
\newblock Cambridge University Press, Cambridge, 1989.

\bibitem[Dun85]{Dun_accessibility}
M.~J. Dunwoody.
\newblock The accessibility of finitely presented groups.
\newblock {\em Invent. Math.}, 81(3):449--457, 1985.

\bibitem[DS99]{DuSa_JSJ}
M.~J. Dunwoody and M.~E. Sageev.
\newblock {J}{S}{J}-splittings for finitely presented groups over slender
  groups.
\newblock {\em Invent. Math.}, 135(1):25--44, 1999.

\bibitem[DS00]{DuSw_algebraic}
M.~J. Dunwoody and E.~L. Swenson.
\newblock The algebraic torus theorem.
\newblock {\em Invent. Math.}, 140(3):605--637, 2000.

\bibitem[Dun93]{Dun_inaccessible}
Martin~J. Dunwoody.
\newblock An inaccessible group.
\newblock In {\em Geometric group theory, Vol.\ 1 (Sussex, 1991)}, pages
  75--78. Cambridge Univ. Press, Cambridge, 1993.

\bibitem[Far98]{Farb_relatively}
B.~Farb.
\newblock Relatively hyperbolic groups.
\newblock {\em Geom. Funct. Anal.}, 8(5):810--840, 1998.

\bibitem[FM99]{FaMo_QI2}
Benson Farb and Lee Mosher.
\newblock Quasi-isometric rigidity for the solvable {B}aumslag-{S}olitar
  groups. {II}.
\newblock {\em Invent. Math.}, 137(3):613--649, 1999.

\bibitem[For02]{For_deformation}
Max Forester.
\newblock Deformation and rigidity of simplicial group actions on trees.
\newblock {\em Geom. Topol.}, 6:219--267 (electronic), 2002.

\bibitem[For03]{For_uniqueness}
Max Forester.
\newblock On uniqueness of {JSJ} decompositions of finitely generated groups.
\newblock {\em Comment. Math. Helv.}, 78(4):740--751, 2003.

\bibitem[For06]{For_splittings}
Max Forester.
\newblock Splittings of generalized {B}aumslag-{S}olitar groups.
\newblock {\em Geom. Dedicata}, 121:43--59, 2006.

\bibitem[FP06]{FuPa_JSJ}
K.~Fujiwara and P.~Papasoglu.
\newblock {JSJ}-decompositions of finitely presented groups and complexes of
  groups.
\newblock {\em Geom. Funct. Anal.}, 16(1):70--125, 2006.

\bibitem[GAnMB06]{GoMa_homeotopy}
Francisco~Javier Gonz\'alez-Acu\~na and Juan~Manuel M\'arquez-Bobadilla.
\newblock On the homeotopy group of the non orientable surface of genus three.
\newblock {\em Rev. Colombiana Mat.}, 40(2):75--79, 2006.

\bibitem[GH15]{GrHu_abelian}
Daniel Groves and Michael Hull.
\newblock Abelian splittings of right-angled {A}rtin groups, 2015.
\newblock arXiv:1502.00129.

\bibitem[Gui98]{Gui_approximation}
Vincent Guirardel.
\newblock Approximations of stable actions on {$\mathbb{R}$}-trees.
\newblock {\em Comment. Math. Helv.}, 73(1):89--121, 1998.

\bibitem[Gui00a]{Gui_dynamics}
Vincent Guirardel.
\newblock Dynamics of {${\rm Out}(F\sb n)$} on the boundary of outer space.
\newblock {\em Ann. Sci. \'Ecole Norm. Sup. (4)}, 33(4):433--465, 2000.

\bibitem[Gui00b]{Gui_reading}
Vincent Guirardel.
\newblock Reading small actions of a one-ended hyperbolic group on
  {$\mathbb{R}$}-trees from its {J}{S}{J} splitting.
\newblock {\em Amer. J. Math.}, 122(4):667--688, 2000.

\bibitem[Gui05]{Gui_coeur}
Vincent Guirardel.
\newblock C\oe ur et nombre d'intersection pour les actions de groupes sur les
  arbres.
\newblock {\em Ann. Sci. \'Ecole Norm. Sup. (4)}, 38(6):847--888, 2005.

\bibitem[Gui08]{Gui_actions}
Vincent Guirardel.
\newblock Actions of finitely generated groups on {$\Bbb R$}-trees.
\newblock {\em Ann. Inst. Fourier (Grenoble)}, 58(1):159--211, 2008.

\bibitem[GL07a]{GL2}
Vincent Guirardel and Gilbert Levitt.
\newblock Deformation spaces of trees.
\newblock {\em Groups Geom. Dyn.}, 1(2):135--181, 2007.

\bibitem[GL07b]{GL1}
Vincent Guirardel and Gilbert Levitt.
\newblock The outer space of a free product.
\newblock {\em Proc. Lond. Math. Soc. (3)}, 94(3):695--714, 2007.

\bibitem[GL09]{GL3a}
Vincent Guirardel and Gilbert Levitt.
\newblock {JSJ} decompositions: definitions, existence and uniqueness. {I}: The
  {JSJ} deformation space.
\newblock arXiv:0911.3173 v2 [math.GR], 2009.

\bibitem[GL10]{GL5}
Vincent Guirardel and Gilbert Levitt.
\newblock {S}cott and {S}warup's regular neighbourhood as a tree of cylinders.
\newblock {\em Pacific J. Math.}, 245(1):79--98, 2010.

\bibitem[GL11]{GL4}
Vincent Guirardel and Gilbert Levitt.
\newblock Trees of cylinders and canonical splittings.
\newblock {\em Geom. Topol.}, 15(2):977--1012, 2011.

\bibitem[GL15a]{GL_McCool}
Vincent Guirardel and Gilbert Levitt.
\newblock Mc{C}ool groups of toral relatively hyperbolic groups.
\newblock {\em Algebr. Geom. Topol.}, 15(6):3485--3534, 2015.

\bibitem[GL15b]{GL6}
Vincent Guirardel and Gilbert Levitt.
\newblock Splittings and automorphisms of relatively hyperbolic groups.
\newblock {\em Groups Geom. Dyn.}, 9(2):599--663, 2015.

\bibitem[GL16]{GL_vertex}
Vincent Guirardel and Gilbert Levitt.
\newblock Vertex finiteness for splittings of relatively hyperbolic groups.
\newblock {\em Israel J. Math.}, 212(2):729--755, 2016.

\bibitem[GLS17]{GLS_finiteindex}
Vincent Guirardel, Gilbert Levitt, and Rizos Sklinos.
\newblock Elementary equivalence vs commensurability for hyperbolic groups,
  2017.
\newblock arXiv:1701.08853.

\bibitem[Hru10]{Hruska_quasiconvexity}
G.~Christopher Hruska.
\newblock Relative hyperbolicity and relative quasiconvexity for countable
  groups.
\newblock {\em Algebr. Geom. Topol.}, 10(3):1807--1856, 2010.

\bibitem[JS79]{JaSh_JSJ}
William~H. Jaco and Peter~B. Shalen.
\newblock Seifert fibered spaces in {$3$}-manifolds.
\newblock {\em Mem. Amer. Math. Soc.}, 21(220):viii+192, 1979.

\bibitem[Joh79]{Johannson_JSJ}
Klaus Johannson.
\newblock {\em Homotopy equivalences of {$3$}-manifolds with boundaries},
  volume 761 of {\em Lecture Notes in Mathematics}.
\newblock Springer, Berlin, 1979.

\bibitem[KM05]{KhMy_effective}
Olga Kharlampovich and Alexei~G. Myasnikov.
\newblock Effective {JSJ} decompositions.
\newblock In {\em Groups, languages, algorithms}, volume 378 of {\em Contemp.
  Math.}, pages 87--212. Amer. Math. Soc., Providence, RI, 2005.

\bibitem[Kou98]{Koubi_croissance}
Malik Koubi.
\newblock Croissance uniforme dans les groupes hyperboliques.
\newblock {\em Ann. Inst. Fourier (Grenoble)}, 48(5):1441--1453, 1998.

\bibitem[Kro90]{Kro_JSJ}
P.~H. Kropholler.
\newblock An analogue of the torus decomposition theorem for certain
  {P}oincar\'e duality groups.
\newblock {\em Proc. London Math. Soc. (3)}, 60(3):503--529, 1990.

\bibitem[KR89a]{KroRol_relative}
P.~H. Kropholler and M.~A. Roller.
\newblock Relative ends and duality groups.
\newblock {\em J. Pure Appl. Algebra}, 61(2):197--210, 1989.

\bibitem[KR89b]{KroRol_splittings3}
P.~H. Kropholler and M.~A. Roller.
\newblock Splittings of {P}oincar\'e duality groups. {III}.
\newblock {\em J. London Math. Soc. (2)}, 39(2):271--284, 1989.

\bibitem[Lev05a]{Lev_automorphisms}
Gilbert Levitt.
\newblock Automorphisms of hyperbolic groups and graphs of groups.
\newblock {\em Geom. Dedicata}, 114:49--70, 2005.

\bibitem[Lev05b]{Lev_rigid}
Gilbert Levitt.
\newblock Characterizing rigid simplicial actions on trees.
\newblock In {\em Geometric methods in group theory}, volume 372 of {\em
  Contemp. Math.}, pages 27--33. Amer. Math. Soc., Providence, RI, 2005.

\bibitem[LP97]{LP}
Gilbert Levitt and Fr{\'e}d{\'e}ric Paulin.
\newblock Geometric group actions on trees.
\newblock {\em Amer. J. Math.}, 119(1):83--102, 1997.

\bibitem[Lin83]{Linnell}
P.~A. Linnell.
\newblock On accessibility of groups.
\newblock {\em J. Pure Appl. Algebra}, 30(1):39--46, 1983.

\bibitem[MM96]{McCulloughMiller_symmetric}
Darryl McCullough and Andy Miller.
\newblock Symmetric automorphisms of free products.
\newblock {\em Mem. Amer. Math. Soc.}, 122(582):viii+97, 1996.

\bibitem[MNS99]{MNS_downunder}
C.~F. Miller, III, Walter~D. Neumann, and G.~A. Swarup.
\newblock Some examples of hyperbolic groups.
\newblock In {\em Geometric group theory down under ({C}anberra, 1996)}, pages
  195--202. de Gruyter, Berlin, 1999.

\bibitem[MS84]{MS_valuationsI}
John~W. Morgan and Peter~B. Shalen.
\newblock Valuations, trees, and degenerations of hyperbolic structures. {I}.
\newblock {\em Ann. of Math. (2)}, 120(3):401--476, 1984.

\bibitem[Osi06]{Osin_relatively}
Denis~V. Osin.
\newblock Relatively hyperbolic groups: intrinsic geometry, algebraic
  properties, and algorithmic problems.
\newblock {\em Mem. Amer. Math. Soc.}, 179(843):vi+100, 2006.

\bibitem[PS09]{PaSw_boundaries}
Panos Papasoglu and Eric Swenson.
\newblock Boundaries and {JSJ} decompositions of {CAT}(0)-groups.
\newblock {\em Geom. Funct. Anal.}, 19(2):559--590, 2009.

\bibitem[Pau89]{Pau_Gromov}
Fr{\'e}d{\'e}ric Paulin.
\newblock The {G}romov topology on {$\mathbb {R}$}-trees.
\newblock {\em Topology Appl.}, 32(3):197--221, 1989.

\bibitem[Pau91]{Pau_arboreal}
Fr{\'e}d{\'e}ric Paulin.
\newblock Outer automorphisms of hyperbolic groups and small actions on
  {$\mathbb {R}$}-trees.
\newblock In {\em Arboreal group theory (Berkeley, CA, 1988)}, pages 331--343.
  Springer, New York, 1991.

\bibitem[Pau04]{Pau_theorie}
Fr{\'e}d{\'e}ric Paulin.
\newblock Sur la th\'eorie \'el\'ementaire des groupes libres (d'apr\`es
  {S}ela).
\newblock {\em Ast\'erisque}, 294:ix, 363--402, 2004.

\bibitem[Per11]{Perin_elementary}
Chlo{\'e} Perin.
\newblock Elementary embeddings in torsion-free hyperbolic groups.
\newblock {\em Ann. Sci. \'Ec. Norm. Sup\'er. (4)}, 44(4):631--681, 2011.

\bibitem[RW10]{ReiWei_MR}
Cornelius Reinfeldt and Richard Weidmann.
\newblock {M}akanin-{R}azborov diagrams for hyperbolic groups, 2010.
\newblock preprint
  \url{http://www.math.uni-kiel.de/algebra/de/weidmann/research/}.

\bibitem[RS97]{RiSe_JSJ}
E.~Rips and Z.~Sela.
\newblock Cyclic splittings of finitely presented groups and the canonical
  {J}{S}{J} decomposition.
\newblock {\em Ann. of Math. (2)}, 146(1):53--109, 1997.

\bibitem[Sco83]{Scott_geometries}
Peter Scott.
\newblock The geometries of $3$-manifolds.
\newblock {\em Bull. London Math. Soc.}, 15(5):401--487, 1983.

\bibitem[SS00]{ScSw_splittings}
Peter Scott and Gadde~A. Swarup.
\newblock Splittings of groups and intersection numbers.
\newblock {\em Geom. Topol.}, 4:179--218 (electronic), 2000.

\bibitem[SS03]{ScSw_regular+errata}
Peter Scott and Gadde~A. Swarup.
\newblock Regular neighbourhoods and canonical decompositions for groups.
\newblock {\em Ast\'erisque}, 289:vi+233, 2003.
\newblock Corrections available at
  \url{http://www.math.lsa.umich.edu/~pscott/}.

\bibitem[SW79]{ScottWall}
Peter Scott and Terry Wall.
\newblock Topological methods in group theory.
\newblock In {\em Homological group theory (Proc. Sympos., Durham, 1977)},
  pages 137--203. Cambridge Univ. Press, Cambridge, 1979.

\bibitem[Sel97a]{Sela_acylindrical}
Z.~Sela.
\newblock Acylindrical accessibility for groups.
\newblock {\em Invent. Math.}, 129(3):527--565, 1997.

\bibitem[Sel97b]{Sela_structure}
Z.~Sela.
\newblock Structure and rigidity in ({G}romov) hyperbolic groups and discrete
  groups in rank $1$ {L}ie groups. {I}{I}.
\newblock {\em Geom. Funct. Anal.}, 7(3):561--593, 1997.

\bibitem[Sel99]{Sela_hopf}
Z.~Sela.
\newblock Endomorphisms of hyperbolic groups. {I}. {T}he {H}opf property.
\newblock {\em Topology}, 38(2):301--321, 1999.

\bibitem[Sel01]{Sela_diophantine1}
Zlil Sela.
\newblock Diophantine geometry over groups. {I}. {M}akanin-{R}azborov diagrams.
\newblock {\em Publ. Math. Inst. Hautes \'Etudes Sci.}, 93:31--105, 2001.

\bibitem[Ser77]{Serre_arbres}
Jean-Pierre Serre.
\newblock {\em Arbres, amalgames, ${\rm {S}{L}}\sb{2}$}.
\newblock Soci\'et\'e Math\'ematique de France, Paris, 1977.
\newblock R\'edig\'e avec la collaboration de Hyman Bass, Ast\'erisque, No. 46.

\bibitem[Sha87]{Sh_introduction}
Peter~B. Shalen.
\newblock Dendrology of groups: an introduction.
\newblock In {\em Essays in group theory}, pages 265--319. Springer, New
  York-Berlin, 1987.

\bibitem[Sha91]{Sh_dendrology}
Peter~B. Shalen.
\newblock Dendrology and its applications.
\newblock In {\em Group theory from a geometrical viewpoint (Trieste, 1990)},
  pages 543--616. World Sci. Publishing, River Edge, NJ, 1991.

\bibitem[Sko96]{Skora_splittings}
Richard~K. Skora.
\newblock Splittings of surfaces.
\newblock {\em J. Amer. Math. Soc.}, 9(2):605--616, 1996.

\bibitem[Sta83]{Stallings_topology}
John~R. Stallings.
\newblock Topology of finite graphs.
\newblock {\em Invent. Math.}, 71(3):551--565, 1983.

\bibitem[Thu80]{Thurston_notes}
William~P. Thurston.
\newblock The geometry and topology of three-manifolds.
\newblock Princeton lecture notes, 1980.

\bibitem[Tou09]{Touikan_finding}
Nicholas~W.M. Touikan.
\newblock Detecting geometric splittings in finitely presented groups.
\newblock arXiv:0906.3902v2 [math.GR]. To appear in Transactions of the A.M.S.,
  2009.

\bibitem[Wal04]{Wall_PoincareGT}
C.~T.~C. Wall.
\newblock Poincar\'e duality in dimension 3.
\newblock In {\em Proceedings of the {C}asson {F}est}, volume~7 of {\em Geom.
  Topol. Monogr.}, pages 1--26 (electronic). Geom. Topol. Publ., Coventry,
  2004.

\bibitem[Wei12]{Weidmann_accessibility}
Richard Weidmann.
\newblock On accessibility of finitely generated groups.
\newblock {\em Q. J. Math.}, 63(1):211--225, 2012.

\end{thebibliography}

{\small
\printindex
}

\bigskip

\begin{flushleft}
Vincent Guirardel\\
Institut de Recherche Math\'ematique de Rennes\\
Universit\'e de Rennes 1 et CNRS (UMR 6625)\\
263 avenue du G\'en\'eral Leclerc, CS 74205\\
F-35042  RENNES C\'edex\\
\emph{e-mail: }\texttt{vincent.guirardel@univ-rennes1.fr}\\[8mm]

Gilbert Levitt\\
Laboratoire de Math\'ematiques Nicolas Oresme (LMNO)\\
Universit\'e de Caen et CNRS (UMR 6139)\\
(Pour Shanghai : Normandie Univ, UNICAEN, CNRS, LMNO, 14000 Caen, France)\\
 \emph{e-mail: }\texttt{levitt@unicaen.fr}\\[8mm]

\end{flushleft}

\end{document}